\documentclass[12pt]{article}
\usepackage{amsmath, amssymb, amsthm, physics, url, mathrsfs, cite, xcolor, mathtools, dsfont, esint, graphicx, subcaption}

\usepackage[shortlabels]{enumitem}
\usepackage[title]{appendix}

\allowdisplaybreaks
\mathtoolsset{showonlyrefs}
\usepackage{lineno}

\definecolor{darkblue}{RGB}{32,57,231}
\definecolor{purple}{RGB}{128,0,160}
\usepackage[colorlinks,
	citecolor=darkblue, 
	linkcolor=darkblue, 
	urlcolor=darkblue,
	bookmarks=true,
	bookmarksopen=true,
	pdfauthor={Fabrice Baudoin, Aobo Chen, Li Chen}, 
	pdftitle={Riesz transform on Ahlfors regular metric trees}
]{hyperref}

\numberwithin{equation}{section}
\numberwithin{figure}{section}
\newtheorem{theorem}{Theorem}[section]
\newtheorem{lemma}[theorem]{Lemma}
\newtheorem{proposition}[theorem]{Proposition}
\newtheorem{corollary}[theorem]{Corollary}
\theoremstyle{definition}
\newtheorem{definition}[theorem]{Definition}

\newtheorem{remark}[theorem]{Remark}
\newtheorem{notation}[theorem]{Notation}

\def\dist{{\mathop {\rm dist}}}

\newcommand*{\dif}{\mathop{}\!\mathrm{d}}
\newcommand{\supp}{\operatorname{supp}}
\newcommand{\one}{\mathds{1}}
\newcommand{\vecone}{\vec{\one}}
\def\sA {{\mathcal {A}}} \def\sB {{\mathcal {B}}} \def\sH {{\mathcal {H}}}
\def\sI {{\mathcal {I}}} \def\sJ {{\mathcal {J}}} \def\sL {{\mathcal {L}}}
\def\sM {{\mathcal {M}}} \def\sP {{\mathcal {P}}} \def\sT {{\mathcal {T}}}
\def\bE {{\mathbb {E}}} \def\bN {{\mathbb {N}}} \def\bP {{\mathbb {P}}}
\def\bR {{\mathbb {R}}} \def\bZ {{\mathbb {Z}}}
\def\ol{\overline}
\def\wt{\widetilde}

\newcommand{\set}[1]{\left\{ #1 \right\}}
\newcommand{\Sett}[2]{\left\{ #1  : \, #2 \right\}}
\DeclareMathOperator*{\einf}{\operatorname*{ess\,inf}}

\newcommand{\restr}[2]{{
		\left.\kern-\nulldelimiterspace 
		#1 
		\vphantom{\big|}
		\right|_{#2}
}} 
\newcommand{\ambient}{X} 
\newcommand{\metric}{d} 
\newcommand{\meas}{m} 
\newcommand{\skeleton}{\mathcal{S}} 
\newcommand{\medm}{\nu} 
\newcommand{\Lip}{\operatorname{Lip}}
\newcommand{\Ker}{\operatorname{Ker}}
\newcommand{\Leb}[0]{\mathcal{L}^{1}}

\def\collectD {{\mathfrak {D}}}
\def\collectJ {{\mathfrak {J}}}

\newcommand{\form}{\mathcal{E}} 
\newcommand{\domain}{\mathcal{F}} 
\newcommand{\gen}{\Delta} 
\newcommand{\wgrad}{\partial} 
\newcommand{\abscon}{\mathrm{AC}}
\newcommand{\Dom}{\mathrm{Dom}} 
\newcommand{\proj}{{E}} 
\newcommand{\riesz}{\mathcal{R}} 
\newcommand{\core}[0]{\mathcal{C}}
\newcommand{\sobolev}[1]{W^{1, #1}}
\newcommand{\AR}[1]{\hyperlink{AR}{\mathrm{AR}( #1 )}}
\newcommand{\dimf}[0]{d_{\mathrm{f}}}
\newcommand{\dimw}[0]{d_{\mathrm{w}}}
\newcommand{\critic}[0]{\gamma}

\font\titlefont=cmbx12 scaled 1400
\title{\titlefont Riesz transform on Ahlfors regular metric trees}

\date{\today}
\author{Fabrice Baudoin, Aobo Chen, Li Chen}
\begin{document}
\maketitle
\vspace{-22pt}
\begin{abstract}
We study the boundedness of the Riesz transform and the validity of reverse Riesz inequalities on proper unbounded Ahlfors regular real trees. For the heat-regularized reverse inequality, we determine how the critical order of the fractional Laplacian depends on the dimension of the tree. At this order, we investigate strong and weak \(L^p\) estimates for the reverse inequality and the boundedness of the corresponding Riesz transform between Lebesgue and Lorentz spaces. We give particular attention to the weak type \((1,1)\) estimates for the Riesz transform on trees of linear and superlinear volume growth. In dimension greater than one, we use a Whitney decomposition of the weak gradient to prove the weak reverse inequality and randomized test functions to establish its sharpness.

\end{abstract}

\tableofcontents

\section{Introduction and main results}
On Euclidean space $\bR^{n}$ with Lebesgue measure $\mathcal L^n$, the Riesz transform associated with the standard Laplacian operator $\gen$ is $\nabla(-\gen)^{-\frac{1}{2}}$. The corresponding upper and lower bounds give the following two inequalities:
\begin{enumerate}[label=\textup{(\alph*)}, align=right, leftmargin=*, topsep=5pt, parsep=0pt, itemsep=2pt]
  \item The Riesz inequality: $\|\nabla f\|_{L^{p}(\bR^n,\mathcal L^n)}\leq C\|(-\gen)^{\frac{1}{2}}f\|_{L^{p}(\bR^n,\mathcal L^n)}$;
  \item The reverse Riesz inequality: $\|(-\gen)^{\frac{1}{2}}f\|_{L^{p}(\bR^n,\mathcal L^n)}\leq C\|\nabla f\|_{L^{p}(\bR^n,\mathcal L^n)}$.
\end{enumerate}
On $\bR^{n}$, both inequalities hold for $p\in(1,\infty)$ and $f\in C_{c}^{\infty}(\bR^{n})$ with constants depending only on $p$ and $n$; see \cite{Ste70,Mey84,Pis88} and \cite[Proposition~5.1.14 and Corollary~5.2.8]{Gra14}.

On a complete Riemannian manifold $M$ with Riemannian volume measure $\operatorname{vol}$ and the Laplace--Beltrami operator, both inequalities hold with constant $1$ for $p=2$, since $\|\nabla f\|_{L^{2}(M,\operatorname{vol})}=\|(-\gen)^{\frac{1}{2}}f\|_{L^{2}(M,\operatorname{vol})}$ for every smooth compactly supported function $f$. Recently, R.~Chen, Jiang, Li and Li \cite{CJLL26} proved that the Riesz transform on every complete non-compact Riemannian manifold is of weak type $(1,1)$ and therefore bounded from $L^{p}(M,\operatorname{vol})$ to the space of vector fields whose pointwise length belongs to $L^{p}(M,\operatorname{vol})$ for $p\in(1,2]$, without additional geometric or heat kernel assumptions, thereby resolving the $p\in(1,2]$ part of the Coulhon--Duong conjecture \cite[Conjecture~1.1]{CD03}. For $p\in(2,\infty)$, the validity of the Riesz inequalities depends on the geometry of the manifold; see \cite[Section~1.2]{ACDH04}. 

The exponent $1/2$ in $(-\gen)^{\frac{1}{2}}$ is not universal, especially for Dirichlet spaces with \emph{sub-Gaussian heat kernel estimates}. It is therefore natural to consider fractional powers $(-\gen)^{\critic}$ in place of $(-\gen)^{\frac{1}{2}}$ and weaker Riesz inequalities that do not rely on global Gaussian heat kernel assumptions. In particular, the third author studied \emph{quasi-Riesz transforms} on manifolds with large-scale sub-Gaussian estimates \cite{Che15}. The corresponding inequalities are:
\begin{enumerate}[label=\textup{(\alph*)}, align=right, leftmargin=*, topsep=5pt, parsep=0pt, itemsep=2pt, resume]
  \item The quasi-Riesz inequality: $\|\nabla P_{1}f\|_{L^{p}(M,\operatorname{vol})}\leq C\|(-\gen)^{\critic}f\|_{L^{p}(M,\operatorname{vol})}$;
  \item The reverse quasi-Riesz inequality: $\|(-\gen)^{\critic}P_{1}f\|_{L^{p}(M,\operatorname{vol})}\leq C\|\nabla f\|_{L^{p}(M,\operatorname{vol})}$.
\end{enumerate}
Here, $P_{t}=\exp(t\gen)$ is the heat semigroup. On complete Riemannian manifolds, the quasi-Riesz inequality for $\critic\in(0,1/2)$ and $p\in(1,2]$ always holds \cite[Theorem~1.2 and Section~2.1]{Che15}. Devyver and Russ studied reverse Riesz inequalities on manifolds \cite{DR22}. See also \cite{DRY23} for quasi-Riesz inequalities on fractal-like cable systems, including the Sierpi\'nski gasket cable system.

\paragraph{Setting.}

In this work, we study the validity of these Riesz inequalities with fractional Laplacian $(-\gen)^{\critic}$ on Ahlfors regular real trees. A \emph{real tree} is a geodesic metric space in which any two points $x,y$ are joined by a \emph{unique} arc $[x,y]$ (Definition~\ref{d.rtree0}). By Ahlfors regularity of a metric space $(\ambient,\metric)$, we mean there exists an exponent $\dimf\in[1,\infty)$, a constant $C_{0}\in[1,\infty)$ and a Radon measure $\meas$ such that
\begin{equation}\label{e.introa}
  C_{0}^{-1}r^{\dimf}\leq\meas(B(x,r))\leq C_{0}r^{\dimf},\ \text{ for all }x\in\ambient\text{ and }r\in(0,\infty);
\end{equation}
see Definition \ref{d.ARdf} and Remark \ref{r.ARdf}. See Figure \ref{f.trees} for some examples from \cite{BC23,BT21,Kig95b}.

\begin{figure}[ht]
\centering
    \begin{subfigure}[b]{0.4\textwidth}            
            \includegraphics[width=\textwidth]{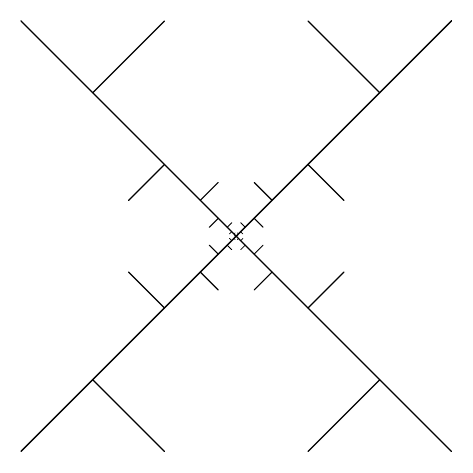}
            \caption{Dyadic comb 
            }
            \label{f.comb}
    \end{subfigure}%
    \begin{subfigure}[b]{0.4\textwidth}
            \centering
            \includegraphics[width=\textwidth]{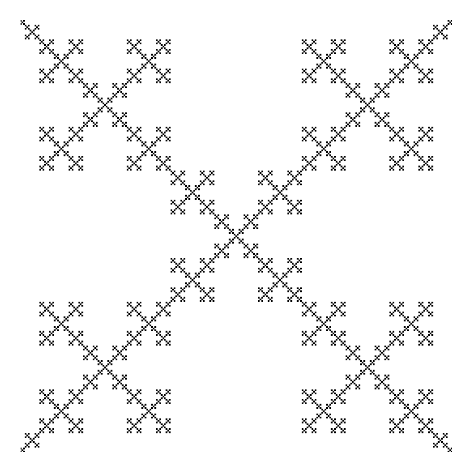}
            \caption{Vicsek tree \cite{BC23}}
            \label{f.vic}
    \end{subfigure}\\
      \begin{subfigure}[b]{0.4\textwidth}
            \centering
            \includegraphics[width=\textwidth]{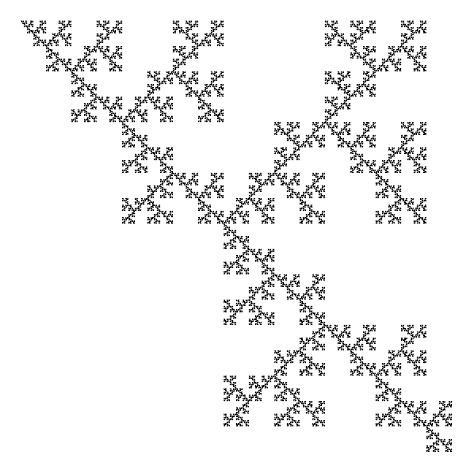}
            \caption{Bonk--Tran tree \cite{BT21}}
            \label{f.bonk-tran}
    \end{subfigure}   
     \begin{subfigure}[b]{0.4\textwidth}
            \centering
            \includegraphics[width=\textwidth]{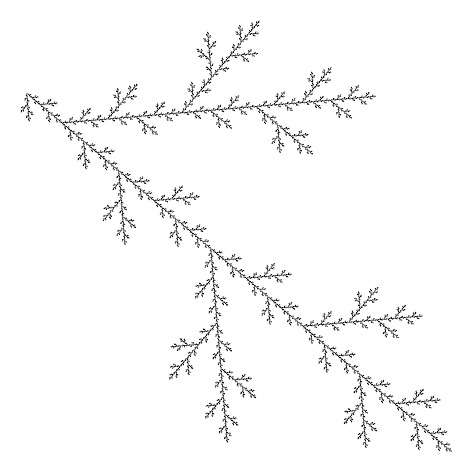}
            \caption{\small Hata tree \cite{Kig95b}}
            \label{f.hata}
    \end{subfigure}  
    \caption{Four Ahlfors regular metric trees}\label{f.trees}
\end{figure}

The union $\skeleton$ of the open arcs of $\ambient$ is the \emph{skeleton} of the tree, and it carries the length measure $\medm$ (Definition~\ref{d.rtree0}). A function that is absolutely continuous along arcs has a weak derivative $\wgrad f$, defined $\medm$-a.e. on $\skeleton$ once an orientation of the arcs is fixed. This \emph{weak gradient} replaces the classical gradient $\nabla$ on Riemannian manifolds and on cable systems; it defines the Sobolev spaces $\sobolev{p}(\ambient)$, $p\in[1,\infty)$, and the canonical Dirichlet form
\begin{equation}
\form(u,v)=\int_{\skeleton}\wgrad u\;\wgrad v\dif\medm,\quad \text{for }u,v\in\domain=\sobolev{2}(\ambient),
\end{equation}
which is a strongly local regular Dirichlet form on $L^{2}(\ambient,\meas)$; see \cite{BCY25} and Sections~\ref{s.geomet} and~\ref{s.analys} for the precise definitions. Let $\gen$ be the generator of this form and let $\{P_{t}\}_{t\in(0,\infty)}$ be its semigroup. By \cite{BCY25}, the heat kernel satisfies two-sided sub-Gaussian estimates with \emph{walk dimension} $\dimw=\dimf+1$; see Theorem \ref{t.dirich}. 

\paragraph{Main results.}

Our first result identifies the critical exponent $\critic_p$ for the reverse
Riesz inequality involving $(-\gen)^{\critic}P_1$ when $1\le p<\infty$.
We set
\begin{equation}\label{e.critp}
\critic_p:=\frac{1}{p}+\left(1-\frac{2}{p}\right)\frac{1}{\dimw},
\qquad p\in[1,\infty].
\end{equation}

We write $L^p(\ambient,\meas)$ and $L^{p,q}(\ambient,\meas)$ for the
Lebesgue and Lorentz spaces on $\ambient$, respectively, and
$L^p(\skeleton,\medm)$ and $L^{p,q}(\skeleton,\medm)$ for the corresponding
spaces on $\skeleton$. See Definition~\ref{d.lorent} for the Lorentz space notation.
\begin{theorem}\label{t.critx}
Let $(\ambient,\metric,\meas)$ be a unbounded proper Ahlfors regular metric tree of dimension $\dimf\in[1,\infty)$. Let $p\in[1,\infty)$.
\begin{enumerate}[label=\textup{(\arabic*)}, align=right, leftmargin=*, topsep=5pt, parsep=0pt, itemsep=2pt]
  \item\label{it.neces} Let $\critic\in(0,1)$. If there exists $C\in(0,\infty)$ such that
\begin{equation}\label{e.assum}
\|(-\gen)^{\critic}P_{1}f\|_{L^{p,\infty}(\ambient,\meas)}\leq C\|\wgrad f\|_{L^{p}(\skeleton,\medm)},\quad \text{for every $f\in\sobolev{p}(\ambient)$,}
\end{equation}
then $\critic\geq\critic_{p}$.
  \item\label{it.suffi} For every $\critic\in(\critic_{p},1)$, there exists $C_{p,\critic}\in(0,\infty)$ such that
\begin{equation}\label{e.subcr}
\|(-\gen)^{\critic}P_{1}f\|_{L^{p}(\ambient,\meas)}\leq C_{p,\critic}\|\wgrad f\|_{L^{p}(\skeleton,\medm)},\quad \text{for every $f\in\sobolev{p}(\ambient)$.}
\end{equation}
\end{enumerate}
\end{theorem}
Thus the strong reverse quasi-Riesz inequality holds for $\critic\in(\critic_{p},1)$ and fails even in weak $L^{p}(\ambient,\meas)$ for $\critic\in(0,\critic_{p})$.

We next study the reverse Riesz inequality for $(-\gen)^{\critic_p}P_1$ at the critical exponent $\critic_p$. The strong and weak bounds depend on whether $\dimf\in(1,\infty)$ or $\dimf=1$. We first consider the case $\dimf\in(1,\infty)$. In this case, the heat kernel satisfies sub-Gaussian bounds with walk dimension $\dimw=\dimf+1>2$, and the reference measure $\meas$ and the length measure $\medm$ are mutually singular (see Lemma~\ref{l.singul}).

\begin{theorem}\label{t.strfl}
  Let $(\ambient,\metric,\meas)$ be a unbounded proper Ahlfors regular metric tree of dimension $\dimf\in(1,\infty)$.
  \begin{enumerate}[label=\textup{(\arabic*)}, align=right, leftmargin=*, topsep=5pt, parsep=0pt, itemsep=2pt]
    \item\label{it.scale} If $p\in[1,2)$, there exists $C\in(0,\infty)$ such that
\begin{equation}\label{e.scalw}
      \sup_{t\in(0,\infty)}\norm{(-\gen)^{\critic_{p}}P_{t}f}_{L^{p,\infty}(\ambient,\meas)} \leq C\|\wgrad f\|_{L^{p}(\skeleton,\medm)},\ \text{for every $f\in \sobolev{p}(\ambient)$}.
\end{equation}
    \item\label{it.strfl} If $p\in[1,2)$, then
\begin{equation}\label{e.strfl}
      \sup\Sett{\|(-\gen)^{\critic_{p}}P_{1}f\|_{L^{p}(\ambient,\meas)}}{f\in \core\text{ and }\|\wgrad f\|_{L^{p}(\skeleton,\medm)}\leq1}=\infty.
\end{equation}
    \item\label{it.p2} If $p=2$, then
\begin{equation}
      \sup_{t\in(0,\infty)}\|(-\gen)^{\frac{1}{2}}P_{t}f\|_{L^{2}(\ambient,\meas)}\leq \|\wgrad f\|_{L^{2}(\skeleton,\medm)},\ \text{ for every $f\in\sobolev{2}(\ambient)$}.
\end{equation}
    \item\label{it.wtrfl} If $p\in(2,\infty)$, then
\begin{equation}\label{e.wtrfl}
      \sup\Sett{\|(-\gen)^{\critic_{p}}P_{1}f\|_{L^{p,\infty}(\ambient,\meas)}}{f\in \core\text{ and }\|\wgrad f\|_{L^{p}(\skeleton,\medm)}\leq1}=\infty.
\end{equation}
  \end{enumerate}
\end{theorem}
The class $\core$, defined in \eqref{e.core}, is contained in $\sobolev{p}(\ambient)$ for every $p\in[1,\infty)$, and in $\Dom((-\gen)^{\critic_{p}})$ for $p\in[2,\infty)$, since $\critic_{p}\leq1/2$. For $p\in[1,2)$, the inclusion $\core\subset\Dom((-\gen)^{\critic_{p}})$ holds exactly when $p\in(2-2/\dimf,2)$; for $p\in[1,2-2/\dimf]$, $\core\cap\Dom((-\gen)^{\critic_{p}})=\{0\}$ (Proposition~\ref{p.thresh}). The heat regularization in Theorem~\ref{t.strfl} is what allows the classification to be stated throughout the range $p\in[1,\infty)$.

The uniform estimates in Theorem~\ref{t.strfl}-\ref{it.scale} and~\ref{it.p2}  give the reverse Riesz inequalities when $t\downarrow0$.
\begin{theorem}\label{t.mainw}
  Let $(\ambient,\metric,\meas)$ be a unbounded proper Ahlfors regular metric tree of dimension $\dimf\in(1,\infty)$. Let $p\in[1,2]$ and $f\in \sobolev{p}(\ambient)$. There exists $G_{f}\in L^{p,\infty}(\ambient,\meas)$ such that
\begin{equation}\label{e.homog}
    \lim_{t\downarrow0} \|(-\gen)^{\critic_{p}}P_{t}f-G_{f}\|_{L^{p,\infty}(\ambient,\meas)}=0,
\end{equation}
  and
\begin{equation}\label{e.mainw}
    \|G_{f}\|_{L^{p,\infty}(\ambient,\meas)} \leq C_{p}\|\wgrad f\|_{L^{p}(\skeleton,\medm)}.
\end{equation}
  If, in addition, $f\in\Dom((-\gen)^{\critic_{p}})$, then $G_{f}=(-\gen)^{\critic_{p}}f$ $\meas$-a.e. on $\ambient$.
\end{theorem}

For each $p\in[1,\infty)$, we initially define the Riesz transform
\begin{equation}\label{e.rieszp}
  \riesz_{p}:\collectD_{p}^{\rm sp}\to L^{2}(\skeleton,\medm),\qquad u\mapsto\wgrad(-\gen)^{-\critic_{p}}u\,,
\end{equation}
where the space $\collectD_{p}^{\rm sp}\subset\Dom(( -\gen)^{-\critic_{p}})$ is constructed in~\eqref{e.specc-strict}; see Lemma~\ref{l.densp} for its properties.

\begin{theorem}\label{t.dualc}
  Let $(\ambient,\metric,\meas)$ be a unbounded proper Ahlfors regular metric tree of dimension $\dimf\in(1,\infty)$. Let $p\in[1,\infty)$, and let $\riesz_{p}$ be the operator~\eqref{e.rieszp}, initially defined on $\collectD_{p}^{\rm sp}$ (see \eqref{e.specc-strict}). The following properties hold:
  \begin{enumerate}[label=\textup{(\arabic*)}, align=right, leftmargin=*, topsep=5pt, parsep=0pt, itemsep=2pt]
    \item\label{it.dlow1} If $p=1$, then
\begin{equation}\label{e.dfail1}
      \sup\Sett{\norm{\riesz_{1} f}_{L^{1,\infty}(\skeleton,\medm)}}{f\in \collectD_{1}^{\rm sp} \text{ and }\norm{f}_{L^{1}(\ambient,\meas)}\leq1}=\infty.
\end{equation}
    In particular, there is no bounded extension $\riesz_{1}:L^{1}(\ambient,\meas)\to L^{1,\infty}(\skeleton,\medm)$.
    \item\label{it.dlowb} If $p\in(1,2)$, then $\riesz_{p}$ extends uniquely to a bounded linear operator $\riesz_{p}:L^{p}(\ambient,\meas)\to L^{p,\infty}(\skeleton,\medm)$.
    \item\label{it.dlowr} If $p\in[1,2)$, then there is no bounded extension $\riesz_{p}:L^{p,1}(\ambient,\meas)\to L^{p}(\skeleton,\medm)$.
    \item\label{it.dequa} If $p=2$, then $\riesz_{2}$ extends uniquely to an isometry, that is,
\begin{equation}\label{e.dequa}
      \norm{\riesz_{2}h}_{L^{2}(\skeleton,\medm)} = \norm{h}_{L^{2}(\ambient,\meas)}, \ \ \text{ for all } h\in L^{2}(\ambient,\meas).
\end{equation}
    \item\label{it.dhigh} If $p\in(2,\infty)$, then $\riesz_{p}$ extends uniquely to a bounded linear operator
\begin{equation}\label{e.dhigh}
      \riesz_{p}: L^{p,1}(\ambient,\meas) \to L^{p}(\skeleton,\medm),
\end{equation}
    but there is no bounded extension $\riesz_{p}: L^{p}(\ambient,\meas) \to L^{p}(\skeleton,\medm)$.
  \end{enumerate}
\end{theorem}

When $\dimf=1$, the measures $\meas$ and $\medm$ are comparable by Lemma~\ref{l.singul}-\ref{it.df=1}. In this case, $\dimw=2$, $\critic_p=\frac{1}{2}$ for every $p\in[1,\infty]$, and the heat kernel associated with $(\form,\domain)$ satisfies Gaussian bounds. At the critical exponent, both the Riesz and reverse Riesz inequalities hold in strong $L^p$ for every $p\in(1,\infty)$, and the Riesz transform is of weak type $(1,1)$. Moreover, every $f\in\core$ belongs to the spectral domain of $(-\gen)^{\frac{1}{2}}$, since $\core\subset\domain=\Dom((-\gen)^{\frac{1}{2}})$.

\begin{theorem}\label{t.one}
  Let $(\ambient,\metric,\meas)$ be a unbounded proper Ahlfors regular metric tree of dimension $\dimf=1$, and let $\riesz_{p}:=\wgrad(-\gen)^{-\frac{1}{2}}$ be the operator \eqref{e.rieszp}, initially defined on $\collectD_{p}^{\rm sp}$.
  \begin{enumerate}[label=\textup{(\arabic*)}, align=right, leftmargin=*, topsep=5pt, parsep=0pt, itemsep=2pt]
    \item\label{it.onestr} Let $p\in(1,\infty)$. The operator $\riesz_{p}$ extends uniquely to a bounded linear operator from $L^{p}(\ambient,\meas)$ to $L^{p}(\skeleton,\medm)$. For every $f\in\sobolev{p}(\ambient)$, the limit $G_{f}:=\lim_{t\downarrow0}(-\gen)^{\frac{1}{2}}P_{t}f$ exists in $L^{p}(\ambient,\meas)$, it equals $(-\gen)^{\frac{1}{2}}f$ if $f\in\Dom((-\gen)^{\frac{1}{2}})$, and
    \begin{equation}\label{e.onestr}
      C_{p}^{-1}\|\wgrad f\|_{L^{p}(\skeleton,\medm)}\leq\|G_{f}\|_{L^{p}(\ambient,\meas)}\leq C_{p}\|\wgrad f\|_{L^{p}(\skeleton,\medm)}.
    \end{equation}
    \item\label{it.onewea} The restriction of $\riesz_{2}$ to $L^{1}(\ambient,\meas)\cap L^{2}(\ambient,\meas)$ extends uniquely to a bounded linear operator from $L^{1}(\ambient,\meas)$ to $L^{1,\infty}(\skeleton,\medm)$, which coincides with $\riesz_{1}$ on $\collectD_{1}^{\rm sp}$; in particular,
    \begin{equation}\label{e.onefor}
      \|\riesz_{1}h\|_{L^{1,\infty}(\skeleton,\medm)}\leq C\|h\|_{L^{1}(\ambient,\meas)},\ \text{ for all }h\in\collectD_{1}^{\rm sp}.
    \end{equation}
    Moreover, for every $f\in\sobolev{1}(\ambient)$, $(-\gen)^{\frac{1}{2}}P_{t}f$ converges in $L^{1,\infty}(\ambient,\meas)$ as $t\downarrow0$, and
    \begin{equation}\label{e.onerev}
      \sup_{t\in(0,\infty)}\|(-\gen)^{\frac{1}{2}}P_{t}f\|_{L^{1,\infty}(\ambient,\meas)}\leq C\|\wgrad f\|_{L^{1}(\skeleton,\medm)}.
    \end{equation}
    \item\label{it.onefai} We have
    \begin{equation}\label{e.onefai}
      \begin{gathered}[c]
      \sup\Sett{\|\riesz_{1}h\|_{L^{1}(\skeleton,\medm)}}{h\in\collectD_{1}^{\rm sp}\text{ and }\|h\|_{L^{1}(\ambient,\meas)}\leq1}=\infty,\\
      \sup\Sett{\|(-\gen)^{\frac{1}{2}}P_{1}f\|_{L^{1}(\ambient,\meas)}}{f\in\core\text{ and }\|\wgrad f\|_{L^{1}(\skeleton,\medm)}\leq1}=\infty.
      \end{gathered}
    \end{equation}
  \end{enumerate}
\end{theorem}

\paragraph{Outline of the proofs.}
We use the weak derivatives and Sobolev spaces on trees from \cite{BCY25,BCC26}. The two-sided heat kernel bounds \eqref{e.HKE} come from \cite[Theorem~5.1]{BCY25}, and the estimate \eqref{e.ddt} for the time derivative of the heat kernel comes from \cite[Lemma~3.7]{BCC26}. 

\begin{enumerate}[label=\textup{({\roman*})}, align=right, leftmargin=*, topsep=5pt, parsep=0pt, itemsep=2pt]
\item To prove the weak reverse inequality \eqref{e.scalw} for $1\le p<2$ in Theorem \ref{t.strfl}-\ref{it.scale}, we first estimate $B_tg$, where $B_t$ is defined in \eqref{e.Bt} and $g$ is a finite linear combination of oriented arc indicators. Let $\Omega_\lambda$ be the set where the maximal function of $|g|^p\dif\medm$ relative to $\meas$ exceeds $\lambda^p$. Since $\dimf>1$, this set contains the support of $|g|\dif\medm$. We split $g$ over the \emph{Whitney balls} as in \eqref{e.weak-pieces}. For each $g_i$, we take an arc primitive $F_i$ and multiply it by a cutoff function $\chi_i$. The identity $g_i=\wgrad(\chi_i F_i)-F_i\wgrad\chi_i$ gives a Calder\'{o}n--Zygmund decomposition \eqref{e.weak-split} into a good part and a bad part. The heat kernel estimates give the $L^1$ estimate \eqref{e.bad} for the bad part, and Lemma~\ref{l.synthe} gives the $L^q$ estimate \eqref{e.good} for the good part, with $q$ chosen in \eqref{e.weak-q}. Chebyshev's inequality and the estimate \eqref{e.Omega} for $\meas(\Omega_\lambda)$ now give the weak estimate for $B_tg$. Taking $g=\wgrad f$ for $f\in\core$ and using density of $\core$ in $\sobolev{p}(\ambient)$ gives \eqref{e.scalw}. The case $p=2$ in Theorem \ref{t.strfl}-\ref{it.p2} follows from the spectral theorem.

For the failures \eqref{e.strfl} in Theorem \ref{t.strfl}-\ref{it.strfl} and \eqref{e.wtrfl} in Theorem \ref{t.strfl}-\ref{it.wtrfl}, we use the tent functions $\psi_I$ defined in Lemma~\ref{l.tent}. Their gradients are supported on the arcs $I$, while \eqref{e.tent} gives a lower bound for $(-\gen)^{\critic_p}P_1\psi_I$ on balls $E_I$ disjoint from $I$. For $1\le p<2$, we subdivide one arc at $N$ scales so that the balls $E_I$ are pairwise disjoint. We then take random linear combinations. If the strong estimate held, comparing \eqref{e.khintc} and \eqref{e.khint1} would give $N\lesssim N^{p/2}$ for every $N$. For $p>2$, we choose pairwise disjoint arcs at $N$ scales. The gradient estimate \eqref{e.wtrflg} and the second moment estimate \eqref{e.weak-second-moment} give a choice of signs for which the ratio of the weak $L^p$ norm to the gradient norm is at least $cN^{1/2-1/p}$.

\item For Theorem~\ref{t.mainw}, Corollary~\ref{c.coreco} gives convergence of $(-\gen)^{\critic_p}P_t h$ in $L^p(\ambient,\meas)$ as $t\downarrow0$ for every $h\in\core$. For $1\le p<2$, the density of $\core$ in $\sobolev{p}(\ambient)$ and the uniform estimate \eqref{e.scalw} shows that $(-\gen)^{\critic_p}P_tf$ is a Cauchy sequence in $L^{p,\infty}(\ambient,\meas)$ for every $f\in\sobolev{p}(\ambient)$. Completeness of $L^{p,\infty}(\ambient,\meas)$ gives the limit in \eqref{e.homog}. If $f\in\Dom((-\gen)^{\critic_p})$, convergence in $L^2(\ambient,\meas)$ identifies this limit with $(-\gen)^{\critic_p}f$. For $p=2$, the spectral theorem gives strong $L^2$ convergence for every $f\in\sobolev{2}(\ambient)$.

\item For Theorem~\ref{t.dualc}, we adapt the duality argument in \cite[p.~1152]{CD99}. Let $1<p<\infty$ and $p'=p/(p-1)$. For $h$ in the function space $\collectD_p^{\rm sp}$ of \eqref{e.specc-strict} and $f\in\core$, the identity \eqref{e.regdu} pairs $\riesz_p P_t h$ with $\wgrad f$ and $(-\gen)^{\critic_{p'}}P_t f$ with $h$. 

For $p>2$, we apply the weak reverse bound \eqref{e.scalw} at $p'$ and H\"older's inequality for Lorentz spaces, then let $t\downarrow0$ to obtain \eqref{e.dhigh}. 

For $1<p<2$, we approximate bounded sign functions on finite unions of arcs by gradients of functions in $\core$. The  Sobolev inequality \eqref{e.restri} at $p'>2$ then gives the weak bound \eqref{e.direct}. By duality, a bound from $L^{p,1}(\ambient,\meas)$ to $L^p(\skeleton,\medm)$ for $1<p<2$ would contradict \eqref{e.wtrfl} at $p'$. A strong $(p,p)$ bound for $p>2$ would contradict \eqref{e.strfl} at $p'$.

At $p=1$, we use length measure on a unit arc to construct the test functions in \eqref{e.p1fail}. Their $L^1$ norms stay bounded, while the norms of their Riesz transforms blow up. At $p=2$, the spectral theorem gives the isometry \eqref{e.dequa}.

\item To prove the necessary condition in Theorem~\ref{t.critx}, we test \eqref{e.assum} on the cutoff function $\chi_{o,r}$. Estimate in \eqref{e.critheatdiff} gives the lower bound \eqref{e.spect2}. For every $q>p$, \eqref{e.critcutq} and H\"older's inequality give an upper bound of $C_{p,q,\critic}r^{\dimf-(\dimf+p-1)/q}$. Letting $r\to\infty$ and then $q\downarrow p$ yields $\critic\ge\critic_p$. For sufficiency, we apply the gradient estimate \eqref{e.gradie} at the conjugate exponent $p'$ in the integral formula \eqref{e.unifie}. Since $\critic_p+\critic_{p'}=1$, the integral $\int_0^\infty s^{-\critic}(1+s)^{-\critic_{p'}}\dif s$ in \eqref{e.unifi1} is finite for $\critic_p<\critic<1$. Estimate \eqref{e.subcr} follows from the duality.

\item For Theorem~\ref{t.one}, write $\dif\meas=w\dif\medm$, where $w$ is bounded above and below by positive constants. By applying \cite[Theorem~5.1]{CJLL26} to $w^{-1/2}\wgrad$, we obtain weak type $(1,1)$ for the Riesz transform and strong type $(p,p)$ for $1<p\le2$. For $p>2$, the Gaussian heat kernel estimates allow us to apply \cite[Theorem~1.3-(ii)]{BCC26}. The identity \eqref{e.onepai} then gives the strong reverse inequality by duality. 

To prove the weak reverse estimate \eqref{e.onerev}, we extend $f$ from the complement of a superlevel set of the maximal function of $|\wgrad f|\dif\medm$ to a Lipschitz function $g$ on $\ambient$. We use \eqref{e.onecz1} for $g$ and \eqref{e.oneczt} for $f-g$ outside that set. To obtain the two failures in \eqref{e.onefai}, we estimate the difference of two values of $(-\gen)^{-1/2}h_L$ along a geodesic ray. For the reverse inequality, we take $f_R$ equal to one on a long segment of the ray and integrate a lower bound of order $1/s$ for $|(-\gen)^{1/2}P_1f_R|$. These give the two logarithmic lower bounds in \eqref{e.onefa1}.
\end{enumerate}

The paper is organized as follows. Section~\ref{s.geomet} recalls weak derivatives and Sobolev spaces on real trees and establishes the existence of cutoff functions and the density of core $\core$ in $\sobolev{p}(\ambient)$. In Section~\ref{s.analys}, we study the heat semigroup and fractional powers of the Laplacian and prove Theorem~\ref{t.critx}. Sections~\ref{ss.revers} and~\ref{s.failure} treat the case $\dimf>1$. In Section~\ref{ss.revers}, we prove the uniform weak reverse inequality \eqref{e.scalw} and the convergence in \eqref{e.homog} as $t\downarrow0$. In Section~\ref{s.failure}, we complete the proof of Theorem~\ref{t.strfl} and establish the Riesz transform bounds and their sharpness in Theorem~\ref{t.dualc}. Section~\ref{s.dimens} treats the case $\dimf=1$ and proves Theorem~\ref{t.one}. Appendix~\ref{s.facts} collects facts about Lorentz spaces and projections on real trees, together with calculations of some integrals used in the proofs.

\begin{notation}\label{n.conven}
We use the following notation and conventions.
\begin{enumerate}[label=\textup{(\roman*)},align=right,leftmargin=*,topsep=5pt,parsep=0pt,itemsep=2pt]
  \item $\bN:=\{1,2,\ldots\}$. Thus $0\notin \bN$.
  \item Let $U$ and $V$ be two open subsets of a topological space. If $U$ is precompact and the closure of $U$ is contained in $V$, then we write $U\Subset V$.
  \item Let $X$ be a non-empty set. We define $\one_{A}\in\bR^{X}$ for $A\subset X$ by
\begin{equation}
    \one_{A}(x):=
    \begin{dcases}
      1 & \mbox{if $x \in A$,}\\
      0 & \mbox{if $x \notin A$.}
    \end{dcases}
\end{equation}
  \item Let $X$ be a topological space. We set
\begin{align}
      C(X)&:=\Sett{f\in\bR^{X}}{\textrm{$f$ is a continuous real-valued function on $X$}},\\
      C_{c}(X)&:=\Sett{f\in C(X)}{\textrm{$X\setminus f^{-1}(0)$ has compact closure in $X$}},\\
      C_{0}(X)&:=\Sett{f\in C(X)}{\textrm{$\{\abs{f}\geq\epsilon\}$ is compact for every $\epsilon\in(0,\infty)$}}.
\end{align}
  For any $f\in C(X)$, we define $\norm{f}_{\sup}:=\sup_{x\in X}\abs{f(x)}$.
  \item Let $X$ be a topological space equipped with a Borel measure $\mu$, and let $f$ be a Borel function. We denote by $\supp_{\mu}[f]$ the support of the measure $\abs{f}\dif\mu$.
  \item Let $(X,d)$ be a metric space. We define the metric ball by
\begin{equation}
    B(x,r):=\Sett{y\in X}{d(x,y)<r},\ \text{ for all }(x,r)\in X\times[0,\infty).
\end{equation}
  Under this convention, we have $B(x,0)=\emptyset$ for $x\in X$.
  \item We use the letters $C$, $c$, etc. to denote positive constants whose values are inessential and may change at each occurrence.
  \item For two extended real numbers $A,B\in\bR\cup\{-\infty,\infty\}$, let $A\wedge B:=\min(A,B)$ and $A\vee B:=\max(A,B)$.
  \item For real-valued quantities $f$ and $g$, if there exists an implicit constant $C\in[1,\infty)$ that depends on inessential parameters such that $f\leq Cg$, then we write $f\lesssim g$.
\end{enumerate}
\end{notation}
\section{Geometry and Sobolev spaces on trees}\label{s.geomet}

In this section, we introduce Sobolev spaces on real trees. We then construct cutoff functions and use them to prove the density of the core $\core$, which we need for Theorem~\ref{t.mainw}. Finally, we study the geodesic rays and branches used in the examples in Section~\ref{s.failure}.

\subsection{Arcs, length measure and weak derivatives}\label{ss.result}

In this subsection, we recall the definitions of \emph{real trees} and \emph{weak derivatives} that we use throughout the paper. We then define the Sobolev spaces $\sobolev{p}(\ambient)$ and prove the fundamental theorem of calculus along an arc in \eqref{e.arcftc}.

\begin{definition}\label{d.rtree0}
  A metric space $(\ambient,\metric)$ is called a \emph{real tree} if it satisfies the following two properties:
  \begin{enumerate}[label=\textup{(\roman*)},align=right,leftmargin=*,topsep=5pt,parsep=0pt,itemsep=2pt]
    \item for every $u,v\in\ambient$, there exists a unique isometric embedding $ \phi_{u,v}\colon [0,\metric(u,v)]\to \ambient$ such that $\phi_{u,v}(0)=u$ and $\phi_{u,v}(\metric(u,v))=v$;
    \item for every injective continuous map $\kappa\colon[0,1]\to\ambient$, one has
    \begin{equation}
      \kappa ( [ 0 , 1 ] ) = \phi _{\kappa(0),\kappa(1)} \bigl( [ 0 , \metric ( \kappa ( 0 ) , \kappa ( 1 ) ) ] \bigr) .
    \end{equation}
  \end{enumerate}
  Let $(\ambient,\metric)$ be a real tree.
  \begin{enumerate}[label=\textup{(\arabic*)},align=right,leftmargin=*,topsep=5pt,parsep=0pt,itemsep=2pt]
    \item For every $x,y\in\ambient$, we write
    \begin{equation}
      [ x , y ] : = \phi _{x,y} \bigl( [ 0 , \metric ( x , y ) ] \bigr) , \text{ and } ( x , y ) : = \phi _{x,y} \bigl( ( 0 , \metric ( x , y ) ) \bigr) .
    \end{equation}
    The set $(x,y)$ is called an \emph{open arc}.
    \item The \emph{skeleton} of $\ambient$ is defined as $\skeleton:=\bigcup_{x,y\in\ambient}(x,y)$, the union of all open arcs.
    \item There exists a unique measure $\medm$ on the $\sigma$-field generated by these open arcs such that
    \begin{equation}
      \text{ $\medm((x,y)):=\metric(x,y)$ whenever $x,y\in\skeleton$.}
    \end{equation}
    The measure $\medm$ is called the \emph{length measure} on $\skeleton$. We extend $\medm$ by $0$ on $\ambient\setminus\skeleton$. Thus $\medm(E)=\medm(E\cap\skeleton)$ for every Borel set $E\subset\ambient$.
  \end{enumerate}
\end{definition}

\begin{remark}\label{r.arcs}
	As in \cite[Remark 2.3-(1)]{BCC26}, if $(\ambient,\metric)$ is a separable real tree, then its skeleton $\skeleton$ is a countable union of arcs.
\end{remark}

\begin{definition}\label{d.rtree}
  Let $(\ambient,\metric)$ be a separable real tree with skeleton $\skeleton$ and length measure $\medm$.
  \begin{enumerate}[label=\textup{(\arabic*)},align=left,leftmargin=*,topsep=5pt,parsep=0pt,itemsep=2pt]
    \item A \emph{local arc chart} is an isometric embedding $\gamma:[0,L]\to\ambient$ with $L\in(0,\infty)$.
    \item We define the collection of \emph{absolutely continuous} functions by
    \begin{equation}\label{e.acdef}
      \abscon(\ambient,\metric):=\Biggl\{u\in C(\ambient)\Biggm|
      \begin{minipage}{270pt}
        for every local arc chart $\gamma:[0,L]\to\ambient$, the function $(u\circ \gamma):[0,L]\to\bR$ is absolutely continuous on $[0,L]$
      \end{minipage}
      \Biggr\}.
    \end{equation}
    \item We say that $\mathfrak{O}=(\gamma_{n})_{n\in\bN}$ is an \emph{orientation} of $(\ambient,\metric)$ if, for each $n\in\bN$, $\gamma_{n}:[0,L_{n}]\to\ambient$ is a local arc chart and $\skeleton=\bigcup_{n\in\bN}\gamma_{n}((0,L_{n}))$.\label{it.orient}
    \item For every orientation $\mathfrak{O}=(\gamma_{n})_{n\in\bN}$ of $(\ambient,\metric)$ and every $u\in\abscon(\ambient,\metric)$, we define the \emph{weak derivative of $u$ with respect to $\mathfrak{O}$} by
    \begin{equation}\label{e.gradde}
      \wgrad_{\mathfrak{O}}u(z):=\sum_{n\in\bN}\one_{A_{n}}(z)(u\circ \gamma_{n})^{\prime}(\gamma_{n}^{-1}(z)),\ \medm\text{-a.e. }z\in\skeleton,
    \end{equation}
    where
    \begin{equation}\label{e.dfan}
      A_{1}:=\gamma_{1}((0,L_{1})),\ A_{n}:=\gamma_{n}((0,L_{n}))\setminus\left(\bigcup_{j=1}^{n-1}\gamma_{j}((0,L_{j}))\right).
    \end{equation}
  \end{enumerate}
\end{definition}
We now consider $(\ambient,\metric,\meas)$, a proper unbounded real metric tree equipped with a Radon measure $\meas$ with full support. For two orientations, their derivatives differ by a measurable sign, independent of the function \cite[Proposition 2.6]{BCC26}. From now on, we fix an orientation $\mathfrak O$ and write $\wgrad:=\wgrad_{\mathfrak O}$.
For $p\in[1,\infty)$, define
\begin{equation}\label{e.sobdef}
  \sobolev{p}(\ambient):=\Biggl\{u\in L^{p}(\ambient,\meas)\;\Biggm|\;
  \begin{minipage}{190pt}
  	there exists $\wt{u}\in \abscon(\ambient,\metric)$ such that $\wt{u}=u$ $\meas$-a.e. on $\ambient$ and $\wgrad \wt{u}\in L^{p}(\skeleton,\medm)$
  \end{minipage}
\Biggr\}
\end{equation}
For every $u\in  \sobolev{p}(\ambient)$,  we always represent $u$ by its continuous version $\wt{u}\in \abscon(\ambient,\metric)$. In this sense, we view $\sobolev{p}(\ambient)$ as a subset of $\abscon(\ambient,\metric)$. We define \begin{equation}
	\norm{u}_{\sobolev{p}(\ambient)}:=\|u\|_{L^{p}(\ambient,\meas)}+\|\wgrad u\|_{L^{p}(\skeleton,\medm)}, \ \text{ for every $u\in  \sobolev{p}(\ambient)$.}
\end{equation}  

For $a,b\in\ambient$ with $a\neq b$, the function
$\metric(a,\cdot)$ is $1$-Lipschitz and therefore belongs to
$\abscon(\ambient,\metric)$. We define
\begin{equation}\label{e.arcind}
  \vecone_{(a,b)}:=\one_{(a,b)}\,\wgrad\metric(a,\cdot),
  \qquad \medm\text{-a.e. on }\skeleton,
\end{equation}
and set $\vecone_{(a,a)}:=0$ for every $a\in\ambient$.

\begin{lemma}\label{l.arcid}
Let $\{\vecone_{(a,b)}\}_{(a,b)\in\ambient\times\ambient}$ be the collection of functions defined in \eqref{e.arcind}.
\begin{enumerate}[label=\textup{({\arabic*})}, align=right, leftmargin=*, topsep=5pt, parsep=0pt, itemsep=2pt]
  \item\label{it.arci1}
  For every $a,b\in\ambient$ with $a\neq b$,
  \begin{equation}\label{e.arcrev}
    \vecone_{(b,a)}=-\vecone_{(a,b)},
      \ \text{ and }\
    |\vecone_{(a,b)}|=\one_{(a,b)},
    \qquad \medm\text{-a.e. on }\skeleton.
  \end{equation}

  \item\label{it.arci2}
  For every $u\in\abscon(\ambient,\metric)$ and every
  $a,b\in\ambient$ with $a\neq b$,
  \begin{equation}\label{e.arcder}
    \vecone_{(a,b)}\,\wgrad u
    =\one_{(a,b)}(u\circ\phi_{a,b})'\circ\phi_{a,b}^{-1},
    \qquad \medm\text{-a.e. on }\skeleton,
  \end{equation}
  where the right-hand side is defined on $(a,b)$ and extended
  by zero to $\skeleton\setminus(a,b)$. In particular,
  \begin{equation}\label{e.arcftc}
    \int_{\skeleton}\vecone_{(a,b)}\,\wgrad u\dif\medm
    =u(b)-u(a).
  \end{equation}
  Identity \eqref{e.arcftc} also holds when $a=b$.

  \item\label{it.arci3}
  For every $o,x,y\in\ambient$,
  \begin{equation}\label{e.arcadd}
    \vecone_{(o,y)}=\vecone_{(o,x)}+\vecone_{(x,y)},
    \qquad \medm\text{-a.e. on }\skeleton.
  \end{equation}
\end{enumerate}
\end{lemma}

\begin{proof}
Fix $a,b\in\ambient$ with $a\neq b$, and put
$L:=\metric(a,b)$ and $\phi:=\phi_{a,b}$.
Let $(\gamma_n)_{n\in\bN}$ be the fixed orientation, and let $(A_n)_{n\in\bN}$ be
the partition in \eqref{e.dfan}. For each $n\in\bN$, define $I_n:=\phi^{-1}\bigl(\gamma_n([0,L_n])\bigr)$. Since both $[a,b]$ and $\gamma_n([0,L_n])$ are compact, $I_n$ is either empty, a singleton, or a compact interval. For each $n\in\bN$, if $\sL^{1}(I_n)>0$, then the map $\gamma_n^{-1}\circ\phi$ is an isometry between two real intervals. Consequently, there exist $c_n\in\bR$ and $\varepsilon_n\in\{-1,1\}$ such that
\begin{equation}
  \gamma_n^{-1}(\phi(t))=c_n+\varepsilon_n t,
  \  \text{ for all } t\in I_n.
\end{equation}
For each fixed $u\in\abscon(\ambient,\metric)$, by the chain rule and \eqref{e.gradde},
\begin{equation}\label{e.arcsig}
  \begin{aligned}
    \wgrad u(\phi(t))
    \overset{\eqref{e.gradde}}{=}
      (u\circ\gamma_n)'(c_n+\varepsilon_n t)=\varepsilon_n(u\circ\phi)'(t),\ \text{ for $\sL^{1}$-a.e. $t\in I_n^{\circ}\cap \phi^{-1}(A_{n})$.}
  \end{aligned}
\end{equation}

\begin{enumerate}[label=\textup{({\arabic*})}, align=right, leftmargin=*, topsep=5pt, parsep=0pt, itemsep=2pt]
  \item[\ref{it.arci1}]
  Since $ \metric(a,\phi(t))=t$ and $\metric(b,\phi(t))=L-t$ whenever $t\in[0,L]$, we have
  \begin{equation}
    \wgrad\metric(a,\cdot)(\phi(t))
    \overset{\eqref{e.arcsig}}{=}\varepsilon_n,
    \ \text{ and }\
    \wgrad\metric(b,\cdot)(\phi(t))
    \overset{\eqref{e.arcsig}}{=}-\varepsilon_n,
  \end{equation}
  which, combined with the definition in \eqref{e.arcind}, gives \eqref{e.arcrev}.

  \item[\ref{it.arci2}]
  For each fixed $u\in\abscon(\ambient,\metric)$,
  \begin{equation}
      \vecone_{(a,b)}(\phi(t))\,\wgrad u(\phi(t))
      \overset{\eqref{e.arcind},\eqref{e.arcsig}}{=}
        \varepsilon_n^2(u\circ\phi)'(t)=(u\circ\phi)'(t),\text{ for $\sL^{1}$-a.e. $t\in (0,L)$}.
  \end{equation}
  This gives \eqref{e.arcder}. Since $u\circ\phi$ is absolutely continuous,
  \begin{equation}
    \int_{\skeleton}
      |\vecone_{(a,b)}\,\wgrad u|\dif\medm
    \overset{\eqref{e.arcder}}{=}
    \int_0^L |(u\circ\phi)'(t)|\dif t<\infty.
  \end{equation}
  The fundamental theorem of calculus now gives
  \begin{equation}
      \int_{\skeleton}\vecone_{(a,b)}\,\wgrad u\dif\medm
      \overset{\eqref{e.arcder}}{=}
        \int_0^L (u\circ\phi)'(t)\dif t=u(b)-u(a).
  \end{equation}
  When $a=b$, both sides of \eqref{e.arcftc} are zero.

  \item[\ref{it.arci3}]
  We first prove subdivision along an arc. Let $a,b\in\ambient$
  and $v\in[a,b]$. If $v\in\{a,b\}$, the convention
  $\vecone_{(v,v)}=0$ gives the subdivision identity directly.
  Otherwise, 
  \begin{equation}
    \metric(a,\phi_{v,b}(t))=\metric(a,v)+t,
    \ \text{ and }\
    \metric(v,\phi_{v,b}(t))=t,\ \text{  for every $t\in[0,\metric(v,b)]$.}
  \end{equation}
  Applying \eqref{e.arcsig} on $(v,b)$ to these two distance   functions gives $\wgrad\metric(a,\cdot)
    \overset{\eqref{e.arcsig}}{=}
    \wgrad\metric(v,\cdot)$, $\medm$-a.e. on $(v,b)$. Thus the decomposition
  $(a,b)=(a,v)\cup\{v\}\cup(v,b)$ gives
  \begin{equation}\label{e.arcspl}
    \vecone_{(a,b)}
    \overset{\eqref{e.arcind}}{=}
    \vecone_{(a,v)}+\vecone_{(v,b)},
    \ \ \medm\text{-a.e. on }\skeleton.
  \end{equation}
  This identity also includes $a=b=v$. Let $c$ be the branch point of $o,x,y$ in \eqref{eq:tripod}. In particular, $c\in[o,x]\cap[x,y]\cap[o,y]$. By \ref{it.arci1},
  \begin{equation}
      \vecone_{(o,x)}+\vecone_{(x,y)}
      \overset{\eqref{e.arcspl}}{=}
        \vecone_{(o,c)}+\vecone_{(c,x)}
        +\vecone_{(x,c)}+\vecone_{(c,y)}=\vecone_{(o,c)}+\vecone_{(c,y)}\overset{\eqref{e.arcspl}}{=}\vecone_{(o,y)},
  \end{equation}
  which is exactly \eqref{e.arcadd}.
\qedhere
\end{enumerate}
\end{proof}
\begin{definition}\label{d.endpoi}
	Let $\{\vecone_{(a,b)}\}_{(a,b)\in\ambient\times\ambient}$ be the collection of functions defined in \eqref{e.arcind}.
	\begin{enumerate}[label=\textup{({\arabic*})}, align=right, leftmargin=*, topsep=5pt, parsep=0pt, itemsep=2pt]
  \item Define
\begin{equation}\label{e.arccla}
  \sA:=\Sett{\sum_{j=1}^{N}\kappa_{j}\vecone_{(a_{j},b_{j})}}{N\in\bN,\ \kappa_{j}\in\bR,\ a_{j},b_{j}\in\ambient,\ a_{j}\neq b_{j}}.
\end{equation}
Clearly, $\sA\subset L^{p}(\skeleton,\medm)$  for every $p\in[1,\infty]$.
\item  For $f=\sum_{j=1}^N\kappa_j\vecone_{(a_j,b_j)}\in\sA$, define
\begin{equation}
 \mu_f:=\sum_{j=1}^N\kappa_j(\delta_{b_j}-\delta_{a_j}),\ \text{ and }\
 Z_f:=\supp(\mu_f).
\label{e.endpoi}
\end{equation}
Write $\mu_f(z)$ for $\mu_f(\{z\})$. Then,
\begin{equation}
 \int_\ambient u\dif\mu_f
 \overset{\eqref{e.endpoi}}{=}\int_{\skeleton}f\cdot\wgrad u\dif\medm,\ \text{for every $u\in\abscon(\ambient,\metric)$}.
\end{equation}
\item We define the \emph{core} by \begin{equation}\label{e.core}
  \core:=\Sett{f\in\abscon(\ambient,\metric)\cap C_{c}(\ambient)}{\wgrad f\in\sA}.
\end{equation}
Clearly, $\core\subset\sobolev{p}(\ambient)$ for every $p\in[1,\infty)$.
\end{enumerate}
\end{definition}

\begin{lemma}\label{l.finarc}
  Let $m\in\bN$, $\{[a_j,b_j]\}_{j=1}^m$ be a finite family of compact arcs with positive $\medm$-measure, and let $H$ be the convex hull of $Z:=\{a_j,b_j:j\in\{1,\ldots,m\}\}$. There exist a finite set $V \subset H$ containing $Z$, an integer $N\in\bN$, and pairs $(u_j,v_j) \in V^2$, $j\in\{1,\ldots, N\}$, with the following properties. 
  \begin{enumerate}[label=\textup{({\arabic*})}, align=right, leftmargin=*, topsep=5pt, parsep=0pt, itemsep=2pt]
  \item\label{it.finarc1} $H=V\sqcup\bigsqcup_{l=1}^N(u_l,v_l)$;
  \item\label{it.finarc2} Set $e_{j}:=[u_{j},v_{j}]$. Then $ e_l\cap e_k
    \subset\{u_l,v_l\}\cap\{u_k,v_k\}$ for distinct $l,k\in\{1,\ldots,N\}$;
  \item\label{it.finarc3} For every $j\in\{1,\ldots,m\}$, we have $[a_j,b_j]
    =\bigcup_{\substack{l\in\{1,\ldots,N\}\\
                       e_l\subset[a_j,b_j]}}e_l$.
   \item\label{it.finarc4} For each $j\in\{1,\ldots,m\}$ and $l\in\{1,\ldots,N\}$, we define a number $\varepsilon_{jl}\in\{-1,0,1\}$ by   
    \begin{equation}\label{e.finedg0}
    \varepsilon_{jl}:=
    \begin{cases}
      1,
        & e_l\subset[a_j,b_j]\ \text{and}\
          \metric(a_j,u_l)<\metric(a_j,v_l),\\
      -1,
        & e_l\subset[a_j,b_j]\ \text{and}\
          \metric(a_j,v_l)<\metric(a_j,u_l),\\
      0,
        & e_l\not\subset[a_j,b_j].
    \end{cases}
  \end{equation}
      For every $\kappa_1,\ldots,\kappa_m\in\bR$, we have
  \begin{equation}\label{e.finedg}
    \sum_{j=1}^m\kappa_j\vecone_{(a_j,b_j)}
    =\sum_{l=1}^N\left(\sum_{j=1}^m\kappa_j\varepsilon_{jl}\right)\vecone_{(u_l,v_l)},
    \quad\medm\text{-a.e. on }\skeleton.
  \end{equation}        
\end{enumerate}
\end{lemma}

\begin{proof}
Write $Z=\{z_1,\ldots,z_n\}$ with distinct points, so $n\geq2$. For $k\in\{1,\ldots,n\}$, define $H_k:=\bigcup_{q=1}^k[z_1,z_q]$. Then $H_k$ is nonempty, compact, and is geodesically convex by Lemma~\ref{l.projt}-\ref{it.proj1}. In particular, $H=H_n$.

For $k\in\{2,\ldots,n\}$, define $p_k:=\pi_{H_{k-1}}(z_k)$. Since $z_1\in H_{k-1}$, Lemma~\ref{l.projt}-\ref{it.proj3}
gives
\begin{equation}
  [z_1,z_k]=[z_1,p_k]\cup[p_k,z_k],\ \text{ and }\  [p_k,z_k]\cap H_{k-1}=\{p_k\}.
\end{equation}
Moreover, $[z_1,p_k]\subset H_{k-1}$ by convexity. Consequently,
\begin{equation}\label{e.finedg1}
  H_k
  =H_{k-1}\cup[z_1,z_k]
  =H_{k-1}\cup[p_k,z_k].
\end{equation}
If $z_k\in H_{k-1}$, then $p_k=z_k$, and the added arc is the singleton $\{z_k\}$. Define $ V:=Z\cup\{p_k\}_{k=2}^{n}$. Fix $k$ with $p_k\neq z_k$, and put $L_k:=\metric(p_k,z_k)$. The finite set $\{t\in[0,L_k]:\phi_{p_k,z_k}(t)\in V\}$ contains $0$ and $L_k$, because $p_k,z_k\in V$, and we list its elements in increasing order as
\begin{equation}
  0=t_{k,0}<t_{k,1}<\cdots<t_{k,r_k}=L_k,
  \ \text{ for some } r_k\in\bN,
\end{equation}
and set $w_{k,i}:=\phi_{p_k,z_k}(t_{k,i})$ for every $i\in\{0,\ldots,r_k\}$. For each $i\in\{1,\ldots,r_k\}$, $\phi_{p_k,z_k}$ is an isometry between $[t_{k,i-1},t_{k,i}]$ and $[w_{k,i-1},w_{k,i}]$. Hence $\metric(w_{k,i-1},w_{k,i})=t_{k,i}-t_{k,i-1}>0$, $[p_k,z_k]=\bigcup_{i=1}^{r_k}[w_{k,i-1},w_{k,i}]$ and $[w_{k,i-1},w_{k,i}]\cap V=\{w_{k,i-1},w_{k,i}\}$. We now re-enumerate these distinct arcs as \begin{equation}
	\label{e.finedg2}\{(u_{j},v_{j})\}_{j=1}^{N}=\Sett{(w_{k,i-1},w_{k,i})}{1\leq i\leq r_{k},\ 2\leq k\leq n \text{ with } p_k\neq z_k },
\end{equation} and write $e_l:=[u_l,v_l]$. Then by the definition of these curves, we know that 
\begin{equation}\label{e.finedg3}
  u_l,v_l\in V,
  \ 
  e_l\cap V=\{u_l,v_l\},
  \ \text{ and }\
  e_l\setminus V=(u_l,v_l).
\end{equation}

  \begin{enumerate}[label=\textup{({\arabic*})}, align=right, leftmargin=*, topsep=5pt, parsep=0pt, itemsep=2pt]
    \item[\ref{it.finarc1}] By applying \eqref{e.finedg1} and \eqref{e.finedg2} inductively, we have  $H=V\cup\bigcup_{l=1}^N(u_l,v_l)$. By \eqref{e.finedg1} and \eqref{e.finedg3}, these sets are pairwise disjoint, so $H=V\sqcup\bigsqcup_{l=1}^N(u_l,v_l)$.

    \item[\ref{it.finarc2}]
    By \ref{it.finarc1} and \eqref{e.finedg3},
    \begin{equation}
      e_l\cap e_k
      \subset(e_l\cap V)\cap(e_k\cap V)
      =\{u_l,v_l\}\cap\{u_k,v_k\}.
    \end{equation}

    \item[\ref{it.finarc3}]
    Fix $j\in\{1,\ldots,m\}$, and list $V\cap[a_j,b_j]$ in order as $w_0=a_j,\ldots,w_r=b_j$. For $i\in\{1,\ldots,r\}$,     $(w_{i-1},w_i)\cap V=\emptyset$, so it lies in some $(u_l,v_l)$. Thus
    \begin{equation}
      \{w_{i-1},w_i\}\subset e_l\cap V=\{u_l,v_l\}.
    \end{equation}
    These two points are distinct, so we must have $[w_{i-1},w_i]=e_l$. Thus
    \begin{equation}
      [a_j,b_j]
      =\bigcup_{i=1}^r[w_{i-1},w_i]
      =\bigcup_{\substack{l\in\{1,\ldots,N\}\\
                         e_l\subset[a_j,b_j]}}e_l.
    \end{equation}
    \item[\ref{it.finarc4}]
    Fix $j\in\{1,\ldots,m\}$, and use the notation from \ref{it.finarc3}. By definition,
    \begin{equation}
      \vecone_{(a_j,b_j)}
      \overset{\eqref{e.arcadd}}{=}
      \sum_{i=1}^r\vecone_{(w_{i-1},w_i)}
      \overset{\eqref{e.finedg0}}{=}\sum_{l=1}^N\varepsilon_{jl}\vecone_{(u_l,v_l)},
      \ \ \medm\text{-a.e. on }\skeleton.
    \end{equation}
    Multiplying these identities by $\kappa_j$ and summing over $j$, we obtain \eqref{e.finedg}.
    \qedhere
  \end{enumerate}
\end{proof}

\subsection{Cutoff functions and the density of the core}\label{ss.cutoff}
In this subsection, we construct cutoff functions whose gradients are supported on finitely many arcs. We then use these cutoff functions to show that $\core$ is dense in $\sobolev{p}(\ambient)$ for $1\le p<\infty$. For Theorem~\ref{t.dualc}, we also show how to approximate bounded functions supported on finitely many arcs by gradients of functions in $\core$, as in \eqref{e.gradd1}.

We first give a formal definition of the Ahlfors regularity of a metric space.

\begin{definition}\label{d.ARdf}
	Let $\dimf\in[1,\infty)$. Let $(\ambient,\metric)$ be a metric space. Let $\meas$ be a Radon measure on $(\ambient,\metric)$. We say that the triple $(\ambient,\metric,\meas)$ is \hypertarget{AR}{\emph{$\dimf$-Ahlfors regular}} $\AR{\dimf}$, if there exists a constant $C_{0}\in[1,\infty)$ satisfying
	\begin{equation}\label{e.ahlfor}
  C_{0}^{-1}r^{\dimf}\leq\meas(B(x,r))\leq C_{0}r^{\dimf},\ \text{ for all }\ (x,r)\in\ambient\times(0,\infty).
\end{equation}
We say that the metric space $(\ambient,\metric)$ is {\emph{$\dimf$-Ahlfors regular}} $\AR{\dimf}$, if there exists a Radon measure $\meas$ on $(\ambient,\metric)$ such that $(\ambient,\metric,\meas)$ is $\AR{\dimf}$.
\end{definition}
\begin{remark}\label{r.ARdf}
	By \cite[Exercise 8.11]{Hei01}, a metric space $(\ambient,\metric)$ is $\AR{\dimf}$ if and only if $(\ambient,\metric,\sH_{\dimf})$ is $\AR{\dimf}$, where $\sH_{\dimf}$ is the Hausdorff $\dimf$-measure on $(\ambient,\metric)$. Therefore, the Ahlfors regularity depends only on the metric space. Whenever we say that a metric space $(\ambient,\metric)$ is $\AR{\dimf}$, we take its reference measure $\meas$ to be the $\dimf$-dimensional Hausdorff measure $\sH_{\dimf}$.
\end{remark}

\begin{lemma}\label{l.cutoff}
Let $(\ambient,\metric)$ be a proper unbounded real tree that is $\AR{\dimf}$ with $\dimf\in[1,\infty)$. There exists $M\in\bN$, depending only on $C_{0}$ and $\dimf$, such that for every $x\in\ambient$ and $r\in(0,\infty)$ there is a function $\chi_{x,r}\in\Lip(\ambient,\metric)\cap C_{c}(\ambient)$ with the following properties.
  \begin{enumerate}[label=\textup{(\arabic*)}, align=right, leftmargin=*, topsep=5pt, parsep=0pt, itemsep=2pt]
    \item\label{it.tcut1} $\one_{\ol{B}(x,r)}\leq\chi_{x,r}\leq \one_{B(x,3r)}$.
    \item\label{it.tcut2} There exist $N\in[1,M]\cap\bN$ and arcs $[a_{1},b_{1}],\ldots,[a_{N},b_{N}]$ with pairwise disjoint interiors, such that
    \begin{equation}\label{e.cutoff0}
    	\max_{1\leq i\leq N}\medm([a_{i},b_{i}])\leq r,\ \bigcup_{1\leq i\leq N} [a_{i},b_{i}]\subset \ol{B}(x,2r)\setminus B(x,r)
    \end{equation}
    and 
    \begin{equation}\label{e.cutoff}
    \wgrad\chi_{x,r}=-\frac{1}{r}\sum_{j=1}^{N}\vecone_{(a_{j},b_{j})}\ \ \medm\text{-a.e. on }\skeleton.
    \end{equation}
    In particular, $\chi_{x,r}\in\core$ and
    \begin{equation}\label{e.cutof1}
      \|\wgrad\chi_{x,r}\|_{L^{q}(\skeleton,\medm)}\leq Mr^{1/q-1},\ \text{ for all }q\in[1,\infty].
    \end{equation}
  \end{enumerate}
\end{lemma}
\begin{proof}
  \begin{enumerate}[label=\textup{(\arabic*)}, align=right, leftmargin=*, topsep=5pt, parsep=0pt, itemsep=2pt]
    \item[\ref{it.tcut1}]  We define a sequence of sets $\{S_{n}\}_{n\in\{0\}\cup\bN}$ as follows. Let $S_0:=\emptyset$. For $n\in\{0\}\cup\bN$, if $\ol{B(x,4r)}\setminus\bigcup_{y\in S_n}B(y,10^{-2}r)$ is nonempty, choose $y_{n+1}$ in this set and put $S_{n+1}:=S_n\cup\{y_{n+1}\}$. Each $S_n$ is $10^{-2}r$-separated. The balls $B(y,2^{-1} 10^{-2}r)$, $y\in S_n$, are disjoint and lie in $B(x,5r)$, so $nC_0^{-1}(2^{-1}10^{-2}r)^{\dimf}\overset{\eqref{e.ahlfor}}{\leq}C_0(5r)^{\dimf}$. Thus the construction stops at an integer $K\in\{1,\ldots,\lceil C_0^2 10^{3\dimf}\rceil\}$, giving a maximal separated set $S_K=\{y_1,\ldots,y_K\}$. In particular,
  \begin{equation}
    \ol{B(x,4r)}\subset\bigcup_{k=1}^{K}B(y_k,10^{-2}r).
  \end{equation}
 Define a compact connected subtree $T_k:=\bigcup_{i=1}^{k}[x,y_i]$ for every $k\in\{1,\ldots, K\}$, and $T:=T_K$. Let $w_{1}:=x$ and $w_k:=\pi_{T_{k-1}}(y_k)$ for $k\in\{2,\ldots,K\}$. Let $T_{0}:=\{x\}$. Lemma~\ref{l.projt} gives
  \begin{equation}
    T_k=T_{k-1}\cup[w_k,y_k],\  [w_k,y_k]\cap T_{k-1}=\{w_k\},\ \text{ and }\  w_k\in[x,y_k].
  \end{equation}
  Define
  \begin{equation}\label{e.cutof2}
    \chi_{x,r}(y):=\left(\left(2-\frac{\metric(x,\pi_T(y))}{r}\vee 0\right)\wedge1\right),\ y\in\ambient.
  \end{equation}
  By Lemma~\ref{l.projt}-\ref{it.proj4}, $\chi_{x,r}$ is $r^{-1}$-Lipschitz and takes values in $[0,1]$. Since $x\in T$, we have
  \begin{equation}
    [x,y]\cap T=[x,\pi_T(y)],\ \text{ and }\
    \metric(x,y)=\metric(x,\pi_T(y))+\metric(\pi_T(y),y).
  \end{equation}
  Consequently, $\chi_{x,r}(y)=1$ when $y\in\ol{B}(x,r)$. If $\metric(x,y)\geq3r$, let $z\in[x,y]$ satisfy $\metric(x,z)=3r$ and put $b:=\pi_T(z)$. The net property implies $\metric(z,b)=\dist(z,T)\leq 10^{-2}r$. Moreover, $b\in[x,z]\cap T\subset[x,\pi_T(y)]$, whence
  \begin{equation}
    \metric(x,\pi_T(y))\geq\metric(x,b)
    =3r-\metric(z,b)\geq3r-10^{-2}r>2r.
  \end{equation}
  Hence $\chi_{x,r}(y)=0$. This proves~\ref{it.tcut1}.

    \item[\ref{it.tcut2}] By \eqref{e.cutof2}, we know that $\chi_{x,r}$ is constant on every component of $\ambient\setminus T$. Thus $\wgrad\chi_{x,r}=0$ $\medm$-a.e. on $\ambient\setminus T$. Discarding the empty and singleton intersections, we enumerate the non-empty set $\{[w_k,y_k]\cap (\ol{B}(x,2r)\setminus B(x,r))\}_{k}$ as $\{[a_{j},b_{j}]\}_{j=1}^{N}$, with $\metric(a_j,x)\leq \metric(b_{j},x)$. Then \eqref{e.cutoff0} is evident, and 
  \begin{equation}
    \chi_{x,r}(y)=2-\frac{\metric(x,a_j)}{r}-\frac{\metric(a_j,y)}{r},\quad
    \wgrad\chi_{x,r}(y)=-r^{-1}\vecone_{(a_j,b_j)}(y)\ \text{ for $\medm$-a.e. }y\in (a_{j},b_{j}),
  \end{equation}
which proves \eqref{e.cutoff}, and hence $\chi_{x,r}\in\core$, with $N\leq K\leq \lceil3C_0^2 10^{3\dimf}\rceil=:M$. Since the arc interiors are disjoint, for all $q\in[1,\infty)$,
  \begin{equation}
    \|\wgrad\chi_{x,r}\|_{L^q(\skeleton,\medm)}^q
    \overset{\eqref{e.cutoff}}{=}r^{-q}\sum_{j=1}^{N}\metric(a_j,b_j)\leq Mr^{1-q},\qquad
    \|\wgrad\chi_{x,r}\|_{L^\infty(\skeleton,\medm)}\leq r^{-1}.
  \end{equation}
  This proves~\eqref{e.cutof1}.
\qedhere
\end{enumerate}
\end{proof}
\begin{lemma}\label{l.primit}
  Let $o\in\ambient$, let $R\in(0,\infty)$, $p\in[1,\infty]$, and let $\eta\in L^{p}(\skeleton,\medm)$ such that $\eta=0$ $\medm$-a.e. on $\skeleton\setminus B(o,R)$. Define
  \begin{equation}\label{e.primit}
    F(x):=\int_{\skeleton}\vecone_{(o,x)}\eta\dif\medm,\ \text{ for every } x\in\ambient.
  \end{equation}
  Then we have the following properties of $F$.\begin{enumerate}[label=\textup{({\arabic*})}, align=right, leftmargin=*, topsep=5pt, parsep=0pt, itemsep=2pt]
  \item\label{it.prim1} $F\in\abscon(\ambient,\metric)$, $F(o)=0$, $\wgrad F=\eta$ $\medm$-a.e.
  \item\label{it.prim2} We have that  \begin{equation}\label{e.primi1}
    F(y)-F(x)=\int_{\skeleton}\vecone_{(x,y)}\eta\dif\medm,\ \text{ and }\
    |F(y)-F(x)|\leq\int_{(x,y)}|\eta|\dif\medm,\ \text{ for every }x,y\in\ambient.
  \end{equation}
  \item\label{it.prim3} If $p\in[1,\infty)$, then $\|F\|_{\sup}\leq R^{1-1/p}\|\eta\|_{L^{p}(\skeleton,\medm)}$.
  \item\label{it.prim4} $F$ is constant on each connected component of $\ambient\setminus B(o,R)$.
\end{enumerate}
\end{lemma}
\begin{proof}
  By H\"older's inequality, we have
  \begin{equation}
    \int_{(a,b)}|\eta|\dif\medm
    \leq\metric(a,b)^{1-1/p}
      \|\eta\|_{L^p(\skeleton,\medm)}<\infty,\ \text{ for every }a,b\in\ambient.
  \end{equation}
  Thus $F$ is well defined, and $F(o)=0$.
  For every $x,y\in\ambient$, 
  \begin{equation}\label{e.primdf}
      F(y)-F(x)
      \overset{\eqref{e.primit}}{=}
        \int_{\skeleton}
          \bigl(\vecone_{(o,y)}-\vecone_{(o,x)}\bigr)
          \eta\dif\medm\overset{\eqref{e.arcadd}}{=}
        \int_{\skeleton}\vecone_{(x,y)}\eta\dif\medm.
  \end{equation}

  \begin{enumerate}[label=\textup{({\arabic*})}, align=right,
    leftmargin=*, topsep=5pt, parsep=0pt, itemsep=2pt]
    \item[\ref{it.prim1}]
    If $p\in(1,\infty]$, H\"older's inequality yields
    \begin{equation}
      |F(y)-F(x)|
      \overset{\eqref{e.primdf}}{\leq}
      \metric(x,y)^{1-1/p}
        \|\eta\|_{L^p(\skeleton,\medm)}.
    \end{equation}
    If $p=1$, then, for every $M\in(0,\infty)$,
    \begin{equation}
      |F(y)-F(x)|
      \overset{\eqref{e.primdf}}{\leq}
      M\metric(x,y)
      +\int_{\{|\eta|>M\}}|\eta|\dif\medm.
    \end{equation}
    The last integral tends to zero as $M\to\infty$ by the dominated convergence theorem, since $\eta\in L^{1}(\skeleton,\medm)$. Given $\varepsilon\in(0,\infty)$,
    choose $M$ so that $\int_{\{|\eta|>M\}}|\eta|\dif\medm<\varepsilon/2$. Then
    $\metric(x,y)<\varepsilon/(2M)$ implies
    $|F(y)-F(x)|<\varepsilon$. Hence $F\in C(\ambient)$.

    Let $\gamma:[0,L]\to\ambient$ be a local arc chart,
    with $a:=\gamma(0)$ and $b:=\gamma(L)$, and set
    $\sigma(s):=\vecone_{(a,b)}(\gamma(s))$ for a.e.
    $s\in(0,L)$. By \eqref{e.arcind}, the indicators
    $\vecone_{(a,\gamma(t))}$ and $\vecone_{(a,b)}$
    agree on $(a,\gamma(t))$. Since $\gamma$ preserves
    length measure, we have
    \begin{equation}
      F(\gamma(t))-F(a)
      \overset{\eqref{e.primdf}}{=}
      \int_0^t\sigma(s)\eta(\gamma(s))\dif s,
      \ \text{ for all }\ t\in[0,L].
    \end{equation}
    The integrand belongs to $L^1((0,L),\Leb)$, because
    $|\sigma|=1$ $\Leb$-a.e. and $\eta$ is integrable on $(a,b)$.
    The fundamental theorem of calculus gives the absolute continuity of $F\circ\gamma$ and
    \begin{equation}
      (F\circ\gamma)'=\sigma(\eta\circ\gamma)
      \qquad\Leb\text{-a.e. on }(0,L).
    \end{equation}
    Thus $F\in\abscon(\ambient,\metric)$.

    For $\gamma=\gamma_n$ from the fixed orientation,
    we have $\metric(\gamma_n(0),\gamma_n(s))=s$.
    Therefore,
    \begin{equation}
      \sigma(s)
      \overset{\eqref{e.arcind},\eqref{e.gradde}}{=}
      \bigl(\metric(\gamma_n(0),\cdot)
        \circ\gamma_n\bigr)'(s)
      =1
      \quad\Leb\text{-a.e. on $\gamma_n^{-1}(A_n)$.}
    \end{equation}
    Definition \eqref{e.gradde} gives $\wgrad F=\eta$
    $\medm$-a.e. on each $A_n$. Since $(A_n)$ is a
    countable partition of $\skeleton$, this equality
    holds $\medm$-a.e. on $\skeleton$.

    \item[\ref{it.prim2}]
    The first identity is \eqref{e.primdf}. Taking absolute
    values and using
    $|\vecone_{(x,y)}|=\one_{(x,y)}$ gives
    \begin{equation}
      |F(y)-F(x)|
      \overset{\eqref{e.primdf}}{\leq}
      \int_{\skeleton}|\vecone_{(x,y)}|\,|\eta|\dif\medm
      =\int_{(x,y)}|\eta|\dif\medm.
    \end{equation}

    \item[\ref{it.prim3}]
    Since distance from $o$ increases at unit speed
    along $[o,x]$, we know that $\medm\bigl((o,x)\cap B(o,R)\bigr)\leq R$. By the assumption that $\restr{\eta}{\ambient\setminus B(o,R)}=0$ $\medm$-a.e., and H\"older's inequality, we have, if $p\in[1,\infty)$,
    \begin{equation}
      |F(x)|
      \overset{\eqref{e.primit}}{\leq}
      \int_{(o,x)\cap B(o,R)}|\eta|\dif\medm
      \leq R^{1-1/p}\|\eta\|_{L^p(\skeleton,\medm)}.
    \end{equation}
    Taking the supremum over $x\in\ambient$ proves
    the bound.

    \item[\ref{it.prim4}]
    Let $C$ be a connected component of
    $\ambient\setminus B(o,R)$. Since this complement
    is closed, $C$ is closed and connected.
    Lemma~\ref{l.projt}-\ref{it.proj1} gives
    $[x,y]\subset C$ for every $x,y\in C$. Consequently,
    \begin{equation}
      F(y)-F(x)
      \overset{\eqref{e.primi1}}{=}
      \int_{\skeleton}\vecone_{(x,y)}\eta\dif\medm
      =0,
    \end{equation}
    because $\eta=0$ $\medm$-a.e. on
    $\skeleton\setminus B(o,R)$. Thus $F$ is constant
    on $C$.
  \end{enumerate}
\end{proof}
\begin{lemma}\label{l.arcden}
  Let $n\in\bN$, $q_1,\ldots,q_n\in[1,\infty)$ and
  $g\in\bigcap_{i=1}^nL^{q_i}(\skeleton,\medm)$. If $\restr{g}{\ambient\setminus B(x_{0},R_{0})}=0$ $\medm$-a.e. for some $(x_{0},R_{0})\in \ambient\times(0,\infty]$, then there exists a sequence \begin{equation}\label{e.arcden3}
 \{g_k\}_{k\in\bN}\subset\operatorname{span}
  \Sett
   { \vecone_{(a,b)}}{
    a,b\in\ambient,\ a\neq b,\ [a,b]\subset B(x_{0},R_{0})}\subset \sA
\end{equation} such that
  \begin{equation}\label{e.arcden1}
   \lim_{k\to\infty} g_k= g\ \text{ in }L^{q_i}(\skeleton,\medm)
    ,\ \text{ for all }\ i\in\{1,\ldots,n\}.
  \end{equation}
and
  \begin{equation}\label{e.arcden2}
    \sup_{k\in\bN}\|g_k\|_{L^s(\skeleton,\medm)}
    \leq\|g\|_{L^s(\skeleton,\medm)},\ \text{ for }s\in\{1,\infty\}.
  \end{equation}
\end{lemma}
\begin{proof}
  If $g=0$ $\medm$-a.e., take $g_k=0$. Assume henceforth that $g\neq0$. By Remark~\ref{r.arcs}, there exist compact arcs $\{I_{j}=[a_{j},b_{j}]\}_{j\in\bN}$ such that each two different arc has $\medm$-null intersections, $I_{j}\subset B(x_{0},R_{0})$, $L_{j}:=\medm(I_{j})>0$ and $\medm(B(x_{0},R_{0})\setminus\bigcup_{j}I_{j})=0$.
  
  Define $E_m:=\bigcup_{j=1}^mI_j$. By the dominated convergence theorem,
  \begin{equation}\label{e.arctai}
    \lim_{m\to\infty}\|g-g\one_{E_m}\|_{L^{q_i}(\skeleton,\medm)}
    =0,\ \text{ for every $i\in\{1,\ldots,n\}$}.
  \end{equation}

For $j\in\bN$, define $\phi_{j}:=\phi_{a_{j},b_{j}}$ and
  \begin{equation}
    h_{j}(t):=\vecone_{(a_{j},b_{j})}(\phi_{j}(t))g(\phi_{j}(t)),
    \ \text{ for }\ t\in(0,L_{j}).
  \end{equation}
  Then $h_j\in L^{q_i}((0,L_j),\Leb)$ for every
  $i\in\{1,\ldots,n\}$.

For $j,k\in\bN$ and $i\in\{1,\ldots,2^{k}\}$, define \begin{equation}\label{e.dyadfn}
\kappa_{j,k,i}:=\frac{2^k}{L_{j}}\int_{(i-1)2^{-k}L_{j}}^{i2^{-k}L_{j}}h_{j}(t)\dif t,\ \text{ and }\	h_{j,k}:=\sum_{i=1}^{2^k}
     \kappa_{j,k,i}
      \one_{((i-1)2^{-k}L_{j},i2^{-k}L_{j})}.
\end{equation}
Then by Jensen's inequality, 
  \begin{equation}\label{e.dyadbd}
    \|h_{j,k}\|_{L^s((0,L_{j}),\Leb)}
    \leq\|h_{j}\|_{L^s((0,L_{j}),\Leb)},
    \ \text{ for every } s\in[1,\infty] \text{ and every }j,k\in\bN.
  \end{equation}
  and by the uniform continuity\begin{equation}\label{e.dyadbd1}
  	\lim_{k\to\infty} \|h_{j,k}-h_{j}\|_{L^q((0,L_{j}),\Leb)}=0\ \text{ for every } q\in\{q_1,\ldots,q_n\} \text{ and every }j\in\bN.
  \end{equation}
For $m,k\in\bN$, we define \begin{equation}\label{e.arcapp}
	f_{m,k}:=\sum_{j=1}^{m}\sum_{i=1}^{2^{k}} \kappa_{j,k,i}\vecone_{(\phi_{j}((i-1)2^{-k}L_{j}),\phi_{j}(i2^{-k}L_{j}))}\in\sA.
\end{equation}

  Thus, for any $m,k\in\bN$ and $j\in\{1,\ldots,m\}$, we have
  \begin{equation}\label{e.arcparam}
    f_{m,k}(\phi_{j}(t))
    \overset{\eqref{e.arcapp},\eqref{e.arcadd}}{=}
    \vecone_{(a_{j},b_{j})}(\phi_{j}(t))h_{j,k}(t)\quad \Leb\text{-a.e. }t\in (0,L_{j})
  \end{equation}
  Since each $\phi_{j}$ preserves length measure,
  for $q\in\{q_1,\ldots,q_n\}$ and $m\in\bN$, we obtain
  \begin{equation}
  \lim_{k\to\infty}\|f_{m,k}-g\one_{E_m}\|_{L^q(\skeleton,\medm)}^q
    \overset{\eqref{e.arcparam}}{=}
    \sum_{j=1}^{m}
       \lim_{k\to\infty}\|h_{j,k}-h_{j}\|_{L^q((0,L_{j}),\Leb)}^q\overset{\eqref{e.dyadbd1}}{=}0.
  \end{equation}
  By \eqref{e.arcparam} and \eqref{e.dyadbd}, we know that 
  \begin{equation}\label{e.arccon}
      \|f_{m,k}\|_{L^s(\skeleton,\medm)}
        \overset{\eqref{e.arcparam},\eqref{e.dyadbd}}{\leq}
        \|g\|_{L^s(\skeleton,\medm)},\ s\in\{1,\infty\}.
  \end{equation}
  For each $k\in\bN$, choose $r_{k}$ so that $\max_{1\leq i\leq n}
      \|f_{k,r_{k}}-g\one_{E_k}
        \|_{L^{q_i}(\skeleton,\medm)}
    <k^{-1}$. Set $g_k:=f_{k,r_{k}}$.
  For every $i\in\{1,\ldots,n\}$,
  \begin{equation}
    \|g_k-g\|_{L^{q_i}(\skeleton,\medm)}
    \leq k^{-1}
      +\|g\one_{E_k}-g\|_{L^{q_i}(\skeleton,\medm)}
    \overset{\eqref{e.arctai}}{\to}0,
    \ \text{ as }\ k\to\infty.
  \end{equation}
  This gives \eqref{e.arcden1}. The inequality \eqref{e.arcden2} follows from~\eqref{e.arccon}.
\end{proof}
We now prove that $\core$ is dense in $\sobolev{p}(\ambient)$ for $p\in[1,\infty)$.
\begin{lemma}\label{l.dense}
  Let $p\in[1,\infty)$ and $f\in\sobolev{p}(\ambient)$. Then
  \begin{equation}\label{e.sup}
    |f(x)|\leq C_{0}^{1/p}\|f\|_{L^{p}(B(x,1),\meas)}+\|\wgrad f\|_{L^{p}(\skeleton\cap B(x,1),\medm)},\ \text{ for all }x\in\ambient.
  \end{equation}
  In particular, $f\in C_{0}(\ambient)$. Moreover, there exists a sequence $\{f_{n}\}_{n\in\bN}\subset\core$ such that $f_{n}\to f$ in $\sobolev{p}(\ambient)$, that is,
  \begin{equation}\label{e.scaled}
    \lim_{n\to\infty}\bigl(\|f_{n}-f\|_{L^{p}(\ambient,\meas)}+\|\wgrad f_{n}-\wgrad f\|_{L^{p}(\skeleton,\medm)}\bigr)=0.
  \end{equation}
\end{lemma}
\begin{proof}
  Fix $x\in\ambient$. For every $y\in B(x,1)$, the fundamental theorem of calculus in \eqref{e.arcftc} and H\"older's inequality give
  \begin{equation}
    |f(x)|\overset{\eqref{e.arcftc}}{\leq}|f(y)|+\int_{(x,y)}|\wgrad f|\dif\medm
    \leq|f(y)|+\|\wgrad f\|_{L^p(\skeleton\cap B(x,1),\medm)}.
  \end{equation}
  The last inequality uses $\metric(x,y)^{1-1/p}\leq1$. Averaging in $y$ and using~\eqref{e.ahlfor}, we obtain \eqref{e.sup}. In particular, $\|f\|_{\sup}\overset{\eqref{e.sup}}{\leq} C_0^{1/p}\|f\|_{L^p(\ambient,\meas)}+\|\wgrad f\|_{L^p(\skeleton,\medm)}$. Fix $o\in\ambient$. If $\metric(o,x)\geq R+1$, then $B(x,1)\subset\ambient\setminus B(o,R)$, so
  \begin{equation}
    \sup_{\ambient\setminus B(o,R+1)}|f|
    \overset{\eqref{e.sup}}{\leq} C_0^{1/p}\|f\|_{L^p(\ambient\setminus B(o,R),\meas)}
    +\|\wgrad f\|_{L^p(\skeleton\setminus B(o,R),\medm)}\to0,\ \text{ as }R\to\infty.
  \end{equation}
So $f\in C_{0}(\ambient)$. For $R\in[1,\infty)$, let $f_R:=\chi_{o,R}f\in\sobolev{p}(\ambient)\cap C_c(\ambient)$, where $\chi_{o,R}$ is the function from Lemma~\ref{l.cutoff}. Then 
\begin{equation}
	\|f_R-f\|_{L^p(\ambient,\meas)}\leq\|f\|_{L^p(\ambient\setminus B(o,R),\meas)},\
    \|f_R-f\|_{\sup}\leq\sup_{\ambient\setminus B(o,R)}|f|,
\end{equation}
and \begin{equation}
	\|\wgrad f_R-\wgrad f\|_{L^p(\skeleton,\medm)}
    \overset{\eqref{e.cutof1}}{\leq}\|\wgrad f\|_{L^p(\skeleton\setminus B(o,R),\medm)}
    +MR^{1/p-1}\sup_{\ambient\setminus B(o,R)}|f|.
\end{equation}
Thus, $f_{R}\to f$ in $\sobolev{p}(\ambient)$ as $R\to\infty$. To prove \eqref{e.scaled}, it suffices to show that $\core$ is dense in $\sobolev{p}(\ambient)\cap C_{c}(\ambient)$, under the norm $\norm{\cdot}_{\sobolev{p}(\ambient)}$.

  Let $f\in\sobolev{p}(\ambient)\cap C_{c}(\ambient)$, and choose $R\in(0,\infty)$ such that $\supp(f)\subset B(o,R)$. Then $\wgrad f$ vanishes $\medm$-a.e. on $\skeleton\setminus B(o,R)$. Lemma~\ref{l.arcden} gives $\eta_j\in\sA$, such that $\restr{\eta_{j}}{\skeleton\setminus B(o,R)}=0$ $\medm$-a.e., and $\lim_{j\to\infty}\|\eta_j-\wgrad f\|_{L^p(\skeleton,\medm)}=0$. Let \begin{equation}
  	F(x):=\int_{\skeleton}\vecone_{(o,x)}\cdot \wgrad f \dif\medm,\ \text{ and }\ F_{j}(x):=\int_{\skeleton}\vecone_{(o,x)}\cdot \eta_{j} \dif\medm\ \text{ for }x\in\ambient.
  \end{equation} By~\eqref{e.arcftc}, $f=f(o)+F$. By Lemma~\ref{l.primit}-\ref{it.prim3}, applied to $\eta_j-\wgrad f$, we have
  \begin{equation}\label{e.denseu}
   \lim_{j\to\infty} \|F_j-F\|_{\sup}\leq \lim_{j\to\infty}R^{1-1/p}\|\eta_j-\wgrad f\|_{L^p(\skeleton,\medm)}=0.
  \end{equation}
  Set $v_j:=\chi_{o,2R}(f(o)+F_j)$. Since $\eta_j$ is bounded, $F_j$ is Lipschitz. Hence $v_j\in\Lip(\ambient,\metric)\cap C_c(\ambient)$. Moreover, by Lemma~\ref{l.primit}-\ref{it.prim4}, $F_j$ is constant on each component of $\ambient\setminus B(o,R)$, and therefore on each component of $\ambient\setminus B(o,2R)$. By Lemma~\ref{l.cutoff}, we have
  \begin{equation}
    \wgrad v_j=\eta_j+(f(o)+F_j)\wgrad\chi_{o,2R}\in\sA.
  \end{equation}
  Thus $v_j\in\core$. Since $f=\chi_{o,2R}f$ and $f\wgrad\chi_{o,2R}=0$, we know that
  \begin{equation}
    v_j-f=\chi_{o,2R}(F_j-F),\ \text{ and }\
    \wgrad v_j-\wgrad f=\eta_j-\wgrad f+(F_j-F)\wgrad\chi_{o,2R}.
  \end{equation}
  The function $v_j-f$ is supported in $\ol{B(o,6R)}$. Consequently,
  \begin{align}
    &\phantom{\ \leq}\|v_j-f\|_{L^p(\ambient,\meas)}+\|v_j-f\|_{\sup}+\|\wgrad v_j-\wgrad f\|_{L^p(\skeleton,\medm)}\\
    &\overset{\eqref{e.denseu},\eqref{e.cutof1}}{\leq}\left(1+R^{1-1/p}\left(\meas(B(o,7R))^{1/p}+1+M(2R)^{1/p-1}\right)\right)\|\eta_j-\wgrad f\|_{L^p(\skeleton,\medm)}\\
&    \to0\ \text{ as }\ j\to\infty.
  \end{align}
Thus, $\core$ is dense in $\sobolev{p}(\ambient)\cap C_{c}(\ambient)$, under the norm $\norm{\cdot}_{\sobolev{p}(\ambient)}$.
\end{proof}

\begin{lemma}\label{l.densg1}
  Fix $q\in(1,\infty)$. Let $\eta\in L^{\infty}(\skeleton,\medm)$ vanish $\medm$-a.e. outside a finite union of arcs. Then there exists a sequence $\{f_{n}\}_{n\in\bN}\subset\core$ such that
  \begin{equation}\label{e.gradd1}
    \lim_{n\to\infty}\wgrad f_{n}=\eta\ \text{ in }L^{q}(\skeleton,\medm)\text{ and in }L^{2}(\skeleton,\medm),\ \text{ and }\ \sup_{n\in\bN}\|\wgrad f_{n}\|_{L^{\infty}(\skeleton,\medm)}\leq2\|\eta\|_{L^{\infty}(\skeleton,\medm)}.
  \end{equation}
\end{lemma}
\begin{proof}
  If $\eta=0$ $\medm$-a.e. on $\skeleton$, then we take $f_n=0$ for every $n\in\bN$. Otherwise both $\|\eta\|_{L^\infty(\skeleton,\medm)}$ and $\|\eta\|_{L^1(\skeleton,\medm)}$ are positive and finite.  Choose $o\in\ambient$ and $R\in(0,\infty)$ so that $\supp_{\medm}[\eta]\subset B(o,R)$. By Lemma~\ref{l.arcden}, after passing to a subsequence, there exist $\eta_n\in\sA$, such that $\supp_{\medm}[\eta_{n}]\subset B(o,R)$ and satisfy
  \begin{equation}
    \|\eta_n-\eta\|_{L^q(\skeleton,\medm)}+\|\eta_n-\eta\|_{L^2(\skeleton,\medm)}\leq \frac{1}{n},\ 
    \|\eta_n\|_{L^\infty(\skeleton,\medm)}\leq \|\eta\|_{L^\infty(\skeleton,\medm)},\ \|\eta_n\|_{L^1(\skeleton,\medm)}\leq \|\eta\|_{L^1(\skeleton,\medm)}.
  \end{equation}
  Let $F_n(x):=\int_{\skeleton}\vecone_{(o,x)}\eta_{n}\dif\medm$ for $x\in\ambient$. Lemma~\ref{l.primit} gives $\wgrad F_n=\eta_n$, $\|F_n\|_{\sup}\leq \|\eta\|_{L^1(\skeleton,\medm)}$, and $F_n$ is constant over each component of $\ambient\setminus B(o,R)$.

  Choose
  \begin{equation}
    R_n:=\max\bigl\{R,n,\|\eta\|_{L^1(\skeleton,\medm)}/\|\eta\|_{L^\infty(\skeleton,\medm)},(nM\|\eta\|_{L^1(\skeleton,\medm)})^{q/(q-1)},(nM\|\eta\|_{L^1(\skeleton,\medm)})^2\bigr\},
  \end{equation}
  and define $f_n:=\chi_{o,R_n}F_n$.
  Then $f_n\in \Lip(\ambient,\metric)\cap C_{c}(\ambient)$. Since the cutoff equals one on the support of $\eta_n$, the product rule gives
  \begin{equation}
    \wgrad f_n=\eta_n+F_n\wgrad\chi_{o,R_n}.
  \end{equation}
  Every component of $\ambient\setminus B(o,R_n)$ is contained in a component of $\ambient\setminus B(o,R)$, so Lemma~\ref{l.cutoff} implies that  $F_n\wgrad\chi_{o,R_n}\in \sA$. Hence $f_n\in\core$. For $s\in\{q,2\}$, the choice of $R_n$ and~\eqref{e.cutof1} imply
  \begin{equation}
    \|F_n\wgrad\chi_{o,R_n}\|_{L^s(\skeleton,\medm)}\overset{\eqref{e.cutof1}}{\leq} M\|\eta\|_{L^1(\skeleton,\medm)} R_n^{-(s-1)/s}\leq n^{-1},
  \end{equation}
  and \begin{equation}
  	\|F_n\wgrad\chi_{o,R_n}\|_{L^\infty(\skeleton,\medm)}\leq \|\eta\|_{L^1(\skeleton,\medm)}/R_n\leq \|\eta\|_{L^\infty(\skeleton,\medm)}.
  \end{equation}
  Therefore $\|\wgrad f_n-\eta\|_{L^s(\skeleton,\medm)}\leq2/n$ for $s\in\{q,2\}$, and $\|\wgrad f_n\|_{L^\infty(\skeleton,\medm)}\leq2\|\eta\|_{L^\infty(\skeleton,\medm)}$. These are exactly \eqref{e.gradd1}.
\end{proof}

\begin{proposition}\label{p.comple}
  For $p\in[1,\infty)$, the space $(\sobolev{p}(\ambient),\norm{\cdot}_{\sobolev{p}(\ambient)})$ is complete. The collection $\core$ is  dense in $(C_{c}(\ambient),\norm{\cdot}_{\sup})$.
\end{proposition}
\begin{proof}
  Let $\{f_{n}\}_{n\in\bN}$ be Cauchy in $(\sobolev{p}(\ambient),\norm{\cdot}_{\sobolev{p}(\ambient)})$. Completeness of $L^{p}(\ambient,\meas)$ and $L^{p}(\skeleton,\medm)$ gives $f\in L^p(\ambient,\meas)$ and $g\in L^p(\skeleton,\medm)$ such that
  \begin{equation}
   \lim_{n\to\infty}\left( \|f_n-f\|_{L^p(\ambient,\meas)}+\|\wgrad f_n-g\|_{L^p(\skeleton,\medm)}\right)=0.
  \end{equation}
  Applying~\eqref{e.sup} to $f_n-f_m$ gives
  \begin{equation}
    \|f_n-f_m\|_{\sup}\overset{\eqref{e.sup}}{\leq} C_0^{1/p}\|f_n-f_m\|_{L^p(\ambient,\meas)}+\|\wgrad f_n-\wgrad f_m\|_{L^p(\skeleton,\medm)}\to0\ \text{ as } m,n\to\infty.
  \end{equation}
  So $\{f_{n}\}_{n\in\bN}$ is also Cauchy in $(C_{0}(\ambient),\norm{\cdot}_{\sup})$. Thus $f\in C(\ambient)$.
  Fix an arc chart $\gamma:[0,L]\to\ambient$, let $a:=\gamma(0)$ and $b:=\gamma(L)$. By \eqref{e.arcftc},
  \begin{equation}\label{e.comple1}
    f_n(\gamma(t))-f_n(a)\overset{\eqref{e.arcftc}}{=}\int_0^t\vecone_{(a,b)}(\gamma(s))\wgrad f_n(\gamma(s))\dif s,\ \text{ for every $t\in[0,L]$}.
  \end{equation}
 Since
  \begin{equation}
    \int_{(a,b)}|\wgrad f_n-g|\dif\medm
    \leq L^{1-1/p}\|\wgrad f_n-g\|_{L^p(\skeleton,\medm)}\to0\quad(n\to\infty),
  \end{equation}
  we can let $n\to\infty$ in \eqref{e.comple1} and obtain
  \begin{equation}
    f(\gamma(t))-f(a)=\int_0^t\vecone_{(a,b)}(\gamma(s))g(\gamma(s))\dif s,\ \text{ for every $t\in[0,L]$}.
  \end{equation}
  Thus $f\in \abscon(\ambient,\metric)$ and $\wgrad f=g$ $\medm$-a.e. on $\skeleton$. Therefore, $f\in\sobolev{p}(\ambient)$ and $f_n\to f$ in $(\sobolev{p}(\ambient),\norm{\cdot}_{\sobolev{p}(\ambient)})$, as $n\to\infty$.

It remains to prove that $\core$ is dense in $(C_c(\ambient),\|\cdot\|_{\sup})$.
First fix a nonnegative $u\in C_c(\ambient)$ and $\eta>0$. We construct
$v\in\sobolev{2}(\ambient)\cap C_c(\ambient)$ such that
$\|u-v\|_{\sup}\le\eta$. If $u=0$, take $v=0$. Suppose that $u\ne0$,
and put $K:=\supp u$. Choose $o\in\ambient$ and $R>0$ such that
$K\subset B(o,R)$. By properness, $u$ is uniformly continuous on the
compact set $\ol{B(o,R+1)}$. If $\metric(x,y)<1$ and at least one of
$x,y$ lies in $K$, then both lie in $B(o,R+1)$. If neither lies in $K$,
then $u(x)=u(y)=0$. Hence we may choose $\delta\in(0,1/3)$ such that
\begin{equation}\label{e.ucmod}
  |u(x)-u(y)|<\eta
  \ \text{whenever }x,y\in\ambient
  \text{ and }\metric(x,y)<3\delta .
\end{equation}
By compactness, choose $x_1,\ldots,x_N\in K$ such that
\begin{equation}\label{e.netdst}
 K\subset\bigcup_{j=1}^N B(x_j,\delta),\ \text{ and }\
 \sup_{x\in K}\min_{1\leq j\leq N}\metric(x,x_j)<\delta.
\end{equation}
Let $\chi_{x_j,\delta}$ be the cutoff function from Lemma~\ref{l.cutoff}. Define
\begin{equation}\label{e.hintrp}
  h_j(x):=u(x_j)\chi_{x_j,\delta}(x),\ \text{ and }\
  v(x):=\max_{1\leq j\leq N}h_j(x),\ \text{ for all } x\in\ambient .
\end{equation}
Fix $x\in\ambient$. If $h_j(x)>0$, then Lemma~\ref{l.cutoff}-\ref{it.tcut1} gives $x\in B(x_j,3\delta)$ and $0<\chi_{x_j,\delta}(x)\le1$, so $h_j(x)\leq u(x_j)\overset{\eqref{e.ucmod}}{<}u(x)+\eta$. The same upper bound holds when $h_j(x)=0$, since $u(x)\ge0$.
If $x\in K$, use~\eqref{e.netdst} to choose $i$ with $x\in B(x_i,\delta)$. Then by Lemma~\ref{l.cutoff}-\ref{it.tcut1},
\begin{equation}
  v(x)\overset{\eqref{e.hintrp}}{\geq} h_i(x)=u(x_i)
       \overset{\eqref{e.ucmod}}{>}u(x)-\eta .
\end{equation}
For $x\notin K$, we have $v(x)\ge0=u(x)$. Consequently,
\begin{equation}\label{e.herror}
  u(x)-\eta\leq v(x)\leq u(x)+\eta\ \ \text{ for all }x\in\ambient,\ \text{ that is }\ \|u-v\|_{\sup}\le\eta .
\end{equation}
By Lemma~\ref{l.cutoff}, $h_{j}\in\Lip(\ambient,\metric)\cap C_{c}(\ambient)$. Since
\begin{equation}\label{e.hlip}
 |v(x)-v(y)|\le\max_{1\leq j\leq N}|h_j(x)-h_j(y)|,\ \text{ and }\
 \supp v\subset\bigcup_{j=1}^N\supp h_j,
\end{equation}
we know that $v\in\Lip(\ambient,\metric)\cap C_{c}(\ambient)$. Since
$|\wgrad v|^2\le\sum_{j=1}^N u(x_j)^2
|\wgrad\chi_{x_j,\delta}|^2$ $\medm$-a.e.\ on $\skeleton$, we know that
\begin{equation}\label{e.hgrad}
 \int_{\skeleton}|\wgrad v|^2\dif\medm
 \leq \sum_{j=1}^N u(x_j)^2
       \|\wgrad\chi_{x_j,\delta}\|_{L^2(\skeleton,\medm)}^2
 \overset{\eqref{e.cutof1}}{<}\infty .
\end{equation}
By~\eqref{e.hgrad} and the definition~\eqref{e.sobdef}, $v\in\sobolev{2}(\ambient)\cap C_c(\ambient)$.

Now let $u\in C_c(\ambient)$ be real-valued and fix $\varepsilon>0$.
Put $u^+:=\max\{u,0\}$, $u^-:=\max\{-u,0\}$ and $\eta:=\varepsilon/4$.
Applying the preceding construction to $u^+$ and $u^-$ gives
$v_+,v_-\in\sobolev{2}(\ambient)\cap C_c(\ambient)$ with
$\|u^+-v_+\|_{\sup}\le\eta$ and $\|u^--v_-\|_{\sup}\le\eta$.
For $v:=v_+-v_-$, linearity gives $v\in\sobolev{2}(\ambient)$ and
$\|u-v\|_{\sup}\le2\eta$ by~\eqref{e.herror}. By Lemma~\ref{l.dense}, applied in~\eqref{e.scaled} with $p=2$, choose
$w_k\in\core$ converging to $v$ in $\sobolev{2}(\ambient)$. By \eqref{e.sup},
\begin{equation}
 \|w_k-v\|_{\sup}
 \overset{\eqref{e.sup}}{\le}
 C_0^{\frac{1}{2}}\|w_k-v\|_{L^2(\ambient,\meas)}
 +\|\wgrad w_k-\wgrad v\|_{L^2(\skeleton,\medm)}
 \overset{\eqref{e.scaled}}{\to}0 .
\end{equation}
For all sufficiently large $k$,
$\|u-w_k\|_{\sup}\le\|u-v\|_{\sup}+\|v-w_k\|_{\sup}
< 3\eta<\varepsilon$. This proves the density of $\core$ in $(C_c(\ambient),\|\cdot\|_{\sup})$.
\end{proof}

\subsection{Geodesic rays}

In this subsection, we study the union $T_o$ of the geodesic rays from a fixed point $o$. We first show that every component of $\ambient\setminus T_o$ is bounded. When $\dimf>1$, we then show that such components contain balls of arbitrarily large radius, which we use to construct the counterexamples in Section~\ref{s.failure}.

\begin{definition}
	Let $(\ambient,\metric)$ be a real tree. Let $o\in\ambient$. A \emph{geodesic ray} from $o$ is an isometric embedding $\xi:[0,\infty)\to\ambient$ with $\xi(0)=o$.
\end{definition}

\begin{remark}
	\begin{enumerate}[label=\textup{({\arabic*})}, align=right, leftmargin=*, topsep=5pt, parsep=0pt, itemsep=2pt]
  \item If $\xi:[0,\infty)\to\ambient$ is a geodesic ray, then its range $\xi([0,\infty))$ is a closed subset of $\ambient$. In fact, if $\xi(t_{n})\to y$ as $n\to\infty$, then $t_{n}=\metric(o,\xi(t_{n}))$ converges, and $y=\xi(\lim_{n\to\infty}t_{n})$.
  \item If $\xi,\xi':[0,\infty)\to\ambient$ are two distinct geodesic rays from $o$, then there exists $s_{0}\in[0,\infty)$ such that $\restr{\xi}{[0,s_{0}]}=\restr{\xi'}{[0,s_{0}]}$, and $\metric(\xi(t),\xi'(t))=2(t-s_{0})$ for $t\in[s_0,\infty)$.
\end{enumerate}
\end{remark}
Fix $o\in\ambient$. We define
\begin{equation}\label{e.Tinfty}
  T_{o}:=\bigcup\Sett{\xi([0,\infty))}{\xi\text{ is a geodesic ray from }o}.
\end{equation}

\begin{lemma}\label{l.ends}
  Let $T_{o}$ be defined by \eqref{e.Tinfty}.
  \begin{enumerate}[label=\textup{(\arabic*)}, align=right, leftmargin=*, topsep=5pt, parsep=0pt, itemsep=2pt]
    \item\label{it.tends1} There are at most $C_0^2 3^{\dimf}$
    geodesic rays from $o$. The set $T_{o}$ is nonempty,
    closed, and connected.

    \item\label{it.tends2} Every connected component $C$ of
    $\ambient\setminus T_{o}$ is bounded, and $\ol C$ is
    compact.

    \item\label{it.tends3} If $\dimf\in(1,\infty)$, then for
    every $R\in(0,\infty)$ there exist a connected component
    $C$ of $\ambient\setminus T_{o}$ and $x\in C$ such that
    $B(x,4R)\subset C$.
  \end{enumerate}
\end{lemma}

\begin{proof}
  We first construct a ray from a sequence
  $x_n\in\ambient\setminus\{o\}$ satisfying
  $\ell_n:=\metric(o,x_n)\to\infty$ as $n\to\infty$, which exists by the unboundedness of $(\ambient,\metric)$. Set
  \begin{equation}\label{e.raygn}
    \gamma_n(t):=\phi_{o,x_n}(t\wedge\ell_n),
    \ \text{ for }\ t\in[0,\infty).
  \end{equation}
  For each $j\in\bN$, the sequence
  $\{\gamma_n(j)\}_{n\in\bN}$ lies in the compact ball
  $\ol{B(o,j)}$. By a diagonal argument, there exists a subsequence such that $\lim_{k\to\infty}\gamma_{n_k}(j)=p_j$ for all $j\in\bN$.
  For $i,j\in\bN$ with $i\leq j$, we have $\metric(o,p_j)=j$ and $\metric(p_i,p_j)=j-i$.  Thus $\metric(o,p_j)=\metric(o,p_i)+\metric(p_i,p_j)$.
  The tripod property and uniqueness of geodesics give $p_i\in[o,p_j]$ and $\restr{\phi_{o,p_j}}{[0,i]}=\phi_{o,p_i}$. Hence the map $\gamma:[0,\infty)\to\ambient$ given by 
  \begin{equation}
    \gamma(t):=\phi_{o,p_j}(t), \ \text{ if $t\leq j$ for some $j\in\bN$},
  \end{equation}
  is well-defined. For $s,t\in[0,\infty)$, choosing
  $j\geq\max\{1,s,t\}$ yields
  \begin{equation}
    \gamma(0)=o,\ \text{ and }\ 
    \metric(\gamma(s),\gamma(t))
    =\metric(\phi_{o,p_j}(s),\phi_{o,p_j}(t))
    =|s-t|.
  \end{equation}
  In particular, $\gamma$ is a geodesic ray from $o$, and
  \begin{equation}\label{e.raylim}
    \lim_{k\to\infty}\gamma_{n_k}(j)=\gamma(j),
    \ \text{ for all }\ j\in\bN.
  \end{equation}
In particular, $T_{o}\neq\emptyset$.
  \begin{enumerate}[label=\textup{(\arabic*)}, align=right, leftmargin=*, topsep=5pt, parsep=0pt, itemsep=2pt]
    \item[\ref{it.tends1}]
    Let $\xi_1,\ldots,\xi_N$ be distinct rays from $o$.
    Let\[s_{ij}:=\max\Sett{s\in[0,\infty)}{\xi_{i}(s)=\xi_{j}(s)},\ \text{ for $i\neq j$}.\] Choose $L>\max_{i\neq j}2s_{ij}$ if $N\geq2$ and take any $L>0$ if $N=1$. Then
    \begin{equation}
      \metric(\xi_i(L),\xi_j(L))
      =2(L-s_{ij})>L,\qquad 1\leq i\neq j\leq N,
    \end{equation}
    and consequently $\bigsqcup_{i=1}^N B(\xi_i(L),L/2)
      \subset B(o,3L/2)$. Therefore,
    \begin{equation}
      \begin{aligned}
        NC_0^{-1}(L/2)^{\dimf}
        \overset{\eqref{e.ahlfor}}{\leq}
          \sum_{i=1}^N\meas(B(\xi_i(L),L/2))\leq\meas(B(o,3L/2))
        \overset{\eqref{e.ahlfor}}{\leq}
          C_0(3L/2)^{\dimf}.
      \end{aligned}
    \end{equation}
    Thus $N\leq C_0^2 3^{\dimf}$ for every finite family
    of distinct rays. There are therefore finitely many
    rays, say $\xi_1,\ldots,\xi_N$, and $T_{o}\overset{\eqref{e.Tinfty}}{=}\bigcup_{i=1}^N\xi_i([0,\infty))$ with $o\in\bigcap_{i=1}^N\xi_i([0,\infty))$. Since each geodesic ray has closed and connected image, their finite union $ T_{o}$ is closed, and their common point $o$ makes $T_{o}$ connected.

    \item[\ref{it.tends2}]
    Let $C$ be a component of $\ambient\setminus T_{o}$,
    and let $g_C\in T_{o}$ be its gate from Lemma~\ref{l.projt}-\ref{it.proj5}.
    Let $a:=\metric(o,g_C)$.
    Suppose otherwise that $C$ is unbounded. Choose $x_n\in C$ such
    that
    \begin{equation}
      \ell_n:=\metric(o,x_n)
      =a+\metric(g_C,x_n)\to\infty
      \ \text{ as }\ n\to\infty.
    \end{equation}
    Let $\gamma_{n}$ be the curves defined in \eqref{e.raygn}.    Lemma~\ref{l.projt}-\ref{it.proj3} and
    \eqref{e.gatec} give $\gamma_n(t)\in(g_C,x_n]\subset C$, $\pi_{T_{o}}(\gamma_n(t))=g_C$ and 
    \begin{equation}\label{e.enddst}
 \dist(\gamma_n(t),T_{o})
          =\metric(\gamma_n(t),g_C)=t-a,
          \qquad t\in(a,\ell_n].
    \end{equation}
    Choose an integer $j>a$. Since distance to a nonempty
    set is $1$-Lipschitz,
    \begin{equation}
        0\overset{\eqref{e.Tinfty}}{=}
           \dist(\gamma(j),T_{o})\overset{\eqref{e.raylim}}{=}
           \lim_{k\to\infty}
           \dist(\gamma_{n_k}(j),T_{o})
          \overset{\eqref{e.enddst}}{=}j-a>0,
    \end{equation}
    a contradiction. Hence $C$ is bounded. Set $D_C:=\sup_{x\in C}\metric(g_C,x)\in(0,\infty)$. Properness and the fact that $\ol C\subset\ol{B(g_C,D_C)}\subset B(g_C,2D_C)$ imply that $\ol C$ is compact.
    \item[\ref{it.tends3}]
    Suppose that $\dimf\in(1,\infty)$, and write
    $T_{o}=\bigcup_{j=1}^N\xi_j([0,\infty))$.
    Fix $R\in(0,\infty)$ and choose $L\in\bN$ such that
    \begin{equation}\label{e.endfar}
      L^{\dimf-1}>NC_0^2 5^{\dimf}.
    \end{equation}
    Define
    \begin{equation}
      U:=\{y\in B(o,LR):\dist(y,T_{o})<4R\}.
    \end{equation}
    For $y\in U$, put $w:=\pi_{T_{o}}(y)$.
    Since $o\in T_{o}$,
    Lemma~\ref{l.projt}-\ref{it.proj3} gives
    \begin{equation}
      \metric(o,w)+\metric(w,y)=\metric(o,y)<LR,
      \qquad
      \metric(y,w)=\dist(y,T_{o})<4R.
    \end{equation}
    Write $w=\xi_j(s)$ with
    $s=\metric(o,w)\in[0,LR)$ and set
    $k:=\lfloor s/R\rfloor\in\{0,\ldots,L-1\}$. Then $ 0\leq s-kR<R$ and 
    \begin{equation}
        \metric(y,\xi_j(kR))
        \leq\metric(y,w)+\metric(\xi_j(s),\xi_j(kR))
        =\metric(y,w)+s-kR<5R.
    \end{equation}
    Consequently,
    \begin{equation}\label{e.endcov}
      U\subset
      \bigcup_{j=1}^N\bigcup_{k=0}^{L-1}
      B(\xi_j(kR),5R),
    \end{equation}
    and
    \begin{equation}
      \begin{aligned}
        \meas(U)
        &\overset{\eqref{e.endcov}}{\leq}
          \sum_{j=1}^N\sum_{k=0}^{L-1}
          \meas(B(\xi_j(kR),5R))\\
        &\overset{\eqref{e.ahlfor}}{\leq}
          NLC_0(5R)^{\dimf}
        \overset{\eqref{e.endfar}}{<}
          C_0^{-1}(LR)^{\dimf}
        \overset{\eqref{e.ahlfor}}{\leq}
          \meas(B(o,LR)).
      \end{aligned}
    \end{equation}
    Choose $x\in B(o,LR)\setminus U$, so that
    $\dist(x,T_{o})\geq4R$, and let $C$ be the
    component containing $x$. For $y\in B(x,4R)$
    and $z\in[x,y]$,
    \begin{equation}
      \dist(z,T_{o})
      \geq\dist(x,T_{o})-\metric(x,z)
      \geq4R-\metric(x,y)>0.
    \end{equation}
    Thus $[x,y]\subset\ambient\setminus T_{o}$,
    which implies $y\in C$. Hence $B(x,4R)\subset C$.
    \qedhere
  \end{enumerate}
\end{proof}

\section{The Laplacian and its fractional powers}\label{s.analys}
In this section, we recall the Dirichlet form on the tree and prove estimates for the heat semigroup. We then study fractional powers of the Laplacian on the core $\core$ and introduce test functions for the Riesz transform. At the end of the section, we prove Theorem~\ref{t.critx} on the critical exponent $\critic_p$.

Recall that we have defined 
\begin{equation}\label{e.extdim}
  \dimw:=\dimf+1,\qquad \critic_p:=\frac{\dimf+p-1}{p\dimw}\ \text{ for }\ p\in[1,\infty),\ \text{ and }\ \critic_\infty:=\frac1{\dimw}.
\end{equation}
Thus $\critic_p+\critic_{q}=1$ whenever $p^{-1}+q^{-1}=1$.

\subsection{The Dirichlet form and heat kernel}
On every Ahlfors regular metric tree, we have a canonical regular Dirichlet form, as given by the following theorem.
\begin{theorem}\label{t.dirich}
  Let $(\ambient,\metric,\meas)$ be $\AR{\dimf}$ with $\dimf\in[1,\infty)$. Fix an arbitrary orientation $\mathfrak{O}$. Let $\wgrad:=\wgrad_{\mathfrak{O}}$. Let
  \begin{equation}\label{e.formde}
    \domain:=\sobolev{2}(\ambient),\ \text{ and }\ \form(f,g):=\int_{\skeleton}\wgrad f\wgrad g\dif\medm\ \text{ for every }(f,g)\in\domain\times\domain.
  \end{equation}
  Then $(\form,\domain)$ is independent of the choice of the orientation $\mathfrak{O}$, and the following assertions hold:
  \begin{enumerate}[label=\textup{(\arabic*)}, align=right, leftmargin=*, topsep=5pt, parsep=0pt, itemsep=2pt]
    \item\label{it.consDF} $(\form,\domain)$ is a conservative, strongly local, regular symmetric Dirichlet form on $L^{2}(\ambient,\meas)$.
    \item\label{it.HKE} Let $\{P_{t}\}_{t\in(0,\infty)}$ be the corresponding strongly continuous Markovian semigroup on $L^{2}(\ambient,\meas)$, with non-positive self-adjoint generator $(\gen,\Dom(\gen))$. Then $\Ker(\gen)=\{0\}$, and there exists a jointly continuous function $p:=p_{t}(x,y):(0,\infty)\times\ambient\times\ambient\to(0,\infty)$ such that for every $t\in(0,\infty)$ and every $f\in L^{2}(\ambient,\meas)$,
    \begin{equation}\label{e.heatke}
      P_{t}f(x)=\int_{\ambient}p_{t}(x,y)f(y)\dif\meas(y),\ \meas\text{-a.e. }x\in\ambient.
    \end{equation}
    There exist constants $C_{1}\in(1,\infty)$ and $c_{2},c_{3}\in(0,\infty)$ such that, for all $(t,x,y)\in (0,\infty)\times\ambient\times\ambient$, we have
    \begin{equation}\label{e.HKE}
      \frac{C_{1}^{-1}}{t^{\dimf/\dimw}} \exp\left(-c_{3}\left(\frac{\metric(x,y)^{\dimw}}{t}\right)^{\frac{1}{\dimw-1}}\right)\leq p_{t}(x,y)\leq \frac{C_{1}}{t^{\dimf/\dimw}} \exp\left(-c_{2}\left(\frac{\metric(x,y)^{\dimw}}{t}\right)^{\frac{1}{\dimw-1}}\right).
    \end{equation}
    \item\label{it.ddt} For every $(x,y)\in \ambient\times\ambient$, the function $t\mapsto p_{t}(x,y)$ is real-analytic on $(0,\infty)$. The function $(t,x,y)\mapsto \frac{\dif}{\dif t}p_{t}(x,y)$ is jointly measurable on $(0,\infty)\times \ambient\times\ambient$, and there exists $C_{2}\in(1,\infty)$ such that for all $(t,x,y)\in (0,\infty)\times\ambient\times\ambient$,
    \begin{equation}\label{e.ddt}
      \abs{\frac{\dif}{\dif t}p_{t}(x,y)}\leq \frac{C_{2}}{t^{1+\frac{\dimf}{\dimw} }} \exp\left(-c_{2}\left(\frac{\metric(x,y)^{\dimw}}{t}\right)^{\frac{1}{\dimw-1}}\right).
    \end{equation}
    \item\label{it.deri2} For $g\in L^{2}(\ambient,\meas)$ and $t\in(0,\infty)$, one has $P_{t}g\in\Dom(-\gen)\subset\domain$. Moreover,
    \begin{equation}\label{e.kernel}
      (-\gen P_{t}g)(y) =-\int_{\ambient}\frac{\dif}{\dif t}p_{t}(x,y)g(x) \dif\meas(x) \  \text{ for $\meas$-a.e. }y\in\ambient.
    \end{equation}
    This integral converges absolutely.
    \item\label{it.riesz}There exists a unique linear isometry $\riesz_{2}:L^2(\ambient,\meas)\to L^2(\skeleton,\medm)$ such that $\riesz_{2}$ is surjective and 
\begin{equation}\label{e.riesz2}
	\riesz_{2} f=\wgrad(-\gen)^{-\frac{1}{2}}f,\ \text{ for all }\ f\in\Dom((-\gen)^{-\frac{1}{2}}).
	\end{equation}
	We denote by $\riesz_{2}^*:L^2(\skeleton,\medm)\to L^2(\ambient,\meas)$ the adjoint of $\riesz_{2}$. Then 
	\begin{equation}\label{e.BTheta.adjoint1}
  \riesz_2^*\circ\riesz_2=\mathrm{id}_{L^2(\ambient,\meas)},\quad \riesz_{2}\circ \riesz_2^*=\mathrm{id}_{L^2(\skeleton,\medm)},
\end{equation}
and \begin{equation}\label{e.BTheta.adjoint2}
	\norm{\riesz_2}_{L^2(\ambient,\meas)\to L^2(\skeleton,\medm)}=\norm{\riesz_2^*}_{L^2(\skeleton,\medm)\to L^2(\ambient,\meas)}=1.
\end{equation}
  \end{enumerate}
\end{theorem}
\begin{proof}
	\begin{enumerate}[label=\textup{({\arabic*})}, align=right, leftmargin=*, topsep=5pt, parsep=0pt, itemsep=2pt]
  \item[\ref{it.consDF}] This is proved in \cite[Proposition 2.15, Remark 3.2]{BCC26}.
  \item[\ref{it.HKE}] That $\Ker(\gen)=\{0\}$ is proved in \cite[Lemma 2.18]{BCC26}. The existence of the heat kernel \eqref{e.heatke} and its estimates \eqref{e.HKE} are proved in \cite[Theorem~5.1]{BCY25}.
  \item[\ref{it.ddt}] The analyticity of $t\mapsto p_{t}(x,y)$ and the measurability of $(t,x,y)\mapsto\frac{\dif}{\dif t}p_{t}(x,y)$ are proved in \cite[Lemma 2.16-(3)]{BCC26}. The estimate \eqref{e.ddt} is proved in \cite[Lemma 3.7]{BCC26}.
  \item[\ref{it.deri2}] This is a consequence of \cite[Lemma 3.8-(2)]{BCC26} and \cite[Theorems~5.1 and~5.2]{BCY25}. 
  \item[\ref{it.riesz}] 
By \cite[Lemma~2.18]{BCC26}, there is a unique linear isometry
$\riesz_2:L^2(\ambient,\meas)\to L^2(\skeleton,\medm)$ satisfying~\eqref{e.riesz2}.
We first show that its range $\riesz_2(L^2(\ambient,\meas))$ contains $\wgrad f$ for every $f\in\domain$. By \ref{it.HKE}, $\Ker(-\gen)=\{0\}$.
Fix $f\in\domain=\Dom((-\gen)^{\frac{1}{2}})$ and let $h:=(-\gen)^{\frac{1}{2}}f$.
Since $\dif\langle\proj_\lambda h,h\rangle
=\lambda\dif\langle\proj_\lambda f,f\rangle$, we have
\begin{equation}
 \int_{(0,\infty)}\lambda^{-1}\dif\langle\proj_\lambda h,h\rangle
 =\int_{(0,\infty)}\dif\langle\proj_\lambda f,f\rangle
 =\|f\|_{L^2(\ambient,\meas)}^2<\infty .
\end{equation}
Thus $h\in\Dom((-\gen)^{-\frac{1}{2}})$ and $(-\gen)^{-\frac{1}{2}}h=f$. In particular,
\begin{equation}
 \riesz_2(-\gen)^{\frac{1}{2}}f
 \overset{\eqref{e.riesz2}}{=}\wgrad(-\gen)^{-\frac{1}{2}}h
 =\wgrad f,\ \text{for all }\ f\in\domain .
\end{equation}

Let $g\in L^2(\skeleton,\medm)$ be orthogonal to $\riesz_2(L^2(\ambient,\meas))$.
Since $\core\subset\domain$, the preceding identity gives
$\langle g,\wgrad f\rangle_{L^2(\skeleton,\medm)}=0$ for every $f\in\core$.
Fix a finite open arc $I\subset\skeleton$ and set
$\eta:=\one_I\operatorname{sgn}(g)$, with $\operatorname{sgn}(0):=0$.
Then $|\eta|\leq\one_I$ and $\medm(I)<\infty$.
By Lemma~\ref{l.densg1} with $q=2$, choose $f_n\in\core$ such that
$\|\eta-\wgrad f_n\|_{L^2(\skeleton,\medm)}\to0$. Hence, by the Cauchy--Schwarz inequality,
\begin{align}
 0\leq\int_I|g|\dif\medm
 &=\langle g,\eta\rangle_{L^2(\skeleton,\medm)}
 =\langle g,\eta-\wgrad f_n\rangle_{L^2(\skeleton,\medm)}\\
 &\leq
 \|g\|_{L^2(\skeleton,\medm)}
 \|\eta-\wgrad f_n\|_{L^2(\skeleton,\medm)}
 \overset{\eqref{e.gradd1}}{\to}0 .
\end{align}
Thus $g=0$ $\medm$-a.e.\ on $I$. By Remark~\ref{r.arcs}, $g=0$ $\medm$-a.e.\ on $\skeleton$, and the range of $\riesz_2$ is dense.
Since $\riesz_2$ is an isometry and $L^2(\ambient,\meas)$ is complete,
its range is also closed. Hence $\riesz_2$ is surjective. The properties \eqref{e.BTheta.adjoint1} and \eqref{e.BTheta.adjoint2} follow from the definition of the adjoint operator. 
\qedhere
\end{enumerate}
\end{proof}

Since $\{P_t\}_{t\in(0,\infty)}$ is a symmetric sub-Markov
semigroup, it extends consistently to a strongly continuous
contraction semigroup on $L^p(\ambient,\meas)$ for every
$p\in[1,\infty)$; see \cite[Theorem~1.4.1]{Dav89}.

Let $\proj$ be the spectral measure of the nonnegative self-adjoint operator $-\gen$. Write $\proj_\lambda:=\proj((-\infty,\lambda])$ for its spectral family, $\lambda\in[0,\infty)$. For every Borel-measurable function $f:[0,\infty)\to[-\infty,\infty]$ such that $\proj(\{\abs{f}=\infty\})=0$, the spectral calculus defines a linear operator $(f(-\gen),\Dom(f(-\gen)))$ by
\begin{equation}\label{e.spec}
  \begin{dcases}
    \Dom(f(-\gen)):=\Sett{u\in L^{2}(\ambient,\meas)}{\int_{[0,\infty)}\abs{f(\lambda)}^{2}\dif\langle \proj_{\lambda}u,u\rangle<\infty};\\
    \langle f(-\gen)u,v\rangle :=\int_{[0,\infty)}f(\lambda)\dif\langle \proj_{\lambda}u,v\rangle, \ u\in \Dom(f(-\gen)) \text{ and } v\in L^{2}(\ambient,\meas).
  \end{dcases}
\end{equation}
Here and throughout the remainder of the paper, all inner products involving the projection operators $\{\proj_{\lambda}\}_{\lambda\in[0,\infty)}$ are taken in the Hilbert space $L^{2}(\ambient,\meas)$. In the rest of this paper, for any $\theta\in(0,\infty)$ and $t\in(0,\infty)$, we denote the bounded linear operator on $L^{2}(\ambient,\meas)$ obtained from \eqref{e.spec} by taking $f(\lambda)=\lambda^{\theta}\exp(-t\lambda)$, $\lambda\in[0,\infty)$, by $(-\gen)^{\theta}P_{t}$.

\subsection{Heat semigroup estimates}
In this subsection, we first estimate the kernel of $(-\gen)^\theta P_t$ in \eqref{e.fracke}. We then prove the gradient estimate \eqref{e.gradie} for $P_t$, which we use to prove the strong reverse inequality in Theorem~\ref{t.critx}-\ref{it.suffi}.

\begin{lemma}\label{l.semes}
  Let $(\ambient,\metric,\meas)$ be $\AR{\dimf}$ with $\dimf\in[1,\infty)$. Let $p\in[1,\infty]$.
  \begin{enumerate}[label=\textup{(\arabic*)}, align=right, leftmargin=*, topsep=5pt, parsep=0pt, itemsep=2pt]
    \item\label{it.analy} For $\theta\in(0,1]$, there exists $C_{p,\theta}\in(0,\infty)$ such that, for all $t\in(0,\infty)$, we have
    \begin{equation}\label{e.analy}
      \|(-\gen)^{\theta} P_{t} h\|_{L^{p}(\ambient,\meas)} \leq C_{p,\theta}t^{-\theta}\|h\|_{L^{p}(\ambient,\meas)},\ \text{ for all }h\in L^{p}(\ambient,\meas)\cap L^{2}(\ambient,\meas).
    \end{equation}
    If $p\in[1,\infty)$, then $(-\gen)^{\theta} P_{t}$ extends uniquely to a bounded operator on $L^{p}(\ambient,\meas)$.
    \item\label{it.ultra} There exists $C_{p,1}\in(0,\infty)$ such that, for all $t\in(0,\infty)$,
    \begin{equation}\label{e.ultra}
      \|P_{t} h\|_{L^{p}(\ambient,\meas)}\leq C_{p,1}t^{-\dimf(1-1/p)/\dimw}\|h\|_{L^{1}(\ambient,\meas)},\ \text{ for all }h\in L^{1}(\ambient,\meas)\cap L^{2}(\ambient,\meas).
    \end{equation}
    \item\label{it.fracke} For every $\theta\in(0,1)$ and $t\in(0,\infty)$, the operator $(-\gen)^{\theta} P_{t}$ admits a measurable kernel $q_{t}^{(\theta)}:\ambient\times\ambient\to\bR$ such that
    \begin{equation}\label{e.fracke}
      \left|q_{t}^{(\theta)}(x,y)\right| \leq \frac{C_{\theta}}{\left(t^{1/\dimw}+\metric(x,y)\right)^{\dimf+\theta \dimw}}, \ \text {for all } x,y\in\ambient\ \text{and all }t\in(0,\infty),
    \end{equation}
    for some constant $C_{\theta}\in(0,\infty)$ depending only on $\theta$ and the geometric and heat kernel constants. In particular,
    \begin{equation}\label{e.frack1}
      \left|q_{t}^{(\theta)}(x,y)\right| \leq C_{\theta} \metric(x,y)^{-\dimf-\theta \dimw},\ \text{ if $\metric(x,y)^{\dimw}\geq t$}.
    \end{equation}
    Moreover, for every $h\in L^{1}(\ambient,\meas)\cap L^{2}(\ambient,\meas)$,
    \begin{equation}\label{e.fracin}
      (-\gen)^{\theta} P_{t} h(x) = \int_{\ambient} q_{t}^{(\theta)}(x,y)h(y)\dif\meas(y),\ \text{ for $\meas$-a.e.\ $x\in \ambient$}.
    \end{equation}
  \end{enumerate}
\end{lemma}
\begin{proof}
  \begin{enumerate}[label=\textup{(\arabic*)}, align=right, leftmargin=*, topsep=5pt, parsep=0pt, itemsep=2pt]
    \item[\ref{it.analy}] Integrating \eqref{e.ddt} over $B(y,t^{1/\dimw})$ and the annuli $B(y,2^jt^{1/\dimw})\setminus B(y,2^{j-1}t^{1/\dimw})$ gives
      \begin{align}
      &\phantom{\ \leq} \sup_{y\in\ambient}\int_{\ambient}\left|\frac{\dif}{\dif t}p_t(x,y)\right|\dif\meas(x)\\
      &\overset{\eqref{e.ddt},\eqref{e.ahlfor}}{\leq} Ct^{-1-\dimf/\dimw}t^{\dimf/\dimw}
      \left(1+\sum_{j=1}^{\infty}2^{j\dimf}e^{-c2^{(j-1)\dimw/(\dimw-1)}}\right)\leq \frac{C}{t},\label{e.Lpt0}
      \end{align}
    which, together with \eqref{e.kernel}, gives
    \begin{equation}
      \norm{(-\gen) P_{t} h}_{L^{1}(\ambient,\meas)}\overset{\eqref{e.kernel},\eqref{e.Lpt0}}{\leq} \frac{C}{t}\norm{h}_{L^{1}(\ambient,\meas)},\ \text{ for all }h\in L^{1}(\ambient,\meas)\cap L^{2}(\ambient,\meas),
    \end{equation}
    and
    \begin{equation}
      \norm{(-\gen) P_{t} h}_{L^{\infty}(\ambient,\meas)}\overset{\eqref{e.kernel},\eqref{e.Lpt0}}{\leq} \frac{C}{t}\norm{h}_{L^{\infty}(\ambient,\meas)},\ \text{ for all }h\in L^{\infty}(\ambient,\meas)\cap L^{2}(\ambient,\meas).
    \end{equation}
    Let $K_t(x,y):=|\frac{\dif}{\dif t}p_t(x,y)|$. H\"older's inequality with respect to $K_t(x,y)\dif\meas(y)$ and then Tonelli's theorem give, for $p\in[1,\infty)$,
    \begin{equation*}
      \begin{aligned}
        &\phantom{\ \leq} \int_{\ambient}\left(\int_{\ambient}K_t(x,y)|h(y)|\dif\meas(y)\right)^p\dif\meas(x)\\
        &\overset{\eqref{e.Lpt0}}{\leq}\left(\frac{C}{t}\right)^{p-1}\int_{\ambient}\int_{\ambient}K_t(x,y)|h(y)|^p\dif\meas(y)\dif\meas(x)\overset{\eqref{e.Lpt0}}{\leq}\left(\frac{C}{t}\right)^p\|h\|_{L^p(\ambient,\meas)}^p.
      \end{aligned}
    \end{equation*}
 Together with the estimates for $p\in\{1,\infty\}$, we have, for every $p\in[1,\infty]$,
    \begin{equation}\label{e.Lpt1}
      \norm{(-\gen) P_{t} h}_{L^{p}(\ambient,\meas)}\leq \frac{C}{t}\norm{h}_{L^{p}(\ambient,\meas)},\ \text{ for all }h\in L^{p}(\ambient,\meas)\cap L^{2}(\ambient,\meas).
    \end{equation}
    For $p\in[1,\infty)$, by the density of $L^{p}(\ambient,\meas)\cap L^{2}(\ambient,\meas)$ in $L^{p}(\ambient,\meas)$, $(-\gen) P_{t}$ extends uniquely to a bounded operator on $L^{p}(\ambient,\meas)$.

    Let $\theta\in(0,1)$. For $h,g\in L^{2}(\ambient,\meas)$, the spectral theorem and Fubini's theorem give
      \begin{align}
        &\phantom{\ \leq} \int_{[0,\infty)}\int_{0}^{\infty} s^{-\theta}\lambda e^{-(t+s)\lambda}\dif s \bigl|\dif\langle\proj_{\lambda}h,g\rangle_{L^{2}(\ambient,\meas)}\bigr|\\
        &=\Gamma(1-\theta)\int_{[0,\infty)}\lambda^{\theta}e^{-t\lambda} \bigl|\dif\langle\proj_{\lambda}h,g\rangle_{L^{2}(\ambient,\meas)}\bigr|\\
        &\leq\Gamma(1-\theta)\left(\theta e^{-1}t^{-1}\right)^{\theta} \|h\|_{L^{2}(\ambient,\meas)}\|g\|_{L^{2}(\ambient,\meas)}<\infty,
      \end{align}
    and therefore,
    \begin{align}
      \langle (-\gen)^{\theta} P_{t} h,g\rangle_{L^{2}(\ambient,\meas)}&\overset{\eqref{e.spec}}{=}\int_{[0,\infty)}\lambda^{\theta}\exp(-t\lambda)\dif\langle\proj_{\lambda}h,g\rangle_{L^{2}(\ambient,\meas)}\\
      &=\int_{[0,\infty)}\frac{1}{\Gamma(1-\theta)} \left(\int_{0}^{\infty} s^{-\theta}\lambda e^{-(t+s)\lambda}\dif s\right)\dif\langle\proj_{\lambda}h,g\rangle_{L^{2}(\ambient,\meas)}\\
      &=\frac{1}{\Gamma(1-\theta)} \int_{0}^{\infty} s^{-\theta} \langle -\gen P_{t+s}h,g\rangle_{L^{2}(\ambient,\meas)}\dif s.\label{e.Lpt2}
    \end{align}

    Take $h\in L^{p}(\ambient,\meas)\cap L^{2}(\ambient,\meas)$ and $g\in L^{p'}(\ambient,\meas)\cap L^{2}(\ambient,\meas)$, where $1 / p + 1 / p'= 1 $. Then
    \begin{equation}\label{e.Lpt21}
      \int_{0}^{\infty}\frac{s^{-\theta}}{t+s}\dif s =t^{-\theta}\int_{0}^{1}u^{-\theta}(1-u)^{\theta-1}\dif u =t^{-\theta}\Gamma(1-\theta)\Gamma(\theta)<\infty,\ \text{ for }\theta\in(0,1).
    \end{equation}
    By \eqref{e.Lpt2} and H\"{o}lder's inequality,
      \begin{align}
        \abs{\langle (-\gen)^{\theta} P_{t} h,g\rangle_{L^{2}(\ambient,\meas)}}&\overset{\eqref{e.Lpt1},\eqref{e.Lpt2}}{\lesssim}\frac{1}{\Gamma(1-\theta)} \int_{0}^{\infty} \frac{s^{-\theta}}{t+s}\dif s \norm{h}_{L^{p}(\ambient,\meas)}\norm{g}_{L^{p'}(\ambient,\meas)}\\
        &\overset{\eqref{e.Lpt21}}{\lesssim} \Gamma(\theta)t^{-\theta}\norm{h}_{L^{p}(\ambient,\meas)}\norm{g}_{L^{p'}(\ambient,\meas)}.
      \end{align}
    which, combined with \cite[Theorem 6.14]{Fol99}, gives \eqref{e.analy} for $p\in[1,\infty]$. For $p\in[1,\infty)$, by the density of $L^{p}(\ambient,\meas)\cap L^{2}(\ambient,\meas)$ in $L^{p}(\ambient,\meas)$, $(-\gen)^{\theta} P_{t}$ extends uniquely to a bounded operator on $L^{p}(\ambient,\meas)$.
    \item[\ref{it.ultra}] By \eqref{e.HKE}, $\sup_{x,y\in\ambient}p_{t}(x,y)\lesssim t^{-\dimf/\dimw}$, which implies
    \begin{equation}
      \|P_{t}h\|_{L^{\infty}(\ambient,\meas)} \overset{\eqref{e.heatke},\eqref{e.HKE}}{\lesssim} t^{-\dimf/\dimw}\|h\|_{L^{1}(\ambient,\meas)},\ \text{ for all }h\in L^{1}(\ambient,\meas)\cap L^{2}(\ambient,\meas).
    \end{equation}
    For $p\in[1,\infty)$, the elementary norm inequality and $L^1(\ambient,\meas)$ contractivity give
    \begin{equation*}
      \|P_th\|_{L^p(\ambient,\meas)}^p\leq\|P_th\|_{L^\infty(\ambient,\meas)}^{p-1}\|P_th\|_{L^1(\ambient,\meas)}
      \overset{\eqref{e.HKE}}{\leq}C_p t^{-(p-1)\dimf/\dimw}\|h\|_{L^1(\ambient,\meas)}^p.
    \end{equation*}
    This proves~\eqref{e.ultra}, including $p=1$. The case $p=\infty$ was proved above.
    \item[\ref{it.fracke}] We write $k(t,x,y):=\frac{\dif}{\dif t}p_{t}(x,y)$. Then $k$ is jointly measurable by Theorem~\ref{t.dirich}-\ref{it.ddt}. Fix $t\in(0,\infty)$ and $\theta\in(0,1)$. Define
    \begin{equation}\label{e.qdef}
      q_{t}^{(\theta)}(x,y) := -\frac{1}{\Gamma(1-\theta)} \int_{0}^{\infty} s^{-\theta}k(t+s,x,y) \dif s,\ x,y\in\ambient.
    \end{equation}
    The integral is absolutely convergent. Indeed, by \eqref{e.ddt} and Corollary~\ref{c.integr},
      \begin{align}
        |q_{t}^{(\theta)}(x,y)| &\overset{\eqref{e.qdef},\eqref{e.ddt}}{\leq} C_{\theta} \int_{0}^{\infty} \frac{s^{-\theta}}{(t+s)^{1+\dimf/\dimw}} \exp\left( -c\left(\frac{\metric(x,y)^{\dimw}}{t+s}\right)^{\frac{1}{\dimw-1}} \right)\dif s\\
        &\overset{\eqref{e.integ2}}{\leq} \frac{C_{\theta}}{\left(t^{1/\dimw}+\metric(x,y)\right)^{\dimf+\theta \dimw}}.
      \end{align}
    Since $k$ is measurable, $q_{t}^{(\theta)}$ is measurable. The estimate \eqref{e.frack1} is evident. By \eqref{e.fracke},  \eqref{e.ahlfor} and a decomposition into dyadic annuli,
      \begin{align}
        &\phantom{\ \leq}\sup_{x\in\ambient}\int_{\ambient}|q_{t}^{(\theta)}(x,y)|\dif\meas(y)\\
         &\overset{\eqref{e.fracke},\eqref{e.ahlfor}}{\leq} C_{\theta}\left(t^{-\frac{\dimf+\theta\dimw}{\dimw}}t^{\frac{\dimf}{\dimw}} +\sum_{j=1}^{\infty}(2^{j-1}t^{1/\dimw})^{-\dimf-\theta\dimw} (2^{j}t^{1/\dimw})^{\dimf}\right)\leq C_{\theta}t^{-\theta}.\label{e.qschur}
      \end{align}
  Thus for $h\in L^{2}(\ambient,\meas)$, by the Cauchy--Schwarz inequality and Fubini's theorem,
      \begin{align}
        &\phantom{\ \leq}\int_{\ambient}\left(\int_{\ambient}|q_{t}^{(\theta)}(x,y)||h(y)|\dif\meas(y)\right)^{2}\dif\meas(x)\\
        &\overset{\eqref{e.qschur}}{\leq} C_{\theta}t^{-\theta}\int_{\ambient}\int_{\ambient} |q_{t}^{(\theta)}(x,y)||h(y)|^{2}\dif\meas(y)\dif\meas(x) \overset{\eqref{e.qschur}}{\leq} C_{\theta}t^{-2\theta}\|h\|_{L^{2}(\ambient,\meas)}^{2}.\label{e.qschu1}
      \end{align}
    In particular, the inner integral converges absolutely for $\meas$-a.e. $x\in\ambient$ and defines an $L^{2}(\ambient,\meas)$ function. For $h \in L ^{1}(\ambient,\meas) \cap L ^{2}(\ambient,\meas)$ and $\phi\in C_{c}(\ambient)$,
      \begin{align}
        &\phantom{\ \leq}\int_{0}^{\infty}\int_{\ambient}\int_{\ambient} s^{-\theta}|k(t+s,x,y)||h(y)||\phi(x)|\dif\meas(y)\dif\meas(x)\dif s\\
        &\overset{\eqref{e.Lpt0}}{\leq}C\|h\|_{L^{1}(\ambient,\meas)}\|\phi\|_{L^{\infty}(\ambient,\meas)} \int_{0}^{\infty}\frac{s^{-\theta}}{t+s}\dif s \\
        &\overset{\eqref{e.Lpt21}}{=}C\Gamma(1-\theta)\Gamma(\theta)t^{-\theta} \|h\|_{L^{1}(\ambient,\meas)}\|\phi\|_{L^{\infty}(\ambient,\meas)}<\infty.
      \end{align}
    Fubini's theorem therefore gives
      \begin{align}
        &\phantom{\ \leq}\left\langle\int_{\ambient}q_{t}^{(\theta)}(\cdot,y)h(y)\dif\meas(y),\phi\right\rangle_{L^{2}(\ambient,\meas)}\\
        &\overset{\eqref{e.qdef}}{=}-\frac{1}{\Gamma(1-\theta)} \int_{\ambient}\left(\int_{\ambient}\left(\int_{0}^{\infty} s^{-\theta}k(t+s,x,y) \dif s\right) h(y)\dif\meas(y)\right)\phi(x)\dif \meas(x)\\
        &=-\frac{1}{\Gamma(1-\theta)}\int_{0}^{\infty} s^{-\theta}\left(\int_{\ambient}\phi(x)\left(\int_{\ambient}k(t+s,x,y)h(y)\dif\meas(y)\right)\dif \meas(x)\right)\dif s\\
        &\overset{\eqref{e.kernel}}{=}-\frac{1}{\Gamma(1-\theta)}\int_{0}^{\infty} s^{-\theta}\langle\gen P_{t+s}h,\phi\rangle_{L^{2}(\ambient,\meas)}\dif s\overset{\eqref{e.Lpt2}}{=}	\langle (-\gen)^{\theta} P_{t} h,\phi\rangle_{L^{2}(\ambient,\meas)}.
      \end{align}
    Since $C_{c}(\ambient)$ is dense in $ \allowbreak L ^{ \allowbreak 2 }( \allowbreak \ambient , \allowbreak \meas ) $, this proves \eqref{e.fracin}.
    \qedhere 
  \end{enumerate}
\end{proof}

\begin{proposition}\label{p.gradie}
 Let $(\ambient,\metric,\meas)$ be $\AR{\dimf}$ with $\dimf\in[1,\infty)$. For every $q\in[1,\infty]$, $t\in(0,\infty)$ and $h\in L^q(\ambient,\meas)\cap L^2(\ambient,\meas)$,
  \begin{equation}\label{e.gradie}
    \|\wgrad P_t h\|_{L^q(\skeleton,\medm)}\leq C_qt^{-\critic_q}\|h\|_{L^q(\ambient,\meas)}.
  \end{equation}
\end{proposition}
\begin{proof}
This is proved in \cite[Theorem 5.7]{BCY25}.
\end{proof}

\subsection{A representation formula}
Suppose $f\in\core$ with $\wgrad f=\sum_{j=1}^{N}\kappa_{j}\vecone_{(a_{j},b_{j})}$. Let $Z_{\wgrad f}$  and  $c_{f}:=\mu_{\wgrad f}$ be defined as in Definition \ref{d.endpoi}. Then, for $s\in(0,\infty)$ and $h\in L^{2}(\ambient,\meas)$,
\begin{align}
  \langle-\gen P_{s}f,h\rangle_{L^{2}(\ambient,\meas)}&=\form(f,P_{s}h)\overset{\eqref{e.formde},\eqref{e.arcftc}}{=}\sum_{j=1}^{N}\kappa_{j}\bigl(P_{s}h(b_{j})-P_{s}h(a_{j})\bigr)\\
  &\overset{\eqref{e.heatke}}{=}\int_{\ambient}\sum_{z\in Z_{\wgrad f}}c_{f}(z)p_{s}(z,y)h(y)\dif\meas(y).
\end{align}
 Since this holds for all $h\in L^2(\ambient,\meas)$, we have
\begin{equation}\label{e.repres}
  -\gen P_{s}f=\sum_{z\in Z_{\wgrad f}}c_{f}(z)p_{s}(z,\cdot)\ \ \meas\text{-a.e. on }\ambient,\ \text{ for all }s\in(0,\infty)\text{ and }f\in\core.
\end{equation}
\begin{lemma}\label{l.semigr}
 Let $(\ambient,\metric,\meas)$ be $\AR{\dimf}$ with $\dimf\in[1,\infty)$. Let $p\in[1,\infty)$, $\theta\in(0,1)$ and $h\in L^{p}(\ambient,\meas)\cap L^{2}(\ambient,\meas)$. Suppose that
  \begin{equation}\label{e.core0}
    \int_{0}^{1}s^{-\theta}\|(-\gen)P_{s}h\|_{L^{p}(\ambient,\meas)}\dif s<\infty.
  \end{equation}
  Then the integral
  \begin{equation}\label{e.coreli}
    T_{h}(x):=\frac{1}{\Gamma(1-\theta)}\int_{0}^{\infty}s^{-\theta}(-\gen)P_{s}h(x)\dif s
  \end{equation}
  converges absolutely for $\meas$-a.e. $x\in\ambient$ and defines a function in $L^{p}(\ambient,\meas)$.  Moreover,
  \begin{equation}\label{e.corePt}
    (-\gen)^{\theta}P_{t}h=P_{t}T_{h}\to T_{h}\quad\text{in }L^{p}(\ambient,\meas)\text{ as }t\downarrow0.
  \end{equation}
  If also $h\in\Dom((-\gen)^{\theta})$, then $T_{h}=(-\gen)^{\theta}h$ $\meas$-a.e.
\end{lemma}
\begin{proof}
  Define
  \begin{equation}
    u(s,x):=-\int_{\ambient}\frac{\dif}{\dif s}p_s(y,x)h(y)\dif\meas(y),
    \qquad (s,x)\in(0,\infty)\times\ambient.
  \end{equation}
   For each $s\in(0,\infty)$ and $x\in\ambient$, integration over $B(x,s^{1/\dimw})$ and the annuli $B(x,2^js^{1/\dimw})\setminus B(x,2^{j-1}s^{1/\dimw})$ gives
  \begin{equation}
    \begin{aligned}
      \int_{\ambient}\left|\frac{\dif}{\dif s}p_s(y,x)\right|^2\dif\meas(y)
      \overset{\eqref{e.ddt},\eqref{e.ahlfor}}{\leq}
        Cs^{-2-\dimf/\dimw}
         \left(1+\sum_{j=1}^\infty 2^{j\dimf}e^{-c2^{(j-1)\dimw/(\dimw-1)}}\right)\leq Cs^{-2-\dimf/\dimw}.
    \end{aligned}
  \end{equation}
 Thus the Cauchy--Schwarz inequality gives the absolute convergence of the integral defining $u(s,x)$, and by Theorem~\ref{t.dirich}, $u$ is jointly measurable.  By \eqref{e.kernel}, 
 \begin{equation}\label{e.p=-genp}
 	u(s,\cdot)=(-\gen)P_sh,\ \text{$\meas$-a.e. for each $s\in(0,\infty)$.}
 \end{equation}    
 By Minkowski's inequality and the monotone convergence theorem, 
  \begin{align}\label{e.coreab}
   &\phantom{\ \leq} \left\|\int_0^\infty s^{-\theta}|u(s,\cdot)|\dif s\right\|_{L^p(\ambient,\meas)}  \leq\int_0^\infty s^{-\theta}\|u(s,\cdot)\|_{L^p(\ambient,\meas)}\dif s\\
  &= \int_{0}^{1}s^{-\theta}\|(-\gen)P_{s}h\|_{L^{p}(\ambient,\meas)}\dif s+\int_1^\infty s^{-\theta}\|u(s,\cdot)\|_{L^p(\ambient,\meas)}\dif s\\
  & \overset{\eqref{e.Lpt1}}{\leq} \int_{0}^{1}s^{-\theta}\|(-\gen)P_{s}h\|_{L^{p}(\ambient,\meas)}\dif s+  C\|h\|_{L^p(\ambient,\meas)}\int_1^\infty s^{-1-\theta}\dif s \overset{\eqref{e.core0}}{<}\infty.
  \end{align}

  Fix $t\in(0,\infty)$ and put $W(x):=\int_0^\infty s^{-\theta}|u(s,x)|\dif s$. By Tonelli's theorem
  \begin{equation}
    \int_0^\infty\int_{\ambient}s^{-\theta}p_t(x,y)|u(s,y)|\dif\meas(y)\dif s
    =P_tW(x)<\infty
    \quad\text{for }\meas\text{-a.e. }x\in\ambient.
  \end{equation}
For each $s\in(0,\infty)$, by \eqref{e.p=-genp}, $P_tu(s,\cdot)=u(t+s,\cdot)$ in $L^2(\ambient,\meas)$. By Fubini's theorem,
  \begin{equation}\label{e.coresh}
    P_tT_h(x)=\frac1{\Gamma(1-\theta)}\int_0^\infty s^{-\theta}u(t+s,x)\dif s
    \quad\text{for }\meas\text{-a.e. }x\in\ambient.
  \end{equation}
  By Minkowski's inequality,
  \begin{equation}
    \|P_tT_h\|_{L^2(\ambient,\meas)}
    \overset{\eqref{e.Lpt1}}{\leq}
      \frac{C\|h\|_{L^2(\ambient,\meas)}}{\Gamma(1-\theta)}
          \int_0^\infty\frac{s^{-\theta}}{t+s}\dif s
    \overset{\eqref{e.Lpt21}}{\leq}
      C\Gamma(\theta)t^{-\theta}\|h\|_{L^2(\ambient,\meas)}.
  \end{equation}
  For every $g\in L^2(\ambient,\meas)$,
  \begin{equation}
    \begin{aligned}
      \langle P_tT_h,g\rangle_{L^2(\ambient,\meas)}
      &=\frac1{\Gamma(1-\theta)}\int_0^\infty
        s^{-\theta}\langle(-\gen)P_{t+s}h,g\rangle_{L^2(\ambient,\meas)}\dif s\\
      &\overset{\eqref{e.Lpt2}}{=}
        \langle(-\gen)^\theta P_th,g\rangle_{L^2(\ambient,\meas)}.
    \end{aligned}
  \end{equation}
  Thus $P_tT_h=(-\gen)^\theta P_th$ $\meas$-a.e. on $\ambient$. By the strong continuity on $L^p(\ambient,\meas)$,
  \begin{equation}
    \lim_{t\downarrow0}\|(-\gen)^\theta P_th-T_h\|_{L^p(\ambient,\meas)}
    = \lim_{t\downarrow0}\|P_tT_h-T_h\|_{L^p(\ambient,\meas)}=0.
  \end{equation}
  If $h\in\Dom((-\gen)^\theta)$, then
  \begin{equation}
   \lim_{t\downarrow0} \|(-\gen)^\theta P_th-(-\gen)^\theta h\|_{L^2(\ambient,\meas)}^2
    \overset{\eqref{e.spec}}{=}\lim_{t\downarrow0}
      \int_{[0,\infty)}\lambda^{2\theta}|e^{-t\lambda}-1|^2
        \dif\langle\proj_\lambda h,h\rangle=0.
  \end{equation}
Thus $T_{h}=(-\gen)^{\theta}h$ $\meas$-a.e. on $\ambient$.
\end{proof}

\begin{corollary}\label{c.coreco}
Let $(\ambient,\metric,\meas)$ be $\AR{\dimf}$ with $\dimf\in[1,\infty)$.  Let $p\in[1,\infty)$ and $h\in\core$. With $\theta=\critic_p$, the integral $T_h$ in \eqref{e.coreli} belongs to $L^p(\ambient,\meas)$ and satisfies \eqref{e.corePt}. If $h\in\Dom((-\gen)^{\critic_p})$, then $T_h=(-\gen)^{\critic_p}h$.
\end{corollary}
\begin{proof}
It suffices to verify the assumption \eqref{e.core0}.
  Let $c_h(z)$, $z\in Z_h$, be the coefficients in \eqref{e.repres}. By \eqref{e.HKE},
  \begin{equation}\label{e.hklp}
    \|p_s(z,\cdot)\|_{L^{p}(\ambient,\meas)}^p
    \leq\|p_s(z,\cdot)\|_{L^{\infty}(\ambient,\meas)}^{p-1}\int_{\ambient} p_s(z,y)\dif\meas(y)
    \overset{\eqref{e.HKE}}{\leq} C_p s^{-(p-1)\dimf/\dimw},\quad p\in[1,\infty).
  \end{equation}
By the triangle inequality,
  \begin{equation}\label{e.coresm}
    \|(-\gen)P_sh\|_{L^{p}(\ambient,\meas)}
    \overset{\eqref{e.repres},\eqref{e.hklp}}{\leq} C_p\left(\sum_{z\in Z_h}|c_h(z)|\right)s^{-\dimf(1-1/p)/\dimw}
    =:C_{p,h}s^{-\dimf(1-1/p)/\dimw}.
  \end{equation}
  The identity
  \begin{equation}\label{e.sumexp}
    \critic_p+\frac{\dimf(1-1/p)}{\dimw}=1-\frac1{p\dimw}<1
  \end{equation}
  yields
  \begin{equation}
    \int_0^1s^{-\critic_p}\|(-\gen)P_sh\|_{L^{p}(\ambient,\meas)}\dif s
    \overset{\eqref{e.coresm},\eqref{e.sumexp}}{\leq} C_{p,h}\int_0^1s^{-1+1/(p\dimw)}\dif s
    =p\dimw C_{p,h}<\infty.
  \end{equation}
  This verifies \eqref{e.core0} in Lemma~\ref{l.semigr}, with $\theta=\critic_p$.
\end{proof}

\subsection{Test functions for the Riesz transform}

For $p\in[1,\infty)$, we define
\begin{equation}\label{e.specc-strict}
  \collectD_{p}^{\rm sp} :=\operatorname{span} \Sett{(I-P_{R})^{2}P_{\varepsilon} h} {\varepsilon,R\in(0,\infty),\ h\in L^{2}(\ambient,\meas)\cap L^{p}(\ambient,\meas) \cap L^{p,1}(\ambient,\meas)},
\end{equation} and
\begin{equation}\label{e.specc}
  \collectD_{p} :=\operatorname{span} \Sett{P_{\varepsilon} h} {\varepsilon\in(0,\infty),\ h\in L^{2}(\ambient,\meas)\cap L^{p}(\ambient,\meas) \cap L^{p,1}(\ambient,\meas)}.
\end{equation}

\begin{lemma}\label{l.densp}
 Let $(\ambient,\metric,\meas)$ be $\AR{\dimf}$ with $\dimf\in[1,\infty)$. Let $p\in[1,\infty)$.
  \begin{enumerate}[label=\textup{(\arabic*)}, align=right, leftmargin=*, topsep=5pt, parsep=0pt, itemsep=2pt]
    \item\label{it.rieszrg} One has $ \collectD_{p}^{\rm sp}\subset\collectD_{p} \subset\Dom(( -\gen)^{1/2-\critic_{p}})$ and
    \begin{equation}\label{e.strict-core-domains}
      \collectD_{p}^{\rm sp}\subset\Dom(( -\gen)^{-\critic_{p}}), \qquad ( -\gen)^{-\critic_{p}}(\collectD_{p}^{\rm sp}) \subset\Dom(( -\gen)^{\frac{1}{2}})=\domain.
    \end{equation}
    \item\label{it.denserg} The space $\collectD_{p}$ is dense in both $L^{p}(\ambient,\meas)$ and $L^{p,1}(\ambient,\meas)$.  If the function being approximated also belongs to $L^{2}(\ambient,\meas)$, the approximating sequence in $\collectD_{p}$ can be chosen to converge in $L^{2}(\ambient,\meas)$ as well. If $p\in(1,\infty)$, the same assertions hold for $\collectD_{p}^{\rm sp}$.
  \end{enumerate}
\end{lemma}

\begin{proof}
  For $h\in L^{2}(\ambient,\meas)\cap L^{p}(\ambient,\meas) \cap L^{p,1}(\ambient,\meas)$ and $\varepsilon,R\in(0,\infty)$,
  \begin{equation}
    (I-P_{R})^{2}P_{\varepsilon} h =P_{\varepsilon} h-2P_{R+\varepsilon}h+P_{2R+\varepsilon}h \in\collectD_{p}.
  \end{equation}
  Hence $\collectD_{p}^{\rm sp}\subset\collectD_{p}$. If $p\in(1,\infty)$, then we fix $1<r_{0}<p<r_{1}<\infty$ and $\theta\in(0,1)$ such that
  \begin{equation}\label{e.denspcrr}
    \frac{1}{p}=\frac{1-\theta}{r_{0}}+\frac{\theta}{r_{1}}.
  \end{equation}
  For $h=h_{0}+h_{1}$ with $h_{i}\in L^{r_{i}}(\ambient,\meas)$, $L ^{r _{i} }$-contractivity gives
  \begin{equation}
    \|P_{t}h_{0}\|_{L^{r_{0}}(\ambient,\meas)}+s\|P_{t}h_{1}\|_{L^{r_{1}}(\ambient,\meas)} \leq\|h_{0}\|_{L^{r_{0}}(\ambient,\meas)}+s\|h_{1}\|_{L^{r_{1}}(\ambient,\meas)},\qquad s,t>0.
  \end{equation}
  Taking the infimum over decompositions, and then using \eqref{e.lorinterproof}, yields
    \begin{align}
      \sup_{t\in(0,\infty)}\|P_{t}h\|_{L^{p,1}(\ambient,\meas)} \leq C_{p}\|h\|_{L^{p,1}(\ambient,\meas)},\ \text{ for all $h\in L^{p,1}(\ambient,\meas)\cap L^{2}(\ambient,\meas)$.}
    \end{align}
 For $p=1$, this follows from $L^{1,1}(\ambient,\meas)=L^{1}(\ambient,\meas)$. In particular,
  \begin{equation}
    \collectD_{p}^{\rm sp}\subset\collectD_{p} \subset L^{2}(\ambient,\meas)\cap L^{p}(\ambient,\meas) \cap L^{p,1}(\ambient,\meas).
  \end{equation}

  \begin{enumerate}[label=\textup{(\arabic*)}, align=right, leftmargin=*, topsep=5pt, parsep=0pt, itemsep=2pt]
    \item[\ref{it.rieszrg}] If $\dimf=1$, then the assertions are evident. It remains to consider $\dimf\in(1,\infty)$. By Theorem \ref{t.dirich}, $\Ker(-\gen)=\{0\}$ and $\proj(\{0\})=0$. Fix $h\in L^{2}(\ambient,\meas)\cap L^{p}(\ambient,\meas) \cap L^{p,1}(\ambient,\meas)$ and $\varepsilon>0$.

    If $p\in[1,2)$, by interpolating \eqref{e.ultra} with $L^{2}$-contractivity, we have
    \begin{equation}
      \|P_{t}h\|_{L^{2}(\ambient,\meas)} \leq C_{p} t^{-\frac{\dimf}{\dimw}\left(\frac{1}{p}-\frac{1}{2}\right)}\|h\|_{L^{p}(\ambient,\meas)}, \ \text{ for all } t\in(0,\infty).
    \end{equation}
    Note that
    \begin{equation}\label{e.lam2p-1}
      \lambda^{-2\critic_{p}+1} =\frac{1}{\Gamma(2\critic_{p}-1)}\int_{0}^{\infty} t^{2\critic_{p}-2}e^{-t\lambda}\dif t,\ \text{ for all $\lambda\in(0,\infty)$},
    \end{equation}
    Combining \eqref{e.lam2p-1} with Fubini's theorem and \eqref{e.spec}, we obtain
      \begin{align}
        &\phantom{\ \leq}\int_{[0,\infty)}\lambda^{-2\critic_{p}+1}e^{-2\varepsilon\lambda} \dif\langle\proj_{\lambda} h,h\rangle\overset{\eqref{e.lam2p-1},\eqref{e.spec}}{=}\frac{1}{\Gamma(2\critic_{p}-1)}\int_{0}^{\infty} t^{2\critic_{p}-2}\|P_{\varepsilon+t/2}h\|_{L^{2}(\ambient,\meas)}^{2} \dif t\\
        &\lesssim\|h\|_{L^{2}(\ambient,\meas)}^{2}\int_{0}^{1}t^{2\critic_{p}-2}\dif t +\|h\|_{L^{p}(\ambient,\meas)}^{2} \int_{1}^{\infty} t^{-1-( 2 \allowbreak/ p - \allowbreak1 )/ \dimw }\dif t<\infty.
      \end{align}
    Here $2\critic_{p}-1>0$ and $2/p-1>0$. Therefore
    \begin{equation}
      P_{\varepsilon} h\in\Dom((-\gen)^{1/2-\critic_{p}}),\ \text{ for all } p\in[1,2) \text{ and all }h\in L^{2}(\ambient,\meas)\cap L^{p}(\ambient,\meas).
    \end{equation}
    If $p=2$, then $1/2-\critic_{2}=0$ and the required inclusion is $P_{\varepsilon} h\in L^{2}(\ambient,\meas)$.

    If $p>2$, then $1/2-\critic_{p}>0$ and $\sup_{\lambda\in[0,\infty)}\lambda^{1/2-\critic_{p}}e^{-\varepsilon\lambda}<\infty$. Therefore,
      \begin{align}
        \|(-\gen)^{1/2-\critic_{p}}P_{\varepsilon}h\|_{L^{2}(\ambient,\meas)}^{2} &=\int_{(0,\infty)}\lambda^{1-2\critic_{p}}e^{-2\varepsilon\lambda} \dif\langle\proj_{\lambda}h,h\rangle\\
        &\leq\left(\frac{1-2\critic_{p}}{2e\varepsilon}\right)^{1-2\critic_{p}} \|h\|_{L^{2}(\ambient,\meas)}^{2}<\infty.
      \end{align}

    Next, let $g:=(I-P_{R})^{2}P_{\varepsilon} h$. Since $0<\critic_{p}<1$ and $0\leq1-e^{-R\lambda}\leq R\lambda$, we have
      \begin{align}
        \int_{(0,\infty)}\lambda^{-2\critic_{p}} \dif\langle\proj_{\lambda} g,g\rangle &=\int_{(0,\infty)}\lambda^{-2\critic_{p}} (1-e^{-R\lambda})^{4}e^{-2\varepsilon\lambda} \dif\langle\proj_{\lambda} h,h\rangle\\
        &\lesssim R^{4}\varepsilon^{-4+2\critic_{p}} \|h\|_{L^{2}(\ambient,\meas)}^{2}<\infty.
      \end{align}
    Hence $g\in\Dom((-\gen)^{-\critic_{p}})$ by \eqref{e.spec}. For $v:=(-\gen)^{-\critic_{p}}g\in L^{2}(\ambient,\meas)$, the spectral calculus now gives
      \begin{align}
        \int_{(0,\infty)}\lambda\dif\langle\proj_{\lambda} v,v\rangle &=\int_{(0,\infty)}\lambda^{1-2\critic_{p}} (1-e^{-R\lambda})^{4}e^{-2\varepsilon\lambda} \dif\langle\proj_{\lambda} h,h\rangle\\
        &\lesssim R^{4}\varepsilon^{-5+2\critic_{p}}\|h\|_{L^{2}(\ambient,\meas)}^{2}<\infty.
      \end{align}
    Thus $(-\gen)^{-\critic_{p}}g\in\Dom((-\gen)^{\frac{1}{2}})$.

    \item[\ref{it.denserg}] Let $Y$ be either $L^{p}(\ambient,\meas)$ or $L^{p,1}(\ambient,\meas)$, and fix $f\in Y$. Fix $o\in\ambient$ and define the bounded simple functions
    \begin{equation}
      h_{n}(x):=\one_{B(o,n)}(x)\operatorname{sgn}(f(x)) \min\set{n,2^{-n}\left\lfloor 2^{n}|f(x)|\right\rfloor}, \qquad n\in\bN.
    \end{equation}
    Then $|h_{n}-f|\leq|f|$ and $h_{n}\to f$ $\meas$-a.e. on $\ambient$. For $ \allowbreak f \in \allowbreak L ^{p}(\ambient,\meas) $, by the dominated convergence theorem, we have $\|h_{n}-f\|_{L^{p}(\ambient,\meas)}\to0$. For $f\in L^{p,1}(\ambient,\meas)$, one has
    \begin{equation}
      \lim_{n\to\infty}\lambda_{h_{n}-f}(\alpha)=0,\text{ and }0\leq\lambda_{h_{n}-f}(\alpha)\leq\lambda_{f}(\alpha),\ \text{ for all }\alpha\in(0,\infty),
    \end{equation}
    and
    \begin{equation}
      \lim_{n\to\infty}\|h_{n}-f\|_{L^{p,1}(\ambient,\meas)} =p\lim_{n\to\infty}\int_{0}^{\infty}\lambda_{h_{n}-f}(\alpha)^{1/p}\dif\alpha=0.
    \end{equation}
    Indeed, $ \allowbreak \meas ( \{ | f | > \alpha \} ) < \infty $ for $\alpha>0$, and $p\int_{0}^{\infty}\lambda_{f}(\alpha)^{1/p}\dif\alpha<\infty$. If $f\in L^{2}(\ambient,\meas)$ as well, the same pointwise bound gives $\|h_{n}-f\|_{L^{2}(\ambient,\meas)}\to0$. Thus
    \begin{equation}
      h_{n}\in L^{1}(\ambient,\meas)\cap L^{2}(\ambient,\meas) \cap L^{p}(\ambient,\meas)\cap L^{p,1}(\ambient,\meas),\ \text{ and }\ \lim_{n\to\infty}\|h_{n}-f\|_{Y}=0.
    \end{equation}

    Fix $n\in\bN$ and let $h=h_{n}$. Since $h\in L^{r}(\ambient,\meas)$ for every $r\in[1,\infty)$, by the strong continuity of the semigroup,
    \begin{equation}\label{e.sconterr}
      \lim_{\varepsilon\downarrow0} \|P_{\varepsilon} h-h\|_{L^{r}(\ambient,\meas)}=0, \ \text{ for all }h\in L^{r}(\ambient,\meas)\ \text{ and all } r\in[1,\infty).
    \end{equation}
    For $p\in(1,\infty)$, \eqref{e.lorinter} also gives
      \begin{align}
        \lim_{\varepsilon\downarrow0}\|P_{\varepsilon} h-h\|_{L^{p,1}(\ambient,\meas)}\overset{\eqref{e.lorinter}}{\leq} C_{p,r_{0},r_{1}} \lim_{\varepsilon\downarrow0} \|P_{\varepsilon} h-h\|_{L^{r_{0}}(\ambient,\meas)}^{1-\theta} \|P_{\varepsilon} h-h\|_{L^{r_{1}}(\ambient,\meas)}^{\theta}\overset{\eqref{e.sconterr}}{=}0.
      \end{align}
    Choose $\varepsilon_{n}\in(0,1/n)$ such that $f_{n}:=P_{\varepsilon_{n}}h_{n}\in\collectD_{p}$ and $\|f_{n}-h_{n}\|_{Y}\leq n^{-1}$. Since $h_{n}\in L^{1}(\ambient,\meas)\cap L^{2}(\ambient,\meas)$, the semigroup
property and~\eqref{e.ultra} with $p=2$ give, for every $R>0$,
\begin{equation}
 \begin{aligned}
 \|(I-P_{R})^{2}f_{n}-f_{n}\|_{L^{2}(\ambient,\meas)}
 &=
   \|-2P_{R+\varepsilon_{n}}h_{n}
      +P_{2R+\varepsilon_{n}}h_{n}\|_{L^{2}(\ambient,\meas)}\\
 &\leq 2\|P_{R+\varepsilon_{n}}h_{n}\|_{L^{2}(\ambient,\meas)}
       +\|P_{2R+\varepsilon_{n}}h_{n}\|_{L^{2}(\ambient,\meas)}\\
 &\overset{\eqref{e.ultra}}{\leq}
  3C_{2,1}R^{-\frac{\dimf}{2\dimw}}\|h_{n}\|_{L^{1}(\ambient,\meas)} .
 \end{aligned}
\end{equation}
    Suppose now that $p\in(1,\infty)$. Then
      \begin{align}
       &\phantom{\ \leq} \|P_{t}h\|_{L^{p,1}(\ambient,\meas)} \overset{\eqref{e.lorinter}}{\leq} C_{p,r_{0},r_{1}} \|P_{t}h\|_{L^{r_{0}}(\ambient,\meas)}^{1-\theta} \|P_{t}h\|_{L^{r_{1}}(\ambient,\meas)}^{\theta}\\
        &\overset{\eqref{e.ultra}}{\leq} C t^{-\frac{\dimf}{\dimw} ((1-\theta)(1-1/r_{0})+\theta(1-1/r_{1}))} \|h\|_{L^{1}(\ambient,\meas)}\overset{\eqref{e.denspcrr}}{=}C t^{-\frac{\dimf}{\dimw}\left(1-\frac{1}{p}\right)}\|h\|_{L^{1}(\ambient,\meas)}.
      \end{align}
    Consequently, for either choice of $Y$, there exists a constant $C$ independent of $h$, $\varepsilon$ and $R$ such that
    \begin{equation}
      \|(I-P_{R})^{2}P_{\varepsilon} h-P_{\varepsilon} h\|_{Y} \leq2\|P_{R+\varepsilon}h\|_{Y}+\|P_{2R+\varepsilon}h\|_{Y}\leq3C R^{-\frac{\dimf}{\dimw}\left(1-\frac{1}{p}\right)}\|h\|_{L^{1}(\ambient,\meas)}.
    \end{equation}
Choose $R_{n}\geq n$ so large that
\begin{equation}
 3\|h_{n}\|_{L^{1}(\ambient,\meas)}
 \left(CR_{n}^{-\frac{\dimf}{\dimw}\left(1-\frac{1}{p}\right)}
       +C_{2,1}R_{n}^{-\frac{\dimf}{2\dimw}}\right)
 \leq \frac{1}{n}.
\end{equation}
Put $g_{n}:=(I-P_{R_{n}})^{2}f_{n}$. By~\eqref{e.specc-strict},
$g_{n}\in\collectD_{p}^{\rm sp}$, and the choice of $R_{n}$ gives
\begin{equation}
 \|g_{n}-f_{n}\|_{Y}+\|g_{n}-f_{n}\|_{L^{2}(\ambient,\meas)}
 \leq \frac{1}{n}.
\end{equation}
Thus $\|g_{n}-f\|_{Y}\leq n^{-1}+\|f_{n}-f\|_{Y}\to0$ for every
$f\in Y$. If $f\in Y\cap L^{2}(\ambient,\meas)$, the same sequence satisfies
\begin{equation}
 \|g_{n}-f\|_{Y}+\|g_{n}-f\|_{L^{2}(\ambient,\meas)}
 \leq \frac{1}{n}+\|f_{n}-f\|_{Y}+\|f_{n}-f\|_{L^{2}(\ambient,\meas)}
 \to 0 .
\end{equation}
This proves the density of $\collectD_{p}^{\rm sp}$.
    \qedhere
  \end{enumerate}
\end{proof}
\begin{remark}
 Note that
  \[
   \collectD_{1}^{\mathrm{sp}}\subset  \Sett{h\in L^{1}(\ambient,\meas)}{\int_{\ambient}h\dif\meas=0}
  \]
  and hence, $\collectD_{1}^{\mathrm{sp}}$ is not dense in $L^{1}(\ambient,\meas)$. Indeed, symmetry and conservativity give, for $g\in L^{1}(\ambient,\meas)$ and $t>0$, $\int_{\ambient}P_{t}g\dif\meas=\int_{\ambient}g\dif\meas$.
  Thus, for $R,\varepsilon>0$, $\int_{\ambient}(I-P_{R})^{2}P_{\varepsilon}g\dif\meas =(1-2+1)\int_{\ambient}g\dif\meas=0$.
  The subspace $\Sett{h\in L^{1}(\ambient,\meas)}{\int_{\ambient}h\dif\meas=0}$ is closed since $\left|\int_{\ambient}h\dif\meas\right|\leq\|h\|_{L^{1}(\ambient,\meas)}$, and is  proper because $\one_{B}\in L^{1}(\ambient,\meas)$ has nonzero integral whenever $0<\meas(B)<\infty$.
\end{remark}

\subsection{Proof of Theorem \ref{t.critx}}\label{s.critic}
We are now able to determine the critical exponent in Theorem \ref{t.critx}.
\begin{proof}[Proof of Theorem \ref{t.critx}]\
	\begin{enumerate}[label=\textup{({\arabic*})}, align=right, leftmargin=*, topsep=5pt, parsep=0pt, itemsep=2pt]
  \item[\ref{it.neces}] Assume~\eqref{e.assum}. Fix $o\in\ambient$ and, for $r\in[1,\infty)$, let $\chi_{o,r}\in \core\subset\sobolev{p}(\ambient)$ be the cutoff function from Lemma~\ref{l.cutoff}. Then
\begin{equation}\label{e.chinor}
 \one_{B(o,r)}\leq\chi_{o,r}\leq\one_{B(o,3r)},\ 
 \|\chi_{o,r}\|_{L^1(\ambient,\meas)}
 \overset{\eqref{e.ahlfor}}{\leq}C_0(3r)^{\dimf},\ \text{ and }\ \|\wgrad\chi_{o,r}\|_{L^p(\skeleton,\medm)}
 \overset{\eqref{e.cutof1}}{\leq}Mr^{\frac{1}{p}-1}.
\end{equation}

If $x\in B(o,r/2)$ and $y\in B(x,1/2)$, then
\begin{equation}
 \metric(o,y)\leq\metric(o,x)+\metric(x,y)
 <\frac{r}{2}+\frac{1}{2}\leq r.
\end{equation}
The lower heat kernel bound in \eqref{e.HKE} therefore gives
\begin{align}
 \langle P_1\chi_{o,r},\chi_{o,r}\rangle_{L^2(\ambient,\meas)}
 &\overset{\eqref{e.heatke},\eqref{e.chinor}}{\geq}
 \int_{B(o,r/2)}\int_{B(x,1/2)}p_1(x,y)\dif\meas(y)\dif\meas(x)\\
 &\overset{\eqref{e.HKE}}{\geq}
 c\int_{B(o,r/2)}\meas(B(x,1/2))\dif\meas(x)
 \overset{\eqref{e.ahlfor}}{\geq}cr^{\dimf}.\label{e.critheatlower}
\end{align}
For every $b\in[1,\infty)$, ultracontractivity gives
\begin{align}
 \langle P_{1+br^{\dimw}}\chi_{o,r},\chi_{o,r}\rangle_{L^2(\ambient,\meas)}
 &\overset{\text{(H\"older)}}{\leq}
 \|P_{1+br^{\dimw}}\chi_{o,r}\|_{L^\infty(\ambient,\meas)}
 \|\chi_{o,r}\|_{L^1(\ambient,\meas)}\\
 &\overset{\eqref{e.ultra}}{\leq}
 C(1+br^{\dimw})^{-\frac{\dimf}{\dimw}}
 \|\chi_{o,r}\|_{L^1(\ambient,\meas)}^2\overset{\eqref{e.chinor}}{\leq}
 Cb^{-\frac{\dimf}{\dimw}}r^{\dimf}.\label{e.critheatupper}
\end{align}
Fix $b\in[1,\infty)$, independent of $r$, so large that the last constant in~\eqref{e.critheatupper} satisfies $Cb^{-\dimf/\dimw}\leq c/2$, where $c$ is the lower-bound constant in~\eqref{e.critheatlower}. Then
\begin{equation}\label{e.critheatdiff}
 \langle(P_1-P_{1+br^{\dimw}})\chi_{o,r},\chi_{o,r}\rangle_{L^2(\ambient,\meas)}
 \overset{\eqref{e.critheatlower},\eqref{e.critheatupper}}{\geq}
 \frac{c}{2}r^{\dimf},\ \text{ for all }\ r\in[1,\infty).
\end{equation}
Since $\critic\in(0,1)$,
\begin{equation}\label{e.critmult}
 1-e^{-v}\leq\min\{v,1\}\leq v^{\critic},\qquad v\geq0.
\end{equation}
Consequently, the spectral theorem gives
\begin{align}\label{e.spect2}
 &\phantom{\ \leq}\langle(-\gen)^{\critic}P_1\chi_{o,r},\chi_{o,r}\rangle_{L^2(\ambient,\meas)}\overset{\eqref{e.spec}}{=}
 \int_{[0,\infty)}\lambda^{\critic}e^{-\lambda}
 \dif\langle\proj_\lambda\chi_{o,r},\chi_{o,r}\rangle\\
 &\overset{\eqref{e.critmult}}{\geq}
 (br^{\dimw})^{-\critic}\int_{[0,\infty)}
 \bigl(e^{-\lambda}-e^{-(1+br^{\dimw})\lambda}\bigr)
 \dif\langle\proj_\lambda\chi_{o,r},\chi_{o,r}\rangle\\
 &\overset{\eqref{e.spec}}{=}
 (br^{\dimw})^{-\critic}
 \langle(P_1-P_{1+br^{\dimw}})\chi_{o,r},\chi_{o,r}\rangle_{L^2(\ambient,\meas)}\overset{\eqref{e.critheatdiff}}{\geq}
 cr^{\dimf-\dimw\critic},\ \text{ for all }\ r\in[1,\infty).
\end{align}

By~\eqref{e.assum},
\begin{equation}\label{e.critcutweak}
 \|(-\gen)^{\critic}P_1\chi_{o,r}\|_{L^{p,\infty}(\ambient,\meas)}^p
 \overset{\eqref{e.assum},\eqref{e.chinor}}{\leq}
 Cr^{1-p}.
\end{equation}
Also, there exists $H\in[1,\infty)$, independent of $r$, such that
\begin{equation}\label{e.critcutbounded}
 \|(-\gen)^{\critic}P_1\chi_{o,r}\|_{L^\infty(\ambient,\meas)}
 \overset{\eqref{e.analy}}{\leq}
 C_{\critic}\|\chi_{o,r}\|_{L^\infty(\ambient,\meas)}
 \overset{\eqref{e.chinor}}{\leq}H.
\end{equation}
Fix any $q\in(p,\infty)$. The layer-cake formula gives
\begin{align}\label{e.critcutq}
 &\phantom{\ \leq}\|(-\gen)^{\critic}P_1\chi_{o,r}\|_{L^q(\ambient,\meas)}^q\overset{\eqref{e.critcutbounded}}{=}
 q\int_0^H\alpha^{q-1}
 \meas\bigl(\{|(-\gen)^{\critic}P_1\chi_{o,r}|>\alpha\}\bigr)\dif\alpha\\
 &\overset{\eqref{e.lorpq}}{\leq}
 q\|(-\gen)^{\critic}P_1\chi_{o,r}\|_{L^{p,\infty}(\ambient,\meas)}^p
 \int_0^H\alpha^{q-p-1}\dif\alpha\\
 &=\frac{qH^{q-p}}{q-p}
 \|(-\gen)^{\critic}P_1\chi_{o,r}\|_{L^{p,\infty}(\ambient,\meas)}^p
 \overset{\eqref{e.critcutweak}}{\leq}C_{p,q,\critic}r^{1-p}.
\end{align}
Thus H\"older's inequality, together with~\eqref{e.spect2}, gives
\begin{align}
 cr^{\dimf-\dimw\critic}
 &\overset{\eqref{e.spect2}}{\leq}
 \langle(-\gen)^{\critic}P_1\chi_{o,r},\chi_{o,r}\rangle_{L^2(\ambient,\meas)}\\
 &\overset{\eqref{e.chinor}}{\leq}
 \|(-\gen)^{\critic}P_1\chi_{o,r}\|_{L^q(\ambient,\meas)}
 \meas(B(o,3r))^{1-\frac{1}{q}}\\
 &\overset{\eqref{e.critcutq},\eqref{e.ahlfor}}{\leq}
 C_{p,q,\critic}r^{\frac{1-p}{q}+\dimf(1-\frac{1}{q})}
 =C_{p,q,\critic}r^{\dimf-\frac{\dimf+p-1}{q}}.
\end{align}
Therefore $r^{\frac{\dimf+p-1}{q}-\dimw\critic}\leq C_{p,q,\critic}$ for all $r\in[1,\infty)$.
For each fixed $q>p$, letting $r\to\infty$ gives $\critic\geq(\dimf+p-1)/(q\dimw)$. Now let $q\downarrow p$ to obtain
\begin{equation}
 \critic\geq\frac{\dimf+p-1}{p\dimw}
 \overset{\eqref{e.critp}}{=}\critic_p.
\end{equation}
  \item[\ref{it.suffi}]Fix $\critic\in(\critic_p,1)$, and let $p'\in(1,\infty]$ be the conjugate exponent, that is, $p^{-1}+p'^{-1}=1$. By~\eqref{e.extdim},
\begin{equation}
 \critic_p+\critic_{p'}=1,\ \text{ and }\
 \critic+\critic_{p'}-1=\critic-\critic_p>0.
\end{equation}
Let $f\in\core$ and $h\in L^{p'}(\ambient,\meas)\cap L^2(\ambient,\meas)$. Then $f\in\domain$ and $P_{s+1}h\in\Dom(-\gen)\subset\domain$ for $s>0$. H\"older's inequality and Proposition~\ref{p.gradie} give
\begin{align}
 |\form(f,P_{s+1}h)|
 &\overset{\eqref{e.formde}}{\leq}
 \|\wgrad f\|_{L^p(\skeleton,\medm)}
 \|\wgrad P_{s+1}h\|_{L^{p'}(\skeleton,\medm)}\\
 &\overset{\eqref{e.gradie}}{\leq}
 C_p(1+s)^{-\critic_{p'}}
 \|\wgrad f\|_{L^p(\skeleton,\medm)}
 \|h\|_{L^{p'}(\ambient,\meas)}.\label{e.formgr}
\end{align}
Note that
\begin{align}
 &\phantom{\ \leq}\langle(-\gen)^{\critic}P_1f,h\rangle_{L^2(\ambient,\meas)}\\
 &\overset{\eqref{e.Lpt2}}{=}
 \frac{1}{\Gamma(1-\critic)}\int_0^\infty s^{-\critic}
 \langle(-\gen)P_{s+1}f,h\rangle_{L^2(\ambient,\meas)}\dif s\\
 &\overset{\eqref{e.spec}}{=}
 \frac{1}{\Gamma(1-\critic)}\int_0^\infty s^{-\critic}\form(f,P_{s+1}h)\dif s\\
 &\overset{\eqref{e.formde}}{=}
 \frac{1}{\Gamma(1-\critic)}\int_0^\infty s^{-\critic}
 \left(\int_{\skeleton}\wgrad f\,\wgrad P_{s+1}h\dif\medm\right)\dif s.\label{e.unifie}
\end{align}
Consequently,
\begin{align}
 &\left|\langle(-\gen)^{\critic}P_1f,h\rangle_{L^2(\ambient,\meas)}\right|\\
 &\overset{\eqref{e.unifie},\eqref{e.formgr}}{\leq}
 \frac{C_p}{\Gamma(1-\critic)}
 \|\wgrad f\|_{L^p(\skeleton,\medm)}\|h\|_{L^{p'}(\ambient,\meas)}
 \int_0^\infty s^{-\critic}(1+s)^{-\critic_{p'}}\dif s\\
 &\overset{\eqref{e.extdim}}{\leq}
 \frac{C_p}{\Gamma(1-\critic)}
 \left(\frac{1}{1-\critic}+\frac{1}{\critic-\critic_p}\right)
 \|\wgrad f\|_{L^p(\skeleton,\medm)}\|h\|_{L^{p'}(\ambient,\meas)}.\label{e.unifi1}
\end{align}

Since $f\in L^p(\ambient,\meas)\cap L^2(\ambient,\meas)$, equation~\eqref{e.analy} implies $(-\gen)^{\critic}P_1f\in L^p(\ambient,\meas)\cap L^2(\ambient,\meas)$. For $p\in(1,\infty)$, take the supremum in~\eqref{e.unifi1} over $h\in L^{p'}\cap L^2$ with $\|h\|_{L^{p'}(\ambient,\meas)}\leq1$. Since $L^{p'}(\ambient,\meas)\cap L^2(\ambient,\meas)$ is dense in $L^{p'}(\ambient,\meas)$, $L^p$ duality gives the desired estimate. At $p=1$, fix $o\in\ambient$ and use
\begin{equation}
 h_n:=\operatorname{sgn}\bigl((-\gen)^{\critic}P_1f\bigr)\one_{B(o,n)}
 \in L^\infty(\ambient,\meas)\cap L^2(\ambient,\meas),
 \qquad\|h_n\|_{L^\infty(\ambient,\meas)}\leq1.
\end{equation}
Then
\begin{equation}
 \int_{B(o,n)}|(-\gen)^{\critic}P_1f|\dif\meas
 \overset{\eqref{e.unifi1}}{\leq}
 C_{1,\critic}\|\wgrad f\|_{L^1(\skeleton,\medm)},
\end{equation}
and the monotone convergence theorem gives the same estimate on $\ambient$. Thus, for every $p\in[1,\infty)$,
\begin{equation}\label{e.unifi2}
 \|(-\gen)^{\critic}P_1f\|_{L^p(\ambient,\meas)}
 \leq C_{p,\critic}\|\wgrad f\|_{L^p(\skeleton,\medm)},
 \ \text{ for all } f\in\core.
\end{equation}

Finally, let $f\in\sobolev{p}(\ambient)$. By Lemma~\ref{l.dense}, there exist $f_n\in\core$ such that
\begin{equation}
\lim_{n\to\infty}(\|f_n-f\|_{L^p(\ambient,\meas)}
 +\|\wgrad f_n-\wgrad f\|_{L^p(\skeleton,\medm)})
 \overset{\eqref{e.scaled}}{=}0.
\end{equation}
By \eqref{e.analy}, we have
\begin{equation}
 \|(-\gen)^{\critic}P_1(f_n-f)\|_{L^p(\ambient,\meas)}
 \overset{\eqref{e.analy}}{\leq}
 C_{p,\critic}\|f_n-f\|_{L^p(\ambient,\meas)}\to 0.
\end{equation}
Passing to the limit in~\eqref{e.unifi2}, we conclude that
\begin{align}
 \|(-\gen)^{\critic}P_1f\|_{L^p(\ambient,\meas)}
 &=\lim_{n\to\infty}\|(-\gen)^{\critic}P_1f_n\|_{L^p(\ambient,\meas)}\\
 &\overset{\eqref{e.unifi2}}{\leq}
 C_{p,\critic}\lim_{n\to\infty}\|\wgrad f_n\|_{L^p(\skeleton,\medm)}
 \overset{\eqref{e.scaled}}{=}
 C_{p,\critic}\|\wgrad f\|_{L^p(\skeleton,\medm)}.
\end{align}
This proves~\eqref{e.subcr}.
\qedhere
\end{enumerate}
\end{proof}

\section{The weak reverse inequality for \texorpdfstring{$\dimf>1$}{d\_f > 1}}\label{ss.revers}
In this section, we prove the weak reverse inequality for $1\le p<2$ and the strong reverse inequality for $p=2$ in Theorem~\ref{t.strfl}-\ref{it.scale} and~\ref{it.p2}. We use these estimates to pass to the limit as $t\downarrow0$ in Theorem~\ref{t.mainw}. 

We first record the comparison of the measures $\meas$ and $\medm$.
\begin{lemma}\label{l.singul}
Suppose that $(\ambient,\metric,\meas)$ is $\AR{\dimf}$ with $\dimf\in[1,\infty)$.
\begin{enumerate}[label=\textup{({\arabic*})}, align=right, leftmargin=*, topsep=5pt, parsep=0pt, itemsep=2pt]
  \item\label{it.df>1} If $\dimf\in(1,\infty)$, then $\meas(\skeleton)=0$ and $\meas\perp\medm$;
  \item\label{it.df=1} If $\dimf=1$, then there exists $C\in[1,\infty)$, depending only on $C_0$, such that
\begin{equation}\label{e.one-measures}
 C^{-1}\meas(E)\leq\medm(E)\leq C\meas(E),
 \ \text{ for all Borel sets } E\subset\ambient.
\end{equation}
Moreover,
\begin{equation}\label{e.onelen}
 r\leq\medm(B(x,r))\leq Cr,
 \ \text{ for all }\ (x,r)\in\ambient\times(0,\infty).
\end{equation}
In particular, $\meas(\ambient\setminus\skeleton)=0$, and there exists a Borel function $w$ such that $\dif\meas=w\dif\medm$ and $C^{-1}\leq w\leq C$ $\medm$-a.e. on $\ambient$.
\end{enumerate}
\end{lemma}
\begin{proof}
\begin{enumerate}[label=\textup{({\arabic*})}, align=right, leftmargin=*, topsep=5pt, parsep=0pt, itemsep=2pt]
  \item[\ref{it.df>1}] Let $a,b\in\skeleton$. Put $\ell:=\metric(a,b)$. For $r>0$, divide $[a,b]$ into $n:=\max\{1,\lceil\ell/r\rceil\}$ equal subarcs and let $z_{1},\ldots,z_{n}$ be their midpoints. Then $[a,b]\subset\bigcup_{k=1}^{n}B(z_k,r)$. Hence
  \begin{equation}
    \meas([a,b])\leq\sum_{k=1}^{n}\meas(B(z_{k},r))
    \overset{\AR{\dimf}}{\leq}C_0r^{\dimf}\max\{1,\lceil\ell/r\rceil\}
   \to0\ \text{ as }\ r\downarrow0.
  \end{equation}
  By Remark~\ref{r.arcs}, there is a countable family of arcs $\{A_n\}_{n\in\bN}$ such that $\skeleton=\bigcup_{n\in\bN}A_{n}$. Then $\meas(\skeleton)\leq\sum_{n\in\bN}\meas(A_{n})=0$. Since $\medm(\ambient\setminus\skeleton)=0$, we have $\meas\perp\medm$.
  \item[\ref{it.df=1}] This follows from the same proof as \cite[Proposition~4.5, (4.29), (4.30)]{BCC26}.
  \qedhere
\end{enumerate}
\end{proof}

\subsection{Estimates for sums of arc indicators}\label{ss.synthe}
In this subsection, we introduce the operators $B_t$ and $\Theta_t$ and express $\Theta_t$ on each arc indicator as a difference of two heat kernels. We then use this formula to estimate $B_t$ on sums of indicator functions of arcs.

\begin{proposition}\label{p.synthe}
	Suppose that $(\ambient,\metric)$ is $\AR{\dimf}$ with $\dimf\in(1,\infty)$. Let $p\in[1,2)$. Then $\critic_p\in(\frac{1}{2},1)$. For each $t\in(0,\infty)$, we define 
	\begin{equation}\label{e.Bt}
  B_{t}:=\big((-\gen)^{\critic_{p}-\frac{1}{2}}P_{t}\big)\circ \riesz_{2}^{*}:L^{2}(\skeleton,\medm)\to L^{2}(\ambient,\meas)
\end{equation}
and 
\begin{equation}\label{e.Theta}
  \Theta_{t}:=\big((-\gen)^{\frac{1}{2}}P_{t}\big)\circ\riesz_{2}^{*}:L^{2}(\skeleton,\medm)\to L^{2}(\ambient,\meas),
\end{equation}
where $(-\gen)^{\theta}P_{t}$ is defined by the spectral calculus of $(-\gen,\Dom(-\gen))$, and $\riesz_{2}^{*}:L^{2}(\skeleton,\medm)\to L^{2}(\ambient,\meas)$ is the adjoint of $\riesz_{2}$ from Theorem~\ref{t.dirich}-\ref{it.riesz}. The follow properties hold.
\begin{enumerate}[label=\textup{({\arabic*})}, align=right, leftmargin=*, topsep=5pt, parsep=0pt, itemsep=2pt]
  \item\label{it.BThe1} For each $t\in(0,\infty)$, the operator $B_{t}:L^{2}(\skeleton,\medm)\to L^{2}(\ambient,\meas)$ is bounded, and  \begin{equation}
  	B_{t}(\wgrad f)=(-\gen)^{\critic_{p}}P_{t}f,\ \text{ for every }\ f\in\domain.
  \end{equation}
  \item\label{it.BThe2} For each $s\in(0,\infty)$, the operator $\Theta_{s}:L^{2}(\skeleton,\medm)\to L^{2}(\ambient,\meas)$ is bounded, and for every $f\in \sA$, we have 
  \begin{equation}\label{e.Pstar}
  \Theta_{s}f=\sum_{z\in Z_{f}}\mu_{f}(z)p_{s}(z,\cdot),\ \
  \langle\Theta_{s}f,h\rangle_{L^{2}(\ambient,\meas)}=\langle f,\wgrad P_{s}h\rangle_{L^{2}(\skeleton,\medm)},\ \text{for all }h\in L^{2}(\ambient,\meas),
\end{equation}
where $Z_{f}$ and $\mu_{f}$ are defined in Definition \ref{d.endpoi}, and 
\begin{equation}\label{e.Pstar1}
  P_{t}\Theta_{s}f{=}\Theta_{t+s}f,\ \text{ for all } s,t\in(0,\infty).
\end{equation}
\end{enumerate}
\end{proposition}
\begin{proof}
Since $\dimw=\dimf+1$, the definition of $\critic_p$ gives
\begin{equation}
  \critic_p-\frac12
  \overset{\eqref{e.critp}}{=}\frac{(\dimf-1)(2-p)}{2p\dimw}>0,
  \ \text{ and }\ 
  1-\critic_p
  \overset{\eqref{e.critp}}{=}\frac{\dimf(p-1)+1}{p\dimw}>0.
\end{equation}
Thus $\critic_p\in(\frac{1}{2},1)$. We first record the identity needed below.
If $u\in\domain=\Dom((-\gen)^{\frac{1}{2}})$ and $v=(-\gen)^{\frac{1}{2}}u$, then,
since $\Ker(-\gen)=\{0\}$, the spectral calculus gives
\begin{equation}
  \int_{(0,\infty)}\lambda^{-1}
       \dif\langle\proj_\lambda v,v\rangle
  \overset{\eqref{e.spec}}{=}
  \int_{(0,\infty)}\dif\langle\proj_\lambda u,u\rangle
  =
  \|u\|_{L^2(\ambient,\meas)}^2<\infty.
\end{equation}
Hence $v\in\Dom((-\gen)^{-\frac{1}{2}})$ and $(-\gen)^{-\frac{1}{2}}v=u$. Thus,
\begin{equation}\label{e.BTheta.riesz}
  \begin{aligned}
    \riesz_2(-\gen)^{\frac{1}{2}}u\overset{\eqref{e.riesz2}}{=}\wgrad u, \text{ and }\
    \riesz_2^*\wgrad u=\riesz_2^*\riesz_2(-\gen)^{\frac{1}{2}}u
      \overset{\eqref{e.BTheta.adjoint1}}{=}(-\gen)^{\frac{1}{2}}u,
      \ \text{ for all }\ u\in\domain.
  \end{aligned}
\end{equation}
For every $\theta\in(0,\infty)$ and $t\in(0,\infty)$,
\begin{equation}\label{e.BTheta.multiplier}
  \|(-\gen)^\theta P_t\|_{L^2(\ambient,\meas)\to L^2(\ambient,\meas)}
  \overset{\eqref{e.spec}}{\leq}\sup_{\lambda\in[0,\infty)}\lambda^\theta e^{-t\lambda}
  =\left(\frac{\theta}{et}\right)^\theta.
\end{equation}

\begin{enumerate}[label=\textup{({\arabic*})}, align=right, leftmargin=*, topsep=5pt, parsep=0pt, itemsep=2pt]
  \item[\ref{it.BThe1}] Since \begin{align}
	\|B_t\|_{L^{2}(\skeleton,\medm)\to L^{2}(\ambient,\meas)}
    &\overset{\eqref{e.Bt}}{\leq}
      \|(-\gen)^{\critic_p-1/2}P_t\|_{L^2(\ambient,\meas)\to L^2(\ambient,\meas)}\|\riesz_2^*\|_{L^{2}(\skeleton,\medm)\to L^{2}(\ambient,\meas)}\\
    &\overset{\eqref{e.BTheta.multiplier},\eqref{e.BTheta.adjoint2}}{\leq}
      \left(\frac{\critic_p-\frac12}{et}\right)^{\critic_p-\frac12},
\end{align}
we know $B_{t}$ is bounded. For $u\in\domain$,
\eqref{e.BTheta.riesz} and the spectral calculus give
\begin{equation}
    B_t(\wgrad u)
    \overset{\eqref{e.Bt}}{=}
      (-\gen)^{\critic_p-1/2}P_t\riesz_2^*\wgrad u\overset{\eqref{e.BTheta.riesz}}{=}
      (-\gen)^{\critic_p-1/2}P_t(-\gen)^{\frac{1}{2}}u\overset{\eqref{e.spec}}{=}(-\gen)^{\critic_p}P_tu.
\end{equation}
  \item[\ref{it.BThe2}] Since   \begin{align}
    \|\Theta_t\|_{L^{2}(\skeleton,\medm)\to L^{2}(\ambient,\meas)}
    &\overset{\eqref{e.Theta}}{\leq}
      \|(-\gen)^{\frac{1}{2}}P_t\|_{L^2(\ambient,\meas)\to L^2(\ambient,\meas)}\|\riesz_2^*\|_{L^{2}(\skeleton,\medm)\to L^{2}(\ambient,\meas)}\\
    &\overset{\eqref{e.BTheta.multiplier},\eqref{e.BTheta.adjoint2}}{\leq}
      (2et)^{-\frac{1}{2}},
  \end{align}
  we know $\Theta_{t}$ is bounded. Fix $s\in(0,\infty)$ and $h\in L^2(\ambient,\meas)$. Since $(-\gen)^{\frac{1}{2}}P_s$ is bounded and
self-adjoint, \eqref{e.BTheta.riesz} implies, for every
$f\in L^2(\skeleton,\medm)$,
\begin{equation}\label{e.BTheta.duality}
    \langle\Theta_sf,h\rangle_{L^2(\ambient,\meas)}
    \overset{\eqref{e.Theta}}{=}
      \langle f,\riesz_2(-\gen)^{\frac{1}{2}}P_sh\rangle_{L^2(\skeleton,\medm)}\overset{\eqref{e.BTheta.riesz}}{=}
      \langle f,\wgrad P_sh\rangle_{L^2(\skeleton,\medm)}.
\end{equation}

Now let
$f=\sum_{j=1}^N\kappa_j\vecone_{(a_j,b_j)}\in\sA$. Let $h\in L^1(\ambient,\meas)\cap L^2(\ambient,\meas)$. By the fact that $P_{s}h\in \domain\subset\abscon(\ambient,\metric)$, identity~\eqref{e.arcftc}, \eqref{e.BTheta.duality}, and
Definition~\ref{d.endpoi}, we have
\begin{align}
    &\phantom{\ \leq}\langle\Theta_sf,h\rangle_{L^2(\ambient,\meas)}
    \overset{\eqref{e.BTheta.duality}}{=}\sum_{j=1}^N\kappa_j
        \int_{\skeleton}\vecone_{(a_j,b_j)}\wgrad P_sh\dif\medm\overset{\eqref{e.arcftc}}{=}\sum_{j=1}^N\kappa_j\bigl(P_{s}h(b_j)-P_{s}h(a_j)\bigr)\\
    &\overset{\eqref{e.endpoi}}{=}\sum_{z\in Z_{\wgrad f}}\mu_f(z)P_{s}h(z)=\int_{\ambient}
        \left(\sum_{z\in Z_{\wgrad f}}\mu_f(z)p_s(z,y)\right)
        h(y)\dif\meas(y).\label{e.BTheta.kerneltest}
  \end{align}
  Since \eqref{e.BTheta.kerneltest} holds for all $h\in L^1(\ambient,\meas)\cap L^2(\ambient,\meas)$, which is dense in $L^2(\ambient,\meas)$, we have \eqref{e.BTheta.kerneltest} holds for all $h\in L^2(\ambient,\meas)$, and therefore
  \begin{equation}
  	\Theta_sf(y)=\sum_{z\in Z_{\wgrad f}}\mu_f(z)p_s(z,y),\ \text{ for $\meas$-a.e. }y\in\ambient.
  \end{equation}
  This proves \eqref{e.Pstar}. Finally, the spectral calculus gives, for every
$f\in L^2(\skeleton,\medm)$ and $s,t\in(0,\infty)$,
\begin{equation}
    P_t\Theta_sf
    \overset{\eqref{e.Theta}}{=}P_t(-\gen)^{\frac{1}{2}}P_s\riesz_2^*f\overset{\eqref{e.spec}}{=}(-\gen)^{\frac{1}{2}}P_{t+s}\riesz_2^*f
      \overset{\eqref{e.Theta}}{=}\Theta_{t+s}f.
\end{equation}
This proves \eqref{e.Pstar1}.
\qedhere
\end{enumerate}
\end{proof}

\begin{lemma}\label{l.synthe}
Suppose that $(\ambient,\metric,\meas)$ is $\AR{\dimf}$ with $\dimf\in(1,\infty)$. Let $L,N\in[1,\infty)$.  Let $p\in[1,2)$ and \begin{equation}\label{e.qrange}
 q\in(p,2]\cap\left(p,\frac{p\dimf}{\dimf-1}\right).
\end{equation}
 For $j\in\bZ$, let $r_{j}:=2^{-j}$ and $\sI_{j}$ be a finite index. Suppose that $\sI_{j}=\emptyset$ for all but finitely many $j\in\bZ$. For each $j\in\bZ$ and each $\alpha\in \sI_{j}$, let $\kappa_{j,\alpha}\in\bR$, $a_{j,\alpha},b_{j,\alpha}\in\ambient$ with $a_{j,\alpha}\neq b_{j,\alpha} $ such that 
 \begin{equation}\label{e.synste2}
 	\sup_{\alpha\in \sI_{j}}\medm([a_{j,\alpha},b_{j,\alpha}])\leq Lr_{j},\ \text{ for all }j\in\bZ
 \end{equation}
 and 
 \begin{equation}\label{e.synste1}
 	\sup_{x\in\ambient}\#\Sett{\alpha\in \sI_{j}}{B(x,r_{j})\cap [a_{j,\alpha},b_{j,\alpha}]\neq\emptyset}\leq N,\ \text{ for all }j\in\bZ.
 \end{equation}
Let $f_{j}:=\sum_{\alpha\in\sI_{j}}\kappa_{j,\alpha}\vecone_{(a_{j,\alpha},b_{j,\alpha})}\in\sA$ and $K_{j}:=r_{j}\sum_{\alpha\in\sI_{j}}|\kappa_{j,\alpha}|^{q}$ for $j\in\bZ$. Then
  \begin{equation}\label{e.synthe}
    \sup_{t\in(0,\infty)}\Bigl\|B_{t}\bigl(\sum_{j\in\bZ}f_{j}\bigr)\Bigr\|_{L^{q}(\ambient,\meas)}^{q}\leq C\sum_{j\in\bZ}r_{j}^{-(\dimf-1)(q-p)/p}K_{j},
  \end{equation}
  where $C$ depends only on $p$, $q$, $L$, $N$, $C_{0}$, $\dimf$ and the constants in \eqref{e.HKE} and \eqref{e.ddt}.
\end{lemma}

\begin{proof}
By Proposition \ref{p.synthe}, since $p\in[1,2)$, we have
\begin{equation}\label{e.synth1}
 \frac{1}{2}<\critic_p<1,\ \text{ and }\
 \dimw\critic_p-1\overset{\eqref{e.critp}}{=}\frac{\dimf-1}{p}.
\end{equation}

\begin{enumerate}[label=\textit{Step {\arabic*}.}, align=right, leftmargin=*, topsep=5pt, parsep=0pt, itemsep=2pt]
  \item \textit{Estimates for $\Theta_s$.}
  
  Let $I_{j,\alpha}=(a_{j,\alpha},b_{j,\alpha})$. Then $\metric(a_{j,\alpha},b_{j,\alpha})=\medm([a_{j,\alpha},b_{j,\alpha}])\overset{\eqref{e.synste2}}{\leq}Lr_{j}$. By~\eqref{e.Pstar} and conservativity in Theorem~\ref{t.dirich}-\ref{it.consDF},
\begin{align}
 \|\Theta_s\vecone_{I_{j,\alpha}}\|_{L^1(\ambient,\meas)}\overset{\eqref{e.Pstar}}{=}
 \int_{\ambient}|p_s(b_{j,\alpha},x)-p_s(a_{j,\alpha},x)|\dif\meas(x)\leq 2.
\end{align}
For every $h\in L^{1}(\ambient,\meas)\cap L^{\infty}(\ambient,\meas)$,
\begin{align}
 |\langle\Theta_s\vecone_{I_{j,\alpha}},h\rangle|
 &\overset{\eqref{e.Pstar}}{=}
 \left|\int_{I_{j,\alpha}}\vecone_{I_{j,\alpha}}\wgrad P_sh\dif\medm\right|\overset{\eqref{e.synste2}}{\leq} Lr_{j}\|\wgrad P_sh\|_{L^\infty(\skeleton,\medm)}\\
 &\overset{\eqref{e.gradie}}{\leq}
 Cr_{j}s^{-\frac{1}{\dimw}}\|h\|_{L^\infty(\ambient,\meas)}.
\end{align}
Fix $o\in\ambient$. Taking $h=\one_{B(o,k)}\operatorname{sgn}(\Theta_s\vecone_{I_{j,\alpha}})$ and then letting $k\to\infty$ gives, by the monotone convergence theorem,
\begin{equation}\label{e.synth4}
 \|\Theta_s\vecone_{I_{j,\alpha}}\|_{L^1(\ambient,\meas)}
 \leq C\min\{1,r_{j}s^{-\frac{1}{\dimw}}\},\quad s\in(0,\infty).
\end{equation}

We next estimate $\Theta_sf_j$ in $L^2(\ambient,\meas)$. For every $x\in\skeleton$, by the Cauchy--Schwarz inequality,
\begin{align}
 |f_j(x)|^2 &\leq
 \left(\sum_{\alpha\in\sI_j}\one_{(a_{j,\alpha},b_{j,\alpha})}(x)\right)
 \left(\sum_{\alpha\in\sI_j}|\kappa_{j,\alpha}|^2\one_{(a_{j,\alpha},b_{j,\alpha})}(x)\right)\\
 &\overset{\eqref{e.synste1}}{\leq}N\sum_{\alpha\in\sI_j}|\kappa_{j,\alpha}|^2\one_{(a_{j,\alpha},b_{j,\alpha})}(x).\label{e.synth-pointwise}
\end{align}
Integrating with respect to $\medm$, we obtain
\begin{equation}
 \|f_j\|_{L^2(\skeleton,\medm)}^2\overset{\eqref{e.synth-pointwise}}{\leq}N\sum_{\alpha\in\sI_j}|\kappa_{j,\alpha}|^2\medm([a_{j,\alpha},b_{j,\alpha}])\overset{\eqref{e.synste2}}{\leq}NLr_j\sum_{\alpha\in\sI_j}|\kappa_{j,\alpha}|^2.\label{e.synth6}
\end{equation}
Here, we count endpoints with their multiplicities in the given representation. By~\eqref{e.synste1}, every ball of radius $r_j$ contains at most $2N$ endpoints. For $R\geq r_j$, let $Y$ be a maximal $r_{j}$-separated set in $B(x,R)$. Then $ \#Y\lesssim (R/r_j)^{\dimf}$, and 
\begin{equation}\label{e.synth8}
 \sum_{\alpha\in\sI_j}
 \bigl(\one_{B(x,R)}(a_{j,\alpha})+
       \one_{B(x,R)}(b_{j,\alpha})\bigr)
 \overset{\eqref{e.synste1}}{\leq}2N\#Y\leq C(R/r_j)^{\dimf}.
\end{equation}
For $s\in(0, r_j^{\dimw}]$, split $\{a_{j,\alpha},b_{j,\alpha}\}_{\alpha\in \sI_{j}}$ according to its intersection with $ B(x,r_j)$ and with $B(x,2^{k+1}r_j)\setminus B(x,2^kr_j)$, $k\in\bN\cup\{0\}$. Then
\begin{align}
 &\phantom{\ \leq}\sum_{\alpha\in\sI_j}
   \bigl(p_s(a_{j,\alpha},x)+p_s(b_{j,\alpha},x)\bigr)\\
 &\overset{\eqref{e.HKE},\eqref{e.synth8}}{\leq}
 Cs^{-\frac{\dimf}{\dimw}}
 \left(1+\sum_{k=0}^{\infty}
 2^{(k+1)\dimf}\exp\left(-c2^{\frac{k\dimw}{\dimw-1}}\right)\right)
 \leq Cs^{-\frac{\dimf}{\dimw}}.\label{e.synth9}
\end{align}
By the Cauchy--Schwarz inequality,
\begin{align}
& \phantom{\ \leq}|\Theta_sf_j(x)|^2
 \overset{\eqref{e.Pstar}}{=}
 \left|\sum_{\alpha\in\sI_j}\kappa_{j,\alpha}
       \bigl(p_s(b_{j,\alpha},x)-p_s(a_{j,\alpha},x)\bigr)\right|^2\\
 &\leq
 \left(\sum_{\alpha\in\sI_j}|\kappa_{j,\alpha}|^2
       \bigl(p_s(a_{j,\alpha},x)+p_s(b_{j,\alpha},x)\bigr)\right)
 \left(\sum_{\alpha\in\sI_j}
       \bigl(p_s(a_{j,\alpha},x)+p_s(b_{j,\alpha},x)\bigr)\right).
\end{align}
Integrating and using~\eqref{e.synth9} and conservativity in Theorem~\ref{t.dirich}-\ref{it.consDF} yields
\begin{equation}\label{e.synt10}
 \|\Theta_sf_j\|_{L^2(\ambient,\meas)}^2
 \overset{\eqref{e.synth9}}{\leq}
 Cs^{-\frac{\dimf}{\dimw}}
 \sum_{\alpha\in\sI_j}|\kappa_{j,\alpha}|^2,
 \ \text{ for all }\ s\in(0, r_j^{\dimw}].
\end{equation}
For every $s\in(0,\infty)$, we also have
\begin{align}\label{e.synth-large}
 \|\Theta_sf_j\|_{L^2(\ambient,\meas)}^2\overset{\eqref{e.Theta},\eqref{e.spec},\eqref{e.BTheta.adjoint2}}{\leq}
 (2es)^{-1}\|f_j\|_{L^2(\skeleton,\medm)}^2\overset{\eqref{e.synth6}}{\leq}
 Cs^{-1}r_j\sum_{\alpha\in\sI_j}|\kappa_{j,\alpha}|^2.
\end{align}

Define 
\begin{equation}\label{e.bandef}
 a_{j,n}:=2^{-n}r_j^{\dimw},\ \text{ and }\
 U_{n,t}:=\sum_j(B_{t+a_{j,n}/2}-B_{t+a_{j,n}})f_j,\ \text{ for $j,n\in\bZ$}.
\end{equation}

\item \textit{The $L^1$ estimate of $U_{n,t}$.}

Fix $t\in(0,\infty)$. For $a\in(0,\infty)$, $f\in L^2(\skeleton,\medm)$ and $h\in L^2(\ambient,\meas)$, we claim that
\begin{equation}\label{e.timwgt}
 \left\langle(B_{t+a/2}-B_{t+a})f,h\right\rangle
 =\int_{a/4}^{3a/4}
 \left\langle\Theta_{a/4}f,
          (-\gen)^{\critic_p}P_{t+u}h\right\rangle\dif u.
\end{equation}
To verify this, by the fundamental theorem of calculus, for $\lambda\in[0,\infty)$,
\begin{equation}
 \lambda^{\frac{1}{2}}e^{-a\lambda/4}
   \int_{a/4}^{3a/4}\lambda^{\critic_p}e^{-(t+u)\lambda}\dif u=
 \lambda^{\critic_p-\frac{1}{2}}
 \bigl(e^{-(t+a/2)\lambda}-e^{-(t+a)\lambda}\bigr).\label{e.timvf1}
\end{equation}
Note that the variation of the spectral measure $\dif\langle\proj_\lambda\riesz_2^*f,h\rangle$ is not greater than $\|\riesz_2^*f\|_{L^2(\ambient,\meas)}\|h\|_{L^2(\ambient,\meas)}$. Therefore,
\begin{align}
 &\phantom{\ \leq}\int_{a/4}^{3a/4}\int_{[0,\infty)}
 \lambda^{\critic_p+\frac{1}{2}}e^{-(t+u+a/4)\lambda}
 \left|\dif\langle\proj_\lambda\riesz_2^*f,h\rangle\right|\dif u\\
 &\overset{\eqref{e.riesz2}}{\leq}
 \frac{a}{2}\sup_{\lambda\in[0,\infty)}
 \left(\lambda^{\critic_p+\frac{1}{2}}e^{-(t+a/2)\lambda}\right)
 \|f\|_{L^2(\skeleton,\medm)}\|h\|_{L^2(\ambient,\meas)}<\infty,
\end{align}
which, combined with \eqref{e.timvf1}, \eqref{e.spec}, \eqref{e.Bt} and \eqref{e.Theta}, proves \eqref{e.timwgt}.

Apply~\eqref{e.timwgt} to an arc $I$ of length $\medm(I)\leq Lr$ and $h\in L^{\infty}(\ambient,\meas)\cap L^{1}(\ambient,\meas)$. Since $\Theta_{a/4}\vecone_I\in L^1(\ambient,\meas)\cap L^2(\ambient,\meas)$ by~\eqref{e.synth4}, we obtain
\begin{align}
 &\phantom{\ \leq}\left|\left\langle(B_{t+a/2}-B_{t+a})\vecone_I,h\right\rangle\right|\\
 &\overset{\eqref{e.timwgt}}{\leq}
 \|\Theta_{a/4}\vecone_I\|_{L^1(\ambient,\meas)}
 \int_{a/4}^{3a/4}
 \|(-\gen)^{\critic_p}P_{t+u}h\|_{L^\infty(\ambient,\meas)}\dif u\\
 &\overset{\eqref{e.synth4},\eqref{e.analy}}{\leq}
 C\min\{1,ra^{-\frac{1}{\dimw}}\}
 \|h\|_{L^\infty(\ambient,\meas)}
 \int_{a/4}^{3a/4}u^{-\critic_p}\dif u\\
 &\leq
 Ca^{1-\critic_p}\min\{1,ra^{-\frac{1}{\dimw}}\}
 \|h\|_{L^\infty(\ambient,\meas)}.
\end{align}
Taking $h=\one_{B(o,k)}\operatorname{sgn}((B_{t+a/2}-B_{t+a})\vecone_I)$ and applying the monotone convergence theorem as $k\uparrow\infty$, we obtain
\begin{equation}\label{e.synth-difference}
 \|(B_{t+a/2}-B_{t+a})\vecone_I\|_{L^1(\ambient,\meas)}
 \leq Ca^{1-\critic_p}\min\{1,ra^{-\frac{1}{\dimw}}\}.
\end{equation}
Using~\eqref{e.synth1}, we have
\begin{equation}
 \dimw(1-\critic_p)
 =1+\dimw(1-\critic_p-1/\dimw)
 \overset{\eqref{e.synth1}}{=}\dimf-\frac{\dimf-1}{p}.
\end{equation}
The triangle inequality therefore yields
\begin{equation}\label{e.bandl1}
 \|U_{n,t}\|_{L^1(\ambient,\meas)}
 \overset{\eqref{e.synth-difference},\eqref{e.bandef}}{\leq}
 C\begin{cases}
 2^{-n(1-\critic_p)}\sum_{\substack{\alpha\in \sI_{j}\\ j\in\bZ}}r_j^{\dimf-\frac{\dimf-1}{p}}|\kappa_{j,\alpha}|,& \text{if }n\in\bN\cup\{0\},\\
 2^{-n(1-\critic_p-\frac{1}{\dimw})}\sum_{\substack{\alpha\in \sI_{j}\\ j\in\bZ}}r_j^{\dimf-\frac{\dimf-1}{p}}|\kappa_{j,\alpha}|,&\text{if }-n\in\bN.
 \end{cases}
\end{equation}
\item \textit{The $L^2$ estimate of $U_{n,t}$.}

For fixed $n\in\bZ$, the intervals $[a_{j,n}/4,3a_{j,n}/4]$ are pairwise disjoint, since $\dimw=\dimf+1>2$ and
\begin{equation}\label{e.synth-disjoint}
 \frac{3a_{j+1,n}}{4}
 \overset{\eqref{e.bandef}}{=}
 3\cdot2^{-\dimw}\frac{a_{j,n}}{4}
 <\frac{a_{j,n}}{4}.
\end{equation}
For $h\in L^2(\ambient,\meas)$, by Fubini's theorem and the spectral theorem,\begin{align}
 &\phantom{\ \leq}\int_0^\infty u^{2\critic_p-1}
       \|(-\gen)^{\critic_p}P_{t+u}h\|_{L^2(\ambient,\meas)}^2\dif u\\
 &\overset{\eqref{e.spec}}{=}
 \int_{(0,\infty)}e^{-2t\lambda}
 \left(\int_0^\infty u^{2\critic_p-1}\lambda^{2\critic_p}
                         e^{-2u\lambda}\dif u\right)
 \dif\langle\proj_\lambda h,h\rangle\\
 &=
 \frac{\Gamma(2\critic_p)}{2^{2\critic_p}}
 \int_{(0,\infty)}e^{-2t\lambda}\dif\langle\proj_\lambda h,h\rangle
 \leq C\|h\|_{L^2(\ambient,\meas)}^2.\label{e.synth-square}
\end{align}
By applying the Cauchy--Schwarz inequality to the integral in~\eqref{e.timwgt}, we have
\begin{align}
|\langle U_{n,t},h\rangle|&\overset{\eqref{e.bandef},\eqref{e.timwgt}}{\leq}
 \sum_j\int_{a_{j,n}/4}^{3a_{j,n}/4}
 \|\Theta_{a_{j,n}/4}f_j\|_{L^2(\ambient,\meas)}
 \|(-\gen)^{\critic_p}P_{t+u}h\|_{L^2(\ambient,\meas)}\dif u\\
 &\leq
 \left(\sum_j\|\Theta_{a_{j,n}/4}f_j\|_{L^2(\ambient,\meas)}^2
       \int_{a_{j,n}/4}^{3a_{j,n}/4}u^{1-2\critic_p}\dif u\right)^{\frac{1}{2}}\\
 &\qquad \times      
 \left(\sum_j\int_{a_{j,n}/4}^{3a_{j,n}/4}u^{2\critic_p-1}
       \|(-\gen)^{\critic_p}P_{t+u}h\|_{L^2(\ambient,\meas)}^2\dif u\right)^{\frac{1}{2}}\\
 &\overset{\eqref{e.synth-disjoint},\eqref{e.synth-square}}{\leq}
 C\left(\sum_j a_{j,n}^{2-2\critic_p}
       \|\Theta_{a_{j,n}/4}f_j\|_{L^2(\ambient,\meas)}^2\right)^{\frac{1}{2}}
 \|h\|_{L^2(\ambient,\meas)}.
\end{align}
Taking the supremum over $\|h\|_{L^2(\ambient,\meas)}\leq1$ yields
\begin{equation}\label{e.synth-l2sum}
 \|U_{n,t}\|_{L^2(\ambient,\meas)}^2
 \leq C\sum_j a_{j,n}^{2-2\critic_p}
       \|\Theta_{a_{j,n}/4}f_j\|_{L^2(\ambient,\meas)}^2.
\end{equation}
If $n\in\bN\cup\{0\}$, then $a_{j,n}/4\leq r_j^{\dimw}$, so
\begin{align}
 a_{j,n}^{2-2\critic_p}
       \|\Theta_{a_{j,n}/4}f_j\|_{L^2(\ambient,\meas)}^2&\overset{\eqref{e.synt10}}{\leq} Ca_{j,n}^{2-2\critic_p-\frac{\dimf}{\dimw}}
       \sum_\alpha|\kappa_{j,\alpha}|^2\\
 &\overset{\eqref{e.bandef},\eqref{e.synth1}}{=}
 C2^{-n(2-2\critic_p-\frac{\dimf}{\dimw})}
       \sum_\alpha r_j^{\dimf-\frac{2(\dimf-1)}{p}}|\kappa_{j,\alpha}|^2.
\end{align}
If $-n\in\bN$, the estimate valid for all $s>0$ gives
\begin{align}
 a_{j,n}^{2-2\critic_p}
       \|\Theta_{a_{j,n}/4}f_j\|_{L^2(\ambient,\meas)}^2 &\overset{\eqref{e.synth-large}}{\leq}
 Ca_{j,n}^{1-2\critic_p}r_j
       \sum_\alpha|\kappa_{j,\alpha}|^2\\
 &\overset{\eqref{e.bandef},\eqref{e.synth1}}{=}
 C2^{-n(1-2\critic_p)}
       \sum_\alpha r_j^{\dimf-\frac{2(\dimf-1)}{p}}|\kappa_{j,\alpha}|^2.
\end{align}
Consequently,
\begin{equation}\label{e.bandl2}
 \|U_{n,t}\|_{L^2(\ambient,\meas)}
 \overset{\eqref{e.synth-l2sum}}{\leq}
 C\begin{cases}
 2^{-n(1-\critic_p-\frac{\dimf}{2\dimw})}\left(\sum_{\substack{\alpha\in \sI_{j}\\ j\in\bZ}}
 r_j^{\dimf-\frac{2(\dimf-1)}{p}}|\kappa_{j,\alpha}|^2\right)^{\frac{1}{2}},&\ n\in\bN\cup\{0\},\\
 2^{-n(\frac{1}{2}-\critic_p)}\left(\sum_{\substack{\alpha\in \sI_{j}\\ j\in\bZ}}
 r_j^{\dimf-\frac{2(\dimf-1)}{p}}|\kappa_{j,\alpha}|^2\right)^{\frac{1}{2}},&\ -n\in\bN.
 \end{cases}
\end{equation}
\item \textit{Real interpolation and summation.}

Define
\begin{equation}
 \sJ:=\{(j,\alpha):\alpha\in\sI_j\}
\ \text{ and }\
 \omega(\{(j,\alpha)\}):=r_j^{\dimf}, \ \text{ for }(j,\alpha)\in\sJ.
\end{equation}
We will view $\omega$ as a measure on the discrete $\sigma$-algebra of $\sJ$. For fixed $n\in\bZ$ and $t\in(0,\infty)$, define the real linear map
\begin{equation}
 T_{n,t}v:=\sum_{(j,\alpha)\in\sJ}
 r_j^{\frac{\dimf-1}{p}}v_{j,\alpha}
 (B_{t+a_{j,n}/2}-B_{t+a_{j,n}})
 \vecone_{(a_{j,\alpha},b_{j,\alpha})}, \text{ for }v=(v_{j,\alpha})_{(j,\alpha)\in\sJ}\in \bR^{\sJ}.
\end{equation}
If $v\in\bR^{\sJ}$ satisfies $v_{j,\alpha}=r_j^{-\frac{\dimf-1}{p}}\kappa_{j,\alpha}$, then
\begin{equation}
 T_{n,t}v=U_{n,t},\qquad
 \|v\|_{L^b(\sJ,\omega)}^b
 =\sum_{j,\alpha}r_j^{\dimf-\frac{b(\dimf-1)}{p}}
       |\kappa_{j,\alpha}|^b,\qquad 1\leq b<\infty.
\end{equation}
Thus~\eqref{e.bandl1} and~\eqref{e.bandl2} are $L^1$ and $L^2$ operator bounds on these two $\sigma$-finite measure spaces. For $q\in(1,2)$, by \eqref{e.realip} and the facts that
\begin{align}
 &(1-\critic_p)\left(\frac{2}{q}-1\right)
  +\left(1-\critic_p-\frac{\dimf}{2\dimw}\right)
       \left(2-\frac{2}{q}\right)\overset{\eqref{e.synth1}}{=}
 \frac{1}{\dimw}\left(\frac{\dimf}{q}-\frac{\dimf-1}{p}\right)>0,
 \label{e.synth-positive}\\
 &\left(1-\critic_p-\frac{1}{\dimw}\right)
       \left(\frac{2}{q}-1\right)
  +\left(\frac{1}{2}-\critic_p\right)\left(2-\frac{2}{q}\right)\overset{\eqref{e.synth1}}{=}
 -\frac{\dimf-1}{\dimw}\left(\frac{1}{p}-\frac{1}{q}\right)<0,
 \label{e.synth-negative}
\end{align}
we obtain
\begin{align}
 &\phantom{\ \leq}\|U_{n,t}\|_{L^q(\ambient,\meas)}\\
 &\overset{\eqref{e.bandl1},\eqref{e.bandl2},\eqref{e.realip}}{\leq}
 C\begin{cases}
 2^{-\frac{n}{\dimw}(\frac{\dimf}{q}-\frac{\dimf-1}{p})}\left(\sum_{j,\alpha}
 r_j^{\dimf-\frac{q(\dimf-1)}{p}}|\kappa_{j,\alpha}|^q\right)^{\frac{1}{q}},&\text{if }n\in\bN\cup\{0\},\\
 2^{\frac{n(\dimf-1)}{\dimw}(\frac{1}{p}-\frac{1}{q})}\left(\sum_{j,\alpha}
 r_j^{\dimf-\frac{q(\dimf-1)}{p}}|\kappa_{j,\alpha}|^q\right)^{\frac{1}{q}},&\text{if }-n\in\bN.
 \end{cases}\label{e.bandlq}
\end{align}
If $q=2$, the same estimate follows directly from~\eqref{e.bandl2}.

For each integer $k\in\bN$, by the triangle inequality,
\begin{align}
 &\phantom{\ \leq}\left\|\sum_{n=-k}^kU_{n,t}\right\|_{L^q(\ambient,\meas)}\\
 &\overset{\eqref{e.bandlq}}{\leq}
 C\left(
 \frac{1}{1-2^{-\frac{1}{\dimw}(\frac{\dimf}{q}-\frac{\dimf-1}{p})}}
 +\frac{1}{2^{\frac{\dimf-1}{\dimw}(\frac{1}{p}-\frac{1}{q})}-1}
 \right)
 \left(\sum_{j,\alpha}
 r_j^{\dimf-\frac{q(\dimf-1)}{p}}|\kappa_{j,\alpha}|^q\right)^{\frac{1}{q}}\\
 &\leq C
 \left(\sum_{j,\alpha}
 r_j^{\dimf-\frac{q(\dimf-1)}{p}}|\kappa_{j,\alpha}|^q\right)^{\frac{1}{q}}.\label{e.synth-partial}
\end{align}
On the other hand, 
\begin{equation}\label{e.synend}
 \sum_{n=-k}^kU_{n,t}
 \overset{\eqref{e.bandef}}{=}
 \sum_j\left(B_{t+2^{-k-1}r_j^{\dimw}}
             -B_{t+2^kr_j^{\dimw}}\right)f_j.
\end{equation}
Strong continuity of $\{P_s\}_{s\in(0,\infty)}$ on $L^2(\ambient,\meas)$ and boundedness of $(-\gen)^{\critic_p-\frac{1}{2}}P_t$ give, for each $j\in\bZ$,
\begin{align}
 &\phantom{\ \leq}\lim_{k\to\infty}\|B_{t+2^{-k-1}r_j^{\dimw}}f_j-B_tf_j\|_{L^2(\ambient,\meas)}\\
 &\overset{\eqref{e.Bt},\eqref{e.spec}}{\leq} \lim_{k\to\infty}
 Ct^{-(\critic_p-\frac{1}{2})}
 \|(P_{2^{-k-1}r_j^{\dimw}}-I)\riesz_2^*f_j\|_{L^2(\ambient,\meas)}=0,
\end{align}
and
\begin{equation}
 \lim_{k\to\infty} \|B_{t+2^kr_j^{\dimw}}f_j\|_{L^2(\ambient,\meas)}
 \overset{\eqref{e.Bt},\eqref{e.spec},\eqref{e.riesz2}}{\leq} \lim_{k\to\infty}
 C(t+2^kr_j^{\dimw})^{-(\critic_p-\frac{1}{2})}
 \|f_j\|_{L^2(\skeleton,\medm)}=0.
\end{equation}
By \eqref{e.synend}, we know that $\sum_{n=-k}^kU_{n,t}\to B_t\sum_jf_j$ in $L^2(\ambient,\meas)$. Choose an $\meas$-a.e. convergent subsequence. By Fatou's lemma, ~\eqref{e.synth-partial} and definition of $K_{j}$,
\begin{align}
 \left\|B_t\sum_jf_j\right\|_{L^q(\ambient,\meas)}^q\overset{\eqref{e.synth-partial}}{\leq}
 C\sum_{j,\alpha}r_j^{\dimf-\frac{q(\dimf-1)}{p}}|\kappa_{j,\alpha}|^q=
 C\sum_jr_j^{-\frac{(\dimf-1)(q-p)}{p}}K_j.
\end{align}
Taking the supremum over $t\in(0,\infty)$ proves~\eqref{e.synthe}.
\qedhere
\end{enumerate}
\end{proof}

\subsection{Whitney decomposition of the gradient}\label{ss.weak}
In this subsection, we prove the weak reverse inequality \eqref{e.scalw} for $1\le p<2$. We then use the spectral theorem to prove Theorem~\ref{t.strfl}-\ref{it.p2}.

We first recall the following Whitney covering from \cite[Proposition~3.2]{Mur24}.
\begin{proposition}\label{p.whitne}
Suppose that $(\ambient,\metric)$ is $\AR{\dimf}$ with $\dimf\in[1,\infty)$. Let $\Omega\subset\ambient$ be nonempty and open, with $\meas(\Omega)<\infty$. There is an at most countable family $\{B_i=B(x_i,r_i)\}_{i\in I}$, such that the balls $B_i/5$ are pairwise disjoint,
\begin{equation}\label{e.whitne}
 \Omega=\bigcup_{i\in I}B_i,\qquad
 10B_i\subset\Omega,\qquad40B_i\setminus\Omega\neq\emptyset,\ \text{ and }\ 
 \sum_{i\in I}\one_{10B_i}\leq C.
\end{equation}
The constant $C$ depends only on $C_0$ and $\dimf$.
\end{proposition}
\begin{proof}
By $\AR{\dimf}$, the metric space $(\ambient,\metric)$ is metric doubling. Thus, we may apply \cite[Proposition~3.2-(a),(d)]{Mur24} with $\varepsilon=1/99$, and obtain disjoint balls $D_i=B(x_i,\delta_i/100)$, where $\delta_i=\dist(x_i,\ambient\setminus\Omega)\in(0,\infty)$ such that
\begin{equation}\label{e.whiten1}
 \Omega=\bigcup_{i\in I}2(1+\varepsilon)D_i,\ \text{ and }\
 \sum_{i\in I}\one_{99D_i}\leq C.
\end{equation}
Set $r_i=\delta_i/20$ and $B_i=5D_i$. Then \eqref{e.whitne} follows from \eqref{e.whiten1} and the definition of $\delta_{i}$.
\end{proof}

\begin{proof}[Proof of Theorem~\ref{t.strfl}-\ref{it.scale}]
Fix $p\in[1,2)$ and choose
\begin{equation}\label{e.weak-q}
 q:=\frac{1}{2}\left(p+\min\left\{2,\frac{p\dimf}{\dimf-1}\right\}\right).
\end{equation}
Then $p<q<2$ and~\eqref{e.qrange} holds. Also, $\dimw=\dimf+1$ and $\frac{1}{2}<\critic_p<1$ by~\eqref{e.critp}. We first prove
\begin{equation}\label{e.fieldw}
 \sup_{t\in(0,\infty)}\|B_tg\|_{L^{p,\infty}(\ambient,\meas)}
 \leq C_p\|g\|_{L^p(\skeleton,\medm)},\ \text{ for all } g\in\sA.
\end{equation}
\begin{enumerate}[label=\textit{Step {\arabic*}.}, align=right, leftmargin=*, topsep=5pt, parsep=0pt, itemsep=2pt]
\item \textit{A Whitney decomposition.}

The case $g=0$ is immediate. Let $0\neq g\in\sA$. By Lemma~\ref{l.finarc} and \eqref{e.finedg}, after discarding zero coefficients, we may find some open arcs $(u_l,v_l)$, $l\in\{1,\ldots,N\}$ such that $(u_l,v_l)\cap(u_k,v_k)=\emptyset$ for $l\neq k$, and write
\begin{equation}\label{e.weak-arcs}
g=\sum_{l=1}^{N}\gamma_l\vecone_{(u_l,v_l)}
 \ \ \medm\text{-a.e.},\ \text{ for some } \gamma_l\neq0.
\end{equation}
Define $K_g:=\bigcup_{l=1}^{N}[u_l,v_l]$, which is nonempty and compact. Set
\begin{equation}
	\sM g(x):=\sup\Sett{\frac{1}{\meas(B)}\int_{B}\abs{g}^{p}\dif\medm}{B\text{ is an open ball in $(\ambient,\metric)$,\ $x\in B$}},\ x\in\ambient.
\end{equation}
For $\lambda\in(0,\infty)$, we set
 \[\sB_{\lambda}:=\Sett{B\text{ is an open metric ball in $(\ambient,\metric)$}}{\int_{B}\abs{g}^{p}\dif\medm>\lambda^{p}\meas(B)}.\] 
Let $\Omega_{\lambda}:=\bigcup_{B\in\sB_{\lambda}}B=\{\sM g>\lambda^{p}\}$. 
If $B(x,r)\in\sB_{\lambda}$, then \begin{equation}\label{e.Omega-1}
	r^{\dimf}\leq C\meas(B(x,r))\leq C\lambda^{-p}\int_{B}\abs{g}^{p}\dif\medm\leq C\lambda^{-p}\norm{g}_{L^{p}(\skeleton,\medm)}^{p}.
\end{equation}So the balls in $\sB_{\lambda}$ have uniformly bounded radii. By the $5B$-covering lemma \cite[Theorem 1.2]{Hei01}, there exists a subfamily of disjoint balls $(B_{i})_{i\in I}$, such that $\Omega_{\lambda}=\bigcup_{B\in\sB_{\lambda}}B\subset \bigcup_{i\in I}5B_{i}$. Thus 
\begin{align}
	\meas(\Omega_{\lambda})\leq \meas\big(\bigcup_{i\in I}5B_{i}\big)\leq C\sum_{i\in I}\meas(B_{i})\overset{\eqref{e.Omega-1}}{\leq} C\lambda^{-p}\norm{g}_{L^{p}(\skeleton,\medm)}^{p}.\label{e.Omega}
\end{align}

For $z\in[u_l,v_l]$ and $\rho\in(0,\metric(u_l,v_l)]$,
\begin{equation}
 \medm\bigl(B(z,\rho)\cap(u_l,v_l)\bigr)
 =\min\{\rho,\metric(u_l,z)\}+\min\{\rho,\metric(z,v_l)\}
 \geq\rho.
\end{equation}
Consequently,
\begin{align}
 \sM g(z)\geq
 \frac{1}{\meas(B(z,\rho))}\int_{B(z,\rho)}\abs{g}^{p}\dif\medm\overset{\eqref{e.weak-arcs},\eqref{e.ahlfor}}{\geq}
 C_0^{-1}|\gamma_l|^p\rho^{1-\dimf}
 \uparrow\infty,\ \text{ as }\rho\downarrow0.
\end{align}
Thus $K_g\subset\Omega_{\lambda}$. Applying Proposition \ref{p.whitne} to $\Omega_{\lambda}$, we obtain a family of Whitney balls satisfying \eqref{e.whitne} that covers $\Omega_{\lambda}$. Since $K_g\subset\Omega_{\lambda}$ and $K_{g}$ is compact, we may select a finite subcover of $K_g$ in this family of Whitney balls, denoted by $\{B_i=B(x_i,r_i)\}_{i=1}^{J}$ such that
\begin{equation}\label{e.weak-whitney}
 K_g\subset\bigcup_{i=1}^{J}B_i,\quad
 10B_i\subset\Omega_{\lambda},\quad40B_i\setminus\Omega_{\lambda}\neq\emptyset,\ \text{ and }\ \sum_{i=1}^{J}\one_{10B_i}\overset{\eqref{e.whitne}}{\leq}C.
\end{equation}
Define $E_{1}:=B_{1}$,
\begin{equation}\label{e.weak-pieces}
 E_i:=B_i\setminus\bigcup_{k<i}B_k,\ 
 g_i:=g\one_{E_i},\ \text{ and }\ G_i:=\|g_i\|_{L^p(\skeleton,\medm)}.
\end{equation}
The sets $(E_i)_{i=1}^{J}$ are disjoint and cover $K_g$. Hence $g=\sum_{i=1}^{J}g_i$ $\medm$-a.e. on $\skeleton$, and $\|g_i\|_{L^2(\skeleton,\medm)}
 \leq\|g\|_{L^2(\skeleton,\medm)}<\infty$. Choose $y_i\in40B_i\setminus\Omega_{\lambda}$. Then $B_i\subset B(y_i,50r_i)$ and $\sM g(y_i)\leq\lambda^p$, so
\begin{align}\label{e.whitn11}
\|g_i\|_{L^p(\skeleton,\medm)}^p
 &\leq\int_{B(y_i,50r_i)}\abs{g}^{p}\dif\medm
 \leq\lambda^p\meas(B(y_i,50r_i))
 \overset{\eqref{e.ahlfor}}{\leq}C\lambda^pr_i^{\dimf},\\
 \sum_{i=1}^{J}\|g_i\|_{L^p(\skeleton,\medm)}^p
 &\overset{\eqref{e.weak-pieces}}{=}
 \sum_{i=1}^{J}\int_{E_i}|g|^p\dif\medm
 =\|g\|_{L^p(\skeleton,\medm)}^p, \label{e.whitn12}
\\
 \sum_{i=1}^{J}r_i^{\dimf}
 &\overset{\eqref{e.ahlfor}}{\leq}C_0\sum_{i=1}^{J}\meas(B_i)
 \overset{\eqref{e.weak-whitney}}{\leq}C\meas(\Omega_{\lambda})
 \overset{\eqref{e.Omega}}{\leq}
 C\lambda^{-p}\|g\|_{L^p(\skeleton,\medm)}^p.\label{e.whitn13}
\end{align}
\item \textit{Split $B_tg$ into two parts.}

For each $i\in\{1,\ldots,J\}$, let
\begin{equation}
 F_i(x):=\int_{\skeleton}\vecone_{(x_i,x)}g_i\dif\medm,
 \qquad x\in\ambient.
\end{equation}
Then, by Lemma~\ref{l.primit}, 
\begin{equation}\label{e.Fi}
F_{i}\in \abscon(\ambient,\metric),\quad  \wgrad F_i=g_i
 \quad\medm\text{-a.e. on $\skeleton$},\ \text{ and }\
 \|F_i\|_{\sup}
 \leq 
 r_i^{1-\frac{1}{p}}\|g_i\|_{L^p(\skeleton,\medm)}.
\end{equation}
Let $\chi_i:=\chi_{x_i,r_i}$ be the cutoff function defined in Lemma \ref{l.cutoff}, and let $u_i:=\chi_iF_i\in \abscon(\ambient,\metric)$ and $c_i:=F_i\wgrad\chi_i$. By Lemma~\ref{l.primit}-\ref{it.prim4}, $F_i$ is constant on each connected component of $\ambient\setminus B_i$. For each $i$, let $[a_{i,l},b_{i,l}]$, $1\leq l\leq N_i\leq M$, be the arcs from Lemma \ref{l.cutoff}-\ref{it.tcut2}, after omitting degenerate arcs. Let $F_{i,l}$ be the constant value of $F_i$ on $[a_{i,l},b_{i,l}]$. Thus
\begin{equation}\label{e.weak-cutoff-field}
F_i\wgrad\chi_i\overset{\eqref{e.cutoff}}{=}
 -\frac{1}{r_i}\sum_{l=1}^{N_i}F_{i,l}\vecone_{(a_{i,l},b_{i,l})}\in\sA,
 \qquad |F_{i,l}|\overset{\eqref{e.Fi}}{\leq}r_i^{1-\frac{1}{p}}\|g_i\|_{L^p(\skeleton,\medm)}.
\end{equation}
By the product rule and Lemma~\ref{l.cutoff}-\ref{it.tcut1},
\begin{align}\label{e.primi2}
 \wgrad (\chi_iF_i)
 \overset{\eqref{e.Fi}}{=}
 \chi_ig_i+F_i\wgrad\chi_i=g_i+c_i,\ \quad |\chi_iF_i|
 \leq
 \|F_i\|_{\sup}\one_{B(x_i,3r_i)},
\end{align}
and 
\begin{equation}
	 \|\chi_iF_i\|_{L^1(\ambient,\meas)}
 \leq\meas(B(x_i,3r_i))\|F_i\|_{\sup}
 \overset{\eqref{e.ahlfor},\eqref{e.Fi}}{\leq}
 Cr_i^{\dimf+1-\frac{1}{p}}\|g_i\|_{L^p(\skeleton,\medm)}.
\end{equation}
Thus,
\begin{equation}
 \begin{gathered}
 \|\chi_iF_i\|_{L^2(\ambient,\meas)}^2
 \overset{\eqref{e.primi2}}{\leq}
 \meas(B(x_i,3r_i))\|F_i\|_{\sup}^2<\infty,\\
 \|\wgrad (\chi_iF_i)\|_{L^2(\skeleton,\medm)}
 \overset{\eqref{e.primi2}}{\leq}
 \|g_i\|_{L^2(\skeleton,\medm)}+\|F_i\wgrad\chi_i\|_{L^2(\skeleton,\medm)}<\infty.
 \end{gathered}
\end{equation}
Therefore $\chi_iF_i\in\sobolev{2}(\ambient)=\domain$ by~\eqref{e.sobdef} and~\eqref{e.formde}. Hence, for all $t\in(0,\infty)$,
\begin{align}
 B_tg&\overset{\eqref{e.primi2}}{=}
 \sum_{i=1}^{J}B_t\wgrad u_i-B_t\left(\sum_{i=1}^{J} F_i\wgrad\chi_i\right)\\
 &\overset{\eqref{e.Bt}}{=}
 \sum_{i=1}^{J}(-\gen)^{\critic_p}P_tu_i
 -B_t\left(\sum_{i=1}^{J}F_i\wgrad\chi_i\right).\label{e.weak-split}
\end{align}
\item \textit{Estimates of the second part.}

We group the Whitney balls by radius, and define
\begin{equation}\label{e.weak-scales}
 I_j:=\Sett{i\in\{1,\ldots,J\}}{2^{-j-1}<r_i\leq2^{-j}}, \text{ and }
 f_j:=\sum_{i\in I_j}F_i\wgrad\chi_i, \text{ for } j\in\bZ.
\end{equation}
For $i\in I_j$, we have
\begin{equation}
 \medm([a_{i,l},b_{i,l}])\overset{\eqref{e.cutoff0}}{\leq}r_i\leq2^{-j},
 \ \text{ and }\
 [a_{i,l},b_{i,l}]\overset{\eqref{e.cutoff0}}{\subset}
 \overline{B(x_i,2r_i)}\setminus B_i.
\end{equation}
If $B(y,2^{-j})\cap [a_{i,l},b_{i,l}]\neq\emptyset$, then $\metric(y,x_i)<2^{-j}+2r_i
 \overset{\eqref{e.weak-scales}}{<}4r_i$, so $y\in10B_i$. Therefore,
\begin{align}
 &\phantom{\ \leq}\#\Sett{(i,l)}{i\in I_j,\ 1\leq l\leq N_i,\
 B(y,2^{-j})\cap[a_{i,l},b_{i,l}]\neq\emptyset}\\
 &\leq M\sum_{i\in I_j}\one_{10B_i}(y)
 \overset{\eqref{e.weak-whitney}}{\leq}CM.
\end{align}
Thus Lemma~\ref{l.synthe} applies to $(f_j)_{j\in\bZ}$ with $L=1$ and overlap bound $N=CM$. Its coefficient satisfies
\begin{align}
 K_j:=2^{-j}\sum_{i\in I_j}\sum_{l=1}^{N_i}|r_i^{-1}F_{i,l}|^q&\overset{\eqref{e.weak-cutoff-field}}{\leq}
 M2^{-j}\sum_{i\in I_j}r_i^{-\frac{q}{p}}\|g_i\|_{L^p(\skeleton,\medm)}^q\\
 &\overset{\eqref{e.weak-scales}}{\leq}
 2M\sum_{i\in I_j}r_i^{1-\frac{q}{p}}\|g_i\|_{L^p(\skeleton,\medm)}^q.\label{e.weak-Kj}
\end{align}
Since $q>p$ and $r_i\leq2^{-j}$ for $i\in I_j$, we know that
\begin{equation}
 2^{\frac{j(\dimf-1)(q-p)}{p}}
 \leq r_i^{-\frac{(\dimf-1)(q-p)}{p}}.
\end{equation}
It follows that
\begin{align}
 &\phantom{\ \leq}\sup_{t\in(0,\infty)}\left\|B_t\left(\sum_{i=1}^{J}(F_i\wgrad\chi_i)\right)\right\|_{L^q(\ambient,\meas)}^q\overset{\eqref{e.synthe}}{\leq}
 C\sum_j2^{\frac{j(\dimf-1)(q-p)}{p}}K_j\\
 &\overset{\eqref{e.weak-Kj},\eqref{e.weak-scales}}{\leq}
 C\sum_{i=1}^{J}r_i^{-\frac{(\dimf-1)(q-p)}{p}+1-\frac{q}{p}}\|g_i\|_{L^p(\skeleton,\medm)}^q=C\sum_{i=1}^{J}
 \left(\frac{\|g_i\|_{L^p(\skeleton,\medm)}^p}{r_i^{\dimf}}\right)^{\frac{q-p}{p}}\|g_i\|_{L^p(\skeleton,\medm)}^p\\
 &\overset{\eqref{e.whitn11}}{\leq}
 C\lambda^{q-p}\sum_{i=1}^{J}\|g_i\|_{L^p(\skeleton,\medm)}^p
 \overset{\eqref{e.whitn12}}{=}
 C\lambda^{q-p}\|g\|_{L^p(\skeleton,\medm)}^p.\label{e.good}
\end{align}
\item \textit{Estimate the first part.}

 By~\eqref{e.weak-whitney}, $\ambient\setminus\Omega_{\lambda}\subset\ambient\setminus10B_i$. For $x\notin10B_i$ and $y\in\supp(u_i)$, we have
\begin{equation}
 \metric(x,y)\geq\metric(x,x_i)-\metric(y,x_i)
 \overset{\eqref{e.primi2}}{\geq}10r_i-3r_i=7r_i.
\end{equation}
Covering $\ambient\setminus10B_i$ by the annuli $B(y,7\cdot2^{k+1}r_i)\setminus B(y,7\cdot2^kr_i)$, $k\in\bN\cup\{0\}$, we have
\begin{align}
 &\phantom{\ \leq}\sup_{t\in(0,\infty)}\int_{\ambient\setminus10B_i}|q_t^{(\critic_p)}(x,y)|\dif\meas(x)\\
 &\overset{\eqref{e.fracke}}{\leq}
 C\sum_{k=0}^{\infty}(7\cdot2^kr_i)^{-\dimf-\dimw\critic_p}
 \meas(B(y,7\cdot2^{k+1}r_i))\\
 &\overset{\eqref{e.ahlfor}}{\leq}
 Cr_i^{-\dimw\critic_p}\sum_{k=0}^{\infty}2^{-k\dimw\critic_p}
 =\frac{Cr_i^{-\dimw\critic_p}}{1-2^{-\dimw\critic_p}}.\label{e.weak-tail}
\end{align}
Using \eqref{e.fracin} and Tonelli's theorem, we obtain
\begin{align}
 &\phantom{\ \leq}\int_{\ambient\setminus\Omega_{\lambda}}
 \left|\sum_{i=1}^{J}(-\gen)^{\critic_p}P_tu_i\right|\dif\meas\\
 &\overset{\eqref{e.fracin}}{\leq}
 \sum_{i=1}^{J}\int_{\supp(u_i)}|u_i(y)|
 \int_{\ambient\setminus10B_i}|q_t^{(\critic_p)}(x,y)|\dif\meas(x)\dif\meas(y)\\
 &\overset{\eqref{e.weak-tail}}{\leq}
 C\sum_{i=1}^{J}r_i^{-\dimw\critic_p}\|u_i\|_{L^1(\ambient,\meas)}\overset{\eqref{e.primi2},\eqref{e.critp}}{\leq}
 C\sum_{i=1}^{J}r_i^{\dimf(1-\frac{1}{p})}G_i\\
 &\overset{\text{(H\"older)}}{\leq}
 C\left(\sum_{i=1}^{J}r_i^{\dimf}\right)^{1-\frac{1}{p}}
 \left(\sum_{i=1}^{J}\|g_i\|_{L^p(\skeleton,\medm)}^p\right)^{\frac{1}{p}}\\
 &\overset{\eqref{e.whitn13}}{\leq}
 C\left(\lambda^{-p}\|g\|_{L^p(\skeleton,\medm)}^p\right)^{1-\frac{1}{p}}
 \|g\|_{L^p(\skeleton,\medm)}=C\lambda^{1-p}\|g\|_{L^p(\skeleton,\medm)}^p,\label{e.bad}
\end{align}
where we used
\begin{equation}
 \dimf+1-\frac{1}{p}-\dimw\critic_p
 \overset{\eqref{e.critp}}{=}
 \dimf+1-\frac{1}{p}-1-\frac{\dimf-1}{p}
 =\dimf\left(1-\frac{1}{p}\right).
\end{equation}
When $p=1$, the step in~\eqref{e.bad} that uses H\"older's inequality is an equality, so the same calculation works.
\item \textit{Proof of \eqref{e.fieldw} and extension to $\sobolev{p}(\ambient).$}

For every $t\in(0,\infty)$, by the decomposition~\eqref{e.weak-split} and Chebyshev's inequality,
\begin{align}
&\phantom{\ \leq} \meas(\{|B_tg|>\lambda\})\\
 &\overset{\eqref{e.weak-split}}{\leq}
 \meas(\Omega_{\lambda})
 +\frac{2}{\lambda}\int_{\ambient\setminus\Omega_{\lambda}}
 \left|\sum_{i=1}^{J}(-\gen)^{\critic_p}P_t(\chi_iF_i)\right|\dif\meas+\frac{2^q}{\lambda^q}
 \left\|B_t\left(\sum_{i=1}^{J}F_i\wgrad\chi_i\right)\right\|_{L^q(\ambient,\meas)}^q\\
 &\overset{\eqref{e.Omega},\eqref{e.good},\eqref{e.bad}}{\leq}
 C\left(\lambda^{-p}+\lambda^{-1}\lambda^{1-p}
                    +\lambda^{-q}\lambda^{q-p}\right)
 \|g\|_{L^p(\skeleton,\medm)}^p\leq C\lambda^{-p}\|g\|_{L^p(\skeleton,\medm)}^p.
\end{align}
The constant is independent of $\lambda$, $t$ and $g$. Taking the supremum over $\lambda,t\in(0,\infty)$ proves~\eqref{e.fieldw}.

For $f\in\core$, we have $\wgrad f\in\sA$ by~\eqref{e.core}, and therefore
\begin{equation}\label{e.weak-core}
 \sup_{t\in(0,\infty)}\|(-\gen)^{\critic_p}P_tf\|_{L^{p,\infty}(\ambient,\meas)}
 \overset{\eqref{e.Bt},\eqref{e.fieldw}}{\leq}
 C_p\|\wgrad f\|_{L^p(\skeleton,\medm)}.
\end{equation}
Let $f\in\sobolev{p}(\ambient)$. By Lemma~\ref{l.dense}, there exist $f_n\in\core$ such that
\begin{equation}
\lim_{n\to\infty} (\|f_n-f\|_{L^p(\ambient,\meas)}
 +\|\wgrad f_n-\wgrad f\|_{L^p(\skeleton,\medm)})
 \overset{\eqref{e.scaled}}{=}0.
\end{equation}
For each fixed $t>0$, the bounded $L^p$ extension in~\eqref{e.analy} gives
\begin{equation}
\lim_{n\to\infty}\|(-\gen)^{\critic_p}P_t(f_n-f)\|_{L^p(\ambient,\meas)}
 \overset{\eqref{e.analy}}{\leq}\lim_{n\to\infty}
 Ct^{-\critic_p}\|f_n-f\|_{L^p(\ambient,\meas)}=0.
\end{equation}
By pass to a subsequence, we may assume that $f_{n}\to f$ $\meas$-a.e. on $\ambient$, and $(-\gen)^{\critic_p}P_tf_{n}\to (-\gen)^{\critic_p}P_tf$ $\meas$-a.e. on $\ambient$. Then,
\begin{equation}
 \one_{\{|(-\gen)^{\critic_p}P_tf|>\lambda\}}
 \leq\liminf_{n\to\infty}
 \one_{\{|(-\gen)^{\critic_p}P_tf_n|>\lambda\}}.
\end{equation}
Thus
\begin{align}
 \meas(\{|(-\gen)^{\critic_p}P_tf|>\lambda\})&\overset{\text{(Fatou)}}{\leq}
 \liminf_{n\to\infty}
 \meas(\{|(-\gen)^{\critic_p}P_tf_n|>\lambda\})\\
 &\overset{\eqref{e.weak-core}}{\leq}
 C\lambda^{-p}\lim_{n\to\infty}
 \|\wgrad f_n\|_{L^p(\skeleton,\medm)}^p
 \overset{\eqref{e.scaled}}{=}
 C\lambda^{-p}\|\wgrad f\|_{L^p(\skeleton,\medm)}^p.
\end{align}
Taking the supremum over $\lambda,t>0$ proves~\eqref{e.scalw}.
\qedhere
\end{enumerate}
\end{proof}

\begin{proof}[Proof of Theorem~\ref{t.strfl}-\ref{it.p2}]
  Let $f\in\sobolev{2}(\ambient)=\domain=\Dom((-\gen)^{\frac{1}{2}})$. Spectral calculus gives $P_tf\in\Dom((-\gen)^{\frac{1}{2}})$ and $(-\gen)^{\frac{1}{2}}P_tf=P_t(-\gen)^{\frac{1}{2}}f$. Hence $L^2$-contractivity of the semigroup and~\eqref{e.spec} imply
  \begin{equation}
   \begin{aligned}
   \|(-\gen)^{\frac{1}{2}}P_tf\|_{L^2(\ambient,\meas)}
   =\|P_t(-\gen)^{\frac{1}{2}}f\|_{L^2(\ambient,\meas)}\leq\|(-\gen)^{\frac{1}{2}}f\|_{L^2(\ambient,\meas)}
    =\|\wgrad f\|_{L^2(\skeleton,\medm)}.
   \end{aligned}
  \end{equation}
  This holds for every $t\in(0,\infty)$, with constant one.
\end{proof}

\subsection{Proof of Theorem~\ref{t.mainw}}
In this subsection, we prove Theorem~\ref{t.mainw}. For $1\le p<2$, we use the weak reverse inequality \eqref{e.scalw} to extend convergence from $\core$ to $\sobolev{p}(\ambient)$, and for $p=2$ we use the spectral theorem. We will also determine for which $p\in[1,2)$ the core is contained in $\Dom((-\gen)^{\critic_p})$.

\begin{proof}[Proof of Theorem~\ref{t.mainw}]
Let $p\in[1,2)$. We first record two consequences of~\eqref{e.lorpq}. For $k\in\{2,3\}$ and measurable functions $v_1,\ldots,v_k$,
\begin{equation}
 \left\{\left|\sum_{j=1}^{k}v_j\right|>\lambda\right\}
 \subset\bigcup_{j=1}^{k}\{|v_j|>\lambda/k\},\qquad\lambda>0.
\end{equation}
Hence
\begin{align}\label{e.homquasi}
 \left\|\sum_{j=1}^{k}v_j\right\|_{L^{p,\infty}(\ambient,\meas)}
\overset{\eqref{e.lorpq}}{\leq}
 k\left(\sum_{j=1}^{k}\|v_j\|_{L^{p,\infty}(\ambient,\meas)}^p\right)^{\frac{1}{p}}
 &\leq k\sum_{j=1}^{k}\|v_j\|_{L^{p,\infty}(\ambient,\meas)}.
\end{align}
Also, if $v_n\to v$ $\meas$-a.e. on $\ambient$, then, for every $\lambda\in(0,\infty)$, by Fatou's lemma,
\begin{equation}
 \one_{\{|v|>\lambda\}}\leq\liminf_{n\to\infty}\one_{\{|v_n|>\lambda\}},
 \ \text{ and }\ 
 \lambda^p\meas(\{|v|>\lambda\})
 \overset{\eqref{e.lorpq}}{\leq}
 \liminf_{n\to\infty}\|v_n\|_{L^{p,\infty}(\ambient,\meas)}^p.
\end{equation}
Taking the supremum over $\lambda\in(0,\infty)$ gives
\begin{equation}\label{e.homfatou}
 \|v\|_{L^{p,\infty}(\ambient,\meas)}
 \leq\liminf_{n\to\infty}\|v_n\|_{L^{p,\infty}(\ambient,\meas)}.
\end{equation}

Fix $f\in\sobolev{p}(\ambient)$. For $h\in\core$, Corollary~\ref{c.coreco} gives that the function $T_h\in L^p(\ambient,\meas)$, defined in \eqref{e.coreli},  satisfies~\eqref{e.corePt}. Therefore,
\begin{align}
 &\phantom{\ \leq}\|(-\gen)^{\critic_p}P_th-(-\gen)^{\critic_p}P_sh\|_{L^p(\ambient,\meas)}\overset{\eqref{e.corePt}}{=}
 \|P_tT_h-P_sT_h\|_{L^p(\ambient,\meas)}\\
 &\leq\|P_tT_h-T_h\|_{L^p(\ambient,\meas)}
       +\|P_sT_h-T_h\|_{L^p(\ambient,\meas)}\to0,\ \text{ as $s,t\downarrow0$}.\label{e.homcore}
\end{align}
Given $\varepsilon>0$, Lemma~\ref{l.dense} supplies $h\in\core$ with $\|\wgrad(f-h)\|_{L^p(\skeleton,\medm)}<\varepsilon$. By linearity,
\begin{align}
 &\phantom{\ \leq}(-\gen)^{\critic_p}P_tf-(-\gen)^{\critic_p}P_sf\\
 &=(-\gen)^{\critic_p}P_t(f-h)
   +(-\gen)^{\critic_p}P_th-(-\gen)^{\critic_p}P_sh-(-\gen)^{\critic_p}P_s(f-h).
\end{align}
Using~\eqref{e.homquasi} and the inequality $\|v\|_{L^{p,\infty}(\ambient,\meas)}\leq\|v\|_{L^p(\ambient,\meas)}$, which follows from Chebyshev's inequality, we obtain
\begin{align}\label{e.homtri}
 &\phantom{\ \leq}\|(-\gen)^{\critic_p}P_tf-(-\gen)^{\critic_p}P_sf\|_{L^{p,\infty}(\ambient,\meas)}\\
 &\overset{\eqref{e.homquasi}}{\leq}
 3\|(-\gen)^{\critic_p}P_t(f-h)\|_{L^{p,\infty}(\ambient,\meas)}\\
 &\qquad+3\|(-\gen)^{\critic_p}P_th-(-\gen)^{\critic_p}P_sh\|_{L^p(\ambient,\meas)}
       +3\|(-\gen)^{\critic_p}P_s(f-h)\|_{L^{p,\infty}(\ambient,\meas)}\\
 &\overset{\eqref{e.scalw}}{\leq}
 6C_p\varepsilon+3\|(-\gen)^{\critic_p}P_th-(-\gen)^{\critic_p}P_sh\|_{L^p(\ambient,\meas)}.
\end{align}
Let $s,t\downarrow0$ and then $\varepsilon\downarrow0$. By~\eqref{e.homcore},
\begin{equation}\label{e.homcauchy}
 \lim_{\delta\downarrow0}\sup_{0<s,t<\delta}
 \|(-\gen)^{\critic_p}P_tf-(-\gen)^{\critic_p}P_sf\|_{L^{p,\infty}(\ambient,\meas)}=0.
\end{equation}

The space $L^{p,\infty}(\ambient,\meas)$ is complete, by \cite[Theorem~1.4.11]{Gra14}. Thus there exists $G_f\in L^{p,\infty}(\ambient,\meas)$ such that
\begin{equation}
 \lim_{n\to\infty}\|(-\gen)^{\critic_p}P_{1/n}f-G_f\|_{L^{p,\infty}(\ambient,\meas)}=0.
\end{equation}
For every $\lambda>0$,
\begin{equation}
\lim_{n\to\infty} \meas\bigl(\{|(-\gen)^{\critic_p}P_{1/n}f-G_f|>\lambda\}\bigr)
 \overset{\eqref{e.lorpq}}{\leq}
 \lim_{n\to\infty}\lambda^{-p}\|(-\gen)^{\critic_p}P_{1/n}f-G_f\|_{L^{p,\infty}(\ambient,\meas)}^p=0.
\end{equation}
Choose a subsequence $(n_k)$ along which $(-\gen)^{\critic_p}P_{1/n_k}f\to G_f$ $\meas$-a.e. For each fixed $t>0$, apply~\eqref{e.homfatou} to the differences to obtain
\begin{equation}
 \|(-\gen)^{\critic_p}P_tf-G_f\|_{L^{p,\infty}(\ambient,\meas)}
 \overset{\eqref{e.homfatou}}{\leq}
 \liminf_{k\to\infty}
 \|(-\gen)^{\critic_p}P_tf-(-\gen)^{\critic_p}P_{1/n_k}f\|_{L^{p,\infty}(\ambient,\meas)}.
\end{equation}
Consequently,
\begin{align}
 &\phantom{\ \leq}\sup_{t\in(0,\delta)}\|(-\gen)^{\critic_p}P_tf-G_f\|_{L^{p,\infty}(\ambient,\meas)}\\
 &\leq\sup_{s,t\in(0,\delta)}
 \|(-\gen)^{\critic_p}P_tf-(-\gen)^{\critic_p}P_sf\|_{L^{p,\infty}(\ambient,\meas)}
 \overset{\eqref{e.homcauchy}}{\to}0\  \text{ as }\delta\downarrow0.
\end{align}
This proves~\eqref{e.homog}. If $\widetilde G_f$ is another limit, then
\begin{align}
 &\phantom{\ \leq}\|G_f-\widetilde G_f\|_{L^{p,\infty}(\ambient,\meas)}\\
 &\overset{\eqref{e.homquasi}}{\leq}
 2\|G_f-(-\gen)^{\critic_p}P_tf\|_{L^{p,\infty}(\ambient,\meas)}+2\|(-\gen)^{\critic_p}P_tf-\widetilde G_f\|_{L^{p,\infty}(\ambient,\meas)}
 \to0.
\end{align}
Hence the limit is unique up to $\meas$-null sets. Moreover,
\begin{align}\label{e.Gbound}
 \|G_f\|_{L^{p,\infty}(\ambient,\meas)}\overset{\eqref{e.homfatou}}{\leq}
 \liminf_{k\to\infty}\|(-\gen)^{\critic_p}P_{1/n_k}f\|_{L^{p,\infty}(\ambient,\meas)}\overset{\eqref{e.scalw}}{\leq}C_p\|\wgrad f\|_{L^p(\skeleton,\medm)},
\end{align}
which proves~\eqref{e.mainw}. If also $f\in\Dom((-\gen)^{\critic_p})$, spectral commutation and strong continuity on $L^2(\ambient,\meas)$ give
\begin{align}
 \|(-\gen)^{\critic_p}P_tf-(-\gen)^{\critic_p}f\|_{L^2(\ambient,\meas)}\overset{\eqref{e.spec}}{=}
 \|P_t((-\gen)^{\critic_p}f)-(-\gen)^{\critic_p}f\|_{L^2(\ambient,\meas)}
 \to0\ \text{ as }\ t\downarrow0.
\end{align}
For every $\varepsilon>0$, it follows by Chebyshev's inequality that
\begin{align}
 &\phantom{\ \leq}\meas\bigl(\{|G_f-(-\gen)^{\critic_p}f|>\varepsilon\}\bigr)\\
 &\overset{\eqref{e.lorpq}}{\leq}
 \left(\frac{2}{\varepsilon}\right)^p
 \|G_f-(-\gen)^{\critic_p}P_tf\|_{L^{p,\infty}(\ambient,\meas)}^p+\left(\frac{2}{\varepsilon}\right)^2
 \|(-\gen)^{\critic_p}P_tf-(-\gen)^{\critic_p}f\|_{L^2(\ambient,\meas)}^2
\\ &\to0,\ \text{ as }t\downarrow0.
\end{align}
Thus $G_f=(-\gen)^{\critic_p}f$ $\meas$-a.e. on $\ambient$.

Finally, let $p=2$. By~\eqref{e.critp}, $\critic_2=\frac{1}{2}$, and Theorem \ref{t.dirich} gives $\sobolev{2}(\ambient)=\Dom((-\gen)^{\frac{1}{2}})$ and $\|(-\gen)^{\frac{1}{2}}f\|_{L^2(\ambient,\meas)} =\|\wgrad f\|_{L^2(\skeleton,\medm)}$. In this case, set $G_f=(-\gen)^{\frac{1}{2}}f$. Then, by Chebyshev's inequality,
\begin{equation}
 \lim_{t\downarrow0}\|(-\gen)^{\frac{1}{2}}P_tf-G_f\|_{L^{2,\infty}(\ambient,\meas)}
 \overset{\eqref{e.spec}}{\leq} \lim_{t\downarrow0}
 \|P_tG_f-G_f\|_{L^2(\ambient,\meas)}=0,
\end{equation}
and \begin{equation}
	\|G_f\|_{L^{2,\infty}(\ambient,\meas)}
 \leq\|G_f\|_{L^2(\ambient,\meas)}
 \overset{\eqref{e.spec},\eqref{e.formde}}{=}
 \|\wgrad f\|_{L^2(\skeleton,\medm)}.
\end{equation}
This proves the assertions for $p=2$.
\end{proof}

\begin{proposition}\label{p.thresh}
Suppose that $(\ambient,\metric)$ is $\AR{\dimf}$ with $\dimf\in(1,\infty)$.  Let $p\in[1,2)$.
\begin{enumerate}[label=\textup{({\arabic*})}, align=right, leftmargin=*, topsep=5pt, parsep=0pt, itemsep=2pt]
  \item  If $p\in(2-2/\dimf,2)$, then $\core\subset\Dom((-\gen)^{\critic_p})$;
  \item If $\dimf\geq2$ and $p\in[1,2-2/\dimf]$, then $\core\cap\Dom((-\gen)^{\critic_p})=\{0\}$.
\end{enumerate}
 
\end{proposition}
\begin{proof}
By Proposition \ref{p.synthe}, $\critic_p\in(\frac{1}{2},1)$. For $h\in L^2(\ambient,\meas)$, the spectral theorem and Tonelli's theorem give
\begin{align}\label{e.thresholdspectral}
 &\phantom{\ \leq}\int_0^\infty s^{1-2\critic_p}\|(-\gen)P_sh\|_{L^2(\ambient,\meas)}^2\dif s\\
 &\overset{\eqref{e.spec}}{=}
 \int_{[0,\infty)}\left(\int_0^\infty
 s^{1-2\critic_p}\lambda^2e^{-2s\lambda}\dif s\right)
 \dif\langle\proj_\lambda h,h\rangle\\
 &=
 2^{2\critic_p-2}\left(\int_0^\infty v^{1-2\critic_p}e^{-v}\dif v\right)
 \int_{[0,\infty)}\lambda^{2\critic_p}\dif\langle\proj_\lambda h,h\rangle\\
 &=2^{2\critic_p-2}\Gamma(2-2\critic_p)
 \int_{[0,\infty)}\lambda^{2\critic_p}\dif\langle\proj_\lambda h,h\rangle.
\end{align}
Since $0<\Gamma(2-2\critic_p)<\infty$, we have 
\begin{equation}\label{e.threshcriterion}
 h\in\Dom((-\gen)^{\critic_p})
\iff
 \int_0^\infty s^{1-2\critic_p}\|(-\gen)P_sh\|_{L^2(\ambient,\meas)}^2\dif s<\infty.
\end{equation}

Fix $0\neq h\in\core$. By~\eqref{e.repres}, for $s\in(0,\infty)$, we can represent $(-\gen)P_sh$ as
\begin{equation}
 (-\gen)P_sh\overset{\eqref{e.repres}}{=}
 \sum_{z\in Z_h}c_h(z)p_s(z,\cdot).
\end{equation}
These coefficients are not all zero. Indeed, otherwise $(-\gen)P_sh=0$ for every $s>0$. Since $P_sh\in\Dom(-\gen)$ and $\Ker(-\gen)=\{0\}$ by Theorem~\ref{t.dirich}-\ref{it.HKE}, we would have $P_sh=0$ for every $s>0$, and strong continuity would imply $h=0$ in $L^2(\ambient,\meas)$. Thus
\begin{equation}\label{e.threshnonzero}
 \sum_{z\in Z_h}|c_h(z)|^2>0.
\end{equation}
By definition, symmetry and the semigroup property, we have
\begin{align}\label{e.thresholdgram}
 \|(-\gen)P_sh\|_{L^2(\ambient,\meas)}^2
 &\overset{\eqref{e.repres}}{=}
 \sum_{z,w\in Z_h}c_h(z)c_h(w)p_{2s}(z,w)\\
 &=\sum_{z\in Z_h}|c_h(z)|^2p_{2s}(z,z)
 +\sum_{\substack{z,w\in Z_h\\z\neq w}}c_h(z)c_h(w)p_{2s}(z,w).
\end{align}
The upper heat kernel bound gives
\begin{equation}\label{e.threshupper}
 \|(-\gen)P_sh\|_{L^2(\ambient,\meas)}^2
 \overset{\eqref{e.thresholdgram},\eqref{e.HKE}}{\leq}
 Cs^{-\frac{\dimf}{\dimw}}
 \left(\sum_{z\in Z_h}|c_h(z)|\right)^2,\qquad s>0.
\end{equation}
For the lower bound, 
\begin{equation}\label{e.threshdiag}
 s^{\frac{\dimf}{\dimw}}\sum_{z\in Z_h}|c_h(z)|^2p_{2s}(z,z)
 \overset{\eqref{e.HKE}}{\geq}
 c\sum_{z\in Z_h}|c_h(z)|^2,\qquad s>0.
\end{equation}
Since $Z_h$ is finite and its points are distinct,
\begin{align}
 s^{\frac{\dimf}{\dimw}}
 \abs{\sum_{\substack{z,w\in Z_h\\z\neq w}}
 c_h(z)c_h(w)p_{2s}(z,w)}&\overset{\eqref{e.HKE}}{\leq}
 C\sum_{\substack{z,w\in Z_h\\z\neq w}}|c_h(z)c_h(w)|
 \exp(-c\left(\frac{\metric(z,w)^{\dimw}}{s}\right)^{\frac{1}{\dimw-1}})\\
 &\to0\ \text{ as }\ s\downarrow0.\label{e.threshoffdiag}
\end{align}
By~\eqref{e.threshnonzero}, we may choose $s_h\in(0,1]$ so that the right-hand side of~\eqref{e.threshoffdiag} is at most $\frac{c}{2}\sum_{z\in Z_h}|c_h(z)|^2$ for $s\in(0,s_h)$, where $c$ is the constant in~\eqref{e.threshdiag}. Combining the diagonal and off-diagonal heat kernel estimates \eqref{e.HKE}, we obtain
\begin{align}
 \frac{c}{2}\left(\sum_{z\in Z_h}|c_h(z)|^2\right)s^{-\frac{\dimf}{\dimw}}
 &\overset{\eqref{e.thresholdgram},\eqref{e.threshdiag},\eqref{e.threshoffdiag}}{\leq}
 \|(-\gen)P_sh\|_{L^2(\ambient,\meas)}^2\\
 &\overset{\eqref{e.threshupper}}{\leq}
 C\left(\sum_{z\in Z_h}|c_h(z)|\right)^2s^{-\frac{\dimf}{\dimw}},\ \text{ for all }\ s\in(0,s_h).\label{e.threshsmall}
\end{align}

Since
\begin{equation}
 \|(-\gen)P_sh\|_{L^2(\ambient,\meas)}
 \overset{\eqref{e.spec}}{\leq}
 \sup_{\lambda\geq0}(\lambda e^{-s\lambda})\|h\|_{L^2(\ambient,\meas)}
 =\frac{1}{es}\|h\|_{L^2(\ambient,\meas)},
\end{equation}
we have
\begin{align}\label{e.threshtail}
 \int_{s_h}^{\infty}s^{1-2\critic_p}\|(-\gen)P_sh\|_{L^2(\ambient,\meas)}^2\dif s
 \leq e^{-2}\|h\|_{L^2(\ambient,\meas)}^2
 \int_{s_h}^{\infty}s^{-1-2\critic_p}\dif s=\frac{s_h^{-2\critic_p}}{2e^2\critic_p}\|h\|_{L^2(\ambient,\meas)}^2<\infty.
\end{align}
Using $\dimw=\dimf+1$ and~\eqref{e.critp}, we obtain
\begin{align}
 h\in\Dom((-\gen)^{\critic_p})
 &\overset{\eqref{e.threshcriterion},\eqref{e.threshsmall},\eqref{e.threshtail}}{\iff}
 \int_0^{s_h}s^{1-2\critic_p-\frac{\dimf}{\dimw}}\dif s<\infty\iff
 2-2\critic_p-\frac{\dimf}{\dimw}>0\\
 &\iff p>2-\frac{2}{\dimf}.
\end{align}
This proves the two assertions.
\end{proof}

\section{Failure of the Riesz transform}\label{s.failure}
In this section, we first construct examples that show where the reverse inequalities fail when $\dimf>1$, completing the proof of Theorem~\ref{t.strfl}. We then prove Theorem~\ref{t.dualc} for $1<p<\infty$ by duality and treat $p=1$ separately.

\subsection{Failure of the reverse inequalities}

For convenience, an arc $[a,b]\subset C$ is called to be \emph{directed away from $T_o$} if $a\in[g_C,b]$.

\begin{lemma}\label{l.branch}
Suppose that $(\ambient,\metric)$ is $\AR{\dimf}$ with $\dimf\in(1,\infty)$. Fix $o\in\ambient$ and let $T_o$ be the union of geodesic rays in~\eqref{e.Tinfty}. For a component $C$ of $\ambient\setminus T_o$, let $g_C\in T_o$ be its gate from Lemma~\ref{l.projt}-\ref{it.proj5}. Let $[a,b]\subset C$ be directed away from $T_o$, that is, $a\in[g_C,b]$.  Let $v\in[a,b)$. Then \begin{enumerate}[label=\textup{({\arabic*})}, align=right, leftmargin=*, topsep=5pt, parsep=0pt, itemsep=2pt]
  \item\label{it.branch1} The component $D$ of $\ambient\setminus\{v\}$ containing $b$ satisfies $D\subset C$. 
  \item\label{it.branch2} $\ol D$ is compact and $\meas(\ol D)\leq\meas(\ol C)<\infty$. 
  \item\label{it.branch3} For every $v\in\ambient$ and every component $D'$ of $\ambient\setminus\{v\}$, the set $D'\cup\{v\}$ is connected.
\end{enumerate}
\end{lemma}
\begin{proof}
Fix $v\in\ambient$, a component $D'$ of $\ambient\setminus\{v\}$, and $x\in D'$. The connected arc $[x,v)$ avoids $v$ and contains $x$. Hence $[x,v)\subset D'$ and $D'\cup\{v\}=\bigcup_{x\in D'}[x,v]$.
The arcs on the right have the common point $v$, so $D'\cup\{v\}$ is connected.

Let $D$ be the component in \ref{it.branch1}. Let $D_g$ be the component of $\ambient\setminus\{v\}$ containing $g_C$. Since $v\in(g_C,b)$, we have $D\cap D_g=\emptyset$. The set $T_o$ is connected by Lemma~\ref{l.ends}-\ref{it.tends1}, contains $g_C$, and avoids $v\in C$. Therefore $ T_o\subset D_g$ and $(D\cup\{v\})\cap T_o=\emptyset$. The connected set $D\cup\{v\}$ contains $v\in C$, so it lies in the component $C$ of $\ambient\setminus T_o$. Thus $\ol D\subset\ol C$. By Lemma~\ref{l.ends}-\ref{it.tends2}, $\ol C$ is compact, and so is its closed subset $\ol D$. Then $\meas(\ol D)\leq\meas(\ol C)<\infty$.
\end{proof}

\begin{lemma}\label{l.tube}
Suppose that $(\ambient,\metric)$ is $\AR{\dimf}$ with $\dimf\in(1,\infty)$. There exists $C_4\in(1,\infty)$, depending only on $C_0$ and $\dimf$, such that for every compact arc $I=[a,b]\subset\ambient$, $z\in I$, $\rho>0$ and $\varepsilon\in(0,1]$, we have
\begin{equation}\label{e.tube}
 \meas\bigl(B(z,\rho)\cap\Sett{y\in\ambient}{\dist(y,I)<\varepsilon\rho}\bigr)
 \leq C_4\varepsilon^{\dimf-1}\rho^{\dimf}.
\end{equation}
\end{lemma}
\begin{proof}
Let $U:=B(z,\rho)\cap\Sett{y\in\ambient}{\dist(y,I)<\varepsilon\rho}$. Since $\pi_I(z)=z$, by Lemma~\ref{l.projt}-\ref{it.proj4},
\begin{equation}
 \metric(\pi_I(y),z)
 \leq
 \metric(y,z)<\rho,\ \text{ for all }\ y\in U.
\end{equation}
Consequently, $\pi_I(U)$ lies in the arc $J:=I\cap\ol{B(z,\rho)}$, whose length is at most $2\rho$. If $J$ is a singleton, then we take $N=1$ and let $w_1\in J$ be its unique point. Otherwise, divide $J$ into $\lceil2/\varepsilon\rceil$ subarcs of equal length and enumerate their distinct endpoints as $\{w_i\}_{i=1}^{N}$. Then
\begin{equation}\label{e.tube-cover}
 \#\{w_i\}_{i=1}^{N}\leq\frac{4}{\varepsilon},\ \
 \sup_{w\in J}\min_{1\leq i\leq N}\metric(w,w_i)\leq\varepsilon\rho,\ \text{ and }\ U\subset\bigcup_{i=1}^{N} B(w_i,2\varepsilon\rho).
\end{equation}
Indeed, for $y\in U$, we can choose $w_i$ with $\metric(w_{i},\pi_{I}(y))\leq\varepsilon\rho$, and use $\metric(y,\pi_I(y))=\dist(y,I)<\varepsilon\rho$. Hence
\begin{equation}
 \meas(U)\overset{\eqref{e.tube-cover}}{\leq}
 \sum_i\meas(B(w_i,2\varepsilon\rho))\overset{\eqref{e.ahlfor},\eqref{e.tube-cover}}{\leq}
 \frac{4}{\varepsilon}C_0(2\varepsilon\rho)^{\dimf}
 =4\cdot2^{\dimf}C_0\varepsilon^{\dimf-1}\rho^{\dimf}.
\end{equation}
This completes the proof.
\end{proof}

\begin{lemma}\label{l.potent}
Suppose that $(\ambient,\metric)$ is $\AR{\dimf}$ with $\dimf\in(1,\infty)$. Let $\theta\in(1/\dimw,1)$ and define
\begin{equation}\label{e.poten}
 H_{1-\theta}(x,z):=\frac{1}{\Gamma(1-\theta)}
 \int_0^\infty s^{-\theta}p_{1+s}(x,z)\dif s.
\end{equation}
Then $H_{1-\theta}\in C(\ambient\times\ambient)$, and there exists $C\in(1,\infty)$ such that
\begin{equation}\label{e.potent}
C^{-1} (1\vee\metric(x,z))^{1-\dimw\theta}\leq H_{1-\theta}(x,z)\leq C
 (1\vee\metric(x,z))^{1-\dimw\theta},\ \text{ for all }x,z\in\ambient.
\end{equation}
If $f\in\core$ has endpoint coefficients $c_f(z)$ as in~\eqref{e.repres}, then
\begin{equation}\label{e.poten1}
 (-\gen)^{\theta}P_1f(x)
 =\sum_{z\in Z_{\wgrad f}}c_f(z)H_{1-\theta}(x,z)
 \qquad\text{for $\meas$-a.e. $x\in\ambient$}.
\end{equation}
\end{lemma}
\begin{proof}
Since $\dimw=\dimf+1$ and $\theta\in(1/\dimw,1)$, we have
\begin{equation}
 0<1-\theta<\frac{\dimf}{\dimw},\ \text{ and }\
 (1-\theta)\dimw-\dimf=1-\dimw\theta<0.
\end{equation}
Thus the heat kernel bounds \eqref{e.HKE} and Lemma~\ref{l.intpot} give
\begin{align}
 H_{1-\theta}(x,z)\asymp (1\vee\metric(x,z)^{\dimw})^{1-\theta-\frac{\dimf}{\dimw}}=(1\vee\metric(x,z))^{1-\dimw\theta}.
\end{align}
For continuity, observe that $s^{-\theta}p_{1+s}(x,z)\overset{\eqref{e.HKE}}{\leq}
 Cs^{-\theta}(1+s)^{-\frac{\dimf}{\dimw}}$, and
\begin{align}
 \int_0^\infty s^{-\theta}(1+s)^{-\frac{\dimf}{\dimw}}\dif s
 \leq\int_0^1s^{-\theta}\dif s
       +\int_1^\infty s^{-\theta-\frac{\dimf}{\dimw}}\dif s=\frac{1}{1-\theta}+\frac{1}{\theta-\frac{1}{\dimw}}<\infty.
\end{align}
Joint continuity of the heat kernel $p$ and the dominated convergence theorem now give joint continuity of $H_{1-\theta}$. Let $f\in\core$, and let $h\in L^{\infty}(\ambient,\meas)\cap L^{1}(\ambient,\meas)$. Then $f,h\in L^2(\ambient,\meas)$ and $h\in L^1(\ambient,\meas)$. For every $z\in Z_{\wgrad f}$,
\begin{align}
 \int_0^\infty\int_{\ambient}
 s^{-\theta}p_{1+s}(x,z)|h(x)|\dif\meas(x)\dif s\overset{\eqref{e.HKE}}{\leq}
 C\|h\|_{L^1(\ambient,\meas)}
 \int_0^\infty s^{-\theta}(1+s)^{-\frac{\dimf}{\dimw}}\dif s<\infty.
\end{align}
By symmetry of the heat kernel and Fubini's theorem,
\begin{align}
 \langle(-\gen)^{\theta}P_1f,h\rangle_{L^2(\ambient,\meas)}&\overset{\eqref{e.Lpt2}}{=}
 \frac{1}{\Gamma(1-\theta)}\int_0^\infty s^{-\theta}
 \langle(-\gen)P_{1+s}f,h\rangle_{L^2(\ambient,\meas)}\dif s\\
 &\overset{\eqref{e.repres}}{=}
 \sum_{z\in Z_{\wgrad f}}\frac{c_f(z)}{\Gamma(1-\theta)}
 \int_0^\infty\int_{\ambient}
 s^{-\theta}p_{1+s}(x,z)h(x)\dif\meas(x)\dif s\\
 &\overset{\eqref{e.poten}}{=}
 \int_{\ambient}\left(\sum_{z\in Z_{\wgrad f}}c_f(z)H_{1-\theta}(x,z)\right)
 h(x)\dif\meas(x).
\end{align}
Since this holds for all $h\in L^{\infty}(\ambient,\meas)\cap L^{1}(\ambient,\meas)$, we obtain \eqref{e.poten1}.
\end{proof}

We next construct the tent functions used in both counterexamples.
\begin{lemma}\label{l.tent}
Suppose that $(\ambient,\metric)$ is $\AR{\dimf}$ with $\dimf\in(1,\infty)$. Fix $p\in[1,\infty)$. There exist $\varepsilon\in(0,1)$, $\delta\in(0,1/6)$, $\ell_0\in[1,\infty)$ and $c_0,c_0'\in(0,\infty)$ such that the following holds. Let $C$ be a component of $\ambient\setminus T_o$, and let $I=[a,b]\subset C$ be directed away from $T_o$, with length $\ell:=\metric(a,b)\geq\ell_0$ and midpoint $z$. There exist $v_I\in I$, a ball $E_I$ and $\psi_I\in\core$ such that:
\begin{enumerate}[label=\textup{(\arabic*)}, align=right, leftmargin=*, topsep=5pt, parsep=0pt, itemsep=2pt]
\item\label{it.ttent1}
$\ell/3<\metric(a,v_I)<2\ell/3$, $E_I\subset B(z,2\delta\ell)$, $E_I\cap I=\emptyset$, $\pi_I(E_I)=\{v_I\}$, and
\begin{equation}\label{e.tentge}
 \begin{gathered}
 \frac{3}{4}\varepsilon\delta\ell\leq\metric(x,v_I)\leq\frac{5}{4}\delta\ell,\ \text{ and }\ \metric(x,a)\wedge \metric(x,b)\geq\frac{\ell}{3}
 \ \text{ for all }\ x\in E_I.
 \end{gathered}
\end{equation}
\item\label{it.ttent2}
$\supp(\psi_I)\subset\ol C$, and there exist $s_-,s_+\in[3/2,3]$ such that 
\begin{equation}
	\wgrad\psi_I=s_-\vecone_{(a,v_I)}-s_+\vecone_{(v_I,b)},\ \medm\text{-a.e. on $\skeleton$}.
\end{equation}
\item\label{it.ttent3}
We have
\begin{equation}\label{e.tent}
 |\wgrad\psi_I|\leq3\one_{(a,b)}\ \medm\text{-a.e. on $\skeleton$},\
 c_0'\ell^{\dimf}\leq\meas(E_I)\leq C_0\ell^{\dimf},\
 \einf_{E_I}(-\gen)^{\critic_p}P_1\psi_I
 \geq c_0\ell^{-\frac{\dimf-1}{p}}.
\end{equation}
\end{enumerate}
All constants depend only on $p$ and the constants in \eqref{e.ahlfor} and \eqref{e.HKE}.
\end{lemma}
\begin{proof}
By~\eqref{e.critp}, $\critic_p\in(1/\dimw,1)$. Lemma~\ref{l.potent} gives $C_3\geq1$ such that
\begin{align}\label{e.tent-kernel}
 C_3^{-1}\metric(x,w)^{-\frac{\dimf-1}{p}}
 \overset{\eqref{e.potent},\eqref{e.critp}}{\leq}
 H_{1-\critic_p}(x,w)\overset{\eqref{e.potent},\eqref{e.critp}}{\leq}
 C_3\metric(x,w)^{-\frac{\dimf-1}{p}},
 \ \text{ if }\metric(x,w)\geq1.
\end{align}
Choose $\varepsilon\in(0,1)$ and then $\delta\in(0,1/6)$ such that
\begin{equation}\label{e.tent-parameters}
 C_4\varepsilon^{\dimf-1}<C_0^{-1},\ 3C_3^{-1}\left(\frac{5\delta}{4}\right)^{-\frac{\dimf-1}{p}}
 \geq12C_3\,3^{\frac{\dimf-1}{p}}, \text{ and let }\
 \ell_0:=\max\left\{3,\frac{4}{3\varepsilon\delta}\right\}.\\
\end{equation}
These choices are possible because $\dimf-1>0$. We now construct $v_I$ and $E_I$. Note that
\begin{align}
 \meas\bigl(B(z,\delta\ell)\cap\Sett{y}{\dist(y,I)<\varepsilon\delta\ell}\bigr)\overset{\eqref{e.tube}}{\leq}
 C_4\varepsilon^{\dimf-1}(\delta\ell)^{\dimf}
 \overset{\eqref{e.tent-parameters}}{<}
 C_0^{-1}(\delta\ell)^{\dimf}
 \overset{\eqref{e.ahlfor}}{\leq}\meas(B(z,\delta\ell)).
\end{align}
Thus, $B(z,\delta\ell)\setminus\Sett{y}{\dist(y,I)<\varepsilon\delta\ell}\neq\emptyset$. Choose $y\in B(z,\delta\ell)\setminus\Sett{y}{\dist(y,I)<\varepsilon\delta\ell}$ and set
\begin{equation}
 h:=\dist(y,I)\in[\varepsilon\delta\ell,\delta\ell),\quad
 v_I:=\pi_I(y),\ \text{ and }\ E_I:=B(y,h/4).
\end{equation}
\begin{enumerate}[label=\textup{({\arabic*})}, align=right, leftmargin=*, topsep=5pt, parsep=0pt, itemsep=2pt]
  \item[\ref{it.ttent1}] Since $\pi_I(z)=z$, by Lemma~\ref{l.projt}-\ref{it.proj4},
\begin{equation}
 \metric(v_I,z)
\leq
 \metric(y,z)<\delta\ell<\frac{\ell}{6},\ \text{ and }\
 \frac{\ell}{3}<\metric(a,v_I)<\frac{2\ell}{3}.
\end{equation}
For every $x\in E_I$,
\begin{equation}
 \dist(x,I)\geq\dist(y,I)-\metric(x,y)>\frac{3h}{4}>0.
\end{equation}
The connected ball $E_I$ therefore lies in a component of $\ambient\setminus I$. Since it contains $y$, Lemma~\ref{l.projt}-\ref{it.proj5} gives $\pi_I(E_I)=\{v_I\}$. Moreover,
\begin{align}
 \frac{3}{4}\varepsilon\delta\ell
 \leq\frac{3h}{4}<\metric(x,v_I)<\frac{5h}{4}
 \leq\frac{5}{4}\delta\ell,
\end{align}
and \begin{equation}
	 \metric(x,z)\leq\metric(x,y)+\metric(y,z)
 <\frac{h}{4}+\delta\ell<2\delta\ell.
\end{equation}
For $w\in\{a,b\}$, by Lemma~\ref{l.projt}-\ref{it.proj3},
\begin{equation}
 \metric(x,w)
=\metric(x,v_I)+\metric(v_I,w)>\frac{\ell}{3}.
\end{equation}
These estimates give \ref{it.ttent1}.
\item[\ref{it.ttent2}] Set $s_-:=\ell/\metric(a,v_I)$ and $s_+:=\ell/\metric(v_I,b)$. Define
\begin{equation}
 F_I(w):=
 \begin{cases}
 s_-\metric(a,w),&w\in[a,v_I],\\
 s_+\metric(w,b),&w\in[v_I,b],
 \end{cases}
 \ \text{ and }\ \psi_I:=F_I\circ\pi_I.
\end{equation}
Then $s_-,s_+\in[3/2,3]$, $F_I$ is $3$-Lipschitz, and so is $\psi_I$ by Lemma~\ref{l.projt}-\ref{it.proj4}. The projection is constant on every component of $\ambient\setminus I$ by Lemma~\ref{l.projt}-\ref{it.proj5}. On the two subarcs of $I$, \eqref{e.arcder} and \eqref{e.arcrev} give
\begin{equation}\label{e.tent-gradient}
 \wgrad\psi_I
 \overset{\eqref{e.arcder},\eqref{e.arcrev}}{=}
 s_-\vecone_{(a,v_I)}-s_+\vecone_{(v_I,b)}
 \qquad\medm\text{-a.e. on $\skeleton$}.
\end{equation}

Let $D$ be the component of $\ambient\setminus\{a\}$ containing $b$. By Lemma~\ref{l.branch}, $D\subset C$ and $\ol D$ is compact. If $\psi_I(x)\neq0$, then $v:=\pi_I(x)\in(a,b)$, since $F_I$ vanishes at both endpoints. By Lemma~\ref{l.projt}-\ref{it.proj3},
\begin{equation}
 [x,b]=[x,v]\cup[v,b],\ \text{ and }\ [x,v]\cap I=\{v\}.
\end{equation}
Neither arc on the right contains $a$, so $x$ and $b$ belong to the same component $D$ of $\ambient\setminus\{a\}$. Consequently, $\{\psi_I\neq0\}\subset D\subset C$ and $\supp(\psi_I)\subset\ol D\subset\ol C$.

The Lipschitz property and~\eqref{e.tent-gradient} imply $\psi_I\in\core$ by~\eqref{e.core}.
\item[\ref{it.ttent3}]The gradient bound in \eqref{e.tent} follows directly from~\eqref{e.tent-gradient} and $s_-,s_+\leq3$. Since
\begin{equation}\label{e.tent-volume}
 C_0^{-1}\left(\frac{\varepsilon\delta}{4}\right)^{\dimf}\ell^{\dimf}
 \overset{\eqref{e.ahlfor}}{\leq}\meas(E_I)
 \overset{\eqref{e.ahlfor}}{\leq}
 C_0\left(\frac{\delta\ell}{4}\right)^{\dimf}
 \leq C_0\ell^{\dimf},
\end{equation}
 we may take $c_0':=C_0^{-1}(\varepsilon\delta/4)^{\dimf}$. By~\eqref{e.tent-gradient} and the definitions and notation in \eqref{e.endpoi}, 
\begin{equation}
 c_{\psi_I}(v_I)=s_-+s_+,\qquad
 c_{\psi_I}(a)=-s_-,\qquad c_{\psi_I}(b)=-s_+.
\end{equation}
Therefore, for $\meas$-a.e. $x\in\ambient$,
\begin{equation}
(-\gen)^{\critic_p}P_1\psi_I(x)\overset{\eqref{e.poten1}}{=}
 (s_-+s_+)H_{1-\critic_p}(x,v_I)-s_-H_{1-\critic_p}(x,a)-s_+H_{1-\critic_p}(x,b).\label{e.tentpo}
\end{equation}
For $x\in E_I$ and $\ell\geq\ell_0$,
\begin{equation}
 \min\{\metric(x,v_I),\metric(x,a),\metric(x,b)\}
 \overset{\eqref{e.tentge}}{\geq}
 \min\left\{\frac{3\varepsilon\delta\ell}{4},\frac{\ell}{3}\right\}
 \overset{\eqref{e.tent-parameters}}{\geq}1.
\end{equation}
Using $3\leq s_-+s_+\leq6$, we obtain for almost every $x\in E_I$
\begin{align}
 (-\gen)^{\critic_p}P_1\psi_I(x)
 &\overset{\eqref{e.tentpo},\eqref{e.tent-kernel},\eqref{e.tentge}}{\geq}
 3C_3^{-1}\left(\frac{5\delta\ell}{4}\right)^{-\frac{\dimf-1}{p}}
 -6C_3\left(\frac{\ell}{3}\right)^{-\frac{\dimf-1}{p}}\\
 &\overset{\eqref{e.tent-parameters}}{\geq}
 6C_3\,3^{\frac{\dimf-1}{p}}\ell^{-\frac{\dimf-1}{p}}.
\end{align}
Taking $c_0:=6C_3\,3^{\frac{\dimf-1}{p}}$ proves the last estimate in~\eqref{e.tent}.
\qedhere 
\end{enumerate}
\end{proof}

Let $(\Omega,\bP)$ be a probability space. 
Let $\varepsilon_1,\ldots,\varepsilon_m:\Omega\to\{-1,1\}$ be independent random variables such that \begin{equation}
	\bP(\varepsilon_j=1)=\bP(\varepsilon_j=-1)=\frac{1}{2},\ \text{ for all } j\in\{1,\ldots,m\}.
\end{equation} 
For $p\in[1,2]$ and $a_1,\ldots,a_m\in\bR$, we will use
\begin{equation}\label{e.random}
 \max_{1\leq i\leq m}|a_i|^p
 \leq\bE\left[\left|\sum_{i=1}^{m}\varepsilon_i a_i\right|^p\right]
 \leq\left(\sum_{i=1}^{m}|a_i|^2\right)^{\frac{p}{2}}.
\end{equation}
For the upper bound, Jensen's inequality and independence give
\begin{align}
 \bE\left[\left|\sum_{i=1}^{m}\varepsilon_i a_i\right|^p\right]
 \overset{\text{(Jensen)}}{\leq}
 \left(\bE\left[\left|\sum_{i=1}^{m}\varepsilon_i a_i\right|^2\right]\right)^{\frac{p}{2}}=\left(\sum_{i,j=1}^{m}a_i a_j\bE[\varepsilon_i\varepsilon_j]\right)^{\frac{p}{2}}
 =\left(\sum_{i=1}^{m}|a_i|^2\right)^{\frac{p}{2}}.
\end{align}
The lower bound is proved in \cite[Lemma~2.7]{APSS24}.

\begin{proof}[Proof of Theorem~\ref{t.strfl}-\ref{it.strfl}]
Fix $p\in[1,2)$ and the constants from Lemma~\ref{l.tent}. For each integer $N\in\bN$, let $R_N:=3^N\ell_0$. By Lemma~\ref{l.ends}-\ref{it.tends3}, there exist a component $C$ of $\ambient\setminus T_o$ and $x_0\in C$ such that $B(x_0,4R_N)\subset C$ and $\metric(g_C,x_0)\geq4R_N$.
Choose $a_0\in[g_C,x_0]$ with $\metric(a_0,x_0)=R_N$. Then $I_0:=[a_0,x_0]\subset C$ is directed away from $T_o$.

Subdivide successively at the points provided by Lemma~\ref{l.tent}:
\begin{equation}
 \sP_0:=\{I_0\},\quad
 \sP_{k+1}:=\bigcup_{I=[a_I,b_I]\in\sP_k}
 \{[a_I,v_I],[v_I,b_I]\},\ \text{ for } k\in\{0,\ldots,N-1\}.
\end{equation}
Each $\sP_k$ is a partition of $I_0$ into arcs with disjoint interiors. By Lemma~\ref{l.tent}-\ref{it.ttent1},
\begin{equation}\label{e.strong-partition}
 \#\sP_k=2^k,\qquad
 \min_{I\in\sP_k}\medm(I)\geq
 3^{N-k}\ell_0,\ \text{ and }\
 \sum_{I\in\sP_k}\medm(I)=R_N.
\end{equation}
Thus Lemma~\ref{l.tent} applies throughout the construction, by induction. Let $\sP:=\bigcup_{k=0}^{N-1}\sP_k$.

We now verify that the balls $(E_I)_{I\in\sP}$, constructed according to Lemma \ref{l.tent}, are pairwise disjoint. For $x\in E_I$ and $a\in I\subset I_0$, we have $\pi_{I_0}(x)\in [x,a]$. Applying Lemma~\ref{l.projt}-\ref{it.proj3} and Lemma~\ref{l.tent}-\ref{it.ttent1}, we have
\begin{equation}
 \pi_I(\pi_{I_0}(x))
=
 \pi_I(x)
=
 v_I\in(a_I,b_I).
\end{equation}
A point of $I_0\setminus I$ projects onto an endpoint of $I$. Hence $\pi_{I_0}(x)\in I$, and the displayed identity implies $\pi_{I_0}(x)=v_I$ for all $x\in E_{I}$. For $I\neq J$, we have $v_I\neq v_J$. Hence,
\begin{equation}\label{e.strong-disjoint}
E_I\cap E_J=\emptyset,
 \ \text{ for }\ I,J\in\sP,\ I\neq J.
\end{equation}

On the finite probability space $\Omega:=\{-1,1\}^{\sP}$ with uniform measure, let $\varepsilon_I:\Omega\to\{-1,1\}$ be the coordinate map, and define
\begin{equation}
 f_\omega:=\sum_{I\in\sP}\varepsilon_I(\omega)\psi_I\in\core,\ \text{ for all }\ \omega\in \Omega.
\end{equation}
Since the arcs in each $\sP$ have pairwise disjoint interior,
\begin{equation}
 \sum_{I\in\sP}|\wgrad\psi_I|^2
 \overset{\eqref{e.tent}}{\leq}
 9\sum_{k=0}^{N-1}\sum_{I=[a_I,b_I]\in\sP_k}\one_{(a_I,b_I)}
 \leq9N\one_{I_0}\quad\medm\text{-a.e. on $\skeleton$.}
\end{equation}
Applying this inequality to $f_{\omega}$ and taking expectations, we obtain
\begin{align}\label{e.khintc}
 \int_{\Omega}\|\wgrad f_\omega\|_{L^p(\skeleton,\medm)}^p\dif\bP(\omega)
 \overset{\eqref{e.random}}{\leq}
 \int_{\skeleton}\left(\sum_{I\in\sP}|\wgrad\psi_I|^2\right)^{\frac{p}{2}}\dif\medm\leq3^pN^{\frac{p}{2}}\medm(I_0)
 =3^pN^{\frac{p}{2}}R_N.
\end{align}
By~\eqref{e.analy}, the operator $(-\gen)^{\critic_p}P_1$ is linear and bounded on $L^p(\ambient,\meas)$. Thus~\eqref{e.random} and the disjointness of the balls give
\begin{align}
 &\phantom{\ \leq} \int_{\Omega}\|(-\gen)^{\critic_p}P_1f_\omega\|_{L^p(\ambient,\meas)}^p\dif\bP(\omega)
 \overset{\eqref{e.random},\eqref{e.strong-disjoint}}{\geq}
 \sum_{I\in\sP}\int_{E_I}|(-\gen)^{\critic_p}P_1\psi_I|^p\dif\meas\\
 &\overset{\eqref{e.tent}}{\geq}
 c_0^pc_0'\sum_{k=0}^{N-1}\sum_{I\in\sP_k}
 \medm(I)^{-(\dimf-1)}\medm(I)^{\dimf}\overset{\eqref{e.strong-partition}}{=}c_0^pc_0'NR_N.\label{e.khint1}
\end{align}

Suppose that the supremum in~\eqref{e.strfl} is finite, and denote it by $K$. Then,
\begin{equation}
 \|(-\gen)^{\critic_p}P_1f\|_{L^p(\ambient,\meas)}^p
 \leq K^p\|\wgrad f\|_{L^p(\skeleton,\medm)}^p,
 \qquad f\in\core.
\end{equation}
Applying this inequality to $f_\omega$ and taking expectation in $(\Omega,\bP)$, we obtain
\begin{align}
 c_0^pc_0'NR_N
 &\overset{\eqref{e.khint1}}{\leq}
 \int_{\Omega}\|(-\gen)^{\critic_p}P_1f_\omega\|_{L^p(\ambient,\meas)}^p\dif\bP(\omega)\\
 &\leq K^p\int_{\Omega}\|\wgrad f_\omega\|_{L^p(\skeleton,\medm)}^p\dif\bP(\omega)
 \overset{\eqref{e.khintc}}{\leq}
 3^pK^pN^{\frac{p}{2}}R_N.
\end{align}
Thus, $N^{1-\frac{p}{2}}\leq{3^pK^p}{c_0^{-p}c_0'^{-1}}$ for all $N\in\bN$, which is impossible because $p<2$. This proves~\eqref{e.strfl}.
\end{proof}

\begin{lemma}\label{l.coreheat}
Suppose that $(\ambient,\metric)$ is $\AR{\dimf}$ with $\dimf\in(1,\infty)$. Let $p\in[1,\infty)$ and $f\in\core$. For every $\varepsilon>0$, one has $P_\varepsilon f\in\sobolev{p}(\ambient)\cap\Dom((-\gen)^{\critic_p})$. Moreover, there is $C_{f,p}\in(0,\infty)$ such that
\begin{equation}\label{e.corehe}
 \|\wgrad P_\varepsilon f-\wgrad f\|_{L^p(\skeleton,\medm)}
 \leq C_{f,p}\varepsilon^{\frac{1}{p\dimw}},
 \  \text{whenever }\ \varepsilon\in(0,1].
\end{equation}
\end{lemma}
\begin{proof}
For $s\in(0,1]$, the semigroup property and the gradient estimate give
\begin{align}\label{e.coregr}
 \|\wgrad(-\gen)P_sf\|_{L^p(\skeleton,\medm)}
 &\overset{\eqref{e.spec}}{=}
 \|\wgrad P_{s/2}((-\gen)P_{s/2}f)\|_{L^p(\skeleton,\medm)}\\
 &\overset{\eqref{e.gradie}}{\leq}
 C_ps^{-\critic_p}\|(-\gen)P_{s/2}f\|_{L^p(\ambient,\meas)}\overset{\eqref{e.coresm},\eqref{e.sumexp}}{\leq}
 C_{f,p}s^{-1+\frac{1}{p\dimw}}.
\end{align}
Fix $\varepsilon\in(0,1]$. Let $h\in L^2(\skeleton,\medm)$. Since $f\in\domain=\Dom((-\gen)^{\frac{1}{2}})$, Tonelli's theorem and the total-variation bound for a spectral measure yield
\begin{align}
\int_0^\varepsilon\int_{[0,\infty)}\lambda e^{-s\lambda}
 \left|\dif\left\langle\proj_\lambda(-\gen)^{\frac{1}{2}}f,\riesz_2^*h\right\rangle\right|\dif s&=\int_{[0,\infty)}(1-e^{-\varepsilon\lambda})
 \left|\dif\left\langle\proj_\lambda(-\gen)^{\frac{1}{2}}f,\riesz_2^*h\right\rangle\right|\\
 &\overset{\eqref{e.riesz2}}{\leq}
 \|(-\gen)^{\frac{1}{2}}f\|_{L^2(\ambient,\meas)}\|h\|_{L^2(\skeleton,\medm)}<\infty.
\end{align}
Using $e^{-\varepsilon\lambda}-1=-\int_0^\varepsilon\lambda e^{-s\lambda}\dif s$, we obtain by Fubini's theorem that
\begin{align}
\left\langle\wgrad P_\varepsilon f-\wgrad f,h\right\rangle_{L^2(\skeleton,\medm)}&\overset{\eqref{e.riesz2},\eqref{e.spec}}{=}
 \int_{[0,\infty)}(e^{-\varepsilon\lambda}-1)
 \dif\left\langle\proj_\lambda(-\gen)^{\frac{1}{2}}f,\riesz_2^*h\right\rangle\\
 &\overset{\eqref{e.spec},\eqref{e.riesz2}}{=}
 -\int_0^\varepsilon
 \left\langle\wgrad(-\gen)P_sf,h\right\rangle_{L^2(\skeleton,\medm)}\dif s.\label{e.coreg1}
\end{align}
Let $p'\in(1,\infty]$ be conjugate to $p$. For $h\in L^{p'}(\skeleton,\medm)\cap L^2(\skeleton,\medm)$, H\"older's inequality gives
\begin{align}
 \left|\left\langle\wgrad P_\varepsilon f-\wgrad f,h\right\rangle_{L^2(\skeleton,\medm)}\right|&\overset{\eqref{e.coreg1},\eqref{e.coregr}}{\leq}
 C_{f,p}\|h\|_{L^{p'}(\skeleton,\medm)}
 \int_0^\varepsilon s^{-1+\frac{1}{p\dimw}}\dif s\\
 &=C_{f,p}p\dimw\,\varepsilon^{\frac{1}{p\dimw}}
 \|h\|_{L^{p'}(\skeleton,\medm)}.\label{e.coregin}
\end{align}
By~\eqref{e.gradie}, $\wgrad P_\varepsilon f-\wgrad f\in L^p(\skeleton,\medm)\cap L^2(\skeleton,\medm)$. For $p>1$, take the supremum in~\eqref{e.coregin} over these tests with $L^{p'}$ norm at most one. Their density in $L^{p'}(\skeleton,\medm)$ and duality give~\eqref{e.corehe}.

For $p=1$, Remark~\ref{r.arcs} provides measurable sets $F_n\uparrow\skeleton$ with $\medm(F_n)<\infty$. Take
\begin{equation}
 h_n:=\one_{F_n}\operatorname{sgn}(\wgrad P_\varepsilon f-\wgrad f)
 \in L^\infty(\skeleton,\medm)\cap L^2(\skeleton,\medm).
\end{equation}
Then
\begin{equation}
 \int_{F_n}|\wgrad P_\varepsilon f-\wgrad f|\dif\medm
 \overset{\eqref{e.coregin}}{\leq}
 C_{f,1}\dimw\,\varepsilon^{\frac{1}{\dimw}},
\end{equation}
and the monotone convergence theorem proves~\eqref{e.corehe} at $p=1$.

For every $\varepsilon\in(0,\infty)$, Theorem~\ref{t.dirich}-\ref{it.deri2} gives $P_\varepsilon f\in\domain$. Also,
\begin{equation}
 \|P_\varepsilon f\|_{L^p(\ambient,\meas)}\leq\|f\|_{L^p(\ambient,\meas)},\qquad
 \|\wgrad P_\varepsilon f\|_{L^p(\skeleton,\medm)}
 \overset{\eqref{e.gradie}}{\leq}
 C_p\varepsilon^{-\critic_p}\|f\|_{L^p(\ambient,\meas)}.
\end{equation}
Thus $P_\varepsilon f\in\sobolev{p}(\ambient)$ by~\eqref{e.sobdef}. Finally,
\begin{align}
 \int_{[0,\infty)}\lambda^{2\critic_p}e^{-2\varepsilon\lambda}
 \dif\langle\proj_\lambda f,f\rangle
 &\leq\sup_{\lambda\in[0,\infty)}(\lambda^{2\critic_p}e^{-2\varepsilon\lambda})
 \|f\|_{L^2(\ambient,\meas)}^2<\infty.
\end{align}
Equation~\eqref{e.spec} therefore gives $P_\varepsilon f\in\Dom((-\gen)^{\critic_p})$.
\end{proof}

\begin{corollary}\label{c.strong-domain-failure}
Suppose that $(\ambient,\metric)$ is $\AR{\dimf}$ with $\dimf\in(1,\infty)$. Let $p\in[1,2)$. Then
\begin{equation}
	\sup\Sett{\|(-\gen)^{\critic_p}f\|_{L^p(\ambient,\meas)}}{f\in\sobolev{p}(\ambient)\cap\Dom((-\gen)^{\critic_p})\ \text{ and }\|\wgrad f\|_{L^p(\skeleton,\medm)}=1}=\infty.
\end{equation}
\end{corollary}
\begin{proof}
Suppose otherwise that there exists $C\in(0,\infty)$ such that
\begin{equation}\label{e.dofail1}
 \|(-\gen)^{\critic_p}f\|_{L^p(\ambient,\meas)}
 \leq C\|\wgrad f\|_{L^p(\skeleton,\medm)},
 \ \text{ for all }\ f\in\sobolev{p}(\ambient)\cap\Dom((-\gen)^{\critic_p}).
\end{equation}
For $f\in\core$ and $\varepsilon\in(0,\infty)$, by Lemma~\ref{l.coreheat}, we may apply  \eqref{e.dofail1} to $P_\varepsilon f$, which, combined with the $L^p$ contractivity, gives
\begin{align}
 \|(-\gen)^{\critic_p}P_{1+\varepsilon}f\|_{L^p(\ambient,\meas)}
 &\overset{\eqref{e.spec}}{=}
 \|P_1((-\gen)^{\critic_p}P_\varepsilon f)\|_{L^p(\ambient,\meas)}\\
 &\leq\|(-\gen)^{\critic_p}P_\varepsilon f\|_{L^p(\ambient,\meas)}
\overset{\eqref{e.dofail1}}{\leq} C\|\wgrad P_\varepsilon f\|_{L^p(\skeleton,\medm)}.
\end{align}
 By~\eqref{e.analy} and strong continuity on $L^p(\ambient,\meas)$,
\begin{align}
 \|(-\gen)^{\critic_p}P_{1+\varepsilon}f
       -(-\gen)^{\critic_p}P_1f\|_{L^p(\ambient,\meas)}\overset{\eqref{e.analy}}{\leq}
 C_p\|P_\varepsilon f-f\|_{L^p(\ambient,\meas)}
 \to0\ \text{ as }\ \varepsilon\downarrow0.
\end{align}
Letting $\varepsilon\downarrow0$ and using~\eqref{e.corehe}, we obtain
\begin{equation}
 \|(-\gen)^{\critic_p}P_1f\|_{L^p(\ambient,\meas)}
 \leq C\|\wgrad f\|_{L^p(\skeleton,\medm)}, \ \text{ for all }\  f\in\core,
\end{equation}
contrary to~\eqref{e.strfl}.
\end{proof}

\begin{proof}[Proof of Theorem~\ref{t.strfl}-\ref{it.wtrfl}]
Fix $p\in(2,\infty)$ and the constants of Lemma~\ref{l.tent}. Since $\dimf>1$, we may choose $\rho\in(0,1/4)$ such that
\begin{equation}\label{e.weak-packing-rho}
 C_4 2^{\dimf-1}C_0^2
 \frac{\rho^{\dimf-1}}{1-\rho^{\dimf-1}}
 \leq\frac{1}{2C_0},
\end{equation}
where $C_4$ is the constant in~\eqref{e.tube}. This choice is independent of $N$. For $N\in\bN$, set $R_N:=\rho^{-N}\ell_0$ and $r_k:=\rho^kR_N$ for $k\in\{1,\ldots,N\}$.
Thus $r_N=\ell_0$ and $r_k<R_N/4$. By Lemma~\ref{l.ends}-\ref{it.tends3}, there are a component $C$ of $\ambient\setminus T_o$ and a point $x_0\in C$ such that $B(x_0,4R_N)\subset C$ and $\metric(g_C,x_0)\geq4R_N$.

We construct the families $\collectJ_k$ for $k=1,\ldots,N$. The families will have the following properties:
\begin{enumerate}[label=\textup{(\roman*)},align=right,leftmargin=*,topsep=5pt,parsep=0pt,itemsep=2pt]
 \item Each arc in $\collectJ_k$ has length $r_k$, is directed away from $T_o$, and is contained in $C\cap B(x_0,2R_N)$.
 \item All the arcs in $\bigcup_{k=1}^{N}\collectJ_k$ are pairwise disjoint.
 \item For every $k\in\{1,\ldots,N\}$,
 \begin{equation}\label{e.packin}
 \frac{1}{2C_0^2 6^{\dimf}}
 \left(\frac{R_N}{r_k}\right)^{\dimf}
 \leq\#\collectJ_k
 \leq C_0^2\left(\frac{R_N}{r_k}\right)^{\dimf},
 \qquad
 \bigcup_{I\in\collectJ_k}E_I\subset B(x_0,2R_N).
 \end{equation}
\end{enumerate}

We first construct the family $\collectJ_1$. Choose a maximal $6r_1$-separated set $\{w_{1,j}\}_{j=1}^{n_1}\subset B(x_0,R_N)$. By \eqref{e.ahlfor}, we have
\begin{equation}\label{e.packing-base-count}
 \frac{1}{C_0^2 6^{\dimf}}
 \left(\frac{R_N}{r_1}\right)^{\dimf}
 \leq n_1
 \leq
 C_0^2\left(\frac{2R_N}{3r_1}\right)^{\dimf}
 \leq C_0^2\left(\frac{R_N}{r_1}\right)^{\dimf}.
\end{equation}

For each $i\in\{1,\ldots,n_1\}$, by the triangle inequality,
\begin{equation}
 \metric(g_C,w_{1,i})
 \geq
 \metric(g_C,x_0)-\metric(x_0,w_{1,i})
 >3R_N>r_1.
\end{equation}
There is therefore a unique point $u_{1,i}\in[g_C,w_{1,i}]$ such that
$\metric(u_{1,i},w_{1,i})=r_1$. Define
\begin{equation}\label{e.packing-base-family}
 \collectJ_1:=\{[u_{1,i},w_{1,i}]:1\leq i\leq n_1\}.
\end{equation}
These arcs have length $r_1$ and are directed away from $T_o$.
Moreover,
\begin{equation}
 [u_{1,i},w_{1,i}]
 \subset[w_{1,i},g_C)
 \overset{\eqref{e.gatec}}{\subset}C,
 \ \text{ and }\
 [u_{1,i},w_{1,i}]
 \subset B(x_0,R_N+r_1)\subset B(x_0,2R_N).
\end{equation}
For distinct $i,i'$, by the triangle inequality and the definition of $\{w_{1,n}\}$,
\begin{equation}\label{e.packing-base-disjoint}
 \dist([u_{1,i},w_{1,i}],[u_{1,i'},w_{1,i'}])
\geq
 \metric(w_{1,i},w_{1,i'})-2r_1
 \geq 4r_1>0.
\end{equation}
Thus $\#\collectJ_1=n_1$, and~\eqref{e.packing-base-count} implies the first estimate in \eqref{e.packin} for $k=1$. Since $r_1\geq\ell_0$, Lemma~\ref{l.tent} applies to every
$I=[u_{1,i},w_{1,i}]\in\collectJ_1$. Its midpoint $z_I$ satisfies
$\metric(z_I,w_{1,i})=r_1/2$, so by Lemma~\ref{l.tent}-\ref{it.ttent1}, the triangle inequality and the fact that $\delta<1/6$, we have
\begin{align}
 E_I
 &\subset
 B(z_I,2\delta r_1)
 \subset
 B\left(w_{1,i},\left(\frac{1}{2}+2\delta\right)r_1\right)\\
 &\subset
 B(w_{1,i},r_1)
 \subset B(x_0,R_N+r_1)
 \subset B(x_0,2R_N),
 \label{e.packing-base-tents}
\end{align}
which is the second estimate in \eqref{e.packin} for $k=1$.

Let $k\in\{2,\ldots,N\}$, and suppose that
$\collectJ_1,\ldots,\collectJ_{k-1}$ have been constructed with the stated
properties. Define
\begin{equation}\label{e.packing-excluded}
 \Gamma_{k-1}:=\bigcup_{j=1}^{k-1}\bigcup_{I\in\collectJ_j}I,
 \ \text{ and }\
 V_k:=\Sett{y\in\ambient}{\dist(y,\Gamma_{k-1})<2r_k}.
\end{equation}
For $1\leq j<k$, we have $r_k/r_j=\rho^{k-j}\leq\rho$.
If $I\in\collectJ_j$ has midpoint $z_I$, then by the fact that $\rho<1/4$ and using the triangle inequality, we have
\begin{equation}
 \frac{r_j}{2}+2r_k
 \leq\left(\frac{1}{2}+2\rho\right)r_j
 <r_j,
 \qquad
 \{y\in\ambient:\dist(y,I)<2r_k\}
 \subset B(z_I,r_j).
\end{equation}
Applying \eqref{e.tube} with radius $r_j$ and
$\varepsilon=2r_k/r_j\in(0,1)$, we have
\begin{equation}\label{e.packing-single-tube}
 \meas\bigl(\{y\in\ambient:\dist(y,I)<2r_k\}\bigr)
 \overset{\eqref{e.tube}}{\leq}
 C_4\left(\frac{2r_k}{r_j}\right)^{\dimf-1}r_j^{\dimf}
 =C_4 2^{\dimf-1}r_jr_k^{\dimf-1}.
\end{equation}
The union defining $\Gamma_{k-1}$ is finite. Hence subadditivity and the
induction hypothesis in~\eqref{e.packin} imply
\begin{align}
 \meas(V_k)
 &\overset{\eqref{e.packing-excluded},\eqref{e.packing-single-tube}}{\leq}
 C_4 2^{\dimf-1}
 \sum_{j=1}^{k-1}\#\collectJ_j\,r_jr_k^{\dimf-1}\overset{\eqref{e.packin}}{\leq}
 C_4 2^{\dimf-1}C_0^2R_N^{\dimf}
 \sum_{j=1}^{k-1}\left(\frac{r_k}{r_j}\right)^{\dimf-1}\\
 &=
 C_4 2^{\dimf-1}C_0^2R_N^{\dimf}
 \sum_{j=1}^{k-1}\rho^{(k-j)(\dimf-1)}\leq 
 C_4 2^{\dimf-1}C_0^2R_N^{\dimf}
 \frac{\rho^{\dimf-1}}{1-\rho^{\dimf-1}}
 \overset{\eqref{e.weak-packing-rho}}{\leq}
 \frac{R_N^{\dimf}}{2C_0}.
 \label{e.packing-excluded-measure}
\end{align}
It follows that the Borel set $W_k:=B(x_0,R_N)\setminus V_k$ satisfies
\begin{equation}\label{e.packing-available}
 \meas(W_k)
 \geq\meas(B(x_0,R_N))-\meas(V_k)
 \overset{\eqref{e.ahlfor},\eqref{e.packing-excluded-measure}}{\geq}
 \frac{R_N^{\dimf}}{2C_0}>0.
\end{equation}

Choose a maximal $6r_k$-separated set
$\{w_{k,1},\ldots,w_{k,n_k}\}\subset W_k$. Then
\begin{equation}\label{e.packing-stage-cover}
 W_k\subset\bigcup_{i=1}^{n_k}B(w_{k,i},6r_k),
 \ \text{ and }\
 \bigsqcup_{i=1}^{n_k}B(w_{k,i},3r_k)
 \subset B(x_0,R_N+3r_k)\subset B(x_0,2R_N).
\end{equation}
Therefore,
\begin{align}
 \frac{R_N^{\dimf}}{2C_0}
 &\overset{\eqref{e.packing-available},\eqref{e.packing-stage-cover}}{\leq}
 \sum_{i=1}^{n_k}\meas(B(w_{k,i},6r_k))
 \overset{\eqref{e.ahlfor}}{\leq}
 n_kC_0(6r_k)^{\dimf},\\
 n_kC_0^{-1}(3r_k)^{\dimf}
 &\overset{\eqref{e.ahlfor},\eqref{e.packing-stage-cover}}{\leq}
 \meas(B(x_0,2R_N))
 \overset{\eqref{e.ahlfor}}{\leq}C_0(2R_N)^{\dimf},
 \label{e.packing-stage-measure}
\end{align}
that is,
\begin{equation}\label{e.packing-stage-count}
 \frac{1}{2C_0^2 6^{\dimf}}
 \left(\frac{R_N}{r_k}\right)^{\dimf}
 \overset{\eqref{e.packing-stage-measure}}{\leq}n_k
 \overset{\eqref{e.packing-stage-measure}}{\leq}
 C_0^2\left(\frac{R_N}{r_k}\right)^{\dimf}.
\end{equation}

Since $w_{k,i}\in B(x_0,R_N)$, by the triangle inequality,
\begin{equation}
 \metric(g_C,w_{k,i})
 \geq 
 \metric(g_C,x_0)-\metric(x_0,w_{k,i})
 >3R_N>r_k.
\end{equation}
Choose the unique point $u_{k,i}\in[g_C,w_{k,i}]$ with
$\metric(u_{k,i},w_{k,i})=r_k$, and define
\begin{equation}\label{e.packing-stage-family}
 \collectJ_k:=\Sett{[u_{k,i},w_{k,i}]}{1\leq i\leq n_k}.
\end{equation}
Every arc in $\collectJ_k$ has length $r_k$, is directed away from $T_o$, and satisfies
\begin{equation}
 [u_{k,i},w_{k,i}]
 \subset[w_{k,i},g_C)
 \overset{\eqref{e.gatec}}{\subset}C,
 \qquad
 [u_{k,i},w_{k,i}]
 \subset B(x_0,R_N+r_k)\subset B(x_0,2R_N).
\end{equation}
For distinct $i,i'$, by the triangle inequality and the definition of $\{w_{k,n}\}$,
\begin{equation}
 \dist([u_{k,i},w_{k,i}],[u_{k,i'},w_{k,i'}])
 \geq
 \metric(w_{k,i},w_{k,i'})-2r_k
 \geq 4r_k>0.
\end{equation}
Furthermore, by the triangle inequality and that $w_{k,i}\notin V_k$,
\begin{align}
 \dist([u_{k,i},w_{k,i}],\Gamma_{k-1})\geq 
 \dist(w_{k,i},\Gamma_{k-1})-r_k\overset{\eqref{e.packing-excluded}}{\geq}
 2r_k-r_k=r_k>0.
\end{align}
Thus the new arcs are mutually disjoint and disjoint from the arcs in $\Gamma_{1},\ldots, \Gamma_{k-1}$. In particular, $\#\collectJ_k=n_k$, so \eqref{e.packing-stage-count} gives the first estimate in \eqref{e.packin}. Finally, $r_k\geq\ell_0$ allows us to apply Lemma~\ref{l.tent} to each
$I=[u_{k,i},w_{k,i}]\in\collectJ_k$. Writing $z_I$ for its midpoint,
we have $\metric(z_I,w_{k,i})=r_k/2$, and hence by Lemma~\ref{l.tent}-\ref{it.ttent1}, the triangle inequality, and the fact that $\delta<1/6$,
\begin{align}
 E_I
 &\subset
 B(z_I,2\delta r_k)
 \subset
 B\left(w_{k,i},\left(\frac{1}{2}+2\delta\right)r_k\right)\\
 &\subset
 B(w_{k,i},r_k)
 \subset B(x_0,R_N+r_k)
\subset B(x_0,2R_N).
\end{align}
This completes the induction for $k=1,\ldots,N$.

Let $(\Omega,\bP)$ be a probability space. Let $\{\varepsilon_I:\Omega\to\{-1,1\}\}_{I\in\bigcup_{k=1}^{N}\collectJ_k}$ be a set of independent random variables with \begin{equation}
	\bP(\varepsilon_I=1)=\bP(\varepsilon_I=-1)=\frac{1}{2},\ \text{ for all }\ I\in\bigcup_{k=1}^{N}\collectJ_k.
\end{equation}
Let $E:=B(x_0,2R_N)$ and define
\begin{equation}
 u_\omega:=\sum_{k=1}^N\sum_{I\in\collectJ_k}
 \varepsilon_I(\omega)r_k^{\frac{\dimf-1}{p}}\psi_I\in\core,\ \text{ for each }\ \omega\in\Omega.
\end{equation}
Then, for every $\omega\in\Omega$,
\begin{align}\label{e.wtrflg}
 \|\wgrad u_\omega\|_{L^p(\skeleton,\medm)}^p
 &=\sum_{k=1}^N\sum_{I\in\collectJ_k}
 r_k^{\dimf-1}\|\wgrad\psi_I\|_{L^p(\skeleton,\medm)}^p\overset{\eqref{e.tent}}{\leq}
 3^p\sum_{k=1}^N\#\collectJ_k\,r_k^{\dimf}
 \overset{\eqref{e.packin}}{\leq}3^pC_0^2NR_N^{\dimf}.
\end{align}
By \eqref{e.spec}, $(-\gen)^{\critic_p}P_1\psi_I\in L^2(\ambient,\meas)$. Hence, by independence of $\{\varepsilon_{I}\}$,
\begin{align}
 \int_{\Omega}\|\one_E(-\gen)^{\critic_p}P_1u_\omega\|_{L^2(\ambient,\meas)}^2\dif\bP(\omega)&=
 \sum_{k=1}^N\sum_{I\in\collectJ_k}
 r_k^{\frac{2(\dimf-1)}{p}}
 \int_E|(-\gen)^{\critic_p}P_1\psi_I|^2\dif\meas\\
 &\overset{\eqref{e.tent},\eqref{e.packin}}{\geq}
 c_0^2\sum_{k=1}^N\sum_{I\in\collectJ_k}\meas(E_I)
 \overset{\eqref{e.tent},\eqref{e.packin}}{\geq}
 cNR_N^{\dimf}.\label{e.weak-second-moment}
\end{align}
Then, there exists $\omega_{N}\in\Omega$ such that
\begin{align}
 \|(-\gen)^{\critic_p}P_1u_{\omega_N}\|_{L^{p,\infty}(\ambient,\meas)}&\overset{\eqref{e.localL2weak}}{\geq}
 \left(\frac{p-2}{p}\right)^{\frac{1}{2}}
 \meas(E)^{\frac{1}{p}-\frac{1}{2}}
 \|\one_E(-\gen)^{\critic_p}P_1u_{\omega_N}\|_{L^2(\ambient,\meas)}\\
 &\overset{\eqref{e.ahlfor},\eqref{e.weak-second-moment}}{\geq}
 cN^{\frac{1}{2}}R_N^{\frac{\dimf}{p}}.
\end{align}
Here $(-\gen)^{\critic_p}P_1u_{\omega_N}\in L^p(\ambient,\meas)\subset L^{p,\infty}(\ambient,\meas)$ by~\eqref{e.analy}, so~\eqref{e.localL2weak} applies. Define
\begin{equation}
 f_N:=\frac{u_{\omega_N}}{3C_0^{\frac{2}{p}}N^{\frac{1}{p}}R_N^{\frac{\dimf}{p}}}\in\core.
\end{equation}
Then
\begin{equation}
 \|\wgrad f_N\|_{L^p(\skeleton,\medm)}
 \overset{\eqref{e.wtrflg}}{\leq}1,
 \ \text{ while }\ 
 \|(-\gen)^{\critic_p}P_1f_N\|_{L^{p,\infty}(\ambient,\meas)}
 \geq cN^{\frac{1}{2}-\frac{1}{p}}\uparrow\infty\ \text{ as }N\uparrow\infty.
\end{equation}
This proves~\eqref{e.wtrfl}.
\end{proof}

\begin{remark}
For $p\in(2,\infty)$, besides \eqref{e.wtrfl}, we also have
\begin{equation}\label{e.wtrflc}
      \sup\Sett{\|(-\gen)^{\critic_{p}}f\|_{L^{p,\infty}(\ambient,\meas)}}{f\in \sobolev{p}(\ambient)\cap\Dom((-\gen)^{\critic_p})\text{ and }\|\wgrad f\|_{L^{p}(\skeleton,\medm)}\leq1}=\infty.
\end{equation}
To see this, let $p\in(1,\infty)$ and $h\in L^2(\ambient,\meas)\cap L^{p,\infty}(\ambient,\meas)$. Fix a Borel set $E\subset\ambient$ with $0<\meas(E)<\infty$. The Markov property and conservativity give
\begin{equation}
 0\leq P_1\one_E\leq1,
 \ \text{ and }\
 \int_{\ambient}P_1\one_E\dif\meas=\meas(E).
\end{equation}
In particular, $\meas(\{P_1\one_E>s\})\leq s^{-1}\meas(E)<\infty$ for every $s\in(0,\infty)$. Positivity, self-adjointness and H\"older's inequality on $(0,1)$ yield
\begin{align}
 \int_E|P_1h|\dif\meas
 &\leq
 \int_{\ambient}|h|P_1\one_E\dif\meas\overset{\text{(Fubini)}}{=}
 \int_0^1\left(\int_{\{P_1\one_E>s\}}|h|\dif\meas\right)\dif s\\
 &\overset{\eqref{e.lorno}}{\leq}
 \|h\|_{(p,\infty;\ambient,\meas)}
 \int_0^1\meas(\{P_1\one_E>s\})^{1-\frac{1}{p}}\dif s\\
 &\overset{\text{(H\"older)}}{\leq}
 \|h\|_{(p,\infty;\ambient,\meas)}
 \left(\int_0^1\meas(\{P_1\one_E>s\})\dif s\right)^{1-\frac{1}{p}}=
 \|h\|_{(p,\infty;\ambient,\meas)}\meas(E)^{1-\frac{1}{p}}.
\end{align}
Taking the supremum over $E$ and using~\eqref{e.loreq} proves
\begin{align}\label{e.Pweak}
 \|P_1h\|_{L^{p,\infty}(\ambient,\meas)}
 \overset{\eqref{e.loreq}}{\leq}
 \|P_1h\|_{(p,\infty;\ambient,\meas)}
 \leq\|h\|_{(p,\infty;\ambient,\meas)}\overset{\eqref{e.loreq}}{\leq}
 \frac{p}{p-1}\|h\|_{L^{p,\infty}(\ambient,\meas)}.
\end{align}

For $p\in(2,\infty)$, one has $\critic_p\in(0,1/2)$. Every $f\in\core\subset\domain$ therefore satisfies
\begin{align}
 \int_{[0,\infty)}\lambda^{2\critic_p}\dif\langle\proj_\lambda f,f\rangle
 \leq\int_{[0,\infty)}(1+\lambda)\dif\langle\proj_\lambda f,f\rangle\overset{\eqref{e.spec},\eqref{e.spec}}{=}
 \|f\|_{L^2(\ambient,\meas)}^2+\form(f,f)<\infty.
\end{align}
Hence $f\in\Dom((-\gen)^{\critic_p})$ by~\eqref{e.spec}. If \eqref{e.wtrflc} fails, then for every $f\in\core$,
\begin{align}
 \|(-\gen)^{\critic_p}P_1f\|_{L^{p,\infty}(\ambient,\meas)}
 &\overset{\eqref{e.spec}}{=}
 \|P_1((-\gen)^{\critic_p}f)\|_{L^{p,\infty}(\ambient,\meas)}\\
 &\overset{\eqref{e.Pweak}}{\leq}
 \frac{p}{p-1}\|(-\gen)^{\critic_p}f\|_{L^{p,\infty}(\ambient,\meas)}
 \leq C_p\|\wgrad f\|_{L^p(\skeleton,\medm)},
\end{align}
which contradicts~\eqref{e.wtrfl}.
\end{remark}

\subsection{A refined Sobolev estimate and duality}
In this subsection, we first prove a Sobolev inequality \eqref{e.restri} and the duality formula \eqref{e.regdu}. We then use them, together with the reverse inequalities and counterexamples proved above, to establish Theorem~\ref{t.dualc} for $1<p<\infty$ and to show where the stronger estimates fail.
\begin{proposition}\label{p.restri}
Suppose that $(\ambient,\metric,\meas)$ is $\AR{\dimf}$ with $\dimf\in(1,\infty)$. For every $q\in(2,\infty)$, there exists $C_q\in(0,\infty)$ such that
  \begin{equation}\label{e.restri}
    \|(-\gen)^{\critic_q}f\|_{L^q(\ambient,\meas)}^q
    \leq C_q\|\wgrad f\|_{L^\infty(\skeleton,\medm)}^{q-2}
                \|\wgrad f\|_{L^2(\skeleton,\medm)}^2,
    \ \text{ for all }\ f\in\core.
  \end{equation}
\end{proposition}
\begin{proof}
  Fix $q>2$ and $f\in\core$. By \eqref{e.critp}, we have $\frac1{\dimw}<\critic_q<\frac12$ and $\frac12-\critic_q
    \overset{\eqref{e.critp}}{=}(1-2q^{-1})(\frac12-\frac1{\dimw})$. For $h\in L^1(\ambient,\meas)\cap L^2(\ambient,\meas)$,
  \begin{equation}\label{e.gradon}
    \|\wgrad P_th\|_{L^1(\skeleton,\medm)}
    \overset{\eqref{e.gradie},\eqref{e.critp}}{\leq}
    Ct^{-1+1/\dimw}\|h\|_{L^1(\ambient,\meas)},\qquad t>0.
  \end{equation}
  Since $f\in\domain$ and $P_th\in\Dom(-\gen)$, we have
  \begin{align}
    &\phantom{\ \leq} |\langle(-\gen)P_tf,h\rangle_{L^2(\ambient,\meas)}|
    \overset{\eqref{e.formde},\eqref{e.spec}}{=}
      \left|\int_\skeleton\wgrad f\,\wgrad P_th\dif\medm\right|\\
    &\overset{\text{(H\"older)}}{\leq}
      \|\wgrad f\|_{L^\infty(\skeleton,\medm)}
      \|\wgrad P_th\|_{L^1(\skeleton,\medm)}\overset{\eqref{e.gradon}}{\leq}
      Ct^{-1+1/\dimw}\|\wgrad f\|_{L^\infty(\skeleton,\medm)}
      \|h\|_{L^1(\ambient,\meas)}.
  \end{align}
  Taking the supremum over such $h$ with $\|h\|_{L^1(\ambient,\meas)}\leq1$ yields
  \begin{equation}\label{e.APinft}
    \|(-\gen)P_tf\|_{L^\infty(\ambient,\meas)}
    \leq
    Ct^{-1+1/\dimw}\|\wgrad f\|_{L^\infty(\skeleton,\medm)}.
  \end{equation}

  Define 
  \begin{equation}
    S_f(x):=\left(\int_0^\infty|(-\gen)P_tf(x)|^2\dif t\right)^{\frac{1}{2}},\ \text{ for each }x\in\ambient.
  \end{equation}
  Then $S_{f}$ is Borel measurable, and
  Tonelli's theorem gives
  \begin{align}\label{e.square}
    \|S_f\|_{L^2(\ambient,\meas)}^2
    &\overset{\eqref{e.spec}}{=}
      \int_{[0,\infty)}\left(\int_0^\infty\lambda^2e^{-2t\lambda}\dif t\right)
          \dif\langle\proj_\lambda f,f\rangle\\
    &=
      \frac12\int_{[0,\infty)}\lambda\dif\langle\proj_\lambda f,f\rangle
      \overset{\eqref{e.spec},\eqref{e.formde}}{=}
      \frac12\|\wgrad f\|_{L^2(\skeleton,\medm)}^2<\infty.
  \end{align}
In particular, $S_f(x)<\infty$ for $\meas$-a.e. $x\in\ambient$. Also, $\lambda^{2\critic_q}\leq1+\lambda$ for $\lambda\in[0,\infty)$, and hence
  \begin{equation}
    \int_{[0,\infty)}\lambda^{2\critic_q}\dif\langle\proj_\lambda f,f\rangle
    \overset{\eqref{e.spec}}{\leq}
    \|f\|_{L^2(\ambient,\meas)}^2+\form(f,f)<\infty.
  \end{equation}
  Thus $f\in\Dom((-\gen)^{\critic_q})$, and by Corollary~\ref{c.coreco},
  \begin{equation}\label{e.integ1}
    (-\gen)^{\critic_q}f(x)
    =
    \frac1{\Gamma(1-\critic_q)}
       \int_0^\infty t^{-\critic_q}(-\gen)P_tf(x)\dif t
    \quad\text{for $\meas$-a.e. }x.
  \end{equation}
  For every $x\in\ambient$ with $S_f(x)<\infty$ and every $T\in(0,\infty)$, splitting the integration over $(0,\infty)$ at $T$ and using the Cauchy--Schwarz inequality, we have
  \begin{align}
    &\phantom{\ \leq}\int_0^\infty t^{-\critic_q}|(-\gen)P_tf(x)|\dif t\\
    &\overset{\eqref{e.APinft}}{\leq}
      S_f(x)\left(\int_0^Tt^{-2\critic_q}\dif t\right)^{\frac{1}{2}}
      +C\|\wgrad f\|_{L^\infty(\skeleton,\medm)}
          \int_T^\infty t^{-\critic_q-1+1/\dimw}\dif t\\
    &=
      \frac{T^{1/2-\critic_q}}{\sqrt{1-2\critic_q}}S_f(x)
      +\frac{C T^{1/\dimw-\critic_q}}{\critic_q-1/\dimw}
          \|\wgrad f\|_{L^\infty(\skeleton,\medm)}.\label{e.integr}
  \end{align}
  When both $S_f(x)$ and $\|\wgrad f\|_{L^\infty(\skeleton,\medm)}$ are positive, choose $T=\left(\frac{\|\wgrad f\|_{L^\infty(\skeleton,\medm)}}{S_f(x)}\right)^{2\dimw/(\dimf-1)}$.
  Since $\dimw=\dimf+1$, 
  \begin{align}
    T^{1/2-\critic_q}S_f(x)
    =T^{1/\dimw-\critic_q}\|\wgrad f\|_{L^\infty(\skeleton,\medm)}\overset{\eqref{e.critp}}{=}
      \|\wgrad f\|_{L^\infty(\skeleton,\medm)}^{1-2/q}S_f(x)^{2/q}.
  \end{align}
  Consequently,
  \begin{equation}\label{e.pointw}
    |(-\gen)^{\critic_q}f(x)|
    \overset{\eqref{e.integ1},\eqref{e.integr}}{\leq}
    C_q\|\wgrad f\|_{L^\infty(\skeleton,\medm)}^{1-2/q}S_f(x)^{2/q}
    \quad\text{for $\meas$-a.e. }x\in\ambient.
  \end{equation}
  If $S_f(x)=0$, the integrand in~\eqref{e.integ1} vanishes for $\sL^{1}$-a.e. $t\in(0,\infty)$, so~\eqref{e.pointw} still holds. If $\|\wgrad f\|_{L^\infty(\skeleton,\medm)}=0$, then~\eqref{e.square} gives $S_f=0$ $\meas$-a.e. on $\ambient$ and the same conclusion follows. Finally,
  \begin{align}
    \|(-\gen)^{\critic_q}f\|_{L^q(\ambient,\meas)}^q
    \overset{\eqref{e.pointw}}{\leq}
      C_q\|\wgrad f\|_{L^\infty(\skeleton,\medm)}^{q-2}
          \int_\ambient S_f(x)^2\dif\meas(x)\overset{\eqref{e.square}}{\leq}
      C_q\|\wgrad f\|_{L^\infty(\skeleton,\medm)}^{q-2}
          \|\wgrad f\|_{L^2(\skeleton,\medm)}^2,
  \end{align}
  proving~\eqref{e.restri}.
\end{proof}

\begin{lemma}\label{l.regdu}
Suppose $(\ambient,\metric,\meas)$ is $\AR{\dimf}$ with $\dimf\in[1,\infty)$.  Let $p\in(1,\infty)$ and $p'=p/(p-1)$. For every $f\in\core$, $h\in\collectD_p^{\rm sp}$ and $t\in(0,\infty)$,
  \begin{equation}\label{e.regdu}
    \int_\skeleton\wgrad f\,\riesz_pP_th\dif\medm
    =\langle(-\gen)^{\critic_{p'}}P_tf,h\rangle_{L^2(\ambient,\meas)}.
  \end{equation}
  Moreover, $\riesz_pP_th\to\riesz_ph$ in $L^2(\skeleton,\medm)$ as $t\downarrow0$. If $f\in\core\cap\Dom((-\gen)^{\critic_{p'}})$, then
  \begin{equation}\label{e.duali}
    \int_\skeleton\wgrad f\,\riesz_ph\dif\medm
    =\langle(-\gen)^{\critic_{p'}}f,h\rangle_{L^2(\ambient,\meas)}.
  \end{equation}
  If $\critic_{p'}\in(0,1/2]$, then $\core\subset \Dom((-\gen)^{\critic_{p'}})$. In particular, this holds for $1<p\leq2$, and for every $p\in(1,\infty)$ when $\dimf=1$.
\end{lemma}
\begin{proof}
By~\eqref{e.strict-core-domains}, both $(-\gen)^{-\critic_p}h$ and $(-\gen)^{-\critic_p}P_th$ belong to $\domain$. Since $f\in\core\subset\domain$, we obtain
  \begin{align}
    \int_\skeleton\wgrad f\,\riesz_pP_th\dif\medm
    &\overset{\eqref{e.rieszp},\eqref{e.formde},\eqref{e.spec}}{=}
      \langle(-\gen)^{\frac{1}{2}}f,(-\gen)^{1/2-\critic_p}P_th\rangle_{L^2(\ambient,\meas)}\\
    &\overset{\eqref{e.spec}}{=}
      \int_{(0,\infty)}\lambda^{1-\critic_p}e^{-t\lambda}
           \dif\langle\proj_\lambda f,h\rangle\overset{\eqref{e.critp},\eqref{e.spec}}{=}
      \langle(-\gen)^{\critic_{p'}}P_tf,h\rangle_{L^2(\ambient,\meas)}.
  \end{align}
  Here $1-\critic_p=\critic_{p'}$. The  spectral integral is absolutely convergent because
  \begin{align}
    \int_{(0,\infty)}\lambda^{1-\critic_p}e^{-t\lambda}
       \bigl|\dif\langle\proj_\lambda f,h\rangle\bigr|\leq
      \sup_{\lambda\in[0,\infty)}(\lambda^{1-\critic_p}e^{-t\lambda})
         \|f\|_{L^2(\ambient,\meas)}\|h\|_{L^2(\ambient,\meas)}<\infty.
  \end{align}
  This proves~\eqref{e.regdu}.

  Lemma~\ref{l.densp} also gives $(-\gen)^{1/2-\critic_p}h\in L^2(\ambient,\meas)$. Therefore
  \begin{align}
    \|\riesz_pP_th-\riesz_ph\|_{L^2(\skeleton,\medm)}
    &\overset{\eqref{e.rieszp},\eqref{e.formde},\eqref{e.spec}}{=}
      \|(P_t-I)(-\gen)^{1/2-\critic_p}h\|_{L^2(\ambient,\meas)}\\
    &\to0\ \text{ as }\ t\downarrow0.
  \end{align}
  If $f\in\Dom((-\gen)^{\critic_{p'}})$, the same argument gives
  \begin{equation}
    \|(-\gen)^{\critic_{p'}}(P_t-I)f\|_{L^2(\ambient,\meas)}
    \overset{\eqref{e.spec}}{=}
    \|(P_t-I)(-\gen)^{\critic_{p'}}f\|_{L^2(\ambient,\meas)}
    \to0\ \text{ as }\ t\downarrow0.
  \end{equation}
 By the Cauchy--Schwarz inequality,
  \begin{align}
    \left|\int_\skeleton\wgrad f\,(\riesz_pP_th-\riesz_ph)\dif\medm\right|
    &\leq
      \|\wgrad f\|_{L^2(\skeleton,\medm)}
      \|\riesz_pP_th-\riesz_ph\|_{L^2(\skeleton,\medm)}\to 0,\\
    |\langle(-\gen)^{\critic_{p'}}(P_t-I)f,h\rangle_{L^2(\ambient,\meas)}|
    &\leq
      \|(-\gen)^{\critic_{p'}}(P_t-I)f\|_{L^2(\ambient,\meas)}
      \|h\|_{L^2(\ambient,\meas)}\to 0.
  \end{align}
  Taking $t\downarrow0$ proves~\eqref{e.duali}.

  Finally, when $0<\critic_{p'}\leq1/2$, we have $\lambda^{2\critic_{p'}}\leq1+\lambda$ for all $\lambda\in[0,\infty)$, so
  \begin{equation}
    \int_{[0,\infty)}\lambda^{2\critic_{p'}}\dif\langle\proj_\lambda f,f\rangle\leq
    \|f\|_{L^2(\ambient,\meas)}^2+\form(f,f)<\infty.
  \end{equation}
  By~\eqref{e.critp}, this condition holds when $1<p\leq2$, and it holds for every $p\in(1,\infty)$ when $\dimf=1$, since then $\critic_{p'}=1/2$.
\end{proof}

\begin{proof}[Proof of Theorem~\ref{t.dualc} for $p\in(1,\infty)$]
\

  Assume $(\ambient,\metric)$ is $\AR{\dimf}$ with $\dimf\in(1,\infty)$ and write $p'=p/(p-1)$. Let $F_n$ be increasing finite unions of finite arcs that cover $\skeleton$ up to a $\medm$-null set, as in Remark~\ref{r.arcs}.
\begin{enumerate}[label=\textup{({\arabic*})}, align=right, leftmargin=*, topsep=5pt, parsep=0pt, itemsep=2pt]
  \item[\ref{it.dlowb}] Let $p\in(1,2)$.   Fix $h\in\collectD_p^{\rm sp}$. By~\eqref{e.strict-core-domains}, $\riesz_ph\in L^2(\skeleton,\medm)$. Fix a Borel set $E\subset F_n$, and let $\eta=\operatorname{sgn}(\riesz_ph)\one_E$. By applying Lemma~\ref{l.densg1} with exponent $2$, we obtain $\{f_j\}_{j\in\bN}\subset\core$ such that
  \begin{equation}\label{e.densin}
    \|\wgrad f_j-\eta\|_{L^2(\skeleton,\medm)}\to 0,
    \ \text{ and }\ \sup_{j\in\bN}\|\wgrad f_j\|_{L^\infty(\skeleton,\medm)}\leq2.
  \end{equation}
  Since $p'>2$, the refined estimate gives
  \begin{align}
    \|(-\gen)^{\critic_{p'}}f_j\|_{L^{p'}(\ambient,\meas)}
    &\overset{\eqref{e.restri}}{\leq}
      C_p\|\wgrad f_j\|_{L^\infty(\skeleton,\medm)}^{1-2/p'}
          \|\wgrad f_j\|_{L^2(\skeleton,\medm)}^{2/p'}\\
          &\overset{\eqref{e.densin}}{\leq}
      C_p2^{1-2/p'}\|\wgrad f_j\|_{L^2(\skeleton,\medm)}^{2/p'}.\label{e.enptre}
  \end{align}
  The convergence in~\eqref{e.densin} permits passage to the scalar pairing with $\riesz_ph\in L^2(\skeleton,\medm)$. Thus
  \begin{align}
    \int_E|\riesz_ph|\dif\medm
    &\overset{\eqref{e.densin}}{=}
      \lim_{j\to\infty}\int_\skeleton\wgrad f_j\,\riesz_ph\dif\medm\overset{\eqref{e.duali}}{=}
      \lim_{j\to\infty}\langle(-\gen)^{\critic_{p'}}f_j,h\rangle_{L^2(\ambient,\meas)}\\
    &\overset{\eqref{e.enptre}}{\leq}
      C_p\|h\|_{L^p(\ambient,\meas)}
          \lim_{j\to\infty}\|\wgrad f_j\|_{L^2(\skeleton,\medm)}^{2/p'}\\
    &\overset{\eqref{e.densin}}{=}
      C_p\|h\|_{L^p(\ambient,\meas)}\|\eta\|_{L^2(\skeleton,\medm)}^{2/p'}
      \leq C_p\|h\|_{L^p(\ambient,\meas)}\medm(E)^{1/p'}.\label{e.settes}
  \end{align}
  For an arbitrary Borel $E\subset\skeleton$ with $\medm(E)<\infty$, the sets $E\cap F_n$ increase to $E$ up to a null set. Therefore, by the monotone convergence theorem,
  \begin{align}
    \int_E|\riesz_ph|\dif\medm
    &\overset{\text{}}{=}
      \lim_{n\to\infty}\int_{E\cap F_n}|\riesz_ph|\dif\medm\\
    &\overset{\eqref{e.settes}}{\leq}
      C_p\|h\|_{L^p(\ambient,\meas)}
      \lim_{n\to\infty}\medm(E\cap F_n)^{1/p'}
      =C_p\|h\|_{L^p(\ambient,\meas)}\medm(E)^{1/p'}.
  \end{align}
  Hence
  \begin{align}
    \|\riesz_ph\|_{L^{p,\infty}(\skeleton,\medm)}
    \overset{\eqref{e.loreq}}{\leq}
      \|\riesz_ph\|_{(p,\infty;\skeleton,\medm)}&\overset{\eqref{e.lorno}}{=}
      \sup_{0<\medm(E)<\infty}\medm(E)^{-1/p'}\int_E|\riesz_ph|\dif\medm\\
      &\overset{\eqref{e.settes}}{\leq}C_p\|h\|_{L^p(\ambient,\meas)}.\label{e.direct}
  \end{align}
By Lemma~\ref{l.densp}, $\collectD_p^{\rm sp}$ is dense in $L^p(\ambient,\meas)$. Therefore, by \eqref{e.direct}, we can extend $\riesz_p$ to all functions in  $L^p(\ambient,\meas)$.

\item[\ref{it.dlowr}] We will prove \ref{it.dlowr} for $p\in(1,2)$. In this case, we have $p'\in(2,\infty)$. Assume otherwise that $\riesz_p:\collectD_{p}^{\rm sp}\to L^{2}(\skeleton,\medm)$ defined in \eqref{e.rieszp} can be extended to a bounded linear operator $\riesz_p:L^{p,1}(\ambient,\meas)\to L^{p}(\skeleton,\medm)$. Let $M:=\norm{\riesz_p}_{L^{p,1}(\ambient,\meas)\to L^{p}(\skeleton,\medm)}$. Fix $f\in\core$. Since $f\in L^{p'}(\ambient,\meas)\cap L^2(\ambient,\meas)$,~\eqref{e.analy} gives
  \begin{equation}
    (-\gen)^{\critic_{p'}}P_1f\in L^{p'}(\ambient,\meas)\cap L^2(\ambient,\meas).
  \end{equation}
  Then for $h\in\collectD_p^{\rm sp}$,
  \begin{align}
    &\phantom{\ \leq}|\langle(-\gen)^{\critic_{p'}}P_1f,h\rangle_{L^2(\ambient,\meas)}|\overset{\eqref{e.regdu}}{=}
      \left|\int_\skeleton\wgrad f\,\riesz_pP_1h\dif\medm\right|\\
    &\overset{\text{(H\"older)}}{\leq}
      \|\wgrad f\|_{L^{p'}(\skeleton,\medm)}\|\riesz_pP_1h\|_{L^p(\skeleton,\medm)}\leq
      C_pM\|\wgrad f\|_{L^{p'}(\skeleton,\medm)}  \|h\|_{L^{p,1}(\ambient,\meas)}.\label{e.strong1}
  \end{align}
   By~\eqref{e.lorchain},  $(-\gen)^{\critic_{p'}}P_1f\in L^{p',\infty}(\ambient,\meas)$, so by the density of $\collectD_p^{\rm sp}$ in $L^{p,1}(\ambient,\meas)$,
  \begin{align}\label{e.dualwe}
    \|(-\gen)^{\critic_{p'}}P_1f\|_{L^{p',\infty}(\ambient,\meas)}
   & \overset{\eqref{e.lordualnorm}}{\leq}
      C_p\sup_{\substack{h\in L^{p,1}(\ambient,\meas)\\
                         \|h\|_{L^{p,1}(\ambient,\meas)}\leq1}}
        \left|\int_\ambient(-\gen)^{\critic_{p'}}P_1f\,h\dif\meas\right|\\
        &\overset{\eqref{e.strong1}}{\leq}C_pM\|\wgrad f\|_{L^{p'}(\skeleton,\medm)}.
  \end{align}
  This contradicts~\eqref{e.wtrfl} and proves~\ref{it.dlowr} for $1<p<2$.

\item[\ref{it.dequa}] For $p=2$, we have $\critic_2=1/2$.  By \eqref{e.strict-core-domains}, \eqref{e.riesz2} and the fact $\Ker(-\gen)=\{0\}$,
\begin{equation}
    \|\riesz_2h\|_{L^2(\skeleton,\medm)}
    \overset{\eqref{e.formde},\eqref{e.spec}}{=}
    \|(-\gen)^{\frac{1}{2}}(-\gen)^{-\frac{1}{2}}h\|_{L^2(\ambient,\meas)}
    \overset{}{=}\|h\|_{L^2(\ambient,\meas)}.
\end{equation}
  Density of $\collectD_2^{\rm sp}$ proves uniqueness of this extension and establishes~\ref{it.dequa}.
\item[\ref{it.dhigh}]
Let $p\in(2,\infty)$. Fix $h\in\collectD_p^{\rm sp}$. Since $p'\in(1,2)$, for every $f\in\core$ and $t\in(0,\infty)$ we have
  \begin{align}
    &\phantom{\ \leq}\left|\int_\skeleton\wgrad f\,\riesz_pP_th\dif\medm\right|
  \overset{\eqref{e.regdu}}{=}
      |\langle(-\gen)^{\critic_{p'}}P_tf,h\rangle_{L^2(\ambient,\meas)}|\\
    &\overset{\eqref{e.lorholder}}{\leq}
      \|(-\gen)^{\critic_{p'}}P_tf\|_{L^{p',\infty}(\ambient,\meas)}
      \|h\|_{L^{p,1}(\ambient,\meas)}\overset{\eqref{e.scalw}}{\leq}
      C_p\|\wgrad f\|_{L^{p'}(\skeleton,\medm)}
          \|h\|_{L^{p,1}(\ambient,\meas)}.
  \end{align}
  By Lemma~\ref{l.regdu}, $\riesz_pP_th\to\riesz_ph$ in $L^2(\skeleton,\medm)$. Since $\wgrad f\in L^2(\skeleton,\medm)$, we obtain
  \begin{equation}
    \left|\int_\skeleton\wgrad f\,\riesz_ph\dif\medm\right|
    \leq
    C_p\|\wgrad f\|_{L^{p'}(\skeleton,\medm)}\|h\|_{L^{p,1}(\ambient,\meas)}.
  \end{equation}
  If $\eta$ is bounded and vanishes outside a finite union of finite arcs, Lemma~\ref{l.densg1} gives $f_j\in\core$ with
  \begin{equation}
    \lim_{j\to\infty}(\|\wgrad f_j-\eta\|_{L^{p'}(\skeleton,\medm)}
    +\|\wgrad f_j-\eta\|_{L^2(\skeleton,\medm)})= 0.
  \end{equation}
  By the Cauchy--Schwarz inequality,
  \begin{equation}
 \lim_{j\to\infty}   \left|\int_\skeleton(\wgrad f_j-\eta)\riesz_ph\dif\medm\right|
    \leq
   \lim_{j\to\infty} \|\wgrad f_j-\eta\|_{L^2(\skeleton,\medm)}
    \|\riesz_ph\|_{L^2(\skeleton,\medm)}=0.
  \end{equation}
Therefore, by Lemma~\ref{l.densg1},
  \begin{equation}\label{e.dualb}
    \left|\int_\skeleton\eta\,\riesz_ph\dif\medm\right|
    \leq
    C_p\|\eta\|_{L^{p'}(\skeleton,\medm)}\|h\|_{L^{p,1}(\ambient,\meas)}.
  \end{equation}
  Apply this with $E_n=F_n\cap\{|\riesz_ph|\leq n\}$ and $\eta_n=\operatorname{sgn}(\riesz_ph)|\riesz_ph|^{p-1}\one_{E_n}$.
  The test function $\eta_n$ is bounded and supported on $F_n$, and $\int_{E_n}|\riesz_ph|^p\dif\medm\leq n^p\medm(F_n)<\infty$. Since $(p-1)p'=p$,
  \begin{equation}
    \int_{E_n}|\riesz_ph|^p\dif\medm
    \overset{\eqref{e.dualb}}{\leq}
    C_p\|h\|_{L^{p,1}(\ambient,\meas)}
       \left(\int_{E_n}|\riesz_ph|^p\dif\medm\right)^{1/p'}.
  \end{equation}
That is,
  \begin{equation}
    \left(\int_{E_n}|\riesz_ph|^p\dif\medm\right)^{1/p}
    \leq C_p\|h\|_{L^{p,1}(\ambient,\meas)};
  \end{equation}
Since $\riesz_ph\in L^2(\skeleton,\medm)$, and the sets $E_n$ increase to $\skeleton$ up to a null set, we have, by the monotone convergence theorem, that
  \begin{equation}\label{e.Rcoreb}
    \|\riesz_ph\|_{L^p(\skeleton,\medm)}
    =
    \lim_{n\to\infty}\left(\int_{E_n}|\riesz_ph|^p\dif\medm\right)^{1/p}
    \leq C_p\|h\|_{L^{p,1}(\ambient,\meas)},\ \text{ for all }h\in\collectD_p^{\rm sp}.
  \end{equation}
  By Lemma~\ref{l.densp}, $\collectD_p^{\rm sp}$ is dense in $L^{p,1}(\ambient,\meas)$. Thus we may extend $\riesz_p$ to $L^{p,1}(\ambient,\meas)$, establishing the positive assertion in~\ref{it.dhigh}.
  
  Assume otherwise that $\riesz_p:\collectD_{p}^{\rm sp}\to L^{2}(\skeleton,\medm)$ defined in \eqref{e.rieszp} can be extended to a bounded linear operator $\riesz_p:L^{p}(\ambient,\meas)\to L^{p}(\skeleton,\medm)$. Fix $f\in\core$. Since $f\in L^{p'}(\ambient,\meas)\cap L^2(\ambient,\meas)$,~\eqref{e.analy} gives
  \begin{equation}
    (-\gen)^{\critic_{p'}}P_1f\in L^{p'}(\ambient,\meas)\cap L^2(\ambient,\meas).
  \end{equation}
  Then, for $h\in\collectD_p^{\rm sp}$,
  \begin{align}
    &\phantom{\ \leq}|\langle(-\gen)^{\critic_{p'}}P_1f,h\rangle_{L^2(\ambient,\meas)}|\overset{\eqref{e.regdu}}{=}
      \left|\int_\skeleton\wgrad f\,\riesz_pP_1h\dif\medm\right|\\
    &\overset{\text{(H\"older)}}{\leq}
      \|\wgrad f\|_{L^{p'}(\skeleton,\medm)}\|\riesz_pP_1h\|_{L^p(\skeleton,\medm)}\leq
      C\|\wgrad f\|_{L^{p'}(\skeleton,\medm)}\|h\|_{L^p(\ambient,\meas)}.\label{e.strong2}
  \end{align}
Since $1<p'<2$, by the density of $\collectD_p^{\rm sp}$ in $L^p(\ambient,\meas)$, we have
  \begin{align}
    \|(-\gen)^{\critic_{p'}}P_1f\|_{L^{p'}(\ambient,\meas)}
    &=
      \sup_{\substack{h\in L^p(\ambient,\meas)\\
                         \|h\|_{L^p(\ambient,\meas)}\leq1}}
        \left|\int_\ambient(-\gen)^{\critic_{p'}}P_1f\,h\dif\meas\right|\overset{\eqref{e.strong2}}{\leq}
      C\|\wgrad f\|_{L^{p'}(\skeleton,\medm)}.
  \end{align}
  This contradicts~\eqref{e.strfl} and proves the failure assertion in~\ref{it.dhigh}.
  \qedhere
\end{enumerate}
\end{proof}

\subsection{Failure of weak type \texorpdfstring{$(1,1)$}{(1,1)} for the Riesz transform}
In this subsection, we will use \eqref{e.gate} to show that $\riesz_1$ fails to be of weak type $(1,1)$. This proves Theorem~\ref{t.dualc}-\ref{it.dlow1} and the case $p=1$ of part~\ref{it.dlowr}.

\begin{lemma}\label{l.gate}
Assume that $(\ambient,\metric)$ is $\AR{\dimf}$ with $\dimf\in[1,\infty)$. Let $J=(a,b)$ be an open arc. For $z\in J$, let $C_z$ be the component of $\ambient\setminus\{z\}$ containing $b$, and suppose that $\ol{C_z}$ is compact for every $z\in J$. With $\sigma_J:=\restr{\vecone_{(a,b)}}{J}$, every $u\in\Dom(-\gen)$ satisfies
\begin{equation}\label{e.gate}
 \sigma_J(z)\wgrad u(z)=\int_{C_z}(-\gen)u\dif\meas,
 \qquad\text{for $\medm$-almost every }z\in J.
\end{equation}
\end{lemma}
\begin{proof}
Fix $z\in J$ and $r\in(0,\metric(z,b))$, and choose $z_r\in(z,b)$ with $\metric(z,z_r)=r$. Define
\begin{equation}
 \eta_{z,r}(x):=\frac{\metric(z,\pi_{[z,z_r]}(x))}{r},\qquad x\in\ambient.
\end{equation}
The projection is $1$-Lipschitz and constant on each component of $\ambient\setminus[z,z_r]$, by Lemma~\ref{l.projt}-\ref{it.proj4} and \ref{it.proj5}. If $x\notin C_z$, then the arc $[x,z_r]$ passes through $z$, so $\pi_{[z,z_r]}(x)=z$. Consequently,
\begin{equation}
 0\leq\eta_{z,r}\leq\one_{C_z},\quad
 \supp(\eta_{z,r})\subset\ol{C_z},\ \text{ and }\
 \wgrad\eta_{z,r}\overset{\eqref{e.arcder}}{=}\frac{1}{r}\vecone_{(z,z_r)}
 \quad\medm\text{-a.e. on }\skeleton.
\end{equation}
Thus $\eta_{z,r}$ is Lipschitz and is compactly supported, with
\begin{equation}
 \|\eta_{z,r}\|_{L^2(\ambient,\meas)}^2\leq\meas(\ol{C_z})<\infty,
\ \text{ and }\
 \|\wgrad\eta_{z,r}\|_{L^2(\skeleton,\medm)}^2=\frac{1}{r}<\infty.
\end{equation}
In particular, $\eta_{z,r}\in\domain$ by~\eqref{e.formde}. The generator identity and the fundamental theorem along arcs give
\begin{align}\label{e.endpgt}
 \int_{\ambient}(-\gen)u\,\eta_{z,r}\dif\meas
 =\form(u,\eta_{z,r})\overset{\eqref{e.formde}}{=}\frac{1}{r}\int_{(z,z_r)}\sigma_J\wgrad u\dif\medm
 \overset{\eqref{e.arcftc}}{=}\frac{u(z_r)-u(z)}{r}.
\end{align}
If $x\in C_z$, the tripod property gives $[z,x]\cap[z,b]=[z,w]$ for some $w\neq z$. Indeed, $w=z$ would imply $z\in[x,b]$, contrary to $x,b\in C_z$. Hence, if $r\in(0,\metric(z,w))$ and $x\in C_z$, then $\eta_{z,r}(x)=1$, and thus $\lim_{r\downarrow0}\eta_{z,r}(x)=\one_{C_z}(x)$ for all $x\in C_z$. Moreover, by the Cauchy--Schwarz inequality,
\begin{equation}
 \int_{C_z}|(-\gen)u|\dif\meas
 \leq
 \meas(\ol{C_z})^{\frac{1}{2}}\|(-\gen)u\|_{L^2(\ambient,\meas)}<\infty.
\end{equation}
Since $u\in\Dom(-\gen)\subset\domain$ is absolutely continuous on $[a,b]$, its derivative in the direction from $a$ to $b$ is $\sigma_J\wgrad u$ almost everywhere. Therefore, for $\medm$-a.e. $z\in J$, by the dominated convergence theorem,
\begin{align}
 \int_{C_z}(-\gen)u\dif\meas
 =\lim_{r\downarrow0}\int_{\ambient}(-\gen)u\,\eta_{z,r}\dif\meas\overset{\eqref{e.endpgt}}{=}\lim_{r\downarrow0}\frac{u(z_r)-u(z)}{r}
 =\sigma_J(z)\wgrad u(z).
\end{align}
\end{proof}

\begin{lemma}\label{l.endpoi}
Assume that $(\ambient,\metric)$ is $\AR{\dimf}$ with $\dimf\in(1,\infty)$. There exists $c_1\in(0,\infty)$ such that, for every integer $N\in\bN$, there exists $g_N\in\collectD_1^{\rm sp}$ satisfying
\begin{equation}\label{e.p1fail}
 \|g_N\|_{L^1(\ambient,\meas)}\leq1,
 \ \text{ and }\
 \|\riesz_1g_N\|_{L^{1,\infty}(\skeleton,\medm)}\geq c_1N.
\end{equation}
\end{lemma}
\begin{proof}
By Lemma~\ref{l.ends}-\ref{it.tends3}, there are a component $C$ of $\ambient\setminus T_o$ and $x_0\in C$ such that $B(x_0,8)\subset C$. Thus the gate $g_C\in T_o$ satisfies $\metric(g_C,x_0)\geq8$. Choose $a\in[g_C,x_0]$ with $\metric(a,x_0)=1$. Then $I:=[a,x_0]\subset C$ is directed away from $T_o$, $\medm(I)=1$ and \begin{equation}
 \Sett{x}{\dist(x,I)<4}\subset B(x_0,5)\subset C.
\end{equation}
Let $\mu:=\restr{\medm}{I}$ be a Borel probability measure on $I$, and define the real-valued functions
\begin{equation}
 P_s\mu(x):=\int_Ip_s(x,y)\dif\mu(y),\ \text{ for all }\ (s,x)\in(0,\infty)\times\ambient.
\end{equation}
Then,
\begin{equation}\label{e.endpoint-source}
 \|P_s\mu\|_{L^1(\ambient,\meas)}
 \overset{\text{(Fubini)}}{=}\int_I\int_{\ambient}p_s(x,y)\dif\meas(x)\dif\mu(y)\leq1,\ \text{ and }\
 \|P_s\mu\|_{L^\infty(\ambient,\meas)}
 \overset{\eqref{e.HKE}}{\leq}Cs^{-\frac{\dimf}{\dimw}},
\end{equation}
and consequently
\begin{equation}
	 \|P_s\mu\|_{L^2(\ambient,\meas)}^2
 \leq\|P_s\mu\|_{L^\infty(\ambient,\meas)}\|P_s\mu\|_{L^1(\ambient,\meas)}
\overset{\eqref{e.endpoint-source}}{ \leq} Cs^{-\frac{\dimf}{\dimw}}
\end{equation}
By Fubini's theorem and the semigroup property, we also have
\begin{equation}\label{e.endpoint-semigroup}
 P_t(P_s\mu)=P_{t+s}\mu,\ \text{ for all }\ s,t\in(0,\infty).
\end{equation}

\begin{enumerate}[label=\textit{Step {\arabic*}.}, align=right, leftmargin=*, topsep=5pt, parsep=0pt, itemsep=2pt]
  \item \textit{For each $(x,r)\in I\times(0,1/16)$ with $1/4\leq\metric(a,x)\leq3/4$, choose an associated point $a_{x,r}\in I$ and an associated arc $J_{x,r}$.}
  
  Choose $\delta\in(0,1)$ such that $C_4\delta^{\dimf-1}<C_0^{-1}$, where $C_4$ is the constant in~\eqref{e.tube}. Fix $(x,r)\in I\times(0,1/16)$ with $1/4\leq\metric(a,x)\leq3/4$. Then
\begin{equation}
 \meas\bigl(B(x,r)\cap\Sett{y\in\ambient}{\dist(y,I)<\delta r}\bigr)\overset{\eqref{e.tube}}{\leq}C_4\delta^{\dimf-1}r^{\dimf}
 <C_0^{-1}r^{\dimf}
 \overset{\eqref{e.ahlfor}}{\leq}\meas(B(x,r)).
\end{equation}
We thus may choose $y\in B(x,r)$ with $\dist(y,I)\geq\delta r$ and denote $a_{x,r}:=\pi_I(y)$. Then
\begin{equation}
 \delta r\leq\metric(a_{x,r},y)<r,\quad
 \metric(a_{x,r},x)<r,\ \text{ and }\
 [a_{x,r},y]\cap I=\{a_{x,r}\}.
\end{equation}
Thus $a_{x,r}\in(a,x_0)$. Both $a_{x,r}$ and $y$ lie in $C$, so convexity gives $[a_{x,r},y]\subset C$. Since $[g_C,a_{x,r}]$ meets $[a_{x,r},y]$ only at $a_{x,r}$, the tripod property gives
\begin{equation}
 [g_C,y]=[g_C,a_{x,r}]\cup[a_{x,r},y].
\end{equation}
In particular, $[a_{x,r},y]$ is directed away from $T_o$. Define the open subarc
\begin{equation}\label{e.Jx}
 J_{x,r}:=\Sett{z\in[a_{x,r},y]}{\delta r/4<\metric(z,a_{x,r})<\delta r/2}.
\end{equation}
Then $\medm(J_{x,r})={\delta r}/{4}$. For $z\in J_{x,r}$, let $D_z$ be the component of $\ambient\setminus\{z\}$ containing $y$. Lemma~\ref{l.branch} gives $D_z\subset C$ and compactness of $\ol{D_z}$. Every arc from $y$ to $I$ passes through $z$, so $D_z\cap I=\emptyset$. Every arc from $w\in D_z$ to $y'\in I$ also passes through $z$, whence
\begin{equation}\label{e.Dzdist}
 \dist(D_z,I)\geq\dist(z,I)=\metric(z,a_{x,r})
 \overset{\eqref{e.Jx}}{>}\frac{\delta r}{4}.
\end{equation}
Choose $w_z\in[a_{x,r},y]$ with $\metric(w_z,a_{x,r})=3\delta r/4$, and set $B_z:=B(w_z,\delta r/8)$. For $z\in J_{x,r}$, we have 
\begin{equation}
 \metric(w_z,z)>\frac{\delta r}{4},\ \text{ and }\
 \metric(w_z,x)+\frac{\delta r}{8}<r+\frac{7\delta r}{8}<2r.
\end{equation}
Then $z\notin B_{z}$ and $w_z\in D_z\cap B_{z}$. Hence
\begin{equation}\label{e.endpoint-ball}
 B_z\subset D_z\cap B(x,2r).
\end{equation}

\item \textit{Choose a sequence of disjoint arcs.} 

Fix an integer $N\in\bN$. Fix $k\in\{0,\ldots,N-1\}$ and put $r=4^{-k}/32$, so that
$0<r\leq1/32<1/16$. Recall that $I=[a,x_0]$ has length one. For each $j\in\{0,\ldots,\lfloor1/(10r)\rfloor\}$, let $x_{k,j}$ be the unique
point of $I$ such that $\metric(a,x_{k,j})=\frac14+5rj$. These points exist and belong to the middle half of $I$, because
\begin{equation}
 \frac14
 \leq\frac14+5rj
 \leq\frac14+5r\left\lfloor\frac1{10r}\right\rfloor
 \leq
 \frac14+\frac12=\frac34.
\end{equation}
Moreover, for distinct indices $i,j$, we have
\begin{equation}
 \metric(x_{k,i},x_{k,j})
 =\left|\metric(a,x_{k,i})-\metric(a,x_{k,j})\right|
 =5r|i-j|
 \geq5r.
\end{equation}
In particular, for this $k$, the points $\{x_{k,i}\}_{i}$ are distinct, and
\begin{equation}
 \#\Sett{x_{k,j}}{0\leq j\leq\lfloor1/(10r)\rfloor}
 =\left\lfloor\frac1{10r}\right\rfloor+1
 >\frac1{10r}.
\end{equation}

Apply the previous step separately to every pair $(x_{k,j},r)$, and denote the
resulting special point and open subarc by $a_{k,j}:=a_{x_{k,j},r}\in I$ and $J_{k,j}:=J_{x_{k,j},r}$, respectively. The preceding construction and~\eqref{e.Jx} give
\begin{equation}
 \metric(x_{k,j},a_{k,j})<r,\qquad
 \pi_I(J_{k,j})=\{a_{k,j}\},\qquad
 \medm(J_{k,j})=\frac{\delta r}{4},
\end{equation}
and
\begin{equation}\label{e.distancbddzI}
 \frac{\delta r}{4}
 \overset{\eqref{e.Jx}}{<}\metric(z,a_{k,j})
=\dist(z,I)
 \overset{\eqref{e.Jx}}{<}\frac{\delta r}{2},\ \text{ for every $z\in J_{k,j}$}.
\end{equation}
For distinct indices $i,j$, the triangle inequality yields
\begin{align}
 \metric(a_{k,i},a_{k,j})
 \geq \metric(x_{k,i},x_{k,j})
 -\metric(x_{k,i},a_{k,i})
 -\metric(x_{k,j},a_{k,j})>5r|i-j|-2r
 \geq3r>0.
\end{align}
Thus $a_{k,i}\neq a_{k,j}$. If a point $z$ belonged to both
$J_{k,i}$ and $J_{k,j}$, the projection identities would imply $a_{k,i}=\pi_I(z)=a_{k,j}$, which is impossible. Hence for each $k$, the arcs $\{J_{k,j}\}_{j=0}^{\lfloor1/(10r)\rfloor}$ are pairwise disjoint. Since they are Borel sets, finite additivity gives
\begin{equation}\label{e.endpoint-scale-length}
 \medm\left(\bigcup_{j=0}^{\lfloor1/(10r)\rfloor}J_{k,j}\right)
 =
 \sum_{j=0}^{\lfloor1/(10r)\rfloor}\medm(J_{k,j})\overset{\eqref{e.Jx}}{=}
 \left(\left\lfloor\frac1{10r}\right\rfloor+1\right)
 \frac{\delta r}{4}
 \geq\frac1{10r}\frac{\delta r}{4}
 =\frac{\delta}{40}.
\end{equation}

 Substituting $r=4^{-k}/32$ into \eqref{e.distancbddzI} gives
\begin{equation}
 \frac{\delta4^{-k}}{128}<\dist(z,I)
 <\frac{\delta4^{-k}}{64},
 \ \text{ for all }\  z\in J_{k,j}.
\end{equation}
Let $0\leq k<l\leq N-1$. Since $l\geq k+1$, we have $ \frac{\delta4^{-l}}{64}
 \leq 
 \frac{\delta4^{-k}}{256}
 <\frac{\delta4^{-k}}{128}$. Therefore, for any admissible indices $i,j$ and any
$z\in J_{k,i}$ and $w\in J_{l,j}$,
\begin{equation}
 \dist(w,I)
 <\frac{\delta4^{-l}}{64}
 \leq\frac{\delta4^{-k}}{256}
 <\frac{\delta4^{-k}}{128}
 <\dist(z,I).
\end{equation}
In particular, $z\neq w$, so $J_{k,i}\cap J_{l,j}=\emptyset$. Thus $\{J_{k,j}\}_{k,j}$ are pairwise disjoint. Finally, define their finite union by
\begin{equation}
 U_N:=
 \bigcup_{k=0}^{N-1}
 \bigcup_{j=0}^{\lfloor32\,4^k/10\rfloor}J_{k,j},
\end{equation}
where the upper index is $\lfloor1/(10r)\rfloor$ with $r=4^{-k}/32$.
Then
\begin{equation}\label{e.UN}
 \medm(U_N)
 =
 \sum_{k=0}^{N-1}
 \sum_{j=0}^{\lfloor32\,4^k/10\rfloor}\medm(J_{k,j})\overset{\eqref{e.endpoint-scale-length}}{\geq}
 \sum_{k=0}^{N-1}\frac{\delta}{40}
 =\frac{\delta}{40}N.
\end{equation}
\item \textit{Estimates on $P_{s}\mu$}.

For $z\in U_N$, define $M_z:[0,\infty)\to[0,\infty)$ by
\begin{equation}\label{e.Mz}
 M_z(s):=\int_{D_z}P_s\mu\dif\meas\ \text{ for }\ s\in(0,\infty),\ \text{ and }\ M_z(0):=0.
\end{equation}
Then $(z,s)\mapsto M_{z}(s)$ is measurable, since $\one_{D_z}(w)=\one_{\{\metric(w,y)<\metric(w,z)+\metric(z,y)\}}$ is a Borel function of $(z,w)$. 

Fix $z\in U_N$, and let $(k,j)$ be the unique pair such that $z\in J_{k,j}$.
Set $x=x_{k,j}$ and $r=4^{-k}/32$, and let $y=y_{x,r}$ be the point chosen in Step~1.
For every $w\in B_z$ and $y'\in I\cap B(x,r)$, we have
\begin{equation}
 \metric(w,y')\leq\metric(w,x)+\metric(x,y')
 \overset{\eqref{e.endpoint-ball}}{<}2r+r=3r. 
\end{equation}
Hence, for every $s\in[r^{\dimw},2r^{\dimw}]$,
\begin{equation}\label{e.Mzlowe-1}
 p_s(w,y')\overset{\eqref{e.HKE}}{\geq}
 \frac{C_1^{-1}}{(2r^{\dimw})^{\dimf/\dimw}}
 \exp\left(-c_3\left(\frac{(3r)^{\dimw}}{r^{\dimw}}\right)^{\frac{1}{\dimw-1}}\right)
 =c r^{-\dimf}. 
\end{equation}
Since $x$ lies in the middle half of $I$ and $r\leq1/32$, we have $\mu(I\cap B(x,r))=2r$. Consequently,
\begin{equation}\label{e.Mzlowe}
 M_z(s)
 \overset{\eqref{e.endpoint-ball},\eqref{e.Mz}}{\geq}
 \int_{B_z}\int_{I\cap B(x,r)}p_s(w,y')\dif\mu(y')\dif\meas(w)\overset{\eqref{e.Mzlowe-1}}{\geq}cr^{-\dimf}\meas(B_z)\,2r
 \overset{\eqref{e.ahlfor}}{\geq}cr.
\end{equation}
Since $\critic_1=1-1/\dimw$, this gives
\begin{equation}\label{e.endpo1}
\int_0^\infty  \frac{s^{-1-\frac{1}{\dimw}}M_z(s)}{\dimw\Gamma(\critic_1)} \dif s
 \overset{\eqref{e.Mzlowe}}{\geq}
 \frac{cr}{\dimw\Gamma(\critic_1)}\int_{r^{\dimw}}^{2r^{\dimw}}s^{-1-\frac{1}{\dimw}}\dif s=\frac{c(1-2^{-1/\dimw})}{\Gamma(\critic_1)}=:c_4>0.
\end{equation}
The constant $c_4$ is independent of $N$, $r$, and $z$.

For fixed $N$, we have
\begin{equation}
 \dist(D_z,I)\overset{\eqref{e.Dzdist}}{\geq}\frac{\delta4^{-N}}{32},
 \ \text{ and }\ \meas(D_z)\leq\meas(\ol C)<\infty,\ \text{ for all $z\in U_N$}.
\end{equation}
By~\eqref{e.HKE} and the definition of $M_{z}$,
\begin{equation}\label{e.endpoint-mass-upper}
 0\leq M_z(s)\leq C\meas(\ol C)s^{-\frac{\dimf}{\dimw}}
 \exp(-c\left(\frac{(\delta4^{-N}/32)^{\dimw}}{s}\right)^{\frac{1}{\dimw-1}}).
\end{equation}
By the dominated convergence theorem, we have 
\begin{equation}\label{e.endpoint-mass-derivative}
 M_z'(s)=\int_{D_z}\int_I\frac{\dif}{\dif s}p_s(w,y)\dif\mu(y)\dif\meas(w),
\end{equation}
and there exists a constant $C_{N}$ independent of $s\in(0,\infty)$ such that
\begin{equation}\label{e.endpoint-mass-derivative1}
	|M_z'(s)|\overset{\eqref{e.ddt}}{\leq}C\meas(\ol C)s^{-1-\frac{\dimf}{\dimw}}
 \exp(-c\left(\frac{(\delta4^{-N}/32)^{\dimw}}{s}\right)^{\frac{1}{\dimw-1}})\leq C_{N}.
\end{equation}
By the mean value theorem, we have
\begin{equation}\label{e.Mzlip}
 |M_z(s)-M_z(t)|\leq C_N|s-t|,
  \text{ and } 0\leq M_z(s)\leq\min\{C_Ns,1\},
  \text{ for all } s,t\geq0.
\end{equation}
In particular,
\begin{equation}
 \int_0^\infty s^{-1-\frac{1}{\dimw}}M_z(s)\dif s
 \overset{\eqref{e.Mzlip}}{\leq}
 C_N\int_0^1s^{-\frac{1}{\dimw}}\dif s
 +\int_1^\infty s^{-1-\frac{1}{\dimw}}\dif s<\infty.
\end{equation}

For $0<a<b<\infty$ and $s\in[a,b]$, equations~\eqref{e.endpoint-semigroup} and~\eqref{e.spec} give
\begin{equation}
 M_z(s)=\int_{[0,\infty)}e^{-(s-a/2)\lambda}
 \dif\langle\proj_\lambda P_{a/2}\mu,\one_{D_z}\rangle.
\end{equation}
Note that the total variation of the measure $\dif\langle\proj_\lambda P_{a/2}\mu,\one_{D_z}\rangle$ is not greater than $\|P_{a/2}\mu\|_{L^2(\ambient,\meas)}\meas(D_z)^{\frac{1}{2}}$. Since $\lambda e^{-(s-a/2)\lambda}\leq2/(ea)$ on this interval, by the dominated convergence theorem, we have for every $s\in(0,\infty)$,
\begin{equation}\label{e.endpoint-derivative-pairing}
 M_z'(s)\overset{\eqref{e.spec}}{=}-\langle(-\gen)P_s\mu,\one_{D_z}\rangle_{L^2(\ambient,\meas)}.
\end{equation}
Fix $\varepsilon\in(0,\infty)$. Applying \eqref{e.Lpt2} with $\theta=1/\dimw$, $t=\varepsilon/2$ and $h=P_{\varepsilon/2}\mu$, we have
\begin{align}
 \int_{D_z}(-\gen)^{\frac{1}{\dimw}}P_\varepsilon\mu\dif\meas
 &\overset{\eqref{e.Lpt2},\eqref{e.critp}}{=}
 \frac{1}{\Gamma(\critic_1)}\int_0^\infty t^{-\frac{1}{\dimw}}
 \langle(-\gen)P_{\varepsilon+t}\mu,\one_{D_z}\rangle\dif t\\
 &\overset{\eqref{e.endpoint-derivative-pairing}}{=}
 -\frac{1}{\Gamma(\critic_1)}\int_0^\infty t^{-\frac{1}{\dimw}}M_z'(\varepsilon+t)\dif t.\label{e.endpoint-fractional-pairing}
\end{align}
In particular,~\eqref{e.endpoint-mass-derivative1} and $\dimw=\dimf+1$ imply
\begin{equation}
 \int_0^\infty t^{-\frac{1}{\dimw}}|M_z'(\varepsilon+t)|\dif t
 \leq C_N\int_0^1t^{-\frac{1}{\dimw}}\dif t
 +C\meas(\ol C)\int_1^\infty t^{-2}\dif t<\infty,
\end{equation}
and 
\begin{equation}
 |M_z(\varepsilon+t)-M_z(\varepsilon)|
 \overset{\eqref{e.Mzlip}}{\leq}\min\{C_Nt,1\}.
\end{equation}
Therefore $t^{-1/\dimw}(M_z(\varepsilon+t)-M_z(\varepsilon))$ tends to zero both at zero and at infinity, and its product with $t^{-1}$ is absolutely integrable. Integrating by parts first on $[a,b]\subset(0,\infty)$ gives
\begin{align}
 &\phantom{\ \leq}\int_a^bt^{-\frac{1}{\dimw}}M_z'(\varepsilon+t)\dif t\\
 &=\Big(t^{-\frac{1}{\dimw}}\bigl(M_z(\varepsilon+t)-M_z(\varepsilon)\bigr)\Big)\Big|_a^b
 +\frac{1}{\dimw}\int_a^bt^{-1-\frac{1}{\dimw}}\bigl(M_z(\varepsilon+t)-M_z(\varepsilon)\bigr)\dif t.
\end{align}
Letting $a\downarrow0$ and $b\to\infty$ in this scalar identity and using~\eqref{e.endpoint-fractional-pairing}, we obtain
\begin{equation}\label{e.endpoint-increments}
 \int_{D_z}(-\gen)^{\frac{1}{\dimw}}P_\varepsilon\mu\dif\meas
 =\frac{1}{\dimw\Gamma(\critic_1)}\int_0^\infty t^{-1-\frac{1}{\dimw}}
 \bigl(M_z(\varepsilon)-M_z(\varepsilon+t)\bigr)\dif t.
\end{equation}
Since $M_z(0)=0$, we have
\begin{align}
 |M_z(\varepsilon)-M_z(\varepsilon+t)+M_z(t)|
 &\leq|M_z(\varepsilon)-M_z(\varepsilon+t)|+|M_z(t)-M_z(0)|\\
& \overset{\eqref{e.Mzlip}}{\leq}2C_Nt,
\end{align}
and 
\begin{align}
	 |M_z(\varepsilon)-M_z(\varepsilon+t)+M_z(t)|
 &\leq|M_z(\varepsilon)-M_z(0)|+|M_z(t)-M_z(\varepsilon+t)|\\
& \overset{\eqref{e.Mzlip}}{\leq}2C_N\varepsilon.
\end{align}
Combining these bounds and splitting the integral at $\varepsilon$, we have
\begin{align}
 &\phantom{\ \leq}\left|\int_{D_z}(-\gen)^{\frac{1}{\dimw}}P_\varepsilon\mu\dif\meas
 +\frac{1}{\dimw\Gamma(\critic_1)}\int_0^\infty t^{-1-\frac{1}{\dimw}}M_z(t)\dif t\right|\\
 &\overset{\eqref{e.endpoint-increments}}{\leq}
 \frac{2C_N}{\dimw\Gamma(\critic_1)}\int_0^\infty t^{-1-\frac{1}{\dimw}}\min\{t,\varepsilon\}\dif t\overset{\eqref{e.critp}}{=}
 \frac{2C_N}{\critic_1\Gamma(\critic_1)}\varepsilon^{\critic_1}
 =:C_N'\varepsilon^{\critic_1}.\label{e.regula}
\end{align}
The constants are independent of $z\in U_N$ and $\varepsilon\in(0,\infty)$.

By \eqref{e.analy} and \eqref{e.endpoint-source}, we have for all $R\in(0,\infty)$ that
\begin{align}
 &\phantom{\ \leq}\left|\int_{D_z}(-\gen)^{\frac{1}{\dimw}}
 \bigl(-2P_{R+\varepsilon}\mu+P_{2R+\varepsilon}\mu\bigr)\dif\meas\right|\\
 &\leq2\|(-\gen)^{\frac{1}{\dimw}}P_R(P_\varepsilon\mu)\|_{L^1(\ambient,\meas)}
 +\|(-\gen)^{\frac{1}{\dimw}}P_{2R}(P_\varepsilon\mu)\|_{L^1(\ambient,\meas)}\\&\overset{\eqref{e.analy},\eqref{e.endpoint-source}}{\leq}CR^{-\frac{1}{\dimw}},\label{e.larget}
\end{align}
where the constant $C$ is independent of $N$, $z$ and $\varepsilon$.
\end{enumerate}

We are now able to construct $g_{N}$. Choose $R\in(1,\infty)$, and then choose $\varepsilon_N>0$ such that
\begin{equation}
 CR^{-\frac{1}{\dimw}}\leq\frac{c_4}{4},\qquad
 C_N'\varepsilon_N^{\critic_1}\leq\frac{c_4}{4}.
\end{equation}
Define
\begin{equation}
 g_N:=\frac{1}{4}(I-P_R)^2P_{\varepsilon_N}\mu
 =\frac{1}{4}(I-P_R)^2P_{\varepsilon_N/2}(P_{\varepsilon_N/2}\mu).
\end{equation}
By~\eqref{e.endpoint-source}, $P_{\varepsilon_N/2}\mu\in L^1(\ambient,\meas)\cap L^2(\ambient,\meas)$. Since $L^{1,1}=L^1$, definition~\eqref{e.specc-strict} gives $g_N\in\collectD_1^{\rm sp}$. Also, by the $L^1$-contractivity of the semigroup,
\begin{equation}
 \|g_N\|_{L^1(\ambient,\meas)}
 \leq
 \frac{1+2+1}{4}\|P_{\varepsilon_N}\mu\|_{L^1(\ambient,\meas)}
 \overset{\eqref{e.endpoint-source}}{\leq}1.
\end{equation}
Let $u_N:=(-\gen)^{-\critic_1}g_N$. Lemma~\ref{l.densp}-\ref{it.rieszrg} gives $u_N\in\domain$ and $\riesz_1g_N=\wgrad u_N$. To apply Lemma~\ref{l.gate}, we need to show that $u_N\in\Dom(-\gen)$. In fact,
\begin{align}
 \int_{[0,\infty)}\lambda^2\dif\langle\proj_\lambda u_N,u_N\rangle&\overset{\eqref{e.spec},\eqref{e.critp}}{=}
 \frac{1}{16}\int_{[0,\infty)}\lambda^{\frac{2}{\dimw}}(1-e^{-R\lambda})^4e^{-\varepsilon_N\lambda}
 \dif\langle\proj_\lambda P_{\varepsilon_N/2}\mu,P_{\varepsilon_N/2}\mu\rangle\\
 &\leq\frac{1}{16}\left(\frac{2}{e\dimw\varepsilon_N}\right)^{\frac{2}{\dimw}}
 \|P_{\varepsilon_N/2}\mu\|_{L^2(\ambient,\meas)}^2<\infty.
\end{align}
Here we used $0\leq1-e^{-R\lambda}\leq1$ and
$\sup_{\lambda\in[0,\infty)}\lambda^{1/\dimw}e^{-\varepsilon_N\lambda/2}=(2/(e\dimw\varepsilon_N))^{1/\dimw}$. Consequently,
\begin{equation}
 (-\gen)u_N\overset{\eqref{e.spec},\eqref{e.critp}}{=}
 \frac{1}{4}(-\gen)^{\frac{1}{\dimw}}(I-P_R)^2P_{\varepsilon_N}\mu,
 \quad\text{$\meas$-a.e. on $\ambient$.}
\end{equation}
By Lemma~\ref{l.gate},
\begin{align}\label{e.meuN}
 |\riesz_1g_N(z)|
 &\overset{\eqref{e.gate}}{=}\left|\int_{D_z}(-\gen)u_N\dif\meas\right|\overset{\eqref{e.endpo1},\eqref{e.regula},\eqref{e.larget}}{\geq}
 \frac{1}{4}\left(c_4-C_N'\varepsilon_N^{\critic_1}-CR^{-\frac{1}{\dimw}}\right)\\
 &\geq\frac{1}{4}\left(c_4-\frac{c_4}{4}-\frac{c_4}{4}\right)=\frac{c_4}{8},
 \qquad\medm\text{-a.e. }z\in U_N.
\end{align}
Therefore,
\begin{equation}
 \|\riesz_1g_N\|_{L^{1,\infty}(\skeleton,\medm)}\overset{\eqref{e.lorpq}}{\geq}\frac{c_4}{16}\medm\Sett{z\in\skeleton}{|\riesz_1g_N(z)|>c_4/16}\overset{\eqref{e.meuN},\eqref{e.UN}}{\geq}\frac{c_4}{16}\frac{\delta N}{40}
 =\frac{c_4\delta}{640}N.
\end{equation}
Thus~\eqref{e.p1fail} holds with $c_1:=c_4\delta/640>0$.
\end{proof}

\begin{proof}[Proof of Theorem~\ref{t.dualc} for $p=1$]
\
\begin{enumerate}[label=\textup{({\arabic*})}, align=right, leftmargin=*, topsep=5pt, parsep=0pt, itemsep=2pt]
  \item[\ref{it.dlow1}] Lemma~\ref{l.endpoi} gives
\begin{equation}
 \sup
\Sett{ \|\riesz_1h\|_{L^{1,\infty}(\skeleton,\medm)}}{h\in\collectD_1^{\rm sp},\ \|h\|_{L^1(\ambient,\meas)}\leq1 }
 \overset{\eqref{e.p1fail}}{\geq}\sup_{N\in\bN}c_1N=\infty.
\end{equation}
This proves~\eqref{e.dfail1}.
\item[\ref{it.dlowr}] Since $L^{1,1}(\ambient,\meas)=L^1(\ambient,\meas)$, by Chebyshev's inequality, we have\begin{equation}
 \|\riesz_1h\|_{L^{1,\infty}(\skeleton,\medm)}
 \leq\|\riesz_1h\|_{L^1(\skeleton,\medm)}.
\end{equation}
A bounded extension from $L^{1,1}(\ambient,\meas)$ to $L^1(\skeleton,\medm)$ would therefore contradict~\eqref{e.dfail1}. This proves~\ref{it.dlowr} at $p=1$.
\qedhere
\end{enumerate}
\end{proof}

\section{The case \texorpdfstring{$\dimf=1$}{d\_f=1}}\label{s.dimens}
In this section, we prove Theorem~\ref{t.one} for $\dimf=1$.
\subsection{The weak reverse inequality at \texorpdfstring{$p=1$}{p=1}}
In this subsection, we prove the weak reverse inequality \eqref{e.oneczc} for $f\in\core$. We first construct a Lipschitz function that agrees with $f$ outside a set of small measure. We then estimate this function and obtain the estimated that will be used in Theorem~\ref{t.one}-\ref{it.onewea}.

\begin{lemma}\label{l.onecz}
Assume that $(\ambient,\metric,\meas)$ is $\AR{1}$. There exists $C\in(0,\infty)$ such that
\begin{equation}\label{e.oneczc}
 \sup_{t\in(0,\infty)}\|(-\gen)^{\frac12}P_tf\|_{L^{1,\infty}(\ambient,\meas)}
 \leq C\|\wgrad f\|_{L^1(\skeleton,\medm)}, \ \text{ for all }\ f\in\core.
\end{equation}
\end{lemma}

\begin{proof}
Fix $f\in\core$ and $\lambda>0$ and set
$\mu=|\wgrad f|\dif\medm$.
If $\mu(\ambient)=0$, then~\eqref{e.arcftc} and the fact that $f\in C_{c}(\ambient)$ imply that $f=0$. We therefore assume $\mu(\ambient)>0$. Let \begin{equation}
	\sM\mu(x):=\sup\Sett{\frac{\mu(B)}{\meas(B)}}{B\text{ is an open ball in $(\ambient,\metric)$,\ $x\in B$}},\ \text{ for every }x\in\ambient.
\end{equation}
For $\lambda\in(0,\infty)$, we define \begin{equation}
	\sB_{\lambda}:=\Sett{B\text{ is an open ball in $(\ambient,\metric)$}}{\mu(B)>\lambda\meas(B)},
\end{equation}
and define an open set $\Omega_{\lambda}:=\bigcup_{B\in\sB_{\lambda}}B=\{\sM\mu>\lambda\}$. If $B(x,r)\in\sB_{\lambda}$, then \begin{equation}\label{e.oneczv-1}
	r\leq C\meas(B(x,r))\leq C\lambda^{-1}\mu(B)\leq C\lambda^{-1}\mu(\ambient).
\end{equation}
So the balls in $\sB_{\lambda}$ have uniformly bounded radii. By the $5B$-covering lemma \cite[Theorem 1.2]{Hei01}, there exists a subfamily of disjoint balls $(B_{i})_{i\in I}$ such that $\Omega_{\lambda}=\bigcup_{B\in\sB_{\lambda}}B\subset \bigcup_{i\in I}5B_{i}$. Thus,
\begin{equation}\label{e.oneczv}
 \meas(\Omega_{\lambda})=\meas\big(\bigcup_{B\in\sB_{\lambda}}B\big)\leq \meas\big(\bigcup_{i\in I}5B_{i}\big)\leq C\sum_{i\in I}\meas(B_{i}) \overset{\eqref{e.oneczv-1}}{\leq}
 C\lambda^{-1}\|\wgrad f\|_{L^1(\skeleton,\medm)}.
\end{equation}
The set $\Omega_{\lambda}$ is bounded. Indeed, choose a compact finite
union of arcs $K$ supporting $\mu$. If $x\notin K$ and a ball $B(z,r)$
contains $x$ and intersects $K$, then $2r>\dist(x,K)$, so
\begin{equation}
 \frac{\mu(B(z,r))}{\meas(B(z,r))}
 \overset{\eqref{e.ahlfor}}{\leq}
 \frac{C_0\|\wgrad f\|_{L^1(\skeleton,\medm)}}{r}
 \leq\frac{2C_0\|\wgrad f\|_{L^1(\skeleton,\medm)}}{\dist(x,K)}.
\end{equation}
Taking the supremum over all such balls, we have \begin{equation}
	\sM\mu(x)\leq \frac{2C_0\|\wgrad f\|_{L^1(\skeleton,\medm)}}{\dist(x,K)},\ \text{ for all }x\in\ambient\setminus K,
\end{equation}
and therefore $\Omega_{\lambda}$ is bounded. Since $\ambient$ is unbounded, $\ambient\setminus\Omega_{\lambda}$ is nonempty.

For distinct $x,y\notin\Omega_{\lambda}$, the arc formula gives
\begin{align}
 |f(x)-f(y)|
 &\overset{\eqref{e.arcftc}}{\leq}\mu([x,y])\leq\lambda\meas(B(x,2\metric(x,y)))\overset{\eqref{e.ahlfor}}{\leq}2C_0\lambda\metric(x,y).
\end{align}
By the McShane extension theorem~\cite[Theorem~6.2]{Hei01}, the function
\begin{equation}
 g(x)=\inf_{y\notin\Omega_{\lambda}}
       \{f(y)+2C_0\lambda\metric(x,y)\},\quad x\in\ambient,
\end{equation}
is a $2C_0\lambda$-Lipschitz extension of $f|_{\ambient\setminus\Omega_{\lambda}}$. Since $g$ vanishes outside $\Omega_{\lambda}\cup\supp(f)$, $g\in C_{c}(\ambient)$. By~\eqref{e.onelen}, $\wgrad g\in L^{2}(\skeleton,\medm)$, hence $g\in\domain$, and $\wgrad g=\wgrad f$ $\medm$-a.e. on $\skeleton\setminus\Omega_{\lambda}$.
Consequently,
\begin{align}\label{e.onegoo}
 \|\wgrad g\|_{L^2(\skeleton,\medm)}^2
 &\leq2C_0\lambda\int_{\skeleton\setminus\Omega_{\lambda}}|\wgrad f|\dif\medm
 +4C_0^2\lambda^2\medm(\Omega_{\lambda})\overset{\eqref{e.one-measures},\eqref{e.oneczv}}{\leq}
 C\lambda\|\wgrad f\|_{L^1(\skeleton,\medm)}.
\end{align}
By the spectral theorem,
\begin{align}
 \|(-\gen)^{\frac{1}{2}}P_tg\|_{L^2(\ambient,\meas)}^2
 &\overset{\eqref{e.spec}}{=}
 \int_{[0,\infty)}u e^{-2tu}\dif\langle\proj_u g,g\rangle
 \leq\int_{[0,\infty)}u\dif\langle\proj_u g,g\rangle \\
 &=\form(g,g)\overset{\eqref{e.formde}}{=}\|\wgrad g\|_{L^2(\skeleton,\medm)}^2
 \overset{\eqref{e.onegoo}}{\leq}C\lambda\|\wgrad f\|_{L^1(\skeleton,\medm)}.
 \label{e.onecz1}
\end{align}

For $y\in\Omega_{\lambda}$, let $\delta(y):=\dist(y,\ambient\setminus\Omega_{\lambda})$.
Since $\Omega_{\lambda}$ is open and its complement is nonempty,
$0<\delta(y)<\infty$. For every $z\in\ambient\setminus\Omega_{\lambda}$,
the identity $g(z)=f(z)$ and the Lipschitz property for $g$ give
\begin{equation}
 |f(y)-g(y)|\leq|f(y)-f(z)|+|g(z)-g(y)|
 \leq4C_0\lambda\metric(y,z). 
\end{equation}
Taking the infimum over $z\in\ambient\setminus\Omega_{\lambda}$, we obtain
\begin{equation}\label{e.oneczb}
 |f(y)-g(y)|\leq4C_0\lambda\delta(y),\qquad y\in\Omega_{\lambda}.
\end{equation}
Since $\dimf=1$ and $\dimw=2$, \eqref{e.fracke} gives
$|q_t^{(1/2)}(x,y)|\leq C\metric(x,y)^{-2}$ for $x\neq y$.
For $y\in\Omega_{\lambda}$ and $x\in\ambient\setminus\Omega_{\lambda}$,
we have $\metric(x,y)\geq\delta(y)$. Integrating over the annuli
$2^k\delta(y)\leq\metric(x,y)<2^{k+1}\delta(y)$, $k\geq0$, gives
\begin{align}
 \int_{\ambient\setminus\Omega_{\lambda}}|q_t^{(1/2)}(x,y)|\dif\meas(x)
 &\leq C\sum_{k=0}^{\infty}
 \frac{\meas\bigl(B(y,2^{k+1}\delta(y))\bigr)}{(2^k\delta(y))^2} \\
 &\overset{\eqref{e.ahlfor}}{\leq}
 \frac{C}{\delta(y)}\sum_{k=0}^{\infty}2^{-k}
 \leq\frac{C}{\delta(y)}. \label{e.oneczk}
\end{align}
Since $f,g\in C_{c}(\ambient)$, we know that $f-g\in L^1(\ambient,\meas)\cap L^2(\ambient,\meas)$.
By \eqref{e.fracin} and Tonelli's theorem,
\begin{align}
 &\phantom{\ \leq}\int_{\ambient\setminus\Omega_{\lambda}}|(-\gen)^{\frac{1}{2}}P_t(f-g)|\dif\meas
 \\
 &\leq\int_{\Omega_{\lambda}}|f(y)-g(y)|
 \left(\int_{\ambient\setminus\Omega_{\lambda}}|q_t^{(1/2)}(x,y)|\dif\meas(x)\right)\dif\meas(y)\\
 &\overset{\eqref{e.oneczk}}{\leq}
 C\int_{\Omega_{\lambda}}\frac{|f(y)-g(y)|}{\delta(y)}\dif\meas(y)
 \overset{\eqref{e.oneczb}}{\leq}C\lambda\meas(\Omega_{\lambda})
 \overset{\eqref{e.oneczv}}{\leq}C\|\wgrad f\|_{L^1(\skeleton,\medm)}. \label{e.oneczt}
\end{align}
If $\Omega_{\lambda}=\emptyset$, then $f=g$ and the left-hand side of
\eqref{e.oneczt} is zero.

Applying Chebyshev's inequality to each term in
\[(-\gen)^{\frac{1}{2}}P_tf=(-\gen)^{\frac{1}{2}}P_tg+(-\gen)^{\frac{1}{2}}P_t(f-g),\] we obtain
\begin{align}
 &\phantom{\ \leq}\lambda\meas\left(\Sett{x\in\ambient}{|(-\gen)^{\frac{1}{2}}P_tf(x)|>\lambda}\right)\\
 &\leq\lambda\meas(\Omega_{\lambda})
 +\frac{4}{\lambda}\|(-\gen)^{\frac{1}{2}}P_tg\|_{L^2(\ambient,\meas)}^2
 +2\int_{\ambient\setminus\Omega_{\lambda}}|(-\gen)^{\frac{1}{2}}P_t(f-g)|\dif\meas
 \\
 &\overset{\eqref{e.oneczv},\,\eqref{e.onecz1},\,\eqref{e.oneczt}}{\leq}
 C\|\wgrad f\|_{L^1(\skeleton,\medm)}. 
\end{align}
Taking the supremum over $t,\lambda\in(0,\infty)$ proves \eqref{e.oneczc}.
\end{proof}
\subsection{Bounds for the Riesz transform}
We record the following estimates on $\riesz_{2}$, which are consequences of \cite{CJLL26} and \cite{BCC26}.
\begin{lemma}\label{l.onerie}
Assume that $(\ambient,\metric)$ is $\AR{1}$. There exists $C\in(0,\infty)$ such that
\begin{equation}\label{e.onerie}
 \|\riesz_2h\|_{L^{1,\infty}(\skeleton,\medm)}
 \leq C\|h\|_{L^1(\ambient,\meas)},
 \qquad h\in L^1(\ambient,\meas)\cap L^2(\ambient,\meas).
\end{equation}
For each $p\in(1,\infty)$, there exists $C_p\in(0,\infty)$ such that
\begin{equation}\label{e.oneri1}
 \|\riesz_2h\|_{L^p(\skeleton,\medm)}
 \leq C_p\|h\|_{L^p(\ambient,\meas)},
 \qquad h\in L^p(\ambient,\meas)\cap L^2(\ambient,\meas).
\end{equation}
\end{lemma}
\begin{proof}
We apply~\cite[Theorem~5.1]{CJLL26} to the real differential
\begin{equation}
D:\domain \to L^2(\ambient,\meas),
\quad u\mapsto w^{-\frac12}\wgrad u\in L^{2}(\ambient,\meas).
\end{equation}
For every $u,v\in\domain$, we have
\begin{equation}\label{e.oneadj}
 \langle Du,Dv\rangle_{L^2(\ambient,\meas)}
=
 \int_{\skeleton}\wgrad u\,\wgrad v\dif\medm
 \overset{\eqref{e.formde}}{=}\form(u,v).
\end{equation}
Fix $u\in\domain$. By the definition of the adjoint, $Du\in\Dom(D^*)$ if and only if there exists $h\in L^2(\ambient,\meas)$ such that
$\langle Du,Dv\rangle_{L^2(\ambient,\meas)}=\langle h,v\rangle_{L^2(\ambient,\meas)}$ for every $v\in\domain$. In that case, $D^*Du=h$.
By \eqref{e.oneadj}, this condition is equivalent to
$\form(u,v)=\langle h,v\rangle_{L^2(\ambient,\meas)}$ for every $v\in\domain$.
By the definition of the generator associated with $\form$, this holds if and only if $u\in\Dom(\gen)$ and $h=-\gen u$.
Thus $\Dom(D^*D)=\Dom(\gen)$ and $D^*Du=-\gen u$ for all $u\in\Dom(D^*D)$.
By applying \cite[Theorem 4.4]{EG25} to each orientation chart, we know that
\begin{equation}
 \one_{\{u=0\}}w^{-\frac12}\wgrad u=0
 \quad\meas\text{-a.e. on } \ambient,\ \text{ for all }\ u\in\domain.
\end{equation}
By \eqref{e.riesz2}, $D=w^{-\frac{1}{2}}\riesz_2\circ(-\gen)^{\frac{1}{2}}$. Therefore, by applying \cite[Theorem~5.1]{CJLL26} with $\sT=w^{-\frac{1}{2}}\riesz_2$, we know that for every
$h\in L^1(\ambient,\meas)\cap L^2(\ambient,\meas)$,
\begin{equation}\label{e.onepol}
 \meas\{|w^{-\frac12}\riesz_2h|>\lambda\}
 \leq
 \frac2\lambda\|h\|_{L^1(\ambient,\meas)},\ \text{ for all }\ \lambda\in(0,\infty),
\end{equation}
and for every $p\in(1,2]$, there exists $C_{p}\in(0,\infty)$ such that, for all $h\in L^p(\ambient,\meas)\cap L^2(\ambient,\meas)$,
\begin{equation}
	\norm{w^{-\frac12}\riesz_2h}_{L^p(\ambient,\meas)}\leq C_{p}\norm{h}_{L^p(\ambient,\meas)}.
\end{equation}
This, combined with Lemma~\ref{l.singul}-\ref{it.df=1}, proves~\eqref{e.onerie} and~\eqref{e.oneri1} for $p\in(1,2]$. The estimate \eqref{e.oneri1} for $p\in(2,\infty)$ follows from $\AR{1}$, \eqref{e.HKE}, \cite[Theorems 5.1 and 5.2]{BCY25} and \cite[Theorem 1.3-(ii)]{BCC26}.
\end{proof}

\subsection{Failure of the strong \texorpdfstring{$L^1$}{L\^ 1} estimates}
In the following lemma, we use a geodesic ray to construct two families of functions, one for the Riesz inequality and one for the reverse Riesz inequality.

\begin{lemma}\label{l.onefai}
 Assume that $(\ambient,\metric)$ is $\AR{1}$. There exist $c,C\in(0,\infty)$ such that, for every $L,R\in[2,\infty)$, we can find functions $h_L\in\collectD_1^{\rm sp}$ and $f_R\in\core$ such that
  \begin{equation}\label{e.onefa1}
    \begin{gathered}
      \|h_L\|_{L^1(\ambient,\meas)}\leq4,\qquad
      \|\riesz_1h_L\|_{L^1(\skeleton,\medm)}\geq c\log L-C,\\
      \|\wgrad f_R\|_{L^1(\skeleton,\medm)}=2,\qquad
      \|(-\gen)^{\frac12}P_1f_R\|_{L^1(\ambient,\meas)}\geq c\log R-C.
    \end{gathered}
  \end{equation}
\end{lemma}
\begin{proof}
Fix $o\in\ambient$ and let $\xi:[0,\infty)\to\ambient$ be a unit-speed
geodesic ray from $o$, as in Lemma~\ref{l.ends}. Fix $L\in[2,\infty)$ and define
\begin{equation}
 \varphi:=\frac{\one_{B(o,1)}}{\meas(B(o,1))},\ \
 h_L:=(I-P_{L^2})^2P_1\varphi,\ \text{ and }\ u_L:=(-\gen)^{-\frac{1}{2}}h_L. 
\end{equation}
Since $\varphi\in L^1(\ambient,\meas)\cap L^2(\ambient,\meas)$ and
$\|\varphi\|_{L^1(\ambient,\meas)}=1$, \eqref{e.specc-strict} gives
$h_L\in\collectD_1^{\rm sp}$. By $L^1$-contractivity,
\begin{equation}
 \|h_L\|_{L^1(\ambient,\meas)}
 =\|P_1\varphi-2P_{1+L^2}\varphi+P_{1+2L^2}\varphi\|_{L^1(\ambient,\meas)}
 \leq 4. 
\end{equation}
Here $\dimf=1$, $\dimw=2$ and $\critic_1=1/2$. Hence
\eqref{e.strict-core-domains} and \eqref{e.rieszp} give
$u_L\in\domain$ and $\riesz_1h_L=\wgrad u_L$. We first show that 
\begin{equation}\label{e.ray-potential}
 u_L(x)=\frac{1}{\sqrt{\pi}}\int_0^\infty
 (I-P_{L^2})^2P_{1+s}\varphi(x)\frac{\dif s}{\sqrt{s}},
 \ \text{ for all } x\in\ambient.
\end{equation}
By differentiating the heat kernel in the integral, we obtain
\begin{equation}
 \left|\frac{\dif}{\dif t}P_t\varphi(x)\right|
 \overset{\eqref{e.heatke},\,\eqref{e.ddt}}{\leq}
 \int_{B(o,1)}C t^{-\frac{3}{2}}\varphi(y)\dif\meas(y)
 =C t^{-\frac{3}{2}},\qquad t\in(0,\infty). 
\end{equation}
By the fundamental theorem of calculus,
\begin{align}
 \left|(I-P_{L^2})^2P_{1+s}\varphi(x)\right|\leq\sum_{j=0}^{1}
 \int_{1+s+jL^2}^{1+s+(j+1)L^2}
 \left|\frac{\dif}{\dif t}P_t\varphi(x)\right|\dif t \leq 2CL^2(1+s)^{-\frac{3}{2}}. \label{e.ray-cancel-bound}
\end{align}
For $0\leq\varepsilon<1<M\leq\infty$, put
\begin{equation}
 U_{\varepsilon,M}(x):=\frac{1}{\sqrt{\pi}}\int_\varepsilon^M
 (I-P_{L^2})^2P_{1+s}\varphi(x)\,\frac{\dif s}{\sqrt{s}},
 \qquad x\in\ambient. 
\end{equation}
Let $U:=U_{0,\infty}$. We use the heat kernel representatives of the three terms in
$(I-P_{L^2})^2P_{1+s}\varphi$. Since $\int_0^\infty s^{-\frac12}(1+s)^{-\frac32}\dif s<\infty$, the integral in \eqref{e.ray-potential} converges absolutely and uniformly in $x$. Then
\begin{align}
 \sup_{x\in\ambient}|U(x)-U_{\varepsilon,M}(x)|
 &\overset{\eqref{e.ray-cancel-bound}}{\leq}
 CL^2\left(\int_0^\varepsilon\frac{\dif s}{\sqrt{s}}
            +\int_M^\infty\frac{\dif s}{s^2}\right) \to0
 \qquad\text{as }\varepsilon\downarrow0,\ M\uparrow\infty. 
\end{align}
By the dominated convergence theorem, $U_{\varepsilon,M}\in C(\ambient)$ for all $0<\varepsilon<1<M<\infty$. Thus $U\in C(\ambient)$ by the uniform convergence.
We next identify $U_{\varepsilon,M}$ in $L^2(\ambient,\meas)$. The semigroup
property and the definition of $h_L$ give
$(I-P_{L^2})^2P_{1+s}\varphi=P_sh_L$. Applying Cauchy--Schwarz inequality, Tonelli's theorem and $L^2$-contractivity, we have
\begin{align}
 \|U_{\varepsilon,M}\|_{L^2(\ambient,\meas)}^2
 &\leq\frac{1}{\pi}\left(\int_\varepsilon^M\frac{\dif s}{\sqrt{s}}\right)
 \int_\varepsilon^M\|P_sh_L\|_{L^2(\ambient,\meas)}^2\frac{\dif s}{\sqrt{s}}
 \\
 &\leq\frac{4}{\pi}(\sqrt{M}-\sqrt{\varepsilon})^2
 \|h_L\|_{L^2(\ambient,\meas)}^2. 
\end{align}
Moreover, for every $v\in L^2(\ambient,\meas)$,
\begin{align}
 \int_\varepsilon^M\int_{\ambient}|P_sh_L(x)v(x)|\dif\meas(x)
        \frac{\dif s}{\sqrt{s}}
 &\leq\int_\varepsilon^M
 \|P_sh_L\|_{L^2(\ambient,\meas)}\|v\|_{L^2(\ambient,\meas)}
        \frac{\dif s}{\sqrt{s}} \\
 &\leq2(\sqrt M-\sqrt\varepsilon)
 \|h_L\|_{L^2(\ambient,\meas)}\|v\|_{L^2(\ambient,\meas)}. 
\end{align}
Fubini's theorem therefore gives
\begin{equation}
 \langle U_{\varepsilon,M},v\rangle
 =\frac{1}{\sqrt{\pi}}\int_\varepsilon^M
 \langle P_sh_L,v\rangle\,\frac{\dif s}{\sqrt{s}}. 
\end{equation}
By \eqref{e.spec},
\begin{equation}
 \langle P_sh_L,v\rangle
 =\int_{[0,\infty)}e^{-s\lambda}(1-e^{-L^2\lambda})^2e^{-\lambda}
       \dif\langle\proj_\lambda\varphi,v\rangle. 
\end{equation}
Since $0\leq e^{-s\lambda}(1-e^{-L^2\lambda})^2e^{-\lambda}\leq1$ and
$\int_\varepsilon^M s^{-1/2}\dif s<\infty$, Fubini's theorem gives
\begin{align}
 \langle U_{\varepsilon,M},v\rangle
 =\int_{[0,\infty)}m_{\varepsilon,M}(\lambda)
                 \dif\langle\proj_\lambda\varphi,v\rangle, \ \text{ with }\
 m_{\varepsilon,M}(\lambda)
 :=\frac{(1-e^{-L^2\lambda})^2e^{-\lambda}}{\sqrt{\pi}}
     \int_\varepsilon^M e^{-s\lambda}\frac{\dif s}{\sqrt{s}}. 
\end{align}
As this holds for every $v\in L^2(\ambient,\meas)$, we have
$U_{\varepsilon,M}=m_{\varepsilon,M}(-\gen)\varphi$ in $L^2(\ambient,\meas)$.
Let 
\begin{equation}
 m_L(\lambda):=\frac{(1-e^{-L^2\lambda})^2e^{-\lambda}}{\sqrt\lambda}
 \quad(\lambda>0),\qquad m_L(0):=0. 
\end{equation}
Then $m_{\varepsilon,M}(\lambda)\to m_L(\lambda)$. Since $1-e^{-L^2\lambda}\leq\min\{1,L^2\lambda\}$, we have
\begin{equation}
 0\leq m_{\varepsilon,M}(\lambda)\leq m_L(\lambda)\leq
 \begin{cases}
  L^4\lambda^{3/2}\leq L,&0<\lambda\leq L^{-2},\\
  \lambda^{-1/2}\leq L,&\lambda>L^{-2}.
 \end{cases} 
\end{equation}
By the definitions of $h_L$ and $u_L$, $u_L=m_L(-\gen)\varphi$.
The positive measure $\dif\langle\proj_\lambda\varphi,\varphi\rangle$
has total mass $\|\varphi\|_{L^2(\ambient,\meas)}^2$. Hence, by the dominated convergence theorem,
\begin{equation}
 \|U_{\varepsilon,M}-u_L\|_{L^2(\ambient,\meas)}^2
 \overset{\eqref{e.spec}}{=}\int_{[0,\infty)}
 |m_{\varepsilon,M}(\lambda)-m_L(\lambda)|^2
 \dif\langle\proj_\lambda\varphi,\varphi\rangle
 \to0, 
\end{equation}
Finally, fix $R>0$. Since $\meas(B(o,R))<\infty$, we have
\begin{align}
 \|U-u_L\|_{L^2(B(o,R),\meas)}
 &\leq\meas(B(o,R))^{1/2}
       \sup_{x\in\ambient}|U(x)-U_{\varepsilon,M}(x)|
       +\|U_{\varepsilon,M}-u_L\|_{L^2(\ambient,\meas)} \\
 &\to0
 \qquad\text{as }\varepsilon\downarrow0,\ M\uparrow\infty. 
\end{align}
Taking $R\in\bN$ gives $U=u_L$ $\meas$-a.e.\ on $\ambient$. Since $U$
and  $u_L\in\domain$ are continuous, we obtain \eqref{e.ray-potential}.

For $s\geq0$, let $\delta_s:=P_{1+s}\varphi(o)-P_{1+s}\varphi(\xi(L))$.
If $y\in B(o,1)$, then $\metric(o,y)<1$ and
$\metric(\xi(L),y)\geq L-1\geq L/2$. Integrating the two bounds in
\eqref{e.HKE} against $\varphi$ gives
\begin{equation}\label{e.onegau}
 P_{1+s}\varphi(o)\geq\frac{c_0}{\sqrt{1+s}},\qquad
 P_{1+s}\varphi(\xi(L))\leq\frac{C_1}{\sqrt{1+s}}
 \exp\left(-\frac{c_2L^2}{4(1+s)}\right).
\end{equation}
Choose $\alpha\in(0,1)$ such that $C_1e^{-c_2/(4\alpha)}\leq c_0/2$.
Suppose first that $\alpha L^2\geq2$, and put $T:=\alpha L^2-1\geq1$.
For $0\leq s\leq T$, \eqref{e.onegau} gives
$\delta_s\geq c_0/(2\sqrt{1+s})$. Also, for $s\geq0$ and $k\in\{1,2\}$,
\begin{equation}\label{e.ray-shift}
 |\delta_{s+kL^2}|\overset{\eqref{e.HKE}}{\leq}
 \frac{2C_1}{\sqrt{1+s+kL^2}}\leq\frac{2C_1}{L}.
\end{equation}
It follows that
\begin{align}
 \int_0^T\bigl(\delta_s-2\delta_{s+L^2}+\delta_{s+2L^2}\bigr)
          \frac{\dif s}{\sqrt{s}}
 &\overset{\eqref{e.ray-shift}}{\geq}
 \frac{c_0}{2}\int_1^T\frac{\dif s}{1+s}-\frac{12C_1\sqrt{T}}{L} \\
 &\geq\frac{c_0}{2}\log\left(\frac{\alpha L^2}{2}\right)-12C_1\sqrt{\alpha}
 \geq c\log L-C. \label{e.ray-small}
\end{align}
For $s\geq T$, apply \eqref{e.ray-cancel-bound} at $o$ and $\xi(L)$. Since
$T=\alpha L^2-1\geq\alpha L^2/2$, we obtain
\begin{align}
 \int_T^\infty\left|\delta_s-2\delta_{s+L^2}+\delta_{s+2L^2}\right|
          \frac{\dif s}{\sqrt{s}}
 &\leq CL^2\int_T^\infty\frac{\dif s}{\sqrt{s}(1+s)^{3/2}} \\
 &\leq CL^2\int_T^\infty\frac{\dif s}{s^2}
 =\frac{CL^2}{T}\leq\frac{2C}{\alpha}. \label{e.ray-tail}
\end{align}
Using $\riesz_1h_L=\wgrad u_L$ and integrating along $[o,\xi(L)]$, we conclude that
\begin{align}
 \|\riesz_1h_L\|_{L^1(\skeleton,\medm)}
 &\overset{\eqref{e.arcrev},\,\eqref{e.arcftc}}{\geq}
 |u_L(o)-u_L(\xi(L))|\geq u_L(o)-u_L(\xi(L)) \\
 &\overset{\eqref{e.ray-potential}}{=}
 \frac{1}{\sqrt{\pi}}\int_0^\infty
 \bigl(\delta_s-2\delta_{s+L^2}+\delta_{s+2L^2}\bigr)\frac{\dif s}{\sqrt{s}}
 \overset{\eqref{e.ray-small},\,\eqref{e.ray-tail}}{\geq} c\log L-C. 
\end{align}
For $2\leq L<\sqrt{2/\alpha}$, the same norm estimate follows by increasing $C$
so that $c\log L-C\leq0$. This proves the first line of \eqref{e.onefa1}.

We then show that $\pi_{T_o}^{-1}(K)$ is compact for every compact $K\subset T_o$.
  Choose $R_K>0$ with $K\subset B(o,R_K)$. Every component $C$ of $\ambient\setminus T_o$ has compact closure by Lemma~\ref{l.ends}, and a gate $g_C\in T_o$ by Lemma~\ref{l.projt}-\ref{it.proj5}.
  If $g_C\in K$ and $C$ contains a point at distance greater than $1$ from $g_C$, then we may choose $z_C\in C$ with $\metric(g_C,z_C)=1$.
  The arcs $(g_C,z_C)$, for all such components $C$ of $\ambient\setminus T_o$, are disjoint and lie in $B(o,R_K+2)$. Consequently,
  \begin{align}
   &\phantom{\ \leq} \#\Sett{C}{
   \begin{minipage}{260pt}
   $C$ is a component of $\ambient\setminus T_o$, $g_C\in K$, and $C$ contains a point at distance greater than $1$ from $g_C$
   \end{minipage}
   }\\
   &=\sum_{C}\medm((g_C,z_C))
    \leq\medm(B(o,R_K+2))
    \overset{\eqref{e.onelen}}{\leq}C(R_K+2).
  \end{align}
  So $\pi_{T_o}^{-1}(K)$ is bounded. It is also closed because $\pi_{T_o}$ is continuous, and hence compact by properness.

  By Lemma~\ref{l.ends}, there are finitely many rays from $o$. Choose $a_0\in[0,\infty)$ such that
  \begin{equation}\label{e.oneray}
    \metric(\xi(t),w)\geq t-a_0,
    \qquad t\geq a_0,\quad
    w\in T_o\setminus\xi((a_0,\infty)).
  \end{equation}
  For $R\in[2,\infty)$, put $a:=a_0+2R$ and let $f_R:\ambient\to\bR$ be defined by
  \begin{equation}\label{e.one-plateau-gradient}
  f_R(x):=\int_{\skeleton}\vecone_{(o,x)} \left(\vecone_{(\xi(a),\xi(a+1))}       -\vecone_{(\xi(a+R+1),\xi(a+R+2))}\right)\dif\medm,\ \text{ for }x\in\ambient.
  \end{equation}
Then 
  \begin{equation}
    f_R(\xi(t))=(t-a)\one_{(a,a+1)}(t)+\one_{[a+1,a+R+1]}(t)+(a+R+2-t)\one_{(a+R+1,a+R+2)},\ \text{ for }t\in[0,\infty).
  \end{equation}
  By Lemma \ref{l.primit}-\ref{it.prim4}, we know that $0\leq f_R\leq1$, $f_R=f_R\circ\pi_{T_o}$ and $\supp(f_R)\subset\pi_{T_o}^{-1}(\xi([a,a+R+2]))$;
  in particular, $f_R\in C_{c}(\ambient)\cap \Lip(\ambient,\metric)$. Hence $f_R\in\core$ and $ \|\wgrad f_R\|_{L^1(\skeleton,\medm)}
    \overset{\eqref{e.one-plateau-gradient}}{=}2$.

  Fix $s\in[2,R]$ and fix $x:=\xi(a-s)$ and $z:=\xi(a+1+s/2)$.
  If $f_R(y)\neq0$, its projection onto $T_o$ equals $\xi(t)$ for some $t>a$, so
  \begin{equation}
    \metric(x,y)\geq\metric(x,\pi_{T_o}(y))=t-a+s>s.
  \end{equation}
  Thus $f_R=0$ on $B(x,s)$. If $y\in B(z,s/2)$, the contraction property of $\pi_{T_o}$ and~\eqref{e.oneray} give
  \begin{equation}
    \pi_{T_o}(y)\in\xi((a+1,a+1+s)),\quad
    f_R(y)=1,\ \text{ and }\
    \metric(x,y)<1+2s\leq4s.
  \end{equation}
  Consequently, if $s^2\leq t\leq2s^2$, then
  \begin{align}\label{e.onepla}
    P_{1+t}f_R(x)
    \geq\int_{B(z,s/2)}p_{1+t}(x,y)\dif\meas(y)\overset{\eqref{e.HKE}}{\geq}cs^{-1}\meas(B(z,s/2))
      \overset{\eqref{e.ahlfor}}{\geq}c.
  \end{align}
  Since $f_R=0$ on $B(x,s)$, \eqref{e.HKE}, \eqref{e.ddt} and $\AR{1}$ imply that
  \begin{align}
  &\phantom{\ \leq} \sup_{\tau\in[1,2]} \left(P_\tau f_R(x)+\abs{\frac{\dif}{\dif\tau} P_\tau f_R(x)}\right)
    \overset{\eqref{e.HKE},\eqref{e.ddt}}{\leq}
       C\sum_{k=0}^\infty e^{-c4^ks^2}\meas(B(x,2^{k+1}s))\\
    &\overset{\eqref{e.ahlfor}}{\leq}
       C\sum_{k=0}^\infty2^ks e^{-c4^ks^2}
       \leq Ce^{-cs^2/2}\sum_{k=0}^\infty e^{-c(4^k-1)/2}
       \leq C'e^{-c's^2}.
  \end{align}
Thus
  \begin{equation}\label{e.onetai}
    0\leq P_1f_R(x)\leq Ce^{-cs^2},\ \text{and}\
    |P_{1+t}f_R(x)-P_1f_R(x)|\leq Cte^{-cs^2},
    \text{ for all }t\in(0,1].
  \end{equation}
  By~\eqref{e.ddt},
  \begin{align}
    \abs{\frac{\dif}{\dif t} P_{1+t}f_R(x)}
    \leq C(1+t)^{-3/2}\|f_R\|_{L^1(\ambient,\meas)},
  \end{align}
  which implies 
  \begin{equation}
  	|P_1f_R(x)-P_{1+t}f_R(x)|
    \leq\min\{Ct\|f_R\|_{L^1(\ambient,\meas)},2\}.
  \end{equation} 
  Also,
  \begin{equation}
    \int_0^\infty t^{-\frac12}\int_\ambient
       \abs{\frac{\dif}{\dif t} p_{1+t}(x,y)}f_R(y)\dif\meas(y)\dif t
    \overset{\eqref{e.ddt}}{\leq}
       C\|f_R\|_{L^1(\ambient,\meas)}
       \int_0^\infty t^{-\frac12}(1+t)^{-3/2}\dif t<\infty.
  \end{equation}
  Therefore using \eqref{e.qdef},~\eqref{e.fracin} and integration by parts, we have 
    \begin{align}
    (-\gen)^{\frac12}P_1f_R(x)
    &\overset{\eqref{e.qdef},\eqref{e.fracin}}{=}
      -\frac1{\sqrt\pi}\int_0^\infty t^{-\frac12}\frac{\dif}{\dif t}P_{1+t}f_R(x)\dif t\\
    &=
      \frac1{2\sqrt\pi}\int_0^\infty
        (P_1f_R(x)-P_{1+t}f_R(x))\frac{\dif t}{t^{3/2}}.\label{e.onefra}
  \end{align}

  At $x=\xi(a-s)$, split~\eqref{e.onefra} at $t=1$ and retain the negative term on $[s^2,2s^2]$. Then
  \begin{align}
    &\phantom{\ \leq}(-\gen)^{\frac12}P_1f_R(\xi(a-s))\\
    &\overset{\eqref{e.onefra},\eqref{e.onetai},\eqref{e.onepla}}{\leq}
       Ce^{-cs^2}\left(\int_0^1t^{-\frac12}\dif t
                    +\int_1^\infty t^{-3/2}\dif t\right)
       -c\int_{s^2}^{2s^2}t^{-3/2}\dif t\\
    &\leq
       Ce^{-cs^2}-2c(1-2^{-\frac{1}{2}})s^{-1}.
  \end{align}
  Choose $s_0\geq2$ such that $Ce^{-cs^2}\leq c(1-2^{-\frac{1}{2}})s^{-1}$  for all $s\geq s_0$. Hence, when $R\geq s_0$, by Lemma~\ref{l.singul}-\ref{it.df=1}
  \begin{equation}
    \|(-\gen)^{\frac12}P_1f_R\|_{L^1(\ambient,\meas)}
    \geq
       c\int_{s_0}^R\frac{\dif s}{s}
    =c\log R-c\log s_0.
  \end{equation}
  By adjusting the constant, this also holds for $R\in[2,s_0)$. This completes~\eqref{e.onefa1}.
\end{proof}
\begin{proof}[Proof of Theorem~\ref{t.one}]
Since $\dimf=1$, \eqref{e.critp} gives $\dimw=2$ and $\critic_p=1/2$ for every $p\in[1,\infty)$.
\begin{enumerate}[label=\textup{({\arabic*})}, align=right, leftmargin=*, topsep=5pt, parsep=0pt, itemsep=2pt]
  \item[\ref{it.onestr}] Let $p\in(1,\infty)$. By~\eqref{e.oneri1}, the restriction of $\riesz_2$ to
$L^p(\ambient,\meas)\cap L^2(\ambient,\meas)$ extends uniquely to a bounded
operator from $L^p(\ambient,\meas)$ to $L^p(\skeleton,\medm)$.
On $\collectD_p^{\rm sp}$, it agrees with
$\wgrad(-\gen)^{-\frac{1}{2}}$ by~\eqref{e.strict-core-domains}
and~\eqref{e.riesz2}. Since $\collectD_p^{\rm sp}$ is dense in $L^p(\ambient,\meas)$
by Lemma~\ref{l.densp}, the extension is precisely $\riesz_p$ and is unique.
In particular,
\begin{equation}
 \riesz_ph=\riesz_2h,\qquad
 h\in L^p(\ambient,\meas)\cap L^2(\ambient,\meas).
\end{equation}

Let $p'=p/(p-1)$. For $f\in\core$, Corollary~\ref{c.coreco} and
$\core\subset\Dom((-\gen)^{\frac{1}{2}})$ give
$(-\gen)^{\frac{1}{2}}f\in L^p(\ambient,\meas)\cap L^2(\ambient,\meas)$.
Moreover, \eqref{e.riesz2} and \eqref{e.BTheta.adjoint1} imply
\begin{equation}\label{e.onepai}
 (-\gen)^{\frac{1}{2}}f=\riesz_2^*\wgrad f,
 \ \text{ for all }\ f\in\core.
\end{equation}
Consequently, for $h\in L^{p'}(\ambient,\meas)\cap L^2(\ambient,\meas)$,
\begin{align}
 \left|\int_\ambient(-\gen)^{\frac{1}{2}}f\,h\dif\meas\right|
 \overset{\eqref{e.onepai}}{=}
 \left|\int_\skeleton\wgrad f\,\riesz_2h\dif\medm\right|\overset{\eqref{e.oneri1}}{\leq}
 C_p\|\wgrad f\|_{L^p(\skeleton,\medm)}
 \|h\|_{L^{p'}(\ambient,\meas)}.
\end{align}
Taking the supremum over $\|h\|_{L^{p'}(\ambient,\meas)}\leq1$, and using
the density of $L^{p'}(\ambient,\meas)\cap L^2(\ambient,\meas)$ in
$L^{p'}(\ambient,\meas)$, gives
\begin{equation}\label{e.onere1}
 \|(-\gen)^{\frac{1}{2}}f\|_{L^p(\ambient,\meas)}
 \leq
 C_p\|\wgrad f\|_{L^p(\skeleton,\medm)}, \ \text{ for all }\  f\in\core.
\end{equation}
The opposite estimate follows from
\begin{equation}
 \wgrad f\overset{\eqref{e.riesz2}}{=}
 \riesz_p\bigl((-\gen)^{\frac{1}{2}}f\bigr), \ \text{ and }\ 
 \|\wgrad f\|_{L^p(\skeleton,\medm)}
 \overset{\eqref{e.oneri1}}{\leq}
 C_p\|(-\gen)^{\frac{1}{2}}f\|_{L^p(\ambient,\meas)}.
\end{equation}

Now let $f\in\sobolev{p}(\ambient)$. Choose $f_n\in\core$ as in
Lemma~\ref{l.dense}, so that
\begin{equation}
\lim_{n\to\infty}( \|f_n-f\|_{L^p(\ambient,\meas)}
 +\|\wgrad f_n-\wgrad f\|_{L^p(\skeleton,\medm)})
 =0.
\end{equation}
Equation~\eqref{e.onere1}, applied to $f_n-f_m\in\core$, shows that
\begin{equation}
 \|(-\gen)^{\frac{1}{2}}f_n-(-\gen)^{\frac{1}{2}}f_m\|_{L^p(\ambient,\meas)}
 \overset{\eqref{e.onere1}}{\leq}
 C_p\|\wgrad f_n-\wgrad f_m\|_{L^p(\skeleton,\medm)}
 \to0.
\end{equation}
Let $G_f$ be the $L^p(\ambient,\meas)$ limit of $\{(-\gen)^{\frac{1}{2}}f_n\}_{n\in\bN}$. Then
\begin{equation}
 \|G_f\|_{L^p(\ambient,\meas)}
 \overset{\eqref{e.onere1}}{\leq}C_p\|\wgrad f\|_{L^p(\skeleton,\medm)},
\ \text{ and }\ 
 \wgrad f=\riesz_pG_f\quad\medm\text{-a.e. on }\skeleton.
\end{equation}
Thus~\eqref{e.oneri1} also gives
$\|\wgrad f\|_{L^p(\skeleton,\medm)}\leq C_p\|G_f\|_{L^p(\ambient,\meas)}$. This proves \eqref{e.onestr}.

For every $t\in(0,\infty)$, \eqref{e.analy} and the $L^p$-contractivity give
\begin{align}
 &\phantom{\ \leq}\|(-\gen)^{\frac{1}{2}}P_tf-P_tG_f\|_{L^p(\ambient,\meas)}\\
 &\overset{\eqref{e.spec}}{\leq}
 \|(-\gen)^{\frac{1}{2}}P_t(f-f_n)\|_{L^p(\ambient,\meas)}
 +\|P_t((-\gen)^{\frac{1}{2}}f_n-G_f)\|_{L^p(\ambient,\meas)}\\
 &\overset{\eqref{e.analy}}{\leq}
 C_pt^{-\frac{1}{2}}\|f-f_n\|_{L^p(\ambient,\meas)}
 +\|(-\gen)^{\frac{1}{2}}f_n-G_f\|_{L^p(\ambient,\meas)}
 \to0.
\end{align}
Hence $(-\gen)^{\frac{1}{2}}P_tf=P_tG_f$, and strong continuity of $\{P_{t}\}$ in $L^{p}(\ambient,\meas)$ gives
\begin{equation}
 \lim_{t\downarrow0}\|(-\gen)^{\frac{1}{2}}P_tf-G_f\|_{L^p(\ambient,\meas)}
 =\lim_{t\downarrow0}\|P_tG_f-G_f\|_{L^p(\ambient,\meas)}=0.
\end{equation}
If also $f\in\Dom((-\gen)^{\frac{1}{2}})$, then
\begin{equation}\label{e.one-spectral-limit}
 \lim_{t\downarrow0} \|(-\gen)^{\frac{1}{2}}P_tf-(-\gen)^{\frac{1}{2}}f\|_{L^2(\ambient,\meas)}
 \overset{\eqref{e.spec}}{=}
  \lim_{t\downarrow0}\|(P_t-I)(-\gen)^{\frac{1}{2}}f\|_{L^2(\ambient,\meas)}=0.
\end{equation}
Therefore $G_f=(-\gen)^{\frac{1}{2}}f$ $\meas$-a.e. on $\ambient$. This proves~\ref{it.onestr}.
\item[\ref{it.onewea}]By~\eqref{e.onerie}, $\riesz_2$ can be extended to a bounded linear operator from $L^1(\ambient,\meas)$ to $L^{1,\infty}(\skeleton,\medm)$. By \eqref{e.strict-core-domains} and~\eqref{e.riesz2}, $\restr{\riesz_2}{\collectD_1^{\rm sp}}=\restr{\riesz_1}{\collectD_1^{\rm sp}}$. This proves the first assertion and~\eqref{e.onefor}. For $f\in\sobolev{1}(\ambient)$, take $f_n\in\core$ converging to $f$ in $\sobolev{1}(\ambient)$. For each fixed $t\in(0,\infty)$,
\begin{equation}
 \|(-\gen)^{\frac{1}{2}}P_t(f_n-f)\|_{L^1(\ambient,\meas)}
 \overset{\eqref{e.analy}}{\leq}
 Ct^{-\frac{1}{2}}\|f_n-f\|_{L^1(\ambient,\meas)}\to0.
\end{equation}
Pass to an $\meas$-a.e. convergent subsequence and use Fatou's lemma in \eqref{e.homfatou}. Then
\begin{align}
 \|(-\gen)^{\frac{1}{2}}P_tf\|_{L^{1,\infty}(\ambient,\meas)}
 &\overset{\eqref{e.homfatou}}{\leq}
 \liminf_{n\to\infty}
 \|(-\gen)^{\frac{1}{2}}P_tf_n\|_{L^{1,\infty}(\ambient,\meas)}\\
 &\overset{\eqref{e.oneczc}}{\leq}
 C\lim_{n\to\infty}\|\wgrad f_n\|_{L^1(\skeleton,\medm)}
 =C\|\wgrad f\|_{L^1(\skeleton,\medm)}.
\end{align}
The constant is independent of $t$, proving~\eqref{e.onerev}.

Corollary~\ref{c.coreco} gives convergence of
$(-\gen)^{\frac{1}{2}}P_th$ in $L^1(\ambient,\meas)$ for every $h\in\core$.
Given $\varepsilon>0$, choose $h\in\core$ with
$\|\wgrad(f-h)\|_{L^1(\skeleton,\medm)}<\varepsilon$ by
Lemma~\ref{l.dense}. Applying~\eqref{e.homquasi} to the three terms in
\begin{equation}
 (-\gen)^{\frac{1}{2}}(P_t-P_s)f
 =(-\gen)^{\frac{1}{2}}P_t(f-h)
 +(-\gen)^{\frac{1}{2}}(P_t-P_s)h
 -(-\gen)^{\frac{1}{2}}P_s(f-h)
\end{equation}
gives
\begin{equation}
 \|(-\gen)^{\frac{1}{2}}(P_t-P_s)f\|_{L^{1,\infty}(\ambient,\meas)}\overset{\eqref{e.homquasi},\eqref{e.onerev}}{\leq}
 6C\varepsilon
 +3\|(-\gen)^{\frac{1}{2}}(P_t-P_s)h\|_{L^1(\ambient,\meas)}.
\end{equation}
Letting $s,t\downarrow0$ and then $\varepsilon\downarrow0$ proves that $\{(-\gen)^{\frac{1}{2}}P_{t}f\}_{t\in(0,\infty)}$ is a Cauchy sequence in $L^{1,\infty}(\ambient,\meas)$ as $t\downarrow0$. Completeness of $L^{1,\infty}(\ambient,\meas)$ \cite[Theorem 1.4.11]{Gra14} and \eqref{e.homquasi} give the convergence of $(-\gen)^{\frac{1}{2}}P_{t}f$ as $t\downarrow0$, exactly as in the proof of Theorem~\ref{t.mainw}. This proves~\ref{it.onewea}.
\item[\ref{it.onefai}]By Lemma~\ref{l.onefai}, for each $L,R\in[2,\infty)$, there exist $h_L\in\collectD_1^{\rm sp}$ and $f_R\in\core$ such that $\|h_L/4\|_{L^1(\ambient,\meas)}\leq1$ and $\|\wgrad(f_R/2)\|_{L^1(\skeleton,\medm)}=1$, while
\begin{equation}
\|\riesz_1(h_L/4)\|_{L^1(\skeleton,\medm)}
 \overset{\eqref{e.onefa1}}{\geq}\frac{c\log L-C}{4},\ \text{ and }\ \|(-\gen)^{\frac{1}{2}}P_1(f_R/2)\|_{L^1(\ambient,\meas)}
 \overset{\eqref{e.onefa1}}{\geq}\frac{c\log R-C}{2}.
\end{equation}
Letting $R,L\uparrow\infty$, we obtain \eqref{e.onefai}.
\qedhere
\end{enumerate}
\end{proof}

\appendix
\section{Useful facts}\label{s.facts}
\begin{definition}\label{d.lorent}
  Let $(\ambient,\meas)$ be a $\sigma$-finite measure space. For every measurable function $h:\ambient\to[-\infty,\infty]$, let
  \begin{equation}\label{e.distri}
    \lambda_{h}(\alpha) := \meas\bigl(\Sett{x\in\ambient}{|h(x)|>\alpha}\bigr), \ \  \alpha\in(0,\infty).
  \end{equation}
  \begin{enumerate}[label=\textup{(\arabic*)}, align=right, leftmargin=*, topsep=5pt, parsep=0pt, itemsep=2pt]
    \item For $p\in [ 1,\infty)$ and $q\in[1,\infty]$, define 
    \begin{equation}\label{e.lorpq}
      \norm{h}_{L^{p,q}(\ambient,\meas)} :=
      \begin{cases}
        \displaystyle \left( p\int_{0}^{\infty} \alpha^{q-1}\lambda_{h}(\alpha)^{q/p}\dif\alpha \right)^{1/q}, & q\in[1,\infty), \\[3mm]
        \displaystyle \sup_{\alpha\in(0,\infty)} \alpha \lambda_{h}(\alpha)^{1/p}, & q=\infty.
      \end{cases}
    \end{equation}
     Let $L^{p,q}(\ambient,\meas)$ be the space of $\meas$-a.e. equivalence classes of measurable functions $h$ such that $\norm{h}_{L^{p,q}(\ambient,\meas)}<\infty$.
    \item For $p\in(1,\infty)$ and $q=\infty$, define
    \begin{equation}\label{e.lorno}
      \|h\|_{(p,\infty;\ambient,\meas)} :=\sup\Sett{\meas(E)^{1/p-1} \int_{E}|h|\dif\meas}{
      \begin{minipage}{120pt}
      	$E\subset \ambient$ is measurable and $\meas(E)\in(0,\infty)$
      \end{minipage}
      }.
    \end{equation}
  \end{enumerate}
\end{definition}
\begin{lemma}\label{l.loren}
  Let $(\ambient,\meas)$ be a $\sigma$-finite measure space. Let $p\in(1,\infty)$ and $q\in[1,\infty]$. Then:
  \begin{enumerate}[label=\textup{(\arabic*)}, align=right, leftmargin=*, topsep=5pt, parsep=0pt, itemsep=2pt]
    \item\label{it.loren1} The functional $\norm{\cdot}_{L^{p,q}}$ is a quasi-norm. If $1\leq q\leq p$, then it is a norm and $L^{p,q}(\ambient,\meas)$ is a Banach space. If $p<q\leq\infty$, then it is equivalent to a Banach norm. Moreover,
    \begin{equation}\label{e.lorchain}
      L^{p,1}(\ambient,\meas) \subset L^{p}(\ambient,\meas) = L^{p,p}(\ambient,\meas) \subset L^{p,\infty}(\ambient,\meas)
    \end{equation}
    continuously. If $q<\infty$, the simple functions with support of finite measure are dense in $L^{p,q}(\ambient,\meas)$. Convergence in $L^{p,q}(\ambient,\meas)$ implies convergence in measure.
    \item\label{it.loren2} For every measurable function $h$,
    \begin{equation}\label{e.loreq}
      \norm{h}_{L^{p,\infty}(\ambient,\meas)} \leq \|h\|_{(p,\infty)} \leq \frac{p}{p-1} \norm{h}_{L^{p,\infty}(\ambient,\meas)}.
    \end{equation}
    In particular, $\|\cdot\|_{(p,\infty)}$ is a norm and $\bigl( L^{p,\infty}(\ambient,\meas), \|\cdot\|_{(p,\infty)} \bigr)$ is a Banach space.
    \item\label{it.loren3} Let
    \begin{equation}
      1<r_{0}<p<r_{1}<\infty, \ \text{ and }\ \frac{1}{p} = \frac{1-\theta}{r_{0}} + \frac{\theta}{r_{1}}, \ \text{ for some }\ \theta\in(0,1).
    \end{equation}
    Then, for every $h\in L^{r_{0}}(\ambient,\meas)\cap L^{r_{1}}(\ambient,\meas)$,
    \begin{equation}\label{e.lorinter}
      \norm{h}_{L^{p,1}(\ambient,\meas)} \leq C_{p,r_{0},r_{1}} \|h\|_{L^{r_{0}}(\ambient,\meas)}^{1-\theta} \|h\|_{L^{r_{1}}(\ambient,\meas)}^{\theta}.
    \end{equation}
    \item\label{it.loren4} Assume in addition that $(\ambient,\meas)$ is nonatomic. Let $p':=\frac{p}{p-1}$. Then, for every $h\in L^{p,1}(\ambient,\meas)$ and $g\in L^{p',\infty}(\ambient,\meas)$,
    \begin{equation}\label{e.lorholder}
      \int_{\ambient}|hg|\dif\meas \leq \norm{h}_{L^{p,1}(\ambient,\meas)} \norm{g}_{L^{p',\infty}(\ambient,\meas)}.
    \end{equation}
    More precisely, there exists $C_{p}\in(1,\infty)$ such that
    \begin{equation}\label{e.lordualnorm}
      C_{p}^{-1}\norm{g}_{L^{p',\infty}} \leq \sup_{\substack{h\in L^{p,1}\\ \norm{h}_{L^{p,1}}\leq1}} \left| \int_{\ambient}hg\dif\meas \right| \leq \norm{g}_{L^{p',\infty}}.
    \end{equation}
    Every continuous linear functional on $L^{p,1}(\ambient,\meas)$ is given by integration against a unique $g\in L^{p',\infty}(\ambient,\meas)$.
    \item\label{it.loren5} If $p>2$, $h\in L^{p,\infty}(\ambient,\meas)$, and $E\subset\ambient$ is measurable with $\meas(E)<\infty$, then
    \begin{equation}\label{e.localL2weak}
      \|h\|_{L^{2}(E,\meas)} \leq \left(\frac{p}{p-2}\right)^{\frac{1}{2}} \meas(E)^{\frac{1}{2}-\frac{1}{p}} \norm{h}_{L^{p,\infty}(\ambient,\meas)}.
    \end{equation}
  \end{enumerate}
\end{lemma}
\begin{proof}
  \begin{enumerate}[label=\textup{(\arabic*)}, align=right, leftmargin=*, topsep=5pt, parsep=0pt, itemsep=2pt]
    \item[\ref{it.loren1}] Let $h^*$ denote the decreasing rearrangement of $|h|$ with respect to $\meas$. For $1\leq q<\infty$, by \eqref{e.lorpq} and \cite[Proposition~1.4.9]{Gra14},
    \begin{equation}
      \norm{h}_{L^{p,q}(\ambient,\meas)}^q
   =
      \int_0^\infty h^*(t)^q t^{q/p-1}\dif t .
    \end{equation}
    If $1\leq q\leq p$, then $q/p-1\leq0$. Applying \cite[Theorem~2.5.1]{CRS07} with exponent $q$ and weight $t^{q/p-1}$, we obtain the triangle inequality for the functional $ \norm{\cdot}_{L^{p,q}(\ambient,\meas)}$. For all $1\leq q\leq\infty$, the quasi-triangle inequality and completeness are given by \cite[(1.4.9) and Theorem~1.4.11]{Gra14}. Thus $\norm{\cdot}_{L^{p,q}(\ambient,\meas)}$ is complete when $q\leq p$. When $p<q<\infty$, \cite[Exercise~1.4.3]{Gra14} gives an equivalent norm, which is complete by the same theorem. When $q=\infty$, part~\ref{it.loren2} below gives the equivalent norm~\eqref{e.lorno}.

    To identify $L^{p}(\ambient,\meas) = L^{p,p}(\ambient,\meas)$ in~\eqref{e.lorchain}, take $q=p$. Then
    \begin{equation}
      \norm{h}_{L^{p,p}(\ambient,\meas)}^p
      \overset{\eqref{e.lorpq}}{=}
      p\int_0^\infty \alpha^{p-1}\lambda_h(\alpha)\dif\alpha
      =
      \int_{\ambient}|h|^p\dif\meas .
    \end{equation}
    The two continuous inclusions in~\eqref{e.lorchain} follow from \cite[Proposition~1.4.10]{Gra14}. Density of simple functions for $q<\infty$ is given in \cite[Theorem~1.4.13]{Gra14}. Finally, if $h_n\to h$ in $L^{p,q}(\ambient,\meas)$, then for each $\varepsilon>0$,
    \begin{align}
      \meas(\{|h_n-h|>\varepsilon\})
      \overset{\eqref{e.lorpq}}{\leq}
        \varepsilon^{-p}\norm{h_n-h}_{L^{p,\infty}(\ambient,\meas)}^p  \leq C_{p,q}\varepsilon^{-p}\norm{h_n-h}_{L^{p,q}(\ambient,\meas)}^p
        \to0.
    \end{align}
    The last inequality is the same inclusion with second indices $(q,\infty)$ if $q<\infty$, and is equality with $C_{p,\infty}=1$ if $q=\infty$.

    \item[\ref{it.loren2}] By Kolmogorov's inequalities in \cite[Exercises~1.1.11-(a) and~1.1.12-(a),(b)]{Gra14}, we know that for every $0<r<p$ and every measurable $h$,
    \begin{equation}\label{e.lorKolmogorov}
      \norm{h}_{L^{p,\infty}(\ambient,\meas)}
      \leq \sup_{0<\meas(E)<\infty}
            \meas(E)^{1/p-1/r}\|h\|_{L^r(E,\meas)} \leq \left(\frac{p}{p-r}\right)^{1/r}
            \norm{h}_{L^{p,\infty}(\ambient,\meas)},
    \end{equation}
    where $E$ ranges over measurable sets. At $r=1$, the supremum is exactly $\|h\|_{(p,\infty)}$ by~\eqref{e.lorno}, and the factor on the right is $p/(p-1)$. This proves~\eqref{e.loreq}. Homogeneity and the triangle inequality follow from~\eqref{e.lorno} by applying the corresponding properties of $\int_E|\cdot|\dif\meas$ for each $E$. The lower bound in~\eqref{e.loreq} shows that $\|h\|_{(p,\infty)}=0$ implies $h=0$ $\meas$-a.e. on $\ambient$. If $(h_n)$ is a Cauchy sequence in $\|\cdot\|_{(p,\infty)}$, the lower bound in~\eqref{e.loreq} makes it also a Cauchy sequence in $L^{p,\infty}(\ambient,\meas)$. By \cite[Theorem~1.4.11]{Gra14}, there exists $h\in L^{p,\infty}(\ambient,\meas)$ such that
    \begin{equation}
      \|h_n-h\|_{(p,\infty)}
      \overset{\eqref{e.loreq}}{\leq}
      \frac{p}{p-1}\norm{h_n-h}_{L^{p,\infty}(\ambient,\meas)}
      \to0 .
    \end{equation}
    This proves completeness.

    \item[\ref{it.loren3}] Since $r_0<p<r_1$ and $1/p=(1-\theta)/r_0+\theta/r_1$, by the real interpolation theorem in \cite[Theorem~5.2.1-(2)]{BL76},
    \begin{equation}\label{e.lorinterproof}
      \bigl(L^{r_0}(\ambient,\meas),L^{r_1}(\ambient,\meas)\bigr)_{\theta,1}
      =
      L^{p,1}(\ambient,\meas),
    \end{equation}
    with equivalent norms. To obtain~\eqref{e.lorinter}, let $h\in L^{r_0}(\ambient,\meas)\cap L^{r_1}(\ambient,\meas)$. For $t>0$, the two decompositions $h=h+0$ and $h=0+h$ in the definition of the $K$-functional give
    \begin{equation}
      K(t,h):=\inf_{h=h_0+h_1}
        \bigl(\|h_0\|_{L^{r_0}}+t\|h_1\|_{L^{r_1}}\bigr) \leq\min\bigl\{\|h\|_{L^{r_0}},\,t\|h\|_{L^{r_1}}\bigr\},
    \end{equation}
    where $h_i\in L^{r_i}(\ambient,\meas)$ for $i=0,1$. By the definition of the interpolation norm,
    \begin{align}
      \norm{h}_{L^{p,1}}
      &\overset{\eqref{e.lorinterproof}}{\leq}
        C_{p,r_0,r_1}\int_0^\infty t^{-\theta}K(t,h)\frac{\dif t}{t} \\
      &\leq C_{p,r_0,r_1}\|h\|_{L^{r_0}}^{1-\theta}\|h\|_{L^{r_1}}^\theta
        \left(\int_0^1 s^{-\theta}\dif s+
              \int_1^\infty s^{-1-\theta}\dif s\right) \\
      &=
        C_{p,r_0,r_1}\left(\frac{1}{1-\theta}+\frac{1}{\theta}\right)
        \|h\|_{L^{r_0}}^{1-\theta}\|h\|_{L^{r_1}}^\theta .
    \end{align}

    \item[\ref{it.loren4}] Since $(\ambient,\meas)$ is $\sigma$-finite and nonatomic and $1<p<\infty$, the duality is given in \cite[Theorem~1.4.16-(v)]{Gra14}. For $g\in L^{p',\infty}(\ambient,\meas)$ and $t>0$, by \cite[Proposition~1.4.9]{Gra14}, we have $ t^{1/p'}g^*(t)\leq 
      \norm{g}_{L^{p',\infty}(\ambient,\meas)}$.
    Applying the Hardy--Littlewood rearrangement inequality \cite[Exercise~1.4.1(b)]{Gra14}, we obtain for $h\in L^{p,1}(\ambient,\meas)$,
    \begin{align}
      \int_{\ambient}|hg|\dif\meas
      \leq
        \int_0^\infty h^*(t)g^*(t)\dif t \leq \norm{g}_{L^{p',\infty}}\int_0^\infty h^*(t)t^{-1/p'}\dif t
       =\norm{g}_{L^{p',\infty}}\norm{h}_{L^{p,1}} .
    \end{align}
    This proves~\eqref{e.lorholder} and the upper bound in~\eqref{e.lordualnorm}. For the lower bound, fix $\alpha>0$ with $\lambda_g(\alpha)>0$ and set $E=\{|g|>\alpha\}$. By~\eqref{e.lorpq}, $\meas(E)\leq \alpha^{-p'}\norm{g}_{L^{p',\infty}}^{p'}<\infty$. Since $|\operatorname{sgn}(g)|=1$ on $E$, we have
    \begin{equation}
      \norm{\operatorname{sgn}(g)\one_E}_{L^{p,1}}
      \overset{\eqref{e.lorpq}}{=}
      p\int_0^1\meas(E)^{1/p}\dif t
      =p\,\meas(E)^{1/p}.
    \end{equation}
    Testing on $\operatorname{sgn}(g)\one_E/(p\,\meas(E)^{1/p})$ gives
    \begin{equation}
      \sup_{\norm{h}_{L^{p,1}}\leq1}\left|\int_{\ambient}hg\dif\meas\right|
      \geq \dfrac{\int_E|g|\dif\meas}{p\,\meas(E)^{1/p}}
      \overset{|g|>\alpha\text{ on }E}{\geq}
      \frac{\alpha}{p}\lambda_g(\alpha)^{1/p'} .
    \end{equation}
   Taking the supremum over $\alpha>0$ gives~\eqref{e.lordualnorm} with $C_p=p$.

    \item[\ref{it.loren5}] If $\meas(E)=0$, both sides of~\eqref{e.localL2weak} vanish. If $0<\meas(E)<\infty$, then $2<p$ allows us to take $r=2$ in~\eqref{e.lorKolmogorov}. For this particular $E$ we obtain
    \begin{equation}
      \meas(E)^{1/p-1/2}\|h\|_{L^2(E,\meas)}
      \overset{\eqref{e.lorKolmogorov}}{\leq}
      \left(\frac{p}{p-2}\right)^{\frac{1}{2}}\norm{h}_{L^{p,\infty}(\ambient,\meas)} .
    \end{equation}
    Multiplying by $\meas(E)^{1/2-1/p}$ proves~\eqref{e.localL2weak}. 
    \qedhere
  \end{enumerate}
\end{proof}

\begin{lemma}\label{l.realip}
  Let $(Y_{0},\mu_{0})$ and $(Y_{1},\mu_{1})$ be $\sigma$-finite measure spaces, and let $T:L^{2}(Y_{0},\mu_{0})\to L^{2}(Y_{1},\mu_{1})$ be real linear. Suppose that $A_{1},A_{2}\in(0,\infty)$ satisfy
  \begin{align}
      \|Tf\|_{L^{1,\infty}(Y_{1},\mu_{1})}&\leq A_{1}\|f\|_{L^{1}(Y_{0},\mu_{0})}
     \ \text{ for all } \ f\in L^{1}(Y_{0},\mu_{0})\cap L^{2}(Y_{0},\mu_{0}),\label{e.realhyp1}\\
      \|Tf\|_{L^{2}(Y_{1},\mu_{1})}&\leq A_{2}\|f\|_{L^{2}(Y_{0},\mu_{0})}
      \ \text{ for all } \   f\in L^{2}(Y_{0},\mu_{0}).\label{e.realhyp2}
  \end{align}
  For $q\in(1,2)$, one has
  \begin{equation}\label{e.realip}
    \|Tf\|_{L^{q}(Y_{1},\mu_{1})}
    \leq C_{q} A_{1}^{2/q-1}A_{2}^{2-2/q}\|f\|_{L^{q}(Y_{0},\mu_{0})},   \ \text{ for all } \ f\in L^{q}(Y_{0},\mu_{0})\cap L^{2}(Y_{0},\mu_{0}).
  \end{equation}
  Consequently $T$ has a unique bounded linear extension from $L^{q}(Y_{0},\mu_{0})$ to $L^{q}(Y_{1},\mu_{1})$.
\end{lemma}
\begin{proof}
  Fix $q\in(1,2)$. Fix $f\in L^{q}(Y_{0},\mu_{0})\cap L^{2}(Y_{0},\mu_{0})$. Let $a\in(0,\infty)$. Then
\begin{equation}
	 \|T(f\one_{\{|f|>a\}})\|_{L^{1,\infty}(Y_1,\mu_1)}
      \overset{\eqref{e.realhyp1}}{\leq}
        A_1\|f\one_{\{|f|>a\}}\|_{L^1(Y_0,\mu_0)}
\end{equation}
and
    \begin{align}
      \|T(f\one_{\{|f|\leq a\}})\|_{L^{2,\infty}(Y_1,\mu_1)}
      \leq
        \|T(f\one_{\{|f|\leq a\}})\|_{L^2(Y_1,\mu_1)}\overset{\eqref{e.realhyp2}}{\leq}
        A_2\|f\one_{\{|f|\leq a\}}\|_{L^2(Y_0,\mu_0)} .
    \end{align}
    By linearity, $Tf=T(f\one_{\{|f|>a\}})+T(f\one_{\{|f|\leq a\}})$ $\mu_1$-a.e. on $Y_{1}$. Solving $1/q=(1-\theta)+\theta/2$ gives $\theta=2-2/q$ and $1-\theta=2/q-1$. Therefore, by \cite[Theorem~1.3.2 and Remark 1.3.3]{Gra14},
    \begin{equation}
      \|Tf\|_{L^q(Y_1,\mu_1)}
      \leq
      C_q A_1^{2/q-1}A_2^{2-2/q}\|f\|_{L^q(Y_0,\mu_0)},
    \end{equation}
    proving~\eqref{e.realip}.  Since $L^{q}(Y_{0},\mu_{0})\cap L^{2}(Y_{0},\mu_{0})$ is dense in $L^q(Y_0,\mu_0)$, we may uniquely extend $T$ to be a bounded linear operator from $L^{q}(Y_{0},\mu_{0})$ to $L^{q}(Y_{1},\mu_{1})$.
\end{proof}

\begin{lemma}\label{l.projt}
  Let $(\ambient,\metric)$ be a real tree. Let $A\subset \ambient$ be nonempty, closed, and connected. Then,
  \begin{enumerate}[label=\textup{(\arabic*)}, align=right, leftmargin=*, topsep=5pt, parsep=0pt, itemsep=2pt]
    \item\label{it.proj1} $A$ is geodesically convex.
    \item\label{it.proj2} There is a unique map $\pi_{A}:\ambient\to A$ satisfying $ \allowbreak\metric (x,\pi_{A}(x))=\dist(x,A)$ for all $x\in\ambient$.
    \item\label{it.proj3} For every $x\in \ambient$ and every $a\in A$, we have
    \begin{equation}
      [x,a]=[x,\pi_{A}(x)]\cup[\pi_{A}(x),a],\ \text{and}\ [x,\pi_{A}(x)]\cap A=\{\pi_{A}(x)\}.
    \end{equation}
    \item\label{it.proj4} For all $x,y\in\ambient$, we have $\metric(\pi_{A}(x),\pi_{A}(y))\leq \metric(x,y)$.
    \item\label{it.proj5} If $C$ is a connected component of $\ambient\setminus A$, then there exists $g_{C}\in A$ such that
    \begin{equation}\label{e.gatec}
      \pi_{A}(x)=g_{C},\text{ and }[x,g_{C})\subset C\ \text{ for all } x\in C.
    \end{equation}
  \end{enumerate}
\end{lemma}
\begin{proof}
  We repeatedly use the standard \emph{tripod property} of a real tree \cite[Lemma~3.22]{Eva08}: for every $u,v,w\in\ambient$ there is a unique point $b=b(u,v,w)$ such that
  \begin{equation}\label{eq:tripod}
    [u,v]\cap[v,w]\cap[w,u]=\set{b},
  \end{equation}
  and
  \begin{equation}\label{eq:tripod-decomp}
    [u,v]=[u,b]\cup[b,v],\qquad [v,w]=[v,b]\cup[b,w],\qquad [w,u]=[w,b]\cup[b,u].
  \end{equation}
Fix $z\in\ambient$. For $u,v\in \ambient\setminus\set{z}$, define $u\sim_{z} v$ if and only if $z\notin[u,v]$. Reflexivity and symmetry are immediate. For transitivity, the tripod decomposition gives
    \begin{align}
      u\sim_{z}v,\ v\sim_{z}w &\ \Longrightarrow \  z\notin[u,v]\cup[v,w],\\
      [u,w]\subset[u,v]\cup[v,w] &\ \Longrightarrow \  z\notin[u,w]\quad\to \quad u\sim_{z}w.
    \end{align}
  If $ u \allowbreak \neq z $, then $B(u,\metric(u,z))\subset\{v\in\ambient\setminus\{z\}:u\sim_{z}v\}$, since $z\in[u,v]$ would give $ \allowbreak \metric ( u, \allowbreak v ) = \metric ( u , z ) + \metric ( z , v ) \geq \metric ( u , z ) $. Each equivalence class is therefore open; its relative complement is a union of the other open classes. A connected subset of $\ambient\setminus \allowbreak \{ z \} $ must lie in one class. In particular,
  \begin{equation}\label{eq:separation}
    z\in(u,v) \ \Longrightarrow \  u\text{ and }v\text{ lie in distinct connected components of } \ambient\setminus\set{z}.
  \end{equation}
  \begin{enumerate}[label=\textup{(\arabic*)}, align=right, leftmargin=*, topsep=5pt, parsep=0pt, itemsep=2pt]
    \item[\ref{it.proj1}] Fix $a_{0},a_{1}\in A$. We claim that $[a_{0},a_{1}]\subset A$. Suppose, to the contrary, that there exists $z\in(a_{0},a_{1})\setminus A$. By \eqref{eq:separation}, $a_{0}$ and $a_{1}$ lie in distinct connected components of $\ambient\setminus\set{z}$. On the other hand, $A\subset\ambient\setminus\set{z}$ is connected and contains both $a_{0}$ and $a_{1}$, which is impossible. Hence $[a_{0},a_{1}]\subset A$. Thus $A$ is geodesically convex.
    \item[\ref{it.proj2}] Fix $x\in\ambient$. If $x\in A$, then necessarily $\pi_{A}(x)=x$, so assume that $x\notin A$. Choose $a_{0}\in A$. Since $[x,a_{0}]$ is isometric to the compact interval $[0,\metric(x,a_{0})]$ and $A$ is closed, $K:=[x,a_{0}]\cap A$ is a nonempty compact subset of $[x,a_{0}]$. Therefore there exists $g\in K$ such that $\metric(x,g)=\min_{z\in K}\metric(x,z)$. In particular,
    \begin{equation}\label{eq:first-no-A}
      [x,g)\cap A=\emptyset.
    \end{equation}
    We show that $g$ is the nearest point of $A$ to $x$. Let $a\in A$, and let $b=b(x,g,a)$ be the branch point of the triple $(x,g,a)$. By \ref{it.proj1}, $[g,a]\subset A$. By \eqref{eq:tripod}, $b\in[x,g]\cap[g,a]$. Consequently $b\in[x,g]\cap A$. In view of \eqref{eq:first-no-A}, we must have $b=g$. By the tripod decomposition \eqref{eq:tripod-decomp}, this is exactly the statement that
    \begin{equation}\label{eq:gate-decomp-pre}
      [x,a]=[x,g]\cup[g,a], \qquad [x,g]\cap[g,a]=\set{g}.
    \end{equation}
    Hence
    \begin{equation}\label{eq:distance-gate}
      \metric(x,a)=\metric(x,g)+\metric(g,a)\geq \metric(x,g), \ \text{ for all } a\in A.
    \end{equation}
    Since $g\in A$, \eqref{eq:distance-gate} gives $\metric(x,g)=\dist(x,A)$. Thus a nearest point exists. To prove uniqueness, suppose that $h\in A$ also satisfies $\metric(x,h)=\dist(x,A)=\metric(x,g)$. Applying \eqref{eq:distance-gate} with $a=h$ gives $\metric(x,h)=\metric(x,g)+\metric(g,h)$. Therefore $\metric(g,h)=0$, so $g=h$. We may consequently define $\pi_{A}(x):=g$. This proves \ref{it.proj2}.
    \item[\ref{it.proj3}] Let $x\in\ambient$ and $a\in A$. The equalities \eqref{eq:gate-decomp-pre}, with $g=\pi_{A}(x)$, already give
    \begin{equation}\label{eq:gate-decomp}
      [x,a]=[x,\pi_{A}(x)]\cup[\pi_{A}(x),a].
    \end{equation}
    Moreover, if $z\in[x,\pi_{A}(x))\cap A$, then $\metric(x,z)<\metric(x,\pi_{A}(x))=\dist(x,A)$, which is impossible. Since $\pi_{A}(x)\in A$, we obtain $[x,\pi_{A}(x)]\cap A=\set{\pi_{A}(x)}$.
    \item[\ref{it.proj4}] Fix $x,y\in\ambient$. If $\pi_{A}(x)=\pi_{A}(y)$, then the required estimate is immediate. Assume $\pi_{A}(x)\neq \pi_{A}(y)$. By \ref{it.proj3},
    \begin{equation}\label{eq:first-intersections}
      [x,\pi_{A}(x)]\cap[\pi_{A}(x),\pi_{A}(y)]=\set{\pi_{A}(x)}, \ [y,\pi_{A}(y)]\cap[\pi_{A}(y),\pi_{A}(x)]=\set{\pi_{A}(y)}.
    \end{equation}
    We claim that
    \begin{equation}\label{eq:xg-yh-disjoint}
      [x,\pi_{A}(x)]\cap[y,\pi_{A}(y)]=\emptyset.
    \end{equation}
    Suppose, to the contrary, that there is a point $z\in[x,\pi_{A}(x)]\cap[y,\pi_{A}(y)]$. Let $b=b(\pi_{A}(x),\pi_{A}(y),z)$ be the branch point. Then
    \begin{equation}
      b\in[\pi_{A}(x),\pi_{A}(y)]\cap[\pi_{A}(x),z]\cap[\pi_{A}(y),z].
    \end{equation}
    Since $z\in[x,\pi_{A}(x)]$ and $z\in[y,\pi_{A}(y)]$, we know that
    \begin{equation}
      [\pi_{A}(x),z]\subset[x,\pi_{A}(x)], \ \text{ and }\ [\pi_{A}(y),z]\subset[y,\pi_{A}(y)].
    \end{equation}
    Together with the fact that $[\pi_{A}(x),\pi_{A}(y)]\subset A$ implied by \ref{it.proj1}, this yields
    \begin{equation}
      b\in A\cap[x,\pi_{A}(x)]=\set{\pi_{A}(x)}, \qquad b\in A\cap[y,\pi_{A}(y)]=\set{\pi_{A}(y)}.
    \end{equation}
    Thus $\pi_{A}(x)=\pi_{A}(y)$, a contradiction. Hence \eqref{eq:xg-yh-disjoint} holds. It follows from \eqref{eq:first-intersections} and \eqref{eq:xg-yh-disjoint} that $[x,\pi_{A}(x)]\cup[\pi_{A}(x),\pi_{A}(y)]\cup[\pi_{A}(y),y]$ is an arc joining $x$ to $y$. Since a real tree has a unique arc joining two points, $[x,y]=[x,\pi_{A}(x)]\cup[\pi_{A}(x),\pi_{A}(y)]\cup[\pi_{A}(y),y]$, with pairwise intersections only at the consecutive endpoints. Therefore
    \begin{equation}
      \metric(x,y) = \metric(x,\pi_{A}(x))+\metric(\pi_{A}(x),\pi_{A}(y))+\metric(\pi_{A}(y),y) \geq \metric(\pi_{A}(x),\pi_{A}(y)).
    \end{equation}
    This proves \ref{it.proj4}.
    \item[\ref{it.proj5}] Let $C$ be a connected component of $\ambient\setminus A$. For $x,y\in C$, if $z\in(x,y)\setminus C$, then \eqref{eq:separation} places the two points of the connected set $C\subset\ambient\setminus\{z\}$ in distinct components. Hence $[x,y]\subset C$ for every $x,y\in C$. Fix $x,y\in C$ and write $g:=\pi_{A}(x)$ and $h:=\pi_{A}(y)$. If $\pi_{A}(x)\neq \pi_{A}(y)$, the arc decomposition in \ref{it.proj4} and convexity in \ref{it.proj1} give
    \begin{equation}
      \emptyset\neq[g,h]=[\pi_{A}(x),\pi_{A}(y)]\subset A\cap[x,y]\subset A\cap C=\emptyset,
    \end{equation}
    a contradiction. Hence $\pi_{A}(x)=\pi_{A}(y)$ for every $x,y\in C$. Thus there exists a unique $g_{C}\in A$ such that
    \begin{equation}\label{eq:projection-constant-C}
      \pi_{A}(x)=g_{C}, \qquad x\in C.
    \end{equation}
    Finally, let $x\in C$. By \ref{it.proj3} and \eqref{eq:projection-constant-C}, $[x,g_{C})\cap A=\emptyset$, so $[x,g_{C})\subset\ambient\setminus A$. The set $[x,g_{C})$ is connected and contains $x$. Since $C$ is the connected component of $\ambient\setminus A$ containing $x$, maximality of connected components gives $[x,g_{C})\subset C$. Together with \eqref{eq:projection-constant-C}, this is exactly \eqref{e.gatec}.\qedhere
  \end{enumerate}
\end{proof}

\begin{lemma}\label{l.intpot}
  Let $a,c,q\in(0,\infty)$ and $\beta\in(0,a)$. For $R\in[0,\infty)$, set
  \begin{equation}\label{e.intpo1}
    \sI(R):= \frac{1}{\Gamma(\beta)} \int_{0}^{\infty} s^{\beta-1}(1+s)^{-a} \exp\left(-c\left(\frac{R}{1+s}\right)^{q}\right) \dif s.
  \end{equation}
  Then there exists $C=C(a,\beta,q,c)\in(1,\infty)$ such that, for every $R\in[0,\infty)$,
  \begin{equation}\label{e.intpot}
    C^{-1}(1\vee R)^{\beta-a} \leq \sI (R) \leq C(1\vee R)^{\beta-a}.
  \end{equation}
\end{lemma}
\begin{proof}
  Suppose that $R\in[0,1]$. Then $\exp(-c\left(\frac{R}{1+s}\right)^{q})\asymp1$, and therefore $\sI (R)\asymp1$. Indeed,
    \begin{align}
      e^{-c}\leq\exp\left(-c\left(\frac{R}{1+s}\right)^{q}\right)&\leq1,\qquad s\in(0,\infty),\\
      \frac{1}{\Gamma(\beta)}\int_{0}^{\infty}s^{\beta-1}(1+s)^{-a}\dif s &=\frac{\Gamma(a-\beta)}{\Gamma(a)}\in(0,\infty),\\
      e^{-c}\frac{\Gamma(a-\beta)}{\Gamma(a)}\leq\sI(R)&\leq\frac{\Gamma(a-\beta)}{\Gamma(a)}.
    \end{align}
  Since $(1\vee R)^{\beta-a}=1$ for $R\in[0,1]$, \eqref{e.intpot} follows in this range. Assume $R\in[1,\infty)$. Split
  \begin{equation}
    \sI(R)=\sI_{0}(R)+\sI_{\infty}(R),
  \end{equation}
  where
    \begin{align}
      \sI_{0}(R)&:=\frac{1}{\Gamma(\beta)}\int_{0}^{1}s^{\beta-1}(1+s)^{-a}\exp\left(-c\left(\frac{R}{1+s}\right)^{q}\right)\dif s,\\
      \sI_{\infty}(R)&:=\frac{1}{\Gamma(\beta)}\int_{1}^{\infty}s^{\beta-1}(1+s)^{-a}\exp\left(-c\left(\frac{R}{1+s}\right)^{q}\right)\dif s.
    \end{align}
  For $0<s<1$, we have $\frac{R}{1+s}\geq\frac{R}{2}$, and therefore
    \begin{align}
      \sI_{0}(R) &\leq \frac{e^{-c2^{-q}R^{q}}}{\Gamma(\beta)} \int_{0}^{1}s^{\beta-1}\dif s= \frac{1}{\beta\Gamma(\beta)}e^{-c2^{-q}R^{q}} \leq C R^{\beta-a}. \label{e.intpotsmall}
    \end{align}
  The last inequality follows from
  \begin{equation}
    R^{a-\beta}e^{-c2^{-q}R^{q}} \leq\sup_{v>0}v^{(a-\beta)/q}e^{-c2^{-q}v} =\left(\frac{a-\beta}{qc2^{-q}e}\right)^{(a-\beta)/q}<\infty.
  \end{equation}
  For $s\geq1$, we have $s\leq1+s\leq 2s$; hence $(1+s)^{-a}\leq s^{-a}$ and $\exp(-c(R/(1+s))^{q})\leq\exp(-c2^{-q}(R/s)^{q})$. Therefore, with the change of variables $s=Ru$ (and using $a>\beta$ and $q>0$),
    \begin{align}
      \sI_{\infty}(R) &\leq \frac{R^{\beta-a}}{\Gamma(\beta)} \int_{R^{-1}}^{\infty} u^{\beta-a-1}\exp(-c2^{-q}u^{-q})\dif u \leq C R^{\beta-a}.\label{e.intpotupper}
    \end{align}
  The dominating integral in \eqref{e.intpotupper} is finite: with $v=u^{-q}$,
    \begin{align}
      \int_{0}^{\infty}u^{\beta-a-1}e^{-c2^{-q}u^{-q}}\dif u &=\frac{1}{q}\int_{0}^{\infty}v^{(a-\beta)/q-1}e^{-c2^{-q}v}\dif v=\frac{\Gamma((a-\beta)/q)}{q(c2^{-q})^{(a-\beta)/q}}<\infty,
    \end{align}
  since $(a-\beta)/q>0$. Together with \eqref{e.intpotsmall}, this proves the upper bound. For the lower bound, restrict \eqref{e.intpo1} to $s\in[R,2R]$. Since $R\geq1$ and $R\leq s\leq2R$,
  \begin{equation}
    s^{\beta-1}\geq c_{\beta} R^{\beta-1}, \ (1+s)^{-a}\geq3^{-a}R^{-a}, \ \text{ and }\ \exp(-c\left(\frac{R}{1+s}\right)^{q})\geq e^{-c}.
  \end{equation}
  Consequently,
    \begin{align}
      \sI(R) &\geq \frac{c_{\beta} e^{-c}3^{-a}}{\Gamma(\beta)} R^{\beta-1-a}\int_{R}^{2R}\dif s\geq c R^{\beta-a}. \label{e.intpotlower}
    \end{align}
  This proves the lower bound.
\end{proof}

\begin{corollary}\label{c.integr}
  Let $\theta\in(0,1)$ and $c\in(0,\infty)$. There exists a constant $C\in(0,\infty)$, depending only on $\theta,c,\dimf,\dimw$, such that, for every $t\in(0,\infty)$ and $R\in[0,\infty)$,
  \begin{equation}\label{e.integ2}
    \int_{0}^{\infty}s^{-\theta}(t+s)^{-1-\frac{\dimf}{\dimw}}
    \exp\left(-c\left(\frac{R^{\dimw}}{t+s}\right)^{\frac{1}{\dimw-1}}\right)\dif s
    \leq C\left(t^{1/\dimw}+R\right)^{-\dimf-\theta\dimw}.
  \end{equation}
\end{corollary}
\begin{proof}
  Set
  \begin{equation}
    a=1+\frac{\dimf}{\dimw},\qquad
    \beta=1-\theta,\qquad q=\frac1{\dimw-1}.
  \end{equation}
  Since $\dimw=\dimf+1\geq2$ and $0<\theta<1$, these parameters satisfy $a>1$, $0<\beta<1<a$, and $q>0$. Thus Lemma~\ref{l.intpot} applies. For $t>0$, making the substitution $s=tu$ in the left-hand side of \eqref{e.integ2}, we obtain
  \begin{equation}
    \begin{aligned}
      &\phantom{\ \leq}\int_0^\infty s^{-\theta}(t+s)^{-1-\dimf/\dimw}
       \exp\left(-c\left(\frac{R^{\dimw}}{t+s}\right)^{1/(\dimw-1)}\right)\dif s\\
      &=t^{-\theta-\dimf/\dimw}
       \int_0^\infty u^{-\theta}(1+u)^{-1-\dimf/\dimw}
       \exp\left(-c\left(\frac{R^{\dimw}/t}{1+u}\right)^{1/(\dimw-1)}\right)\dif u\\
      &=\Gamma(1-\theta)t^{-\theta-\dimf/\dimw}\sI(R^{\dimw}/t)\leq C t^{-\theta-\dimf/\dimw}
        (1\vee(R^{\dimw}/t))^{-\theta-\dimf/\dimw}\\
      &=C(t\vee R^{\dimw})^{-\theta-\dimf/\dimw}
       =C(t^{1/\dimw}\vee R)^{-\dimf-\theta\dimw}.
    \end{aligned}
  \end{equation}
  Finally, $t^{1/\dimw}\vee R\geq(t^{1/\dimw}+R)/2$, and the exponent $\dimf+\theta\dimw$ is positive. Thus
  \begin{equation}
    (t^{1/\dimw}\vee R)^{-\dimf-\theta\dimw}
    \leq2^{\dimf+\theta\dimw}(t^{1/\dimw}+R)^{-\dimf-\theta\dimw}.
  \end{equation}
  Absorbing this factor and $\Gamma(1-\theta)$ into the constant proves \eqref{e.integ2}.
\end{proof}

\vspace{0.4cm}

\noindent \textbf{Acknowledgments.} F.B.~is partially supported by grant 10.46540/4283-00175B from the Independent Research Fund Denmark, by the Villum Investigator grant \emph{Stochastic Analysis in Aarhus}, and by the European Research Council (ERC) under the European Union's Horizon Europe research and innovation programme (RanGe project, Grant Agreement No.~101199772). A.C.~is partially funded by the Villum Investigator grant \emph{Stochastic Analysis in Aarhus}. L.C.~is partially supported by grant 10.46540/4283-00175B from the Independent Research Fund Denmark.

\vspace{0.4cm}

\noindent \textbf{AI disclosure statement.} The authors used generative AI tools (ChatGPT) to assist with literature searches, to check, revise, and simplify mathematical arguments in earlier drafts, to proofread the manuscript, to polish the wording, and to generate TikZ code for the figures. The authors take full responsibility for the content of the manuscript.

\noindent Fabrice~Baudoin
\vspace{-3pt}

\noindent Department of Mathematics, Aarhus University, 8000 Aarhus C, Denmark
\vspace{-3pt}

\noindent \texttt{fbaudoin@math.au.dk}
\vspace{5pt}

\noindent Aobo~Chen
\vspace{-3pt}

\noindent Department of Mathematics, Aarhus University, 8000 Aarhus C, Denmark
\vspace{-3pt}

\noindent \texttt{aobochen.math@hotmail.com} / \texttt{aobochen@math.au.dk}

\vspace{5pt}

\noindent Li~Chen
\vspace{-3pt}

\noindent Department of Mathematics, Aarhus University, 8000 Aarhus C, Denmark
\vspace{-3pt}

\noindent \texttt{lchen@math.au.dk}
\end{document}